\documentclass{memo-l}

\usepackage{amssymb,amsmath,array,amscd,amsthm,hhline,stmaryrd,latexsym}
\usepackage[mathscr]{euscript}
\DeclareMathAlphabet{\sfsl}{OT1}{cmss}{m}{sl}

\newtheorem{theorem}{Theorem}[section]
\newtheorem{proposition}[theorem]{Proposition}
\newtheorem{lemma}[theorem]{Lemma}
\newtheorem{corollary}[theorem]{Corollary}
\newtheorem{definition}[theorem]{Definition}

\newtheorem{remark}[theorem]{Remark}

\renewcommand{\theenumi}{\roman{enumi}}
\renewcommand{\labelenumi}{(\theenumi)}

\makeatletter \@addtoreset{equation}{chapter}
\renewcommand{\theequation}{\arabic{chapter}.\arabic{equation}}
\makeatother

\numberwithin{section}{chapter}

\def\smod#1{#1{\tt{-smod}}}

\def\ll{\left\llbracket}
\def\rr{\right\rrbracket}
\def\<{\langle}
\def\>{\rangle}
\renewcommand{\[}{\left[}

\def\bi{\text{\boldmath$i$}}

\def\cond{\;\begin{picture}(0,0)\put(3.5,3.6){\circle{9.6}}\put(0,0){{\rm\small if}}\end{picture}\;\;\;\,}
\def\cscript{\;\begin{picture}(0,0)\put(2.5,2.5){\circle{7.5}}\put(0,0){$\scriptstyle\mathrm{if}$}\end{picture}\,\;\;\,}

\def\cf{\mathop{\rm cf}\nolimits}

\def\I{\mathcal{I}}

\def\lm{\lambda}
\def\Z{\mathbb Z}

\def\R{\mathbf P}
\def\0{\mathfrak 0}
\def\1{\mathfrak 1}
\def\T{\mathcal T}
\def\U{\mathcal U}

\renewcommand{\and}{\mathbin\&}
\renewcommand\epsilon{\varepsilon}
\newcommand\de{\delta}
\newcommand\si{\sigma}
\newcommand\ga{\gamma}
\renewcommand\phi{\varphi}

\renewcommand{\O}{\mathcal O}

\def\ind{{\operatorname{ind}}}

\def\al{{\alpha}}
\def\be{{\beta}}

\def\suchthat{\mathbin{\rm |}}

\renewcommand{\emptyset}{\varnothing}
\def\id{\mathop{\rm id}\nolimits}
\def\odd#1{\overline{#1}}
\def\res{\mathop{\rm Res}\nolimits}
\def\Hom{\mathop{\rm Hom}\nolimits}

\def\ev{\mathop{\rm ev}\nolimits}
\def\Dist{\mathop{\rm Dist}\nolimits}
\def\Rem{\mathop{\rm Rem}\nolimits}
\def\Add{\mathop{\rm Add}\nolimits}
\def\height{\mathop{\rm ht}}

\def\Mtype{\mathtt{M}}
\def\Qtype{\mathtt{Q}}
\def\Se{\mathcal{Y}}
\def\Cl{\mathcal{C}}
\def\la{\lambda}
\def\Res{{\operatorname{cont}}}
\def\cont{{\operatorname{cont}}}
\def\eps{{\varepsilon}}
\def\phi{{\varphi}}

\def\u{\mathfrak{u}}

\def\cp{\mathop{\rm cp}\nolimits}
\def\pr{\mathop{\rm pr}\nolimits}

\def\tbinom#1#2{{\left(\textstyle\genfrac{}{}{0pt}{}{\!#1\!}{#2}\right)}}

\def\le{\leqslant}
\def\ge{\geqslant}
\renewcommand{\(}{\left(}
\renewcommand{\)}{\right)}
\def\={\equiv}

\newdimen\hoogte    \hoogte=12pt
\newdimen\breedte   \breedte=14pt
\newdimen\dikte     \dikte=0.5pt

\newenvironment{Young}{\begingroup
       \def\vr{\vrule height0.89\hoogte width\dikte depth 0.2\hoogte}
       \def\fbox##1{\vbox{\offinterlineskip
                    \hrule height\dikte
                    \hbox to \breedte{\vr\hfill##1\hfill\vr}
                    \hrule height\dikte}}
       \vbox\bgroup \offinterlineskip \tabskip=-\dikte \lineskip=-\dikte
            \halign\bgroup &\fbox{##\unskip}\unskip  \crcr }
       {\egroup\egroup\endgroup}
\def\diagram#1{\relax\ifmmode\vcenter{\,\begin{Young}#1\end{Young}\,}\else%
              $\vcenter{\,\begin{Young}#1\end{Young}\,}$\fi}

\makeindex

\begin{document}

\frontmatter

\title{Modular branching rules for projective representations of symmetric groups and lowering operators for the supergroup Q(n)}


\author{Alexander Kleshchev}
\address{Department of Mathematics, University of Oregon, Eugene, OR 97403, USA}
\email{klesh@uoregon.edu}
\thanks{This collaboration began when the second author visited University of Oregon in 2007. The second author is grateful to the Department of Mathematics, University of Oregon, for hospitality. The second author would like to acknowledge the financial support from the Russian Federation President Grant MK-2304.2007.1 which made the visit possible. The first author supported in part by NSF grant no. DMS-0654147. Both authors supported by the Isaac Newton Institute in Cambridge, U.K. Both authors are grateful to Jon Brundan for many useful comments and suggestions.}

\author{Vladimir Shchigolev}
\address{Department of Algebra, \  Faculty of Mathematics,\ \  Lomonosov Moscow State University, Leninskiye Gory, Moscow 119899, Russia}
\email{shchigolev\_vladimir@yahoo.com}
\thanks{}

\date{}

\subjclass[2000]{Primary }


\begin{abstract}
There are two approaches to projective representation theory of symmetric and alternating groups, which are powerful enough to work for modular representations. One is based on Sergeev duality, which connects projective representation theory of the symmetric group and representation theory of the algebraic supergroup $Q(n)$  via appropriate Schur (super)algebras and Schur functors. 
The second approach follows the work of Grojnowski for classical affine and cyclotomic Hecke algebras and connects   projective representation theory of symmetric groups in characteristic $p$ to the crystal graph of the basic module of the twisted affine Kac-Moody algebra of type $A_{p-1}^{(2)}$. 

The goal of this work is to connect the two approaches mentioned above and to obtain new branching results for projective representations of symmetric groups. This is achieved by developing the theory of lowering operators for the supergroup $Q(n)$ which is parallel to (although much more intricate than) the similar theory for $GL(n)$ developed by the first author. 
The theory of lowering operators for $GL(n)$ is a non-trivial generalization of Carter's work in characteristic zero, 
and it has received a lot of attention. 
So this part of our work might be of independent interest.

One of the applications of lowering operators is to tensor products of irreducible $Q(n)$-modules with natural and dual natural modules, which leads to important special translation functors. We describe the socles and primitive vectors in such tensor products.

\end{abstract}

\maketitle


\setcounter{page}{4}

\tableofcontents

\chapter*{Introduction}

\section*{Set up}
There are two approaches to {\em projective}\, representation theory of symmetric and alternating groups, which are powerful enough to work for modular representations. One is based on Sergeev duality, which connects projective representation theory of the symmetric group and representation theory of the algebraic supergroup $Q(n)$ via appropriate Schur (super)algebras and Schur functors. This approach has been developed in \cite{Brundan_Kleshchev_Projective_representations}.

The second approach follows the work of Grojnowski for the classical affine and cyclotomic Hecke algebras and connects   projective representation theory of symmetric groups in characteristic $p$ to the crystal graph of the basic module of the twisted affine Kac-Moody algebra of type $A_{p-1}^{(2)}$. This approach has been developed in \cite{Brundan_Kleshchev_Hecke-Clifford_superalgebras} and~\cite{Kbook}.

The goal of this work is to connect the two approaches described above and to obtain new branching results for projective representations of symmetric groups. This is achieved by developing the theory of lowering operators for the supergroup $Q(n)$ which is parallel to (although much more intricate than) the similar theory for $GL(n)$ developed in \cite{KBrII}.

The theory of lowering operators for $GL(n)$ is a non-trivial  generalization of the Carter's work \cite{Carter} in characteristic zero, and it has received quite a lot of attention recently, see for example  \cite{KDec,KBrIV,BrBr,Brundan_operators,BKtr,Kujawa,ShchItLow,ShchGenLow,Shchigolev_Rectangular_low_level_case,ShchWeyl,RT}.
So this part of our work might be of independent interest, since it should be a useful tool for studying representation theory of $Q(n)$ and for obtaining further results on projective representations of symmetric groups.

One of the applications of lowering operators is to tensor products of irreducible $Q(n)$-modules with natural and dual natural modules, which leads to important special translation functors. In this paper we describe the socles and primitive vectors in such tensor products.

\section*{Projective representations and Sergeev algebra}
We now describe the contents of this work more carefully.
Let $\mathbb F$ be an algebraically closed field of characteristic $p\neq 2$. Let $S_n$ be the symmetric group on $n$ letters, and $A_n$ be the alternating group on $n$ letters.

Studying {\em projective}\, representations of $S_n$ over $\mathbb F$ is equivalent to studying {\em linear}\,  representations of the twisted group algebra $\T_n$ of $S_n$,  see for example \cite[Section 13.1]{Kbook}. Explicitly, $\T_n$ is the $\mathbb F$-algebra generated by the elements $t_1,\dots,t_{n-1}$ subject only to the relations
\begin{eqnarray*}
&t_i^2  =  1\quad &(1\leq i<n),\label{twrel1}
\\
&t_it_{i+1}t_i=t_{i+1}t_it_{i+1}&(1\leq i\leq n-2),\label{twrel2}
\\
&t_it_j=-t_jt_i &(1\leq i,j<n,\ |i-j|>1).\label{twrel3}
\end{eqnarray*}
Inside the algebra $\T_n$ we have the subalgebra
$$
\U_n:=\operatorname{span}\{t_g\mid g\in A_n\}.
$$
This is a twisted group algebra of the alternating group $A_n$, and its representation theory is equivalent to the projective representation theory of $A_n$ over $\mathbb F$.

We consider $\T_n$ as a superalgebra with respect to the following grading:
$$
(\T_n)_{\0}=\U_n,\quad (\T_n)_{\1}=\operatorname{span}\{t_g\mid g\in S_n\setminus A_n\}.
$$
(when $\Z/2\Z$ is used for grading, its zero and identity elements are be denoted $\0$ and $\1$).
To understand the usual irreducible modules over $\T_n$ and $\U_n$,
it suffices to understand the irreducible supermodules over $\T_n$. This is explained precisely in \cite[Proposition 12.2.11]{Kbook}. So from now on, we are mainly interested in supermodules over the superalgebra $\T_n$.
Note that understanding irreducible supermodules over $\T_n$, among other things, now entails understanding their {\em type}, which can be $\Mtype$ or $\Qtype$, see \cite[Section 12.2]{Kbook}.

Let $\Cl_n$ be the Clifford (super)algebra given by odd generators $c_1,\dots,c_n$ subject only to the relations
\begin{align*}
c_i^2=1\qquad&(1\leq i\leq n),
\\
c_ic_j=-c_jc_i\qquad&(1\leq i\neq j\leq n).
\end{align*}
The superalgebra $\T_n$ is `Morita superequivalent' to the {\em Sergeev superalgebra}
$$
\Se_n:=\T_n\otimes \Cl_n,
$$
where the tensor product is the tensor product of superalgebras. This is explained in detail, for example, in \cite[Section 13.2]{Kbook}. In particular, the information about  irreducible supermodules, including their type, is easily transferable between $\T_n$ and $\Se_n$. It turns out that $\Se_n$ is a little easier to work with than $\T_n$, so from now on let us concentrate on $\Se_n$.

\section*{Crystal graph approach}
This approach has been realized in \cite{Brundan_Kleshchev_Hecke-Clifford_superalgebras}, see also \cite{Kbook}, following the work of Grojnowski \cite{Gr} for the usual symmetric groups (and cyclotomic Hecke algebras). The original idea here is due to Leclerc and Thibon \cite{LT}.

Set
$$
\ell:=\left\{
\begin{array}{ll}
\infty&\hbox{if $p=0$,}\\
(p-1)/2&\hbox{if $p>0$;}
\end{array}\right.
$$
and
$$
I:=\left\{
\begin{array}{ll}
\Z_{\geq 0}&\hbox{if $p=0$,}\\
\{0,1,\dots,\ell\}&\hbox{if $p>0$.}
\end{array}\right.
$$
The block components of the restriction $\operatorname{res}^{\Se_n}_{\Se_{n-1}}$ 
are naturally labeled by the elements of $I$. For any irreducible $\Se_n$-supermodule $L$, this gives us a natural decomposition into block components \cite[Section 19.1]{Kbook}:
$$
\operatorname{res}^{\Se_n}_{\Se_{n-1}} L=\bigoplus_{i\in I}\operatorname{res}_i L.
$$

If $\operatorname{res}_i L\neq 0$, then up to a (not necessarily even) isomorphism, there is only one irreducible $\Se_{n-1}$-supermodule, denoted $\tilde e_i L$, such that
$$\Hom_{\Se_{n-1}}(\tilde e_i L,\operatorname{res}_i L)\neq 0.$$
If $\operatorname{res}_i L=0$, we set $\tilde e_i L=0$.
Now, let $B_n$ be the set of the isomorphism classes of irreducible $\Se_n$-supermodules, and $B:=\bigsqcup_{n\geq 0}B_n$. We make  $B$ into an $I$-colored graph as follows: $[L_1]\stackrel{i}{\rightarrow}[L_2]$ if and only if $L_1\cong \tilde e_i L_2$.

One of the main result of \cite{Brundan_Kleshchev_Hecke-Clifford_superalgebras} is that the colored graph $B$ is the crystal graph $B(\Lambda_0)$ of the basic representation of the twisted affine Kac-Moody Lie algebra of type $A_{p-1}^{(2)}$ (interpreted as type $B_{\infty}$ if $p=0$).

Kang \cite{Kang} has given a convenient combinatorial description
of the crystal graph $B(\Lambda_0)$ in terms of Young diagrams,
which we now explain. The following notions of $p$-strict and $p$-restricted partitions were first suggested in \cite{LT}. 
These notions arise naturally in \cite{Brundan_Kleshchev_Projective_representations} and \cite{Kleshchev_Brundan_Modular_Representations_of_the_supergroup_Q(n)_I} from completely different Lie theoretic considerations.

For any $n \geq 0$, a partition $\la=(\la_1,\la_2,\dots)$ of $n$
is {\em $p$-strict} if $\la_r=\la_{r+1}$ for some $r$ implies $p\mid\la_r$.
A $p$-strict partition $\la$
is {\em $p$-restricted} if in addition
$$
\left\{
\begin{array}{ll}
\la_r-\la_{r+1}< p &\hbox{if $p|\la_r$},\\
\la_r-\la_{r+1}\leq p &\hbox{if $p \nmid \la_r$}
\end{array}
\right.
$$
for each $r \geq 1$. If $p = 0$, we interpret both $p$-strict and $p$-restricted partitions as {\em strict partitions}, i.e. partitions all of whose non-zero parts are distinct.
Let ${\mathcal P}_p(n)$ denote the set of all $p$-strict partitions of $n$,
and ${\mathcal{RP}}_p(n)\subseteq {\mathcal P}_p(n)$
denote the set of all $p$-restricted partitions
of $n$. 

Let $\la$ be a $p$-strict partition.
As usual, we identify $\la$ with its {\em Young diagram}, which consists of certain nodes (or boxes). A node $(r,s)$ is the node in row $r$ and column $s$. We label the nodes
with ($p$-){\em contents} which are elements of the set
$I$.
The labelling depends only on the column and follows the repeating pattern
$$
0,1,\dots,\ell-1,\ell,\ell-1,\dots,1,0,
$$
starting fom the first column and going to the right.
For example, let $p=5$, so $\ell=2$.
The partition $\la=(16, 11,10,10,9,5,1)$ belongs to $\mathcal{RP}_5$,
and the contents of its nodes are as follows:
\vspace{1 mm}
$$
\newdimen\hoogte    \hoogte=12pt
\newdimen\breedte   \breedte=14pt
\newdimen\dikte     \dikte=0.5pt
\diagram{
$0$ & $1$ & $2$ & $1$ & $0$ & $0$& $1$ & $2$ & $1$ & $0$ & $0$ & $1$ & $2$ & $1$ & $0$ & $0$\cr
$0$ & $1$ & $2$ & $1$ & $0$ & $0$ & $1$ & $2$ & $1$ & $0$ & $0$ \cr
$0$ & $1$ & $2$ & $1$ & $0$ & $0$ & $1$ & $2$ & $1$ & $0$  \cr
$0$ & $1$ & $2$ & $1$ & $0$ & $0$ & $1$ & $2$ & $1$ & $0$  \cr
$0$ & $1$ & $2$ & $1$ & $0$ & $0$ & $1$ & $2$ & $1$ \cr
$0$ & $1$ & $2$ & $1$ & $0$\cr
$0$ \cr
}
\vspace{1 mm}
$$
\vspace{1 mm}
The content of the node $A$ is denoted $\Res_p A$.

Let $\la$ be a $p$-strict partition and $i \in I$ be some fixed content.
A node $A = (r,s)\in \la$ is {\em $i$-removable} (for $\la$) if one of the following
holds:
\begin{enumerate}
\item[(R1)] $\Res_p A = i$ and
$\la_A:=\la-\{A\}$ is again a $p$-strict partition;
\item[(R2)] the node $B = (r,s+1)$ immediately to the right of $A$
belongs to $\la$,
$\Res_p A = \Res_p B = i$,
and both $\la_B = \la - \{B\}$ and
$\la_{A,B} := \la - \{A,B\}$ are $p$-strict partitions.
\end{enumerate}
Similarly, a node $B = (r,s)\notin\la$ is
{\em $i$-addable} (for $\la$) if one of the following holds:
\begin{enumerate}
\item[(A1)] $\Res_p B = i$ and
$\la^B:=\la\cup\{B\}$ is again an $p$-strict partition;
\item[(A2)]
the node $A = (r,s-1)$
immediately to the left of $B$ does not belong to $\la$,
$\Res_p A = \Res_p B = i$, and both
$\la^A = \la \cup \{A\}$ and
$\la^{A, B} := \la \cup\{A,B\}$ are $p$-strict partitions.
\end{enumerate}
Of course, (R2) and (A2) above are only possible if $i = 0$.

Now label all $i$-addable
nodes of the diagram $\la$ by $+$ and all $i$-removable nodes by $-$.
Then the {\em $i$-signature} of
$\la$ is the sequence of pluses and minuses obtained by going along the rim of the Young diagram from top right to bottom left and reading off
all the signs.
The {\em reduced $i$-signature} of $\la$ is obtained
from the $i$-signature
by successively erasing all neighbouring
pairs of the form $-+$.

Note the reduced $i$-signature always looks like a sequence
of $+$'s followed by $-$'s.
Nodes corresponding to  $-$'s in the reduced $i$-signature are
called {\em $i$-normal} or {\em normal nodes for $\la$ of content $i$}. 
The top $i$-normal node 
(corresponding to the leftmost $-$ in the reduced $i$-signature) is called {\em $i$-good}.
Nodes corresponding to  $+$'s in the reduced $i$-signature are
called {\em $i$-conormal} or {\em conormal nodes for $\la$ of content $i$}. 
The bottom $i$-conormal node 
(corresponding to the rightmost $+$ in the reduced $i$-signature) is called {\em $i$-cogood}.
Continuing with the example above, the $0$-addable and $0$-removable nodes are as labelled in the diagram:
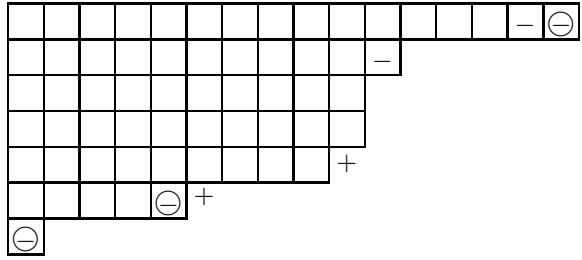
\vspace{2mm}
$$
\begin{picture}(320,95)
\put(63,40)
{$\newdimen\hoogte    \hoogte=12pt
\newdimen\breedte   \breedte=14pt
\newdimen\dikte     \dikte=0.5pt
\diagram{
 &  &  &  &  & &  &  &  &  &  &  &  &  & $-$ & $-$ \cr
 &  &  &  &  &  &  &  &  &  & $-$ \cr
 &  &  &  &  &  &  &  &  &   \cr
 &  &  &  &  &  &  &  &  &  \cr
 &  &  &  &  &  &  &  & \cr
 &  &  &  & $-$\cr
$-$ \cr
}
$}
\put(123.4, 14.4){\circle{9}}
\put(69.4, .8){\circle{9}}
\put(271.9, 82.3){\circle{9}}
\put(137,13.5){\makebox(0,0)[b]{$+$}}
\put(191,26){\makebox(0,0)[b]{$+$}}
\end{picture}
\vspace{5 mm}
$$
The $0$-signature of $\la$ is
$-,-,-,+,+,-,-$,
and the reduced $0$-signature is
$-,-,-$.
The nodes corresponding to the $-$'s in the reduced $0$-signature
have been circled in the diagram. The top of them, which is the node $(1,16)$, is $0$-good. There are no conormal or cogood nodes of content $0$ for $\la$.

Set
$$
\tilde e_i \la =
\left\{
\begin{array}{ll}
\la_A&\hbox{if $A$ is the $i$-good node,}\\
0&\hbox{if $\la$ has no $i$-good nodes},
\end{array}
\right.
$$
The definitions imply that $\tilde e_i(\la)$ is $p$-restricted (or zero)
if $\la$ itself is $p$-restricted. We make  $\mathcal{RP}_p:=\bigsqcup_{n\geq 0}\mathcal{RP}_p(n)$ into an $I$-colored graph as follows: $\la\stackrel{i}{\rightarrow}\mu$ if and only if $\la= \tilde e_i \mu$.
Kang \cite[7.1]{Kang} proves that this graph is isomorphic to the crystal graph $B(\Lambda_0)$.

Now we can {\em canonically} label the irreducible $\Se_n$-supermodules by the restricted  $p$-strict partitions of $n$. This is done as follows. Let $\la\in \mathcal{RP}_p(n)$. The supermodule $G(\la)$ is defined inductively through its branching properties. Indeed, $G(\emptyset)$ will have to be the trivial supermodule since $\Se_0$ is by convention the trivial algebra $\mathbb  F$. Now assume that $n>0$ and we have already defined $G(\mu)$ for all $\mu \in \mathcal{RP}_p(n-1)$. Let $\la\in \mathcal{RP}_p(n)$. Then for some $i$, the partition $\la$  has a good $i$-node. Let $\mu=\tilde e_i\la$. Then $G(\la)$ is defined as the only irreducible $\Se_n$-supermodule such that $\tilde e_i G(\la)=G(\mu)$. We refer the reader to \cite[Section 22.2]{Kbook} for further details.

It is proved in \cite{Brundan_Kleshchev_Hecke-Clifford_superalgebras}, see also \cite[Theorem 22.2.1]{Kbook}, that
\begin{equation}\label{EGs}
\{G(\la)\mid\la\in  \mathcal{RP}_p(n)\}
\end{equation}
is a complete and irredundant set of irreducible $\Se_n$-supermodules up to isomorphism. Finally, for $\la\in {\mathcal{RP}}_p(n)$ we have that $G(\la)$ is of type
$\Mtype$ if $h_{p'}(\la)$ is even and of type $\Qtype$ if $h_{p'}(\la)$ is odd, where the {\em $p'$-height} $h_{p'}(\la)$ of $\la$ is defined by:
$$
h_{p'}(\la):=\big|\{i\mid1\leq i\leq n\ \text{and}\ p\nmid\la_i\}\big| \qquad(\la\in {\mathcal{RP}}_p(n)).
$$

\section*{Schur functor approach}
We now review in more detail the work \cite{Brundan_Kleshchev_Projective_representations} and \cite{Kleshchev_Brundan_Modular_Representations_of_the_supergroup_Q(n)_I}. Let us denote by $G$ the algebraic supergroup $Q(n)$ (it will be discussed in more detail in Section~\ref{supergroupQn}).
The supergroup $G$ has a ``maximal torus'' $H$ and a ``Borel subgroup'' $B$ with the  natural supergroup epimorphism $B\to H$, so each $H$-supermodule inflates to a $B$-supermodule.

The irreducible $H$-supermodules are parametrized by the set of {\em weights} $X(n):=\Z^n$. If $\la=(\la_1,\dots,\la_n)\in X(n)$, we denote the corresponding irreducible $H$-supermodule by $\mathfrak u(\la)$, see \cite[Lemma 6.4]{Kleshchev_Brundan_Modular_Representations_of_the_supergroup_Q(n)_I}. We point out that $H$ is not commutative, and so $\mathfrak u(\la)$ does not need to be $1$-dimensional. In fact,
$$
\dim \mathfrak u(\la)=2^{\lfloor\frac{ h_{p'}(\la)+1}{2}\rfloor},
$$
where
$$
h_{p'}(\la):=\big|\{i\mid1\leq i\leq n\ \text{and}\ p\nmid\la_i\}\big| \qquad(\la\in X(n)).
$$
The supergroup $G$ has a ``Chevalley anti-involution'' $\tau$ which defines the duality $V\mapsto V^\tau$ on finite dimensional $G$-supermodules.
For $\la\in X(n)$,
define the {\em Weyl supermodule}
$$
V(\la):=\big(\ind_B^G\, \mathfrak u(\la)\big)^\tau.
$$

A weight $\la=(\la_1,\dots,\la_n)\in X(n)$ is {\em dominant} if $\la_1\geq\dots\geq \la_n$.
A dominant weight $\la=(\la_1,\dots,\la_n)\in X(n)$ is {\em $p$-strict} if $\la_r=\la_{r+1}$ for some $1\leq r<n$ implies $p\mid \la_r$. The set of all $p$-strict dominant weights in $X(n)$ is denoted by $X^+_p(n)$.
It is proved in \cite{Kleshchev_Brundan_Modular_Representations_of_the_supergroup_Q(n)_I} that $V(\la)$ is non-zero if and only if $\la\in X^+_p(n)$, in which case $V(\la)$ is finite dimensional and has a unique highest weight $\la$. Moreover, $V(\la)$ is the ``universal highest weight supermodule'' of weight $\la$, see \cite[Lemma 6.13]{Kleshchev_Brundan_Modular_Representations_of_the_supergroup_Q(n)_I}; $V(\la)$  has a unique irreducible quotient $L(\la)$; and
$$
\{L(\la)\mid \la\in X^+_p(n)\}
$$
is a complete irredundant set of irreducible $G$-supermodules up to isomorphism. Finally, $L(\la)$ is of type $\Mtype$ if $h_{p'}(\la)$ is even and of type $\Qtype$ if $h_{p'}(\la)$ is odd. Observe the remarkable fact that, unlike in the ``even case'' of $GL(n)$ and other ``purely even'' reductive algebraic groups, the parameterization of the irreducible supermodules for $Q(n)$ depends on $p$.

{\em Sergeev duality} allows us to push some of the results for the algebraic supergroup $Q(n)$ to results about $\Se_n$, just like the classical Schur-Weyl duality connects the general linear group $GL(n)$ and the symmetric group $S_n$. In order to apply Sergeev duality we need to restrict our attention to {\em polynomial representations} of $Q(n)$ of degree $n$. These are defined and studied in \cite[Section 10]{Kleshchev_Brundan_Modular_Representations_of_the_supergroup_Q(n)_I} and \cite{Brundan_Kleshchev_Projective_representations}. The irreducible supermodule $L(\la)$ and the Weyl supermodule $V(\la)$ are polynomial of degree $n$ if and only if $\la$ is a ($p$-strict) partition of $n$.

The category of finite dimensional polynomial representations of $Q(n)$ of degree $n$ is equivalent to the category of finite dimensional supermodules over certain {\em Schur superalgebra}\, $S(n,n)$ of type $Q(n)$.  There is an idempotent $e\in S(n,n)$ such that
$$eS(n,n)e\cong\Se_n,$$
see \cite[Theorem 6.2]{Brundan_Kleshchev_Projective_representations}. This defines the {\em Schur functor} from $S(n,n)$-supermodules to $\Se_n$-supermodules:
$$
\mathcal F: M\mapsto eM.
$$

Any partition in ${\mathcal{P}}_p(n)$ can be considered as a weight in $X_p^+(n)$. Then
the main results of \cite{Brundan_Kleshchev_Projective_representations} can be stated as follows. Define the {\em Specht supermodules} $S(\la)$ and the supermodules $D(\la)$ over $\Se_n$ as:
$$
S(\la):=\mathcal F(V(\la)),\quad D(\la):=\mathcal F(L(\la))\quad\qquad(\la\in {\mathcal{P}}_p(n)).
$$
Then $D(\la)\neq 0$ if and only if $\la\in {\mathcal{RP}}_p(n)$, in which case $D(\la)$ is the simple head of $S(\la)$. Moreover,
\begin{equation}\label{EDs}
\{D(\la)\mid\la\in {\mathcal{RP}}_p(n)\}
\end{equation}
is a complete and irredundant set of irreducible $\Se_n$-supermodules up to isomorphism. Finally, for $\la\in {\mathcal{RP}}_p(n)$ we have that $D(\la)$ is of type
$\Mtype$ if $h_{p'}(\la)$ is even and of type $\Qtype$ if $h_{p'}(\la)$ is odd.



\section*{Modular branching rules}

Our first main result is the following {\em modular branching rule} for irreducible supermodules over Sergeev supealgebras:

\vspace{2.5mm}
\noindent
{\bf Theorem A.}
{\em
Let $\la\in\mathcal{RP}_p(n)$ and $\mu\in \mathcal{RP}_p(n-1)$. Then
\begin{enumerate}
\item[{\rm (i)}] $\mu$ is obtained from $\la$ by removing a good node if and only if
$$
\Hom_{\Se_{n-1}}(D(\mu),\operatorname{res}^{\Se_n}_{\Se_{n-1}}D(\la))\neq 0.
$$
\item[{\rm (ii)}] $\mu$ is obtained from $\la$ by removing a normal node if and only if
$$
\Hom_{\Se_{n-1}}(S(\mu),\operatorname{res}^{\Se_n}_{\Se_{n-1}}D(\la))\neq 0.
$$
In particular, if $\mu$ is obtained from $\la$ by removing a normal node then $D(\mu)$ is a composition factor of the restriction $\operatorname{res}^{\Se_n}_{\Se_{n-1}}D(\la)$.
\item[{\rm (iii)}] $\la$ is obtained from $\mu$ by adding a cogood node if and only if
$$
\Hom_{\Se_{n}}(D(\la),\operatorname{ind}^{\Se_n}_{\Se_{n-1}}D(\mu))\neq 0.
$$
\item[{\rm (iv)}] $\la$ is obtained from $\mu$ by adding a conormal node if and only if
$$
\Hom_{\Se_{n}}(S(\la),\operatorname{ind}^{\Se_n}_{\Se_{n-1}}D(\mu))\neq 0.
$$
In particular, if $\la$ is obtained from $\mu$ by adding a conormal node then $D(\la)$ is a composition factor of the induced module $\operatorname{ind}^{\Se_n}_{\Se_{n-1}}D(\mu)$.
\end{enumerate}
}

We believe that Theorem A(ii) provides an especially important information, as many  applications of branching rules for {\em linear}\, representations involve normal nodes, see for example \cite{KDec,BarK,BKZ,BKtr,BKIrrRes}.

The reader is referred to Section~\ref{SFinal} for translation from  $\Se_n$-supermodules to $\T_n$-supermodules.

\section*{Connecting the two approaches}

Looking at (\ref{EGs}) and (\ref{EDs}), we would of course like to know that $D(\la)\cong G(\la)$ for all $\la$. In other words, do the Schur functor approach and the crystal graph approach, described above, lead to the same classification of irreducible $\Se_n$-supermodules, and hence to the same classification of projective representations of the symmetric groups in arbitrary characteristic? Unfortunately, this is far from clear.

This situation is quite annoying, because different sets of results are available for the modules $D(\la)$ and for the modules $G(\la)$. For example, the works \cite{Brundan_Kleshchev_Projective_representations, BKReg} obtain results on $D(\la)$'s, while the works \cite{Brundan_Kleshchev_Hecke-Clifford_superalgebras, Kbook, KT,BKDurham,BKCartan,Ph,ArS} obtain results on $G(\la)$'s. (Do not be mislead by the fact that the irreducible modules are usually denoted $D(\la)$ in all cases!)
The second main result of this work removes the ``annoyance'' described above:

\vspace{2.5mm}
\noindent
{\bf Theorem B.}
{\em
We have $D(\la)\cong G(\la)$ for all $\la\in\mathcal{RP}_p$.
}

\vspace{2.5mm}
We point out that the similar issue for the {\em linear}\, representations of symmetric groups has been resolved. We know of three different arguments, all of which rely on heavy machinery. 

The {\em first}\, argument, outlined in
\cite[Remark 11.2.2]{Kbook}, appeals to the work \cite{KBrI,KBrII,KBrIII}, which shows that the modules $D(\la)$, defined using Specht modules or Schur functors from $GL(n)$, satisfy the socle branching rules described by the combinatorics of good nodes, just like $G(\la)$'s do by definition. This is sufficient to identify $D(\la)$ with $G(\la)$.

The {\em second}\, argument is due to Ariki \cite{ABr}. It relies on the powerful Ariki's categorification theorem \cite{ACat} and subtle reduction modulo $p$ arguments.

The {\em third} approach \cite{BKllt} also relies on the Ariki's categorification theorem as well as the idea of an extremal weight from \cite{BKDurham}.

As the projective analogue of Ariki's categorification theorem is not available, the only feasible approach to the proof of Theorem B is the first one---through the branching rules for the modules $D(\la)$, that is through Theorem A. This is what we implement in this paper.


\section*[Tensor products over $Q(n)$]{Some tensor products over $Q(n)$}

The category of finite dimensional $Q(n)$-supermodules can be considered as a category of integrable finite dimensional supermodules over the corresponding distribution algebra, which is naturally isomorphic to the hyperalgebra $U(n)$, see Section~\ref{supergroupQn}. It is convenient to work in a larger {\em integral category $\O_p$} of $U(n)$-supermodules. This is described in detail in Section~\ref{SHWT}. Here we just mention that the irreducible supermodules in the category $\O_p$ are labelled by all integral weights:
$$
\{L(\la)\mid \la\in X(n)\},
$$
and there are Verma modules
$$
\{M(\la)\mid \la\in X(n)\}.
$$
The notions of normal, good, conormal, and cogood nodes for partitions are generalized to those of normal, good, conormal, and cogood indices for arbitrary weights $\la\in X(n)$, see Chapter~\ref{ChComb}. Finally, let $\eps_i$ denote the weight
$$
\eps_i:=(0,\dots,0,1,0,\dots,0)\in X(n)
$$
with $1$ in the $i$th position. Now our main result on tensor products is as follows:

\vspace{2.5mm}
\noindent
{\bf Theorem C.}
{\em
Let $\la,\mu\in X(n)$. Then:
\begin{enumerate}
\item[{\rm (i)}] $\Hom_{U(n)}(M(\mu),L(\la)\otimes V^*)\neq 0$ if and only if $\mu=\la-\eps_i$ for some $\la$-normal index $i$.
\item[{\rm (ii)}] $\Hom_{U(n)}(L(\mu),L(\la)\otimes V^*)\neq 0$ if and only if $\mu=\la-\eps_i$ for some $\la$-good index $i$.
\item[{\rm (iii)}] $\Hom_{U(n)}(M(\mu),L(\la)\otimes V)\neq 0$ if and only if $\mu=\la+\eps_i$ for some $\la$-conormal index $i$.
\item[{\rm (iv)}] $\Hom_{U(n)}(L(\mu),L(\la)\otimes V)\neq 0$ if and only if $\mu=\la+\eps_i$ for some $\la$-cogood index $i$.
\end{enumerate}
}

\vspace{2.5mm}

\mainmatter

\chapter{Preliminaries}\label{ChPrel}

\section{General Notation}\label{notation}
Throughout the paper $\mathbb F$ is an algebraically closed field of characteristic $p\ne2$.
We identify $\Z/p\,\Z$ with the simple subfield of $\mathbb F$. Residues modulo $p$ will be denoted by bold, for example, $\mathbf0=0+p\,\Z$, $\mathbf2=2+p\,\Z$. For any $j\in\Z$, we set
\begin{equation}\label{EDefRes}
\res_p j:=j(j-1)+p\Z\in\Z/p\Z\subset\mathbb F.
\end{equation}

For any condition $\pi$, let $\cond_\pi$ denote $1$ if $\pi$ is satisfied and $0$ otherwise.

When $\Z/2\Z$ is used for grading, its zero and identity elements are be denoted $\0$ and $\1$, respectively. We adopt the convention $(-1)^\0=1$ and $(-1)^\1=1$.

For $s,t\in\Z\cup\{\pm\infty\}$, we use the following notation for ``segments'' in $\Z$:
$$
\begin{array}{ll}\label{int}
[s..t]=\{x\in\Z\cup\{\pm\infty\}\suchthat s\le x\le t\},& [s..t)=\{x\in\Z\cup\{\pm\infty\}\suchthat s\le x <  t\},\\
(s..t]=\{x\in\Z\cup\{\pm\infty\}\suchthat s <  x\le t\},& (s..t)=\{x\in\Z\cup\{\pm\infty\}\suchthat s <  x <  t\}.
\end{array}
$$

We often denote by $x^n$ the sequence $(x,x,\dots,x)$ ($x$ repeated $n$ times).

For a sequence $\lm$ of length $n$ and an index $i=1,\ldots,n$, we denote by $\lm_i$ the $i$th entry
of $\lm$. In other words, $\lm=(\lm_1,\ldots,\lm_n)$. We also set $\sum\lm:=\lm_1+\cdots+\lm_n$ (when this makes sense).
Considering $\lm$ as a function on $[1..n]$, for any $I\subset [1..m]$ we denote by $\lm|_I$ the restriction considered as a sequence.
For example, $\lm|_{[1..n)}=(\lm_1,\ldots,\lm_{n-1})$.
If $f:T\to M$ is a function, its value at a point $t\in T$ will be denoted $f(t)$ or $f_t$.
Moreover, denote $\sum f:=\sum_{t\in T}f(t)$ and ${\sum^2} f:=\sum_{s,t\in T,s<t}f_sf_t$ (when makes sense).

Let $S\subset\Z$. A subset $R\subset S$ is called a {\it beginning} (resp. {\it end}) of $S$ if for any $r\in R$, all elements $x\in S$ such that
$x\le r$ (resp. $x\ge r$) belong to $R$. In other words, a beginning of $S$ is a set of the form $S\cap (-\infty..i)$
and an end of $S$ is a set of the form $S\cap (i..+\infty)$ for some $i\in\Z\cup\{\pm\infty\}$.

If $M$ is a set and $x\in M$, we denote by $M_{x\mapsto y}$ the set $M\setminus\{x\}\cup\{y\}$.
Similar notation makes sense for multisets.

For any integer $n\in\Z$, we consider its {\em odd}\, counterpart $\bar n$
and let $\bar\Z:=\{\bar n\suchthat n\in\Z\}$. We assume $\Z\cap\bar\Z=\emptyset$.
A {\em signed set} is a subset $S\subset\Z\cup\bar\Z$ such that both $n$ and $\bar n$ belong to $S$ for no $n\in\Z$.
Elements of $\Z$ are called {\em even} and elements of $\bar\Z$ are called {\em odd}.
We use the following order on $\Z\cup \bar\Z$:
$
\bar n<m,\quad n<\bar m,\quad\bar n<\bar m
$
for arbitrary $n,m\in\Z$ satisfying $n<m$ in the usual order on $\Z$.
Moreover, for $n\in\Z$, it is convenient to use the following notation:
$
n\sim n,\quad \bar n\sim n,\quad n\sim \bar n,\quad\bar n\sim \bar n
$.
For any $x\in\Z\cup\bar\Z$, the element $y\in\Z$ such that $x\sim y$ is called the
{\it absolute value} of $x$. For example, $7$ is the absolute value of $7$ and $\bar7$.

For a finite signed set $S=\{n_1,\ldots,n_k\}\cup\{\bar m_1,\ldots,\bar m_l\}$
the {\em height}\, of $S$ is
$$\height S=n_1+\cdots+n_k+m_1+\cdots+m_l$$
and the {\em parity}\, of $S$ is
$$\|S\|=l\cdot\1\in\Z/2\Z.$$
By convention, $\height\emptyset=-\infty$ and $\|\emptyset\|=\0$.
Let $I\subset \Z$. A signed set $S$ is a signed $I$-set if the absolute value of each element of $S$ belongs to $I$.
For example, let $S=\bigl\{1,\bar 3,5,\bar6,\bar 7\bigr\}$. Then, $S$ is a signed $[1..7]$-set but not a signed $[2..7]$-set. Moreover, $\height S=22$, $\|S\|=\1$, $\min S=1$ and
$\max S=\bar7$.

We can apply usual operations of set theory to signed $\Z$-sets, simply considering them as subsets of $\Z\cup\bar\Z$.
For example, $\bigl\{2,5,\bar8\bigr\}\cup\bigl\{\bar3,9\bigr\}=\bigl\{2,\bar3,5,\bar8,9\bigr\}$,
$\bigl\{3,\bar4,7,8\bigr\}\setminus\bigl\{\bar4,8\bigr\}=\bigl\{3,7\bigr\}$ and $\bigl\{1,\bar4,5\bigr\}_{\bar4\mapsto\bar3}=\bigl\{1,\bar3,5\bigr\}$.
Let $M$ be a signed set and $S\subset\Z$. Then we denote by $M_S$ the signed set consisting of all
elements $x\in M$ whose absolute values belong to $S$. For example, $\{1,\bar 2,3,\bar5,6\}_{(2..5]}=\{3,\bar5\}$.

If $A=A_\0\oplus A_\1$ is a superalgebra, and $a\in A$ is a homogeneous element of degree $\delta\in\{\0,\1\}$, then we write $\|a\|:=\delta$. If $a$ and $b$ are two homogeneous elements of $A$, we denote the {\em supercommutator}
\begin{equation}\label{ESuperComm}
[a,b]:=ab-(-1)^{\|a\|\|b\|}ba.
\end{equation}
If $V$ is an $A$-supermodule and $B$ is a subsuperalgebra of $B$ then the restriction of $V$ from $A$ to $B$ is denoted $\operatorname{res}^A_B V$ or simply $V_B$.

In this paper, we will deal with root systems of type $A_{n-1}$. Denote $\eps_i:=(0,\dots,0,1,0,\dots,0)\in \mathbb R^n$ with $1$ in the $i$th position. For $i=1,\ldots,n-1$, set
$\alpha_i:=\eps_i-\eps_{i+1}$ and for $1\le i<j\le n$, set
$\alpha(i,j):=\eps_i-\eps_{j}$. 
Thus $\alpha_1,\ldots,\alpha_{n-1}$ are all simple roots and
$\{\alpha(i,j)\suchthat 1\le i<j\le n\}$ are all positive roots.
We denote $X(n):=\bigoplus_{i=1}^n \Z\cdot \eps_i\cong\Z^n$ and sometimes refer to it as the {\em weight lattice}. Elements of $X(n)$ are called {\em weights}.
We denote the non-negative part of the root lattice by
$$
Q_+(n):=\bigoplus_{i=1}^{n-1}\Z_{\geq 0}\cdot\al_i.
$$
We use the usual {\em dominance order} on $X(n)$:
$\lm\ge\mu$ if and only if $\lm-\mu\in Q_+$.
A weight $\la=(\la_1,\dots,\la_n)\in X(n)$ is {\em dominant} if $\la_1\geq\dots\geq \la_n$.
A dominant weight $\la=(\la_1,\dots,\la_n)\in X(n)$ is {\em $p$-strict} if $\la_i=\la_{i+1}$ for some $1\leq i<n$ implies $p\mid \la_i$. The set of all $p$-strict dominant weights in $X(n)$ is denoted by $X^+_p(n)$.
Denote
$$
h_{p'}(\la):=\big|\{i\mid1\leq i\leq n\ \text{and}\ p\nmid\la_i\}\big| \quad\qquad(\la\in X(n)).
$$

\section[The supergroup $Q(n)$]{The supergroup $Q(n)$ and its hyperalgebra}\label{supergroupQn}
We review some basic facts on the supergroup $Q(n)$ and its hyperagebra referring the reader to \cite{Kleshchev_Brundan_Modular_Representations_of_the_supergroup_Q(n)_I} for details.

The complex Lie superalgebra $\mathfrak q(n,\mathbb C)$, as a vector superspace, consists of all complex $2n\times2n$ matrices of the form
$$
\left(
{\arraycolsep=1pt
\begin{array}{r|l}
S\,&\,S'\\
\hline
\vphantom{A^{A^A}}
S'\,&\,S
\end{array}}
\right),
$$
where $S$ and $S'$ are $n\times n$ matrices with complex entries. Such a matrix is even, i.e. belongs to  $\mathfrak q(n,\mathbb C)_\0$, if $S'=0$. Such a matrix  is odd, i.e. belongs to  $\mathfrak q(n,\mathbb C)_\1$, if $S=0$.
The {\it supercommutator} of two homogeneous elements $x,y\in \mathfrak q(n,\mathbb C)$ is defined by
$[x,y]=xy-(-1)^{\|x\|\cdot\|y\|}yx$, where $\|x\|,\|y\|\in\Z_2$ denote the parity of elements $x$, $y$ respectively.

Let $e_{s,t}$ denote the $2n\times2n$ matrix with $1$ in the $(s,t)$-position and zeros elsewhere. For $1\leq i,j\leq n$ set
$$X_{i,j}:=e_{i,j}+e_{n+i,n+j},\quad\bar X_{i,j}:=e_{i,n+j}+e_{n+i,j}\qquad(1\leq i,j\leq n).$$
Then $\{X_{i,j}, \bar X_{i,j}\mid 1\leq i,j\leq n\}$ is a basis of $\mathfrak q(n,\mathbb C)$.
Let $U_\Z(n)$ be the $\Z$-subalgebra of the universal enveloping superalgebra $U_\mathbb C(n)$ generated by
{\begin{itemize}
\itemsep=1pt
\item $X_{i,j}^{(m)}:=X_{i,j}^m/m!$\quad ($1\leq i\neq j\leq n$, $m\in \Z_{\ge0}$);
\item $\bar X_{i,j}$\quad ($1\leq i\neq j\leq n$);
\item $\tbinom{X_{i,i}}m:=X_{i,i}(X_{i,i}-1)\cdots(X_{i,i}-m+1)/m!$\quad ($1\leq i\leq n$, $m\in \Z_{\ge0}$);
\item $\bar X_{i,i}$\quad ($1\leq i\leq n$).
\end{itemize}}

\begin{proposition}[\mbox{\cite[Lemma 4.3]{Kleshchev_Brundan_Modular_Representations_of_the_supergroup_Q(n)_I}}]\label{proposition:intro:1}
The $\Z$-superalgebra $U_\Z(n)$ is free  over $\Z$ with basis given by all products of $X_{i,j}^{(a_{i,j})}$,
$\bar X_{i,j}^{b_{i,j}}$ $\tbinom{X_{i,i}}{c_i}$, $\bar X_{i,i}^{d_i}$ in a (fixed) arbitrary order,
where $a_{i,j},c_i$ are nonnegative integers and $b_{i,j},d_i\in\{0,1\}$.
\end{proposition}

Set $U(n):=U_\Z(n)\otimes_\Z\mathbb F$. The elements $X_{i,j}^{(m)}\otimes 1$, $\bar X_{i,j}\otimes 1$, $\tbinom{X_{i,i}}m\otimes 1$, $\bar X_{i,i}\otimes 1$ in $U(n)$
are denoted again by $X_{i,j}^{(m)}$, $\bar{X}_{i,j}^{b_{i,j}}$, $\tbinom{X_{i,i}}m$, $\,\bar{\!X}_{i,i}$, respectively.
We call $U_\Z(n)$ and $U(n)$ the {\it hyperalgebras} of $Q(n)$ over $\Z$ and $\mathbb F$ respectively. We often re-denote:
\begin{eqnarray*}
E_{i,j}^{(m)}\!:= X_{i,j}^{(m)},\ \bar{\!E}_{i,j}\!:=\bar{\! X}_{i,j},\  F_{i,j}^{(m)}\!:= X_{j,i}^{(m)},\ \bar{\! F}_{i,j}\!:=\bar{\! X}_{j,i}\label{EF}\  (1\le i<j\le n,\ m\in\Z_{\geq 0});
\\
\tbinom{ H_i}m:=\tbinom{ X_{i,i}}m,\ \,\bar{\! H}_i:=\,\bar{\! X}_{i,i},\label{H}\
(1\le i\le n,\ m\in\Z_{\geq 0}).
\end{eqnarray*}

The hyperalgebra $U(n)$ is naturally a superalgebra: the elements $ X_{i,j}^{(a_{i,j})}$ and $\tbinom{ X_{i,i}}{c_i}$
are even and the elements $\,\bar{\! X}_{i,j}$ and $\,\bar{\! X}_{i,i}$ are odd.
Moreover, $U(n)$ has a grading by the root lattice of the root system $A_{n-1}$.
We use the table bellow to describe the {\it weights} of the generators of $U(n)$ in this grading:

\vspace{ 1mm}

\begin{center}
\begin{tabular}{|c||c|c|c|c|c|c|}
\hline
Element&$E_{i,j}^{(m)}$&$\,\bar{\!E}_{i,j}$&$ F_{i,j}^{(m)}$&$\,\bar{\! F}_{i,j}$&$\tbinom{ H_i}m$\vphantom{$\binom{A^{A^A}}m$}&$\,\bar{\! H}_i$\\
\hline
Weight&$m\alpha(i,j)$&$\alpha(i,j)$&$-m\alpha(i,j)$&$-\alpha(i,j)$&0&0\vphantom{$A^{A^A}$}\\
\hline
\end{tabular}
\end{center}
\vspace{1 mm}
Finally, $U(n)$ inherits the structure of a {\em Hopf}\, superalgebra from $U_{\mathbb C}(n)$, but this will not be important here.
The hyperalgebra has the {\it triangular decomposition} $U(n)=U^-(n) U^0(n) U^+(n)$, where
\begin{itemize}
\item $U^-(n)$ is the subalgebra generated by $ F_{i,j}^{(m)}$ and $\,\bar{\! F}_{i,j}$;
\item $U^0(n)$ is the subalgebra generated by $\tbinom{ H_i}m$ and $\,\bar{\! H}_i$;
\item $U^+(n)$ is the subalgebra generated by $E_{i,j}^{(m)}$ and $\,\bar{\!E}_{i,j}$.
\end{itemize}
We will also need the sub(super)algebras $U^{\leq 0}(n):=U^-(n) U^0(n)$ and $U^{\geq 0}(n):=U^0(n) U^+(n)$. Similarly we have a triangular decomposition over $\Z$:
$$U_\Z(n)=U_\Z^-(n) U_\Z^0(n) U_\Z^+(n).$$

Recall that $U(n)$ is a Hopf superalgebra
whose comultiplication $\Delta$ and antipode $\sigma$ are given by
$$
\begin{array}{ll}
\Delta\bigl(X_{i,j}^{(m)}\bigr)=\sum_{t=0}^mX_{i,j}^{(t)}\otimes X_{i,j}^{(m-t)},&\quad \Delta\bigl(\bar X_{i,j}\bigr)=\bar X_{i,j}\otimes1+1\otimes\bar X_{i,j}\\[12pt]
\Delta\(\tbinom{H_i}m\)=\sum_{t=0}^m\tbinom{H_i}t\otimes\tbinom{H_i}{m-t},&\quad\Delta(\bar H_i)=\bar H_i\otimes 1+1\otimes\bar H_i\\[12pt]
\eta\bigl(X_{i,j}^{(m)}\bigr)=(-1)^mX_{i,j}^{(m)},&\quad\eta\bigl(\bar X_{i,j}\bigr)=-\bar X_{i,j}\\[8pt]
\eta\(\tbinom{H_i}m\)=\tbinom{-H_i}m,&\quad\eta(\bar H_i)=-\bar H_i.
\end{array}
$$
As usual, the comultiplication is used to define the structure of a $U(n)$-supermodule on the tensor product of two $U(n)$-supermodules, and the antipode is used to define the structure of a $U(n)$-supermodule on the dual of a $U(n)$-supermodule. The dual supermodule to the supermodule $M$ will be denoted $M^*$.  

The hyperalgebra 
plays an important role because it is isomorphic to the algebra of distributions of the supergroup $Q(n)$. Recall that a {\em supergroup} (over $\mathbb F$) is a functor from the category $\mathfrak{salg}_\mathbb F$ of commutative $\mathbb F$-superalgebras and even homomorphisms to the category of groups.
In particular, the supergroup $Q(n)$\label{Qn} by definition is a functor $G$ which to a commutative $\mathbb F$-superalgebra $A=A_\0\oplus A_\1$ assigns the group $G(A)$ of {\em invertible}\,
$2n\times 2n$ matrices of the form
\begin{equation}\label{EQnMat}
\left(
{\arraycolsep=1pt
\begin{array}{r|l}
S\,&\,S'\\
\hline
\vphantom{A^{A^A}}
{-}S'\,&\,S
\end{array}}
\right),
\end{equation}
where $S$ is an $n\times n$ matrix with entries in $A_\0$ and $S'$ is an $n\times n$ matrix with entries in $A_\1$. To a morphism $f:A\to B$ the functor $G$ assigns the morphism $G(f):Q(n)(A)\to Q(n)(B)$
acting as $f$ on each entry of the matrix $G(A)$.

Actually, $G$ is an {\em algebraic}\, supergroup, i.e. an affine superscheme with the coordinate ring $\mathbb F[G]$ is described as follows. Let $M$ be the functor from $\mathfrak{salg}_\mathbb F$ to the category of {\em sets}\, which assigns to a commutative superalgebra $A$ the set of {\em all} $2n\times 2n$ matrices of the form (\ref{EQnMat}). Then $M$ is isomorphic to the affine supescheme $\mathbb A^{n^2|n^2}$ with coordinate ring $\mathbb F[M]$ being the free commutative superalgebra on even generators $s_{i,j}$ and odd generators $s_{i,j}'$ for $1\leq i,j\leq n$. It is known \cite[Section 3]{Kleshchev_Brundan_Modular_Representations_of_the_supergroup_Q(n)_I} that $G$ is the principal open subset of $M$ defined by the function
$\det:\left(
{\arraycolsep=1pt
\begin{array}{r|l}
S\,&\,S'\\
\hline
\vphantom{A^{A^A}}
{-}S'\,&\,S
\end{array}}
\right)
\mapsto \det S$. In particular, $\mathbb F[G]$ is the localization of $\mathbb F[M]$ at the function $\det:=\det\|s_{ij}\|_{1\leq i,j\leq n}$. The Hopf superalgebra structure of $\mathbb F[G]$ is described explicitly in \cite[Section 3]{Kleshchev_Brundan_Modular_Representations_of_the_supergroup_Q(n)_I}.

Let $e$ be the identity matrix in $G(\mathbb F)$, and $I_e$ be the kernel of the evaluation map $\mathbb F[G]\to\mathbb F, f\mapsto f(e)$.
The {\em algebra of distributions}\, $\Dist G$ is defined 
as
$\Dist G:=\bigcup_{m\ge0}\Dist_mG$,
where
$
\Dist_mG=\bigl\{\mu\in\mathbb F[G]^*\suchthat\mu(I_e^{m+1})=0\bigr\}.
$
The Hopf superalgebra structure on $\mathbb F[G]$ yields a cocommutative Hopf superalgebra structure on $\Dist G$.

\begin{theorem}[\mbox{\cite[Theorem 4.4]{Kleshchev_Brundan_Modular_Representations_of_the_supergroup_Q(n)_I}}]\label{proposition:intro:2}
The Hopf superalgebras $\Dist G$ and $U(n)$ are isomorphic.
\end{theorem}

From now on we identify $U(n)$ with $\Dist G$ and interchange them freely. Let $M$ be a  $U^0(n)$-supermodule and $\la=\sum_{i=1}^n\la_i\eps_i\in X(n)$ be a weight. Define the {\em $\la$-weight space}\, of $M$:
$$
M^\la:=\{v\in M\mid \tbinom{ H_i}m v=\tbinom{\la_i}m\ \text{for all $i=1,\dots,n$, $m\geq 1$}\}.
$$
A $U(n)$-supermodule is a $U^0(n)$-supermodule on restriction, so its weight spaces are also defined.
A $U(n)$-supermodule $M$ is {\em integrable} if $M$ is locally finite and $M=\bigoplus_{\la\in X(n)}M^\la$. A structure of $G$-supermodule on a vector superspace $M$ canonically gives rise to a structure of an integrable $U(n)$-supermodule on $M$. It is proved in \cite{Kleshchev_Brundan_Modular_Representations_of_the_supergroup_Q(n)_I} that the converse is also true, i.e. a structure of an integrable $U(n)$-supermodule on a vector superspace $M$ canonically lifts to a structure of a $G$-supermodule on $M$:

\begin{theorem} 
{\rm \cite[Corollary 5.7]{Kleshchev_Brundan_Modular_Representations_of_the_supergroup_Q(n)_I}} 
The category of $G$-supermodules is isomorphic to the category of integrable $U(n)$-supermodules.
\end{theorem}

In view of the theorem, we will switch freely between $G$-supermodules and integrable $U(n)$-supermodules.

\section{Highest weight theory}\label{SHWT}

The classification of irreducible $G$-supermodules was obtained in~\cite{Kleshchev_Brundan_Modular_Representations_of_the_supergroup_Q(n)_I}.

\begin{proposition}[\mbox{\cite[Theorem 6.11]{Kleshchev_Brundan_Modular_Representations_of_the_supergroup_Q(n)_I}}]\label{proposition:intro:3}
For any $\lm\in X^+_p(n)$ there exists a $G$-supermodule $L(\lm)$ with highest weight $\lm$, i.e. $L(\lm)^\la\neq 0$ and $L(\lm)^\mu\neq 0$ only if $\mu\leq \la$.
Moreover, $\bigl\{L(\lm)\suchthat\lm\in X^+_p(n)\bigr\}$ is a complete and irredundant set of irreducible $G$-supermodules up to isomorphism.
\end{proposition}

Consider the algebra antiautomorphism $\tau=\tau_n$ of $U(n)$, i.e. a linear map such that $\tau(xy)=\tau(y)\tau_n(x)$ for all $x,y\in U(n)$,
defined by
$$
\begin{array}{lll}
\tau(E_{i,j}^{(m)})=F_{i,j}^{(m)},&\tau(F_{i,j}^{(m)})=E_{i,j}^{(m)},&\tau\left(\tbinom{H_i}m\right)=\tbinom{H_i}m,\\[6pt]
\tau(\bar E_{i,j})=\bar F_{i,j},&\tau(\bar F_{i,j})=\bar E_{i,j},&\tau(\bar H_i)=\bar H_i.
\end{array}
$$

For any $U(n)$-supermodule $M$ which has a weight space decomposition $M=\bigoplus_{\la\in X(n)}M^\la$, we denote by $M^{\tau}$ the superspace $\bigoplus_{\la\in X(n)}(M^\la)^*$ of linear functions $f:M\to\mathbb F$
with the grading $M^{\tau}_\0=\{f\in M^{\tau}\suchthat f(M_\1)=0\}$, $M^{\tau}_\1=\{f\in M^{\tau}\suchthat f(M_\0)=0\}$.
This superspace is made into a $U(n)$-supermodule via $(xf)(m)=f({\tau}(x)m)$
for $x\in U(n), f\in M^\tau, m\in M$. The supermodule $M^\tau$ is referred to as the {\em contravariant dual} of $M$ and should be distinguished from the usual dual $M^*$.
For any $U(n)$-homomorphism $\zeta:M\to N$,
we denote by $\zeta^{\tau}$ the map from $N^{\tau}\to M^{\tau}$ defined by
$\zeta^{\tau}(f)=f\circ\zeta$ for $f\in N^{\tau}$.
Then $\zeta^{\tau}$ is a homomorphism of $U(n)$-supermodules.

A vector $v$ in a $U(n)$-supermodule $M$ is called {\it primitive} if it belongs to some weight space $M^\mu$ and
$E^{(m)}_{i,j}v=\,\bar{\!E}_{i,j}\,v=0$ for all $1\le i<j\le n$ and $m>0$.
Likewise a subset of $M$ is called {\it primitive} if each its vector is primitive.
It follows from the results of \cite[Section 6]{Kleshchev_Brundan_Modular_Representations_of_the_supergroup_Q(n)_I} that the set of primitive vectors of $L(\la)$ is precisely $L(\la)^\la$ and that $L(\la)^\la$ is the irreducible $U^0(n)$ module $\mathfrak u(\la)$ described by the following proposition:

\begin{proposition} \label{PUMu}
\label{proposition:intro:4}
For each $\mu\in X(n)$, there exists a unique (up to isomorphism)
irreducible $U^0(n)$-supermodule $\mathfrak u(\mu)$ such that $\mathfrak u(\mu)^\mu=\mathfrak u(\mu)$. Moreover:
\begin{enumerate}
\item[{\rm (i)}] $\dim \mathfrak u(\mu)=2^{\lfloor(h_{p'}(\mu)+1)/2\rfloor}$.
\item[{\rm (ii)}] if $\mu_i\=0\pmod p$ for some $1\leq i\leq n$, then $\bar{\! H}_i\,\mathfrak u(\mu)=0$
\item[{\rm (iii)}] if $\mu\in X_p^+(n)$, then $L(\mu)^\mu$ is isomorphic to $\mathfrak u(\mu)$ as a $U^0(n)$-module.
\item[{\rm (iv)}] $\mathfrak u(\mu)$ is of type $\Mtype$ if and only if $h_{p'}(\mu)$ is even.
\end{enumerate}
\end{proposition}

For each $\la\in X(n)$, one can form the induced module $H^0(\la):=\operatorname{ind}_B^G\mathfrak u(\la)$, as in \cite[(6.5)]{Kleshchev_Brundan_Modular_Representations_of_the_supergroup_Q(n)_I}, and its contravariant dual $V(\la):=H^0(\la)^\tau$ \cite[(10.14)]{Kleshchev_Brundan_Modular_Representations_of_the_supergroup_Q(n)_I}. Then

\begin{theorem} \label{TUnivProperty}{\rm \cite[Section 6]{Kleshchev_Brundan_Modular_Representations_of_the_supergroup_Q(n)_I}} 
We have $V(\la)\neq 0$  if and only if $\la\in X^+_p(\la)$, in which case $V(\la)$ is finite dimensional,
$V(\la)^\la\cong\u(\la)$, $V(\la)$ has simple head $L(\la)$, and $V(\la)$
is universal among all finite dimensional $U(n)$-supermodules generated by
a primitive $U^0(n)$-subsupermodule isomorphic to $\mathfrak u(\lm)$.
\end{theorem}


It is often slightly more convenient to work in the larger category than the category of finite dimensional $U(n)$-supermodules. To define this larger category, given $\la\in X(n)$, set
$$
X(\la):=\{\mu\in X(n)\mid \mu\leq \la\}.
$$
Now, define the {\em integral category $\O=\O(n)$} as the full subcategory of $U(n)$-supermodules which consists of supermodules $M$ such that $M=\bigoplus_{\la\in X(n)} M^\la$, $\dim M^\la<\infty$ for all $\la\in X(n)$, and there are $\la^{(1)},\dots,\la^{(r)}\in X(n)$ such that for any $\mu\in X(n)$, we have that $M^\mu\neq 0$ implies $\mu\in\cup_{i=1}^r X(\la^{(i)})$. This is analogous to the category considered in  \cite[Section 2.3]{Kujawa} for $\mathfrak{gl}(m,n)$.
Note that any module in $\O$ is locally finite over $U^{\geq 0}(n)$.
For any $\la\in X(n)$, we have the {\em Verma supermodule}
$$
M(\la):=U(n)\otimes_{U^{\geq 0}(n)}\u(\la),
$$
where we have inflated $\u(\la)$ from $U^0(n)$ to $U^{\geq 0}(n)$.
A standard argument yields:

\begin{lemma}\label{Vermamod} 
Let $\la\in X(n)$. Then $M(\la)$ is an object in $\O$,
$M(\la)^\la\cong\u(\la)$, $M(\la)$ has simple head $L(\la)$,
and $M(\la)$ is universal among all supermodules in $\O$ generated by
a primitive $U^0(n)$-subsupermodule isomorphic to $\mathfrak u(\lm)$.
Finally, $\{L(\la)\mid\la\in X(n)\}$ is a complete and irredundant set of irreducible modules in $\O$ up to isomorphism.
\end{lemma}

The category of integrable finite dimensional $U(n)$-supermodules is a subcategory of the category $\O$, and $L(\la)$ is finite dimensional if and only if $\la\in X^+_p(n)$.
The contravariant duality $\tau$ preserves the category $\O$, and $L(\la)^\tau\cong L(\la)$ for all $\la\in X(n)$. Moreover, let us
consider $\u(\la)$ as a $U^{\leq 0}(n)$-module via inflation along the natural surjection $U^{\leq 0}(n)\to U^0(n)$. Then it is easy to see that
\begin{equation}\label{EVermaDual}
M(\la)^\tau\cong \operatorname{coind}_{U^{\leq 0}(n)}^{U(n)}\u(\la),
\end{equation}
where $\operatorname{coind}_B^A V$ denotes the $A$-module $\Hom_B(A,V)$ obtained from a $B$-module $V$ by coinduction from a  subalgebra $B\subseteq A$.

\begin{lemma} \label{LHomVermaDualVerma}
For $\la,\mu\in X(n)$, we have $\Hom_{U(n)}(M(\la),M(\mu)^\tau)=0$ unless $\la=\mu$, and $\Hom_{U(n)}(M(\la),M(\la)^\tau)\cong\operatorname{End}_{U^0(n)}(\u(\la))$.
\end{lemma}
\begin{proof}
Using (\ref{EVermaDual}), we get
$$\Hom_{U(n)}(M(\la),M(\mu)^\tau)\cong\Hom_{U^{\leq 0}(n)}(M(\la),\u(\mu)),$$
which easily implies the result.
\end{proof}

\begin{lemma} \label{LVermaFilt}
Let $\la\in X(n)$ and $W$ be an integrable finite dimensional $U(n)$-supermodule. Then the supermodule
$M(\la)\otimes W$ has a finite filtration with factors of the form $M(\la+\mu)$, where $\mu$ is a weight of $W$.
\end{lemma}
\begin{proof}
It is easy to check that as usual we have
$$
M(\la)\otimes W=(\ind_{U^{\geq 0}(n)}^{U(n)}\u(\la))\otimes W\cong \ind_{U^{\geq 0}(n)}^{U(n)}(\u(\la)\otimes \operatorname{res}^{U(n)}_{U^{\geq 0}(n)}W).
$$
Now $\operatorname{res}^{U(n)}_{U^{\geq 0}(n)}W$ has a finite filtration whose composition factors are of the form $\u(\mu)$, where $\mu$ is a weight of $W$, and $\u(\la)\otimes \u(\mu)$ has a finite filtration all of whose factors are of the form $\u(\la+\mu)$.
\end{proof}

\begin{corollary} \label{CVermaTensHom}
Let $\la,\mu\in X(n)$ and $W$ be an integrable finite dimensional $U(n)$-supermodule. Then $\Hom_{U(n)}(M(\la),L(\mu)^\tau\otimes W)\neq 0$ implies $\la=\mu+\nu$ for some weight $\nu$ of $W$.
\end{corollary}
\begin{proof}
We have
$$\Hom_{U(n)}(M(\la),L(\mu)^\tau\otimes W)\subseteq \Hom_{U(n)}(M(\la),M(\mu)^\tau\otimes W),
$$
and by dualizing in Lemma~\ref{LVermaFilt}, we conclude that $M(\mu)^\tau\otimes W$ has a finite filtration all of whose factors are of the form $M(\mu+\nu)^\tau$, where $\nu$ is a weight of $W$. Now the result follows from Lemma~\ref{LHomVermaDualVerma}.
\end{proof}

Denote
\begin{equation}\label{equation:intro:0}
\begin{split}
&E^\0_i:=E^{(1)}_{i,i+1},\ E^\1_i:=\bar E_{i,i+1},\ F_{i,j}^\0:=F^{(1)}_{i,j},
\\
&F_{i,j}^\1:=\bar F_{i,j},\ H_i^\0:=\tbinom{H_i}1,\ H_i^\1:=\bar H_i.
\end{split}
\end{equation}
Extending scalars from $\Z$ to $\mathbb F$, we get the corresponding elements of $E^\0_i$, $E^\1_i$, $\sfsl F_{i,j}^\0$, $\sfsl F_{i,j}^\1$, $ H_i^\0$, $ H_i^\1$ of $U(n)$.
Moreover, it will be convenient to use the  notation $X:=X^\0$ and $\bar X:=X^\1$ when the right hand sides make sense for the  symbol $X$.
Thus, for example, $H_i=H^\0_i$ and $\bar E_i=E^\1_i$.

Recalling the supercommutation notation (\ref{ESuperComm}), we can write the supercommutator of any two generators of $U_\Z(n)$ and $U(n)$, using the corresponding relation in the universal enveloping algebra $U_\mathbb C(n)$. The following relations will be used especially often:
\begin{equation}\label{equation:intro:1}
\begin{array}{l}
[H_i^\de,H_i^\eps]=(1 -(-1)^{\de\eps}) H_i^{\de+\eps};
\\[1pt]
\bar H_i^2= H_i;
\\[1pt]
[H_i^\de, H_j^\eps]=0\qquad(
i\neq j);
\\[1pt]
(\bar H_i\pm\bar H_j)^2=(H_i+H_j)\qquad(
i\neq j);
\\[1pt]
[E_i^\delta, F_{i,j}^\epsilon]=-(-1)^{\delta\epsilon}F_{i+1,j}^{\epsilon+\delta}\qquad
(j>i+1);
\\[1pt]
[E_{j-1}^\delta, F_{i,j}^\epsilon]=F_{i,j-1}^{\epsilon+\delta}\qquad
(j>i+1);
\\[1pt]
[E_k^\delta, F_{i,j}^\epsilon]=0\qquad (k\ne i, j-1);
\\[1pt]
[E_i^\delta, F_{i}^\epsilon]=H_i^{\epsilon+\delta}-(-1)^{\delta\epsilon}H_{i+1}^{\epsilon+\delta};
\\[1pt]
[H_i^\de,E_{i,j}^\eps]=E_{i,j}^{\de+\eps}; 
\\[1pt]
[H_j^\de,E_{i,j}^\eps]=-(-1)^{\de\eps}E_{i,j}^{\de+\eps}; 
\\[1pt]
[H_i^\de,F_{i,j}^\eps]=-(-1)^{\de\eps}F_{i,j}^{\de+\eps}; 
\\[1pt]
[H_j^\de,F_{i,j}^\eps]=F_{i,j}^{\de+\eps}; 
\\[1pt]
[H_l^\de,E_{i,j}^\eps]=[H_l^\de,F_{i,j}^\eps]=0; \qquad(k\neq i,j).
\\[1pt]
\end{array}
\end{equation}

Often we will only need to know commutation relations modulo ideals generated by
certain elements of $U^+_\Z(n)$. Specifically, for $S\subset [1,n)$, let $I^+_S$\label{I+S} denote the left ideal of $U_\Z(n)$ generated by $E_s$ and $\bar E_s$, where $s\in S$.
We will abbreviate $I^+_l:=I^+_{\{l\}}$. For example the relation $[E_i^\eps, H_i^\de]=-(-1)^{\eps\de}E_i^{\eps+\de}$ above implies $[E_i^\eps, H_i^\de]\equiv 0\pmod{I_i}$.

We will use the following elements of $U_\Z(n)$ or $U(n)$ for $1\leq i,j\leq n$:
$$
C(i,j):=H_i(H_i-1)-H_j(H_j-1),\quad B(i,j):=H_i(H_i-1)-(H_j+1)H_j.\label{CB}
$$
Note that for any vector $v$ of weight $\lm$ we have
$$
 C(i,j)v=(\res_p\lm_i-\res_p\lm_j)v,\quad  B(i,j)v=(\res_p\lm_i-\res_p(\lm_j+1))v.
$$

Let $k\leq n$. We have natural inclusions $U(k)\subset U(n)$ and $U^0(k)\subset U^0(n)$, under which generators go to the generators with exactly the same names.
Given a $U(n)$-supermodule $V$ and a weight $\mu\in X(n)$, we can speak of $U(k)$-primitive vectors of weight $\mu$ in $V$, i.e. vectors $v\in V^\mu$ such that $E_{i,j}^{(m)}v=\bar E_{i,j}v=0$ for all $1\leq i<j\leq k$ and $m>0$.  One of the main goals of this paper is to understand $U(n-1)$-primitive vectors of weight $\la-\al(i,n)$ in the irreducible $U(n)$-module $L(\la)$. These primitive vectors will be obtained by  applying appropriate lowering operators to highest weight vectors.

The proof of the following useful fact is standard:

\begin{proposition}\label{proposition:intro:5}
Let $\la\in X(n)$ and $\al\in Q_+(n)$. A vector $v\in L(\lm)^{\lm-\alpha}$ is nonzero if and only if there exists some $E\in U^+(n)^\al$
such that $E v\ne0$.
\end{proposition}

\begin{lemma}\label{lemma:socle:1.5}
Let $\lm\in X(n)$, $1\le h<i<n$, $F\in U(n)^{-\alpha(h,i)}$,
and $v$ be a $U(n-1)$-primitive vector of $L(\lm)^{\lm-\alpha(i,n)}$.
Suppose that $ E_l^\delta Fv=0$ for all $l\in[h..i-1)$ and $\delta\in\{\0,\1\}$, and that
$ E_h^{\delta_h}\cdots E_{i-1}^{\delta_{i-1}}Fv=0$ for all $\delta_h,\ldots,\delta_{i-1}\in\{\0,\1\}$.
Then $Fv$ is a $U(n-1)$-primitive vector.
\end{lemma}
\begin{proof}
It suffices to prove that $ E_{i-1}^\delta Fv=0$ for any $\delta$. By Proposition~\ref{proposition:intro:5},
we must prove that $P E_{i-1}^\delta Fv=0$ for any product $P$ of the elements $ E^{\delta_i}_i,\ldots, E^{\delta_{n-1}}_{n-1}$
and the elements $ E^{\delta_h}_h,\ldots, E^{\delta_{i-2}}_{i-2}$ in an arbitrary order.
The elements of the first group supercommute with the elements of the second group, so
we can write $P=\pm P'P''$, where $P'$ is a product of the elements $ E^{\delta_i}_i,\ldots, E^{\delta_{n-1}}_{n-1}$ and $P''$ is a product of elements $ E^{\delta_h}_h,\ldots, E^{\delta_{i-2}}_{i-2}$.
Now $P'' E_{i-1}^\delta F\,v=0$, except when $P''= E^{\delta_h}_h\cdots E^{\delta_{i-2}}_{i-2}$.
However, in this case $P'' E_{i-1}^\delta F\,v=0$  by assumption.
\end{proof}

Recall that our ground field $\mathbb F$ is algebraically closed and of characteristic  different from $2$.
Fix a square root $\sqrt{-1}\in\mathbb F$. Let $\bi\in\mathbb C$ be a primitive 4th root of $1$, and $\Z[\bi]$ be the ring of Gaussian integers. We may extend the natural $\Z$-action on $\mathbb F$ to a $\Z[\bi]$-action such that $\bi$ acts with multiplication by $\sqrt{-1}$.

Let $w_0$ be the longest element of the symmetric group $S_n$, i.e. $w_0 i=n+1-i$ for all $i=1,\dots,n$. It is now easy to check that there is an even automorphism $\si$ of Lie superalgebra $\mathfrak q(n,\mathbb C)$ such that
\vspace{1mm}
\begin{equation}\label{ESigmaForm}
\si:\mathfrak q(n,\mathbb C)\to\mathfrak q(n,\mathbb C),\ X^\eps_{i,j}\mapsto  -(\bi)^{\cond_{\eps=\1}}X^\eps_{w_0 j,w_0 i}.
\end{equation}
The automorphism $\si$ restricts to the subset $\Z[\bi]\cdot U_\Z(n)\subset U_{\mathbb C}(n)$, which is isomorphic as a $\Z[\bi]$-algebra to the $\Z[\bi]$-form $U_{\Z[\bi]}(n)=U_\Z(n)\otimes\Z[\bi]$ of $U_{\mathbb C}(n)$, and then extends to the automorphism of the hyperalgebra $U(n)=U_{\Z[\bi]}(n)\otimes \mathbb F$:
\begin{equation}\label{ESigma}
\si:U(n)\to U(n)
\end{equation}
Given a $U(n)$-supermodule $M$, we can twist it with the automorphism (\ref{ESigma}) to get the $U(n)$-supermodule which we denote $M^\si$.
If $\la=(\la_1,\dots,\la_n)\in X(n)$ we set:
$$
-w_0\la:=(-\la_n,-\la_{n-1},\dots,-\la_1).
$$
Using  (\ref{ESigmaForm}), we deduce:

\begin{lemma} \label{LSigmaAppl}
If $\la\in X(n)$, then $L(\la)^\si\cong L(-w_0\la)$ and $M(\la)^\si\cong M(-w_0\la)$. In particular $(V^*)^\si\cong V$.
\end{lemma}

\chapter{Lowering operators}\label{Q(n):lovering operators}

\section{Definitions}\label{Q(n):lovering operators:def}
In this section, we define the {\em lowering operators} $S^{\,\epsilon}_{i,j}(M)\in U_\Z^{\leq 0}(n)$, where
\begin{enumerate}
\item[$\bullet$] $1\le i<j\le n$;
\item[$\bullet$] $\epsilon\in\{\0,\1\}$,
\item[$\bullet$] $M$ is a signed $(i..j]$-set containing either $\bar\jmath$ or $j$.
\end{enumerate}
These three assumptions are assumed to hold whenever we talk about lowering operators. We will also denote $S_{i,j}(M):=S^{\0}_{i,j}(M)$ and $\bar S_{i,j}(M):=S^{\1}_{i,j}(M)$.
The lowering operators are defined by induction on $\height M$ as follows. First, we set
{
\renewcommand{\theequation}{{\bf S-\arabic{equation}}}
\begin{equation}
\label{S1}
S^{\,\epsilon}_{i,j}(\{\bar\jmath\})=F^{\,\epsilon}_{i,j},
\end{equation}
\begin{equation}
\label{S2}
\begin{split}
S^{\,\epsilon}_{i,j}(\{j\})=\sum_{\gamma+\sigma=\epsilon}\Big((-1)^{\sigma} F_{i,j}^{\,\gamma}(H_i^\sigma+(-1)^{\gamma\sigma+\eps} H_j^\sigma)
+(-1)^{\gamma\epsilon}\sum_{i<k<j}\!F^\gamma_{i,k}F^\sigma_{k,j}\Big).
\end{split}
\end{equation}

Now suppose that $M$ is not equal to $\{\bar\jmath\}$ or $\{j\}$. Then of course $\min M<j$. There are four cases:

{\em Case 1:} $\min M=\odd{i{+}1}$. In this case we set
\begin{equation}
\label{S3}
\begin{split}
S_{i,j}^\epsilon(M)=\sum_{\gamma+\sigma=\epsilon}
\Big(({-}1)^{\gamma(\1+\epsilon+\|M_{(i+1..j]}\|)}
 S^{\,\gamma}_{i,i+1}\(\left\{\odd{i{+}1}\right\}\)\!S_{i+1,j}^{\,\sigma}\(M_{(i+1..j]}\)\\
+
(-1)^{\sigma(\1+\|M\|)}S_{i,j}^{\,\gamma}\(M\setminus\left\{\odd{i{+}1}\right\}\)H^\sigma_i\Big).
\end{split}
\end{equation}

{\em Case 2:} $\min M=i+1$. In this case we set
\begin{equation}
\label{S4}
\begin{split}
S_{i,j}^\epsilon(M)=\sum_{\gamma+\sigma=\epsilon}
(-1)^{\gamma(\1+\epsilon+\|M_{(i+1..j]}\|)}S^{\,\gamma}_{i,i+1}\(\left\{i{+}1\right\}\)S_{i+1,j}^{\,\sigma}\(M_{(i+1..j]}\)\\
+S_{i,j}^{\,\epsilon}(M\setminus\{i{+}1\})\,C(i,i{+}1).
\end{split}
\end{equation}

{\em Case 3:} $\min M=\bar m>i+1$. In this case we set
\begin{equation}
\label{S5}
\begin{split}
S_{i,j}^\epsilon(M)=\sum_{\gamma+\sigma=\epsilon}(-1)^{\gamma(\1+\epsilon+\|M_{(m..j]}\|)}\,S_{i,m}^{\,\gamma}\(\left\{\bar m\right\}\)S_{m,j}^{\,\sigma}\(M_{(m..j]}\)
\\
+S_{i,j}^{\,\epsilon}\(M_{\odd{\vphantom{1}m}\mapsto\odd{m{-}1}}\).
\end{split}
\end{equation}

{\em Case 4:} $\min M= m>i+1$. In this case we set
\begin{equation}
\label{S6}
\begin{split}
S_{i,j}^\epsilon(M)=\sum_{\gamma+\sigma=\epsilon}
(-1)^{\gamma(\1+\epsilon+\|M_{(m..j]}\|)}\,S^{\,\gamma}_{i,m}(\{m\})S_{m,j}^{\,\sigma}\(M_{(m..j]}\)\\
+S_{i,j}^{\,\epsilon}(M_{m\mapsto m{-}1})+S_{i,j}^{\,\epsilon}(M\setminus\{m\})\,C(m{-}1,m).
\end{split}
\end{equation}}


The following result can be easily proved by induction on $\height M$.

\begin{proposition}\label{proposition:ops:1}
$S_{i,j}^{\,\epsilon}(M)$ is a degree $\epsilon$ homogeneous  integral polynomial in terms of the form $F^\delta_{k,l}$ and $H^\gamma_t$, where $i\le k<l\le j$ and $i\le t\le j-\cond_{\bar\jmath\in M}$.
\end{proposition}

\begin{corollary}\label{corollary:ops:1}
$H^\delta_l$ supercommutes with $S_{i,j}^{\epsilon}(M)$ if $l<i$ or $l>j$.
\end{corollary}

\begin{corollary}\label{corollary:ops:2}
$E_l^\delta$ supercommutes with  $S_{i,j}^{\,\epsilon}(M)$
if $l<i$, or $l>j$, or $l=j$ and $M$ contains $\bar\jmath$.
\end{corollary}

\section{Properties of $S^{\,\epsilon}_{i,j}(\{\bar\jmath\})$ and $S^{\,\epsilon}_{i,j}(\{j\})$}\label{Q(n):lovering operators:propertiesofSijbarjandSij}
These operators are defined explicitly by~(\ref{S1}) and~(\ref{S2}). We study their properties first, before using induction to investigate general lowering operators. The main issue is to establish commutation formulas with various $E$'s and $H$'s. We will repeatedly use the commutation formulas (\ref{equation:intro:1}).

\begin{lemma}\label{lemma:ops:1}
We have
$
[E^\delta_j,S^{\,\epsilon}_{i,j}\(\{j\}\)]=
\sum_{\gamma+\sigma=\epsilon+\delta}(-1)^{\1+(\delta+\gamma)\epsilon}F_{i,j}^{\,\gamma}E_j^{\,\sigma}.
$
\end{lemma}
\begin{proof} 
The only generator appearing in the right hand side of (\ref{S2}) that does not supercommute with $E_j^\delta$ is $H_j^\si$. So $\displaystyle [E^\delta_j,S^{\,\epsilon}_{i,j}\(\{j\}\)]$ equals
\begin{align*}
&\sum_{\gamma+\sigma=\epsilon}(-1)^{\sigma} F_{i,j}^{\,\gamma}(-1)^{\gamma\sigma+\eps}(-1)^{\de\ga} [E^\delta_j,H_j^\sigma]
\\
=&\sum_{\gamma+\sigma=\epsilon}(-1)^{\sigma} F_{i,j}^{\,\gamma}(-1)^{\gamma\sigma+\eps}(-1)^{\de\ga}(-(-1)^{\de\si})E_j^{\de+\si}
\end{align*}
It remains to apply the substitution $\sigma:=\delta+\si$ and simplify the sign. 
\end{proof}

\begin{lemma}\label{lemma:ops:2}
We have
\begin{enumerate}
\item\label{lemma:ops:2:i} $E^\delta_iS_{i,j}^{\,\epsilon}(\{j\})\=(-1)^{\1+\delta(\epsilon+\1)}S_{i+1,j}^{\,\epsilon+\delta}(\{j\})\pmod{I^+_i}$ if $i+1<j$;
\item\label{lemma:ops:2:iv} $E^\delta_iS^\epsilon_{i,i+1}(\{i{+}1\})\=\cond_{\delta=\epsilon}B(i,i{+}1)\pmod{I^+_i}$.
\item\label{lemma:ops:2:ii} $E^\delta_lS_{i,j}^{\,\epsilon}(\{j\})\=0\pmod{I^+_l}$ if $l\ne i,j-1$;
\item\label{lemma:ops:2:iii} $E^\delta_{j-1}S^\epsilon_{i,j}(\{j\})\=S^{\epsilon+\delta}_{i,j-1}(\{j{-}1\})\pmod{I^+_{j-1}}$ if $i+1<j$;
\end{enumerate}
\end{lemma}

\begin{proof} We write ``$\equiv$'' for ``$\equiv\pmod{I}$'' for an appropriate ideal $I$, which is clear from the context.

(\ref{lemma:ops:2:i}) By~(\ref{S2}), we get
\begin{align*}
E^\delta_iS_{i,j}^{\,\epsilon}(\{j\})\=-\sum_{\gamma+\sigma=\epsilon}(-1)^{\sigma+\delta\gamma}F_{i+1,j}^{\,\gamma+\delta}(H_i^\sigma+(-1)^{\gamma\sigma+\eps} H_j^\sigma)
\\
-\sum_{\gamma+\sigma=\epsilon}(-1)^{\gamma\epsilon+\delta\gamma}\sum_{i+1<k<j}\!F^{\gamma+\delta}_{i+1,k}F^\sigma_{k,j}
+\sum_{\gamma+\sigma=\epsilon}(-1)^{\gamma\epsilon}(H_i^{\delta+\gamma}-(-1)^{\delta\gamma}H_{i+1}^{\delta+\gamma})F^\sigma_{i+1,j}
\\
\shoveleft{
=-\!\!\sum_{\gamma+\sigma=\epsilon+\delta}
\Big((-1)^{\sigma+\delta(\gamma+\delta)}F_{i+1,j}^{\,\gamma}(H_i^\sigma+(-1)^{(\gamma+\delta)\sigma+(\gamma+\delta)+\sigma} H_j^\sigma)}
\\
+
(-1)^{(\gamma+\delta)\epsilon+\delta(\gamma+\delta)}
\hspace{-2mm}
\sum_{i+1<k<j}F^{\gamma}_{i+1,k}F^\sigma_{k,j}
-(-1)^{(\gamma+\delta)\epsilon}(H_i^{\gamma}
-(-1)^{\delta(\gamma+\delta)}H_{i+1}^{\gamma})F^\sigma_{i+1,j}\Big).
\end{align*}
Here we have applied the substitution $\gamma:=\gamma+\delta$.
Now, the last summand equals
$$
-(-1)^{(\ga+\de)\eps+\gamma\sigma}F^\sigma_{i+1,j}(H_i^{\gamma}-(-1)^{\delta(\gamma+\delta)}H_{i+1}^{\gamma})\\
-(-1)^{(\gamma+\delta)\epsilon+\delta(\gamma+\delta)+\gamma\sigma}F^{\gamma+\sigma}_{i+1,j}.
$$
By considering the cases $\eps+\de=\0$ and $\eps+\de=1$ separately, we observe that the signs match in such a way that
$$
\sum_{\gamma+\sigma=\epsilon+\delta}
(-1)^{(\gamma+\delta)\epsilon+\delta(\gamma+\delta)+\gamma\sigma}
F^{\gamma+\sigma}_{i+1,j}=0.
$$
Next, swapping $\ga$ and $\si$ we obtain
\begin{align*}
\sum_{\gamma+\sigma=\epsilon+\delta}
(-1)^{(\ga+\de)\eps+\gamma\sigma}F^\sigma_{i+1,j}(H_i^{\gamma}-(-1)^{\delta(\gamma+\delta)}H_{i+1}^{\gamma})
\\
=\sum_{\gamma+\sigma=\epsilon+\delta}
(-1)^{(\si+\de)\eps+\sigma\gamma}F^\gamma_{i+1,j}(H_i^{\sigma}-(-1)^{\delta(\sigma+\delta)}H_{i+1}^{\sigma}).
\end{align*}
Thus, modulo $I_i^+$, the expression $E^\delta_iS_{i,j}^{\,\epsilon}(\{j\})$ is equal to
\begin{align*}
-\sum_{\gamma+\sigma=\epsilon+\delta}\Big(
(-1)^{\sigma+\delta(\gamma+\delta)}F_{i+1,j}^{\,\gamma}(H_i^\sigma+(-1)^{(\gamma+\delta)\sigma+(\gamma+\delta)+\sigma} H_j^\sigma)
\\
+(-1)^{(\gamma+\delta)\epsilon+\delta(\gamma+\delta)}
\hspace{-3mm}
\sum_{i+1<k<j}
\hspace{-3mm}
F^{\gamma}_{i+1,k}F^\sigma_{k,j}
-(-1)^{(\si+\de)\eps+\sigma\gamma}F^\gamma_{i+1,j}(H_i^{\sigma}-(-1)^{\delta(\sigma+\delta)}H_{i+1}^{\sigma})\Big).
\end{align*}
The first and the third terms inside the summation give
\begin{align*}
F_{i+1,j}^{\,\gamma}\bigl((-1)^{(\si+\de)\eps+\sigma\gamma+\delta(\sigma+\delta)}H_{i+1}^\sigma
+(-1)^{\sigma+\delta(\gamma+\delta)+(\gamma+\delta)\sigma+\gamma+\delta+\sigma}H_j^\sigma\bigr),
\end{align*}
and the result follows by comparing the signs with ~(\ref{S2}).

(\ref{lemma:ops:2:iv}) By definitions, we get using commutation relations
\begin{align*}
E^\delta_iS^\epsilon_{i,i+1}(\{i{+}1\})\equiv
\sum_{\gamma+\sigma=\epsilon}(-1)^{\sigma}
(H_i^{\delta+\gamma}-(-1)^{\delta\gamma}H_{i+1}^{\delta+\gamma})(H_i^\sigma+(-1)^{\gamma\sigma+\gamma+\sigma}H_{i+1}^\sigma)
\\
=\sum_{\gamma+\sigma=\epsilon}(-1)^{\sigma}
\Big(H_i^{\delta+\gamma}H_i^\sigma
+(-1)^{\gamma\sigma+\gamma+\sigma}H_i^{\delta+\gamma}H_{i+1}^\sigma
\\
-(-1)^{\delta\gamma+\sigma(\delta+\gamma)}H_i^\sigma H_{i+1}^{\delta+\gamma}-(-1)^{\delta\gamma+\gamma\sigma+\gamma+\sigma}H_{i+1}^{\delta+\gamma}H_{i+1}^\sigma\Big)\\
=H_i^{\delta+\epsilon}H_i+(-1)^{\epsilon}H_i^{\delta+\epsilon}H_{i+1}
-(-1)^{\delta\epsilon}H_iH_{i+1}^{\delta+\epsilon}-(-1)^{\delta\epsilon+\epsilon}H_{i+1}^{\delta+\epsilon}H_{i+1}
-H_i^{\delta+\epsilon+\1}\bar H_i
\\
+H_i^{\delta+\epsilon+\1}\bar H_{i+1}
-(-1)^{\delta\epsilon+\epsilon}\bar H_iH_{i+1}^{\delta+\epsilon+\1}-(-1)^{\delta\epsilon+\delta}H_{i+1}^{\delta+\epsilon+\1}\bar H_{i+1}).
\end{align*}
Considering separately the cases $\delta=\epsilon$ and  $\delta=\epsilon+\1$, we obtain the required result.

(\ref{lemma:ops:2:ii}) It suffices to consider the case $i<l<j-1$. Note that $E^\delta_l$ supercommutes with
$F_{i,j}^{\,\gamma}$ as well as with $F_{i,k}^\gamma$ for $k\ne l+1$ and with $F^\sigma_{k,j}$ for $k\ne l$.
Hence
\begin{align*}
E^\delta_lS_{i,j}^{\,\epsilon}(\{j\})
\=\sum_{\gamma+\sigma=\epsilon}(-1)^{\gamma\epsilon+\delta\gamma}F^\gamma_{i,l}E^\delta_lF^\sigma_{l,j}
+\sum_{\gamma+\sigma=\epsilon}(-1)^{\gamma\epsilon}E^\delta_lF^\gamma_{i,l+1}F^\sigma_{l+1,j}
\\
\=-\sum_{\gamma+\sigma=\epsilon}(-1)^{\gamma\epsilon+\delta\gamma+\sigma\delta}F^\gamma_{i,l}F^{\sigma+\delta}_{l+1,j}
+\sum_{\gamma+\sigma=\epsilon}(-1)^{\gamma\epsilon}F^{\gamma+\delta}_{i,l}F^\sigma_{l+1,j}
\\
=-\sum_{\gamma+\sigma'=\epsilon+\delta}(-1)^{\gamma\epsilon+\delta\gamma+(\sigma'+\delta)\delta}F^\gamma_{i,l}F^{\sigma'}_{l+1,j}
+\sum_{\gamma'+\sigma=\epsilon+\delta}(-1)^{(\gamma'+\delta)\epsilon}F^{\gamma'}_{i,l}F^\sigma_{l+1,j}=0.
\end{align*}
Here we applied the substitutions $\sigma':=\sigma+\delta$ and $\gamma':=\gamma+\delta$.

(\ref{lemma:ops:2:iii}) Using commutation relations, we can write $E^\delta_{j-1}S^\epsilon_{i,j}(\{j\})$ modulo $I_{j-1}$ as
\begin{align*}
\sum_{\gamma+\sigma=\epsilon}\Big((-1)^{\sigma}F_{i,j-1}^{\delta+\gamma}(H_i^\sigma+(-1)^{\gamma\sigma+\eps} H_j^\sigma)
+(-1)^{\gamma\epsilon+\delta\gamma}\sum_{i<k<j}F^\gamma_{i,k}E^{\,\delta}_{j-1}F^\sigma_{k,j}\Big)
\\
=\sum_{\gamma+\sigma=\epsilon}\Big((-1)^{\sigma}F_{i,j-1}^{\,\delta+\gamma}(H_i^\sigma+(-1)^{\gamma\sigma+\eps} H_j^\sigma)+
(-1)^{\gamma\epsilon+\delta\gamma}\sum_{i<k<j-1}\!F^\gamma_{i,k}F^{\,\delta+\sigma}_{k,j-1}
\\
+(-1)^{\gamma\epsilon+\delta\gamma}F^\gamma_{i,j-1}(H^{\delta+\sigma}_{j-1}-(-1)^{\delta\sigma}H^{\delta+\sigma}_j)\Big)
\end{align*}
\begin{align*}
=\sum_{\gamma+\sigma=\epsilon+\delta}\Big((-1)^{\sigma}F_{i,j-1}^{\gamma}(H_i^\sigma+(-1)^{(\gamma+\delta)\sigma+\gamma+\delta+\sigma} H_j^\sigma)
\\
+(-1)^{\gamma\epsilon+\delta\gamma}\sum_{i<k<j-1}F^\gamma_{i,k}F^{\,\sigma}_{k,j-1}
+(-1)^{\gamma\epsilon+\delta\gamma}F^\gamma_{i,j-1}(H^{\sigma}_{j-1}-(-1)^{\delta(\sigma+\delta)}H^{\sigma}_j)\Big),
\end{align*}
where we have applied the substitutions $\gamma:=\gamma+\delta$ in the first term and $\sigma:=\sigma+\delta$ for the second and the third terms.
It is easy to see that the last expression equals $S^{\epsilon+\delta}_{i,j-1}(\{j{-}1\})$.
\end{proof}

\begin{lemma}\label{lemma:ops:5}
We have
\begin{enumerate}
\item\label{lemma:ops:5:i}  $[H_i^\delta, S_{i,j}^\epsilon(\{j\})]=(-1)^{\1+\delta(\epsilon+\1)}S_{i,j}^{\epsilon+\delta}(\{j\})$.
\item\label{lemma:ops:5:ii} $[H_l^\delta, S_{i,j}^\epsilon(\{j\})]=0$ if $l\ne i,j$.
\item\label{lemma:ops:5:iii} $[H_j^\delta, S_{i,j}^\epsilon(\{j\})]=S_{i,j}^{\epsilon+\delta}(\{j\})$.
\end{enumerate}
\end{lemma}
\begin{proof}
(\ref{lemma:ops:5:i}) As $[H_i^\delta,H_j^\si]=0$ and $[H_i^\delta,F^\sigma_{k,j}]=0$ for $k>i$, we have by Leibnitz rule that $[H_i^\delta, S_{i,j}^\epsilon(\{j\})]$ equals
\begin{align*}
\sum_{\gamma+\sigma=\epsilon}\Big((-1)^{\sigma} [H_i^\delta,F_{i,j}^{\gamma}](H_i^\sigma+(-1)^{\gamma\sigma+\eps} H_j^\sigma)
+(-1)^{\sigma+\ga\de} F_{i,j}^{\gamma}[H_i^\delta,H_i^\sigma]
\\
+(-1)^{\gamma\epsilon}\sum_{i<k<j}[H_i^\delta, F^\gamma_{i,k}]F^\sigma_{k,j}\Big).
\end{align*}
Now substitute $-(-1)^{\de\ga}F_{i,j}^{\ga+\de}$ for $[H_i^\delta,F_{i,j}^{\gamma}]$,
$-(-1)^{\de\ga}F_{i,k}^{\ga+\de}$ for $[H_i^\delta,F_{i,k}^{\gamma}]$,
and $(1-(-1)^{\de\si})H_i^{\de+\si}$ for $[H_i^\delta,H_i^\sigma]$, and simplify.

(\ref{lemma:ops:5:ii}) and (\ref{lemma:ops:5:iii}) are proved similarly. \end{proof}

The analogues of Lemmas~\ref{lemma:ops:2} and~\ref{lemma:ops:5} for $S_{i,j}^{\,\epsilon}(\{\bar\jmath\})$ are much easier to prove, so we just record those omitting the proofs.

\begin{lemma}\label{lemma:ops:2'}
We have
\begin{enumerate}
\item\label{lemma:ops:2':i} $E^\delta_iS_{i,j}^{\,\epsilon}(\{\bar\jmath\})\=(-1)^{\1+\delta\epsilon}S_{i+1,j}^{\,\epsilon+\delta}(\{\bar\jmath\})\pmod{I_i^+}$ if $i+1<j$;
\item\label{lemma:ops:2':iv} $E^\delta_iS^\epsilon_{i,i+1}(\{\odd{i{+}1}\})\=H^{\delta+\epsilon}_i-(-1)^{\delta\epsilon}H^{\delta+\epsilon}_{i+1}\pmod{I_i^+}$.
\item\label{lemma:ops:2':ii} $E^\delta_lS_{i,j}^{\,\epsilon}(\{\bar\jmath\})\=0\pmod{I_l^+}$ if $l\ne i,j-1$;
\item\label{lemma:ops:2':iii} $E^\delta_{j-1}S^\epsilon_{i,j}(\{\bar\jmath\})\=S^{\epsilon+\delta}_{i,j-1}(\{\odd{j{-}1}\})\pmod{I_{j-1}^+}$ if $i+1<j$;
\end{enumerate}
\end{lemma}

\begin{lemma}\label{lemma:ops:5'}
We have
\begin{enumerate}
\item\label{lemma:ops:5':i}  $[H_i^\delta, S_{i,j}^{\,\epsilon}(\{\bar\jmath\})]=(-1)^{\1+\delta\epsilon}S_{i,j}^{\,\epsilon+\delta}(\{\bar\jmath\})$.
\item\label{lemma:ops:5':ii} $[H_l^\delta, S_{i,j}^{\,\epsilon}(\{\bar\jmath\})]=0$ if $l\ne i,j$.
\item\label{lemma:ops:5':iii} $[H_j^\delta, S_{i,j}^{\,\epsilon}(\{\bar\jmath\})]=S_{i,j}^{\,\epsilon+\delta}(\{\bar\jmath\})$.
\end{enumerate}
\end{lemma}

\section{Supercommutator $[H_k^\delta,S_{i,j}^{\,\epsilon}(M)]$}

\begin{lemma}\label{lemma:ops:5.5}
We have
\begin{enumerate}
\item\label{lemma:ops:5.5:i} $[H^\delta_i,S_{i,j}^{\,\epsilon}(M)]=(-1)^{\1+\delta(\epsilon+\1+\|M\|)}S_{i,j}^{\,\epsilon+\delta}(M)$;
\item\label{lemma:ops:5.5:ii}
$[H^\delta_l,S_{i,j}^{\,\epsilon}(M)]=0$ if $l\ne i,j$.
\item\label{lemma:ops:5.5:iii}
$[H^\delta_j,S_{i,j}^{\,\epsilon}(M)]=S_{i,j}^{\,\epsilon+\delta}(M)$
\end{enumerate}
\end{lemma}
\begin{proof}
We apply induction on $\height M$. The cases $M=\{j\}$ and $M=\{\bar\jmath\}$ come from  Lemmas~\ref{lemma:ops:5} and~\ref{lemma:ops:5'}.
We assume now that $M$ is distinct from $\{j\}$ and $\{\bar\jmath\}$.

(\ref{lemma:ops:5.5:i})
{\it Case~1: $\odd{i+1}\in M$.} By~(\ref{S3}), Lemma~\ref{lemma:ops:5'}(\ref{lemma:ops:5':i}),
the inductive hypothesis and Corollary~\ref{corollary:ops:1}, we get that $[H^\delta_i,S_{i,j}^{\,\epsilon}(M)]$ equals
\begin{align*}
\sum_{\gamma+\sigma=\epsilon}
\Big(({-}1)^{\gamma(\1+\epsilon+\|M_{(i+1..j]}\|)}[H^\delta_i,S^{\gamma}_{i,i+1}\(\left\{\odd{i{+}1}\right\}\)]S_{i+1,j}^{\,\sigma}\(M_{(i+1..j]}\)
\\
+
(-1)^{\sigma(\1+\|M\|)}[H^\delta_i,S_{i,j}^{\,\gamma}\(M\setminus\left\{\odd{i{+}1}\right\}\)]H^\sigma_i
\\
+
(-1)^{\sigma(\1+\|M\|)+\de\ga}S_{i,j}^{\gamma}\(M\setminus\left\{\odd{i{+}1}\right\}\)[H^\delta_i,H^\sigma_i]\Big)
\\
=
\sum_{\gamma+\sigma=\epsilon}
\Big(({-}1)^{\gamma(\1+\epsilon+\|M_{(i+1..j]}\|)+\1+\de\ga}S^{\de+\gamma}_{i,i+1}\(\left\{\odd{i{+}1}\right\}\) S_{i+1,j}^{\,\sigma}\(M_{(i+1..j]}\)
\\
+
(-1)^{\sigma(\1+\|M\|)+\1+\delta(\ga+\1+\|M\setminus\left\{\odd{i{+}1}\right\}\|)}S_{i,j}^{\de+\gamma}\(M\setminus\left\{\odd{i{+}1}\right\}\)H^\sigma_i
\\
+
(-1)^{\sigma(\1+\|M\|)+\de\ga}S_{i,j}^{\,\gamma}\(M\setminus\left\{\odd{i{+}1}\right\}\)(1-(-1)^{\delta\sigma})H^{\de+\sigma}_i\Big)
\\
=
\sum_{\gamma+\sigma=\epsilon+\de}
({-}1)^{(\gamma+\de)(\1+\epsilon+\|M_{(i+1..j]}\|)+\1+\de(\ga+\de)}
S^{\gamma}_{i,i+1}\(\left\{\odd{i{+}1}\right\}\) S_{i+1,j}^{\,\sigma}\(M_{(i+1..j]}\)
\\
+\sum_{\gamma+\sigma=\epsilon+\de}
(-1)^{\sigma(\1+\|M\|)+\1+\delta(\ga+\de+\|M\|)}S_{i,j}^{\gamma}\(M\setminus\left\{\odd{i{+}1}\right\}\)H^\sigma_i
\\
+\sum_{\gamma+\sigma=\epsilon+\de}
(-1)^{(\sigma+\de)(\1+\|M\|)+\de\ga}S_{i,j}^{\,\gamma}\(M\setminus\left\{\odd{i{+}1}\right\}\)
(1-(-1)^{\delta(\sigma+\delta)})H^{\si}_i,
\end{align*}
which is easily checked to equal to
$(-1)^{\1+\delta(\epsilon+\1+\|M\|)}S_{i,j}^{\,\epsilon+\delta}(M)$.

{\it Case~2: $i{+}1\in M$.}
By~(\ref{S4}), and Lemma~\ref{lemma:ops:5}(\ref{lemma:ops:5:i}), the inductive hypothesis and Corollary~\ref{corollary:ops:1},
we get that $[H^\delta_i,S_{i,j}^{\,\epsilon}(M)]$ equals
\begin{align*}
\sum_{\gamma+\sigma=\epsilon}
(-1)^{\gamma(\1+\epsilon+\|M_{(i+1..j]}\|)}
\left[H^\delta_i,S^{\,\gamma}_{i,i+1}\(\left\{i{+}1\right\}\)\right]
S_{i+1,j}^{\,\sigma}\(M_{(i+1..j]}\)
\\
+\left[H^\delta_i,S_{i,j}^{\,\epsilon}(M\setminus\{i{+}1\})\right]\,C(i,i{+}1)\\
=\sum_{\gamma+\sigma=\epsilon}
(-1)^{\gamma(\1+\epsilon+\|M_{(i+1..j]}\|)+\1+\delta(\gamma+\1)}
S^{\,\gamma+\delta}_{i,i+1}\(\left\{i{+}1\right\}\)
S_{i+1,j}^{\,\sigma}\(M_{(i+1..j]}\)
\\
+(-1)^{\1+\delta(\epsilon+\1+\|M_{(i+1..j]}\|)}S_{i,j}^{\,\epsilon+\delta}(M\setminus\{i{+}1\})\,C(i,i{+}1)
\\
=\sum_{\gamma+\sigma=\epsilon+\delta}
(-1)^{(\gamma+\delta)(\1+\epsilon+\|M_{(i+1..j]}\|)+\1+\delta(\gamma+\delta+\1)}
S^{\,\gamma}_{i,i+1}\(\left\{i{+}1\right\}\)
S_{i+1,j}^{\,\sigma}\(M_{(i+1..j]}\)
\\
+(-1)^{\1+\delta(\epsilon+\1+\|M\|)}S_{i,j}^{\,\epsilon+\delta}(M\setminus\{i{+}1\})\,C(i,i{+}1),
\end{align*}
which is easily checked to equal to $(-1)^{\1+\delta(\epsilon+\1+\|M\|)}S_{i,j}^{\,\epsilon+\delta}(M)$.

{\it Case~3: $i+1<\min M=\odd m<j$.} By~(\ref{S5}), Lemma~\ref{lemma:ops:5'}\ref{lemma:ops:5':i},
the inductive hypothesis and Corollary~\ref{corollary:ops:1}, $[H_i^\delta, S_{i,j}^\epsilon(M)]$ equals
\begin{align*}
\sum_{\gamma+\sigma=\epsilon}(-1)^{\gamma(\1+\epsilon+\|M_{(m..j]}\|)}[H_i^\delta,S_{i,m}^{\,\gamma}\(\left\{\bar m\right\}\)] S_{m,j}^{\,\sigma}\(M_{(m..j]}\)
+[H_i^\delta,S_{i,j}^{\,\epsilon}\(M_{\odd{\vphantom{1}m}\mapsto\odd{m{-}1}}\)]
\\
=\sum_{\gamma+\sigma=\epsilon}(-1)^{\gamma(\1+\epsilon+\|M_{(m..j]}\|)+\1+\delta\gamma}\,S_{i,m}^{\,\gamma+\delta}\(\left\{\bar m\right\}\)S_{m,j}^{\,\sigma}\(M_{(m..j]}\)
\\
+(-1)^{\1+\delta(\epsilon+\1+\|M_{\odd{\vphantom{1}m}\mapsto\odd{m{-}1}}\|)}S_{i,j}^{\,\epsilon+\delta}\(M_{\odd{\vphantom{1}m}\mapsto\odd{m{-}1}}\)
\\
=\sum_{\gamma+\sigma=\epsilon+\de}(-1)^{(\gamma+\de)(\1+\epsilon+\|M_{(m..j]}\|)+\1+\delta(\gamma+\de)}\,S_{i,m}^{\gamma}\(\left\{\bar m\right\}\)S_{m,j}^{\,\sigma}\(M_{(m..j]}\)
\end{align*}
\begin{align*}
+(-1)^{\1+\delta(\epsilon+\1+\|M\|)}S_{i,j}^{\,\epsilon+\delta}\(M_{\odd{\vphantom{1}m}\mapsto\odd{m{-}1}}\),
\end{align*}
which is easily checked to equal to $(-1)^{\1+\delta(\epsilon+\1+\|M\|)}S_{i,j}^{\,\epsilon+\delta}(M)$.

{\it Case~4: $i+1<\min M=m<j$.} This case is similar to Case 2.

(\ref{lemma:ops:5.5:ii}) If $l\not\sim\min M$ then the required formula follows immediately from~(\ref{S3})--(\ref{S6}),
parts~(ii) of Lemmas~\ref{lemma:ops:5} and~\ref{lemma:ops:5'} and the inductive hypothesis. So let $l\sim\min M$. By~(\ref{S3})--(\ref{S6}), the inductive hypothesis, and parts~(iii) of Lemmas~\ref{lemma:ops:5},~\ref{lemma:ops:5'}, the supercommutator $[H_l^\delta, S^{\,\epsilon}_{i,j}(M)]$ equals
\begin{align*}
\sum_{\gamma+\sigma=\epsilon}(-1)^{\gamma(\1+\epsilon+\|M_{(l..j]}\|)}
\Big(
[H_l^\delta,S^{\,\gamma}_{i,l}(\{\min M\})]S_{l,j}^{\,\sigma}\(M_{(l..j]}\)
\\
+
(-1)^{\de\gamma}S^{\gamma}_{i,l}(\{\min M\})[H_l^\delta, S_{l,j}^{\,\sigma}\(M_{(l..j]}\)]\Big)
\\
=\sum_{\gamma+\sigma=\epsilon}(-1)^{\gamma(\1+\epsilon+\|M_{(l..j]}\|)}
\Big(S^{\,\gamma+\delta}_{i,l}(\{\min M\}) S_{l,j}^{\,\sigma}\(M_{(l..j]}\)
\\
+
(-1)^{\delta\gamma+\1+\delta(\sigma+\1+\|M_{(l..j]}\|)}S^{\,\gamma}_{i,l}(\{\min M\})S_{l,j}^{\,\sigma+\delta}\(M_{(l..j]}\)\Big)
\\
=\sum_{\gamma'+\sigma=\epsilon+\delta}(-1)^{(\gamma'+\delta)(\1+\epsilon+\|M_{(l..j]}\|)} S^{\,\gamma'}_{i,l}(\{\min M\}) S_{l,j}^{\,\sigma}\(M_{(l..j]}\)
\\
+
\sum_{\gamma+\sigma'=\epsilon+\delta}(-1)^{\gamma(\1+\epsilon+\|M_{(l..j]}\|)+\delta\gamma+\1+\delta(\sigma'+\delta+\1+\|M_{(l..j]}\|)}
S^{\,\gamma}_{i,l}(\{\min M\})S_{l,j}^{\,\sigma'}\(M_{(l..j]}\).
\end{align*}
Here, as usual, we have applied the substitutions $\gamma'=\gamma+\delta$ and $\sigma'=\sigma+\delta$. The last expression is now easily checked to be zero.

(\ref{lemma:ops:5.5:iii})  is similar to (\ref{lemma:ops:5.5:i}) but much easier, so we skip the details.
\end{proof}

\section{Supercommutator $[E_j^\delta,S_{i,j}^{\,\epsilon}(M)]$}
By Corollary~\ref{corollary:ops:2}, $[E_j^\delta,S_{i,j}^{\,\epsilon}(M)]=0$ if $\bar\jmath\in M$.
So we just need to compute the supercommutator in the case where $j\in M$.

\begin{lemma}\label{lemma:socle:1}
Let $j\in M$.
Then
$$
[E^{\delta}_j,S^{\,\epsilon}_{i,j}\(M\)]=
\sum_{\gamma+\sigma=\epsilon+\delta}(-1)^{\1+(\delta+\gamma)\epsilon}S_{i,j}^{\,\gamma}(M_{j\mapsto \bar\jmath})E_j^{\,\sigma}.
$$
\end{lemma}
\begin{proof}
We apply induction on $\height M$. The base case $M=\{j\}$ follows from Lemma~\ref{lemma:ops:1}. Denote $L:=M_{j\mapsto \bar\jmath}$.

{\it Case~1: $\min M=\odd{i+1}$}.
By~(\ref{S3}) and the inductive hypothesis, $[E^{\,\delta}_j,S^{\,\epsilon}_{i,j}(M)]$ equals
\begin{align*}
\sum_{\rho+\pi=\epsilon}\Big(
({-}1)^{\rho(\1+\epsilon+\|M_{(i+1..j]}\|)+\delta\rho}S^{\,\rho}_{i,i+1}\(\left\{\odd{i{+}1}\right\}\)[E^{\,\delta}_j,S_{i+1,j}^{\,\pi}\(M_{(i+1..j]}\)]
\\
+(-1)^{\pi(\1+\|M\|)}[E^{\,\delta}_j,S_{i,j}^{\,\rho}\(M\setminus\left\{\odd{i{+}1}\right\}\)]H^\pi_i\Big)
\\
=
\hspace{-2mm}
\sum_{\rho+\pi=\epsilon}
({-}1)^{\rho(\1+\epsilon+\|M_{(i+1..j]}\|)+\delta\rho}S^{\,\rho}_{i,i+1}(\left\{\odd{i{+}1}\right\})
\hspace{-3mm}
 \sum_{\xi+\sigma=\pi+\delta}
 \hspace{-2mm}
 (-1)^{\1+(\delta+\xi)\pi}S_{i+1,j}^{\,\xi}(L_{(i+1..j]})E_j^{\sigma}
\\
+\sum_{\rho+\pi=\epsilon}(-1)^{\pi(\1+\|M\|)}
\sum_{\zeta+\sigma=\rho+\delta}(-1)^{\1+(\delta+\zeta)\rho}
S_{i,j}^{\,\zeta}\(L\setminus\left\{\odd{i{+}1}\right\}\)E^{\,\sigma}_jH^\pi_i
\\
=
\hspace{-3mm}
\sum_{\rho+\xi+\sigma=\epsilon+\delta}
\hspace{-3mm}
({-}1)^{\rho(\1+\epsilon+\|M_{(i+1..j]}\|)+\delta\rho+\1+(\delta+\xi)(\rho+\epsilon)} S^{\,\rho}_{i,i+1}\(\left\{\odd{i{+}1}\right\}\)S_{i+1,j}^{\,\xi}\(L_{(i+1..j]}\)E_j^{\,\sigma}
\end{align*}
\begin{align*}
+\sum_{\zeta+\sigma+\pi=\epsilon+\delta}(-1)^{\pi(\1+\|M\|)+\1+(\delta+\zeta)(\pi+\epsilon))+\sigma\pi} S_{i,j}^{\,\zeta}\(L\setminus\left\{\odd{i{+}1}\right\}\)H^\pi_iE^{\,\sigma}_j.
\end{align*}
Substitute $\gamma=\rho+\xi$ in the first sum and $\gamma=\zeta+\pi$ in the second sum to get
\begin{align*}
\sum_{\gamma+\sigma=\epsilon+\delta}(-1)^{\1+(\delta+\gamma)\epsilon}
\Big[\sum_{\rho+\xi=\gamma}(-1)^{\rho(\1+\gamma+\|L_{(i+1..j]}\|)}S^{\,\rho}_{i,i+1}\(\left\{\odd{i{+}1}\right\}\)S_{i+1,j}^{\,\xi}\(L_{(i+1..j]}\)
\\
+\sum_{\zeta+\pi=\gamma}(-1)^{\pi(\1+\|L\|)}S_{i,j}^{\,\zeta}\(L\setminus\left\{\odd{i{+}1}\right\}\)H^\pi_i\Big]E^{\sigma}_j
=\sum_{\gamma+\sigma=\epsilon+\delta}(-1)^{\1+(\delta+\gamma)\epsilon}S_{i,j}^{\,\gamma}(L)E^{\sigma}_j.
\end{align*}

{\it Case~2: $\min N=i+1$}. This case is similar to Case 1.

{\it Case~3: $i+1<\min N=\bar m$}. This case is similar to Case 4 which we now do in detail.

{\it Case~4: $i+1<\min N=m$}. By~(\ref{S6}) and the inductive hypothesis, we have that $[E_j^{\delta},S_{i,j}^{\epsilon}(M)]$ equals
\begin{align*}
\sum_{\rho+\pi=\epsilon}
(-1)^{\rho(\1+\epsilon+\|M_{(m..j]}\|)+\delta\rho}\,S^{\,\rho}_{i,m}(\{m\})[E_j^{\delta},S_{m,j}^{\pi}\(M_{(m..j]}\)]
\\
+[E_j^{\delta},S_{i,j}^\eps(M_{m\mapsto m{-}1})]+[E_j^{\delta},S_{i,j}^\eps(M\setminus\{m\})]C(m-1,m)
\\
=\sum_{\rho+\pi=\epsilon}
(-1)^{\rho(\1+\epsilon+\|M_{(m..j]}\|)+\delta\rho}\,S^{\,\rho}_{i,m}(\{m\})
\sum_{\xi+\sigma=\pi+\delta}(-1)^{\1+(\delta+\xi)\pi}S_{m,j}^{\,\xi}\(L_{(m..j]}\)E_j^{\sigma}
\\
+\sum_{\gamma+\sigma=\epsilon+\delta}(-1)^{\1+(\delta+\gamma)\epsilon}S_{i,j}^{\,\gamma}\(L_{m\mapsto m{-}1}\)E_j^{\sigma}
\\
+\sum_{\gamma+\sigma=\epsilon+\delta}(-1)^{\1+(\delta+\gamma)\epsilon}S_{i,j}^{\,\gamma}\(L\setminus\{m\}\)E_j^{\sigma}C(m{-}1,m)
\\
=\sum_{\rho+\xi+\sigma=\epsilon+\delta}(-1)^{\rho(\1+\epsilon+\|M_{(m..j]}\|)+\delta\rho+\1+(\delta+\xi)(\rho+\epsilon)}S^{\,\rho}_{i,m}(\{m\})S_{m,j}^{\,\xi}\(L_{(m..j]}\)E_j^{\sigma}
\\
+\sum_{\gamma+\sigma=\epsilon+\delta}(-1)^{\1+(\delta+\gamma)\epsilon}S_{i,j}^{\,\gamma}\(L_{m\mapsto m{-}1}\)E_j^{\sigma}\\
+\sum_{\gamma+\sigma=\epsilon+\delta}(-1)^{\1+(\delta+\gamma)\epsilon}S_{i,j}^{\,\gamma}\(L\setminus\{m\}\)C(m{-}1,m)E_j^{\sigma}.
\end{align*}
Introducing the new parameter $\gamma=\rho+\xi$ in first sum, we get
\begin{align*}
\sum_{\gamma+\sigma=\epsilon+\delta}(-1)^{\1+(\delta+\gamma)\epsilon}
\Big[\sum_{\rho+\xi=\gamma}(-1)^{\rho(\1+\gamma+\|L_{(m..j]}\|)}\,S^{\,\rho}_{i,m}(\{m\})S_{m,j}^{\,\xi}\(L_{(m..j]}\)
\\
+S_{i,j}^{\,\gamma}\(L_{m\mapsto m{-}1}\)
+S_{i,j}^{\,\gamma}\(L\setminus\{m\}\)\,C(m{-}1,m)\Big]E_j^{\sigma},
\end{align*}
which is $\sum_{\gamma+\sigma=\epsilon+\delta}(-1)^{\1+(\delta+\gamma)\epsilon}S_{i,j}^{\gamma}(L)\,E_j^{\sigma}.
$
\end{proof}

\section{More on $E_l^{\delta}S_{i,j}^\epsilon(M)$}
First we consider the case $l=i$:

\begin{lemma}\label{lemma:ops:6}
Let $i<j-1$. Modulo $I^+_i$, we have
{
\renewcommand{\labelenumi}{{\rm \theenumi}}
\renewcommand{\theenumi}{{\rm(\roman{enumi})}}
\begin{enumerate}
\item\label{lemma:ops:6:i}  If $i+1\in M$ then $E^\delta_iS_{i,j}^{\,\epsilon}(M)\=0\,;$
\item\label{lemma:ops:6:iii} If $\odd{i{+}1}\in M$ then
\begin{align*}
E^\delta_iS_{i,j}^{\,\epsilon}(M)\=
\sum_{\gamma+\sigma=\epsilon+\delta}
(-1)^{(\epsilon+\gamma)(\1+\|M\|)+\1+\delta\epsilon} S^{\,\gamma}_{i+1,j}\(M{\setminus}\left\{\overline{i{+}1}\right\}\)H^{\,\sigma}_{i+1};
\end{align*}
\item\label{lemma:ops:6:ii}  If $i+1,\odd{i+1}\notin M$ then $E^\delta_iS_{i,j}^{\,\epsilon}(M)\=(-1)^{\1+\delta(\epsilon+\1+\|M\|)}S_{i+1,j}^{\,\epsilon+\delta}(M)$.
\end{enumerate}}
\end{lemma}
\begin{proof} We apply induction on $\height M$. The base cases $M=\{j\}$ and $M=\{\bar\jmath\}$ follow from Lemmas~\ref{lemma:ops:2} and~\ref{lemma:ops:2'}. Let $M$ be distinct from
$\{\bar\jmath\}$ and $\{j\}$.

\smallskip

{\it Case~1: $\odd{i+1}\in M$.} By~(\ref{S3}), Lemma~\ref{lemma:ops:2'}(\ref{lemma:ops:2':iv}), Corollary~\ref{corollary:ops:2},
part~\ref{lemma:ops:6:ii} of the inductive hypothesis and Lemma~\ref{lemma:ops:5.5}(\ref{lemma:ops:5.5:i}),(\ref{lemma:ops:5.5:ii}),
we get
\begin{align*}
E^\delta_iS_{i,j}^{\,\epsilon}(M)\=
\sum_{\gamma+\sigma=\epsilon}
({-}1)^{\gamma(\1+\epsilon+\|M_{(i+1..j]}\|)}(H_i^{\gamma+\delta}-(-1)^{\delta\gamma}H_{i+1}^{\gamma+\delta})\,S_{i+1,j}^{\,\sigma}\(M_{(i+1..j]}\)\\
+\sum_{\gamma+\sigma=\epsilon}(-1)^{\sigma(\1+\|M\|)+\1+\delta(\gamma+\1+\|M\setminus\{\odd{i{+}1}\}\|)}S_{i,j}^{\,\gamma+\delta}\(M\setminus\left\{\odd{i{+}1}\right\}\)H^\sigma_i\\
=\sum_{\gamma+\sigma=\epsilon}
({-}1)^{\gamma(\1+\epsilon+\|M_{(i+1..j]}\|)+\sigma(\gamma+\delta)}S_{i+1,j}^{\,\sigma}\(M_{(i+1..j]}\)(H_i^{\gamma+\delta}-(-1)^{\delta\gamma}H_{i+1}^{\gamma+\delta})\\
+\sum_{\gamma+\sigma=\epsilon}({-}1)^{\gamma(\1+\epsilon+\|M_{(i+1..j]}\|)+\delta\gamma+(\gamma+\delta)(\sigma+\1+\|M_{(i+1..j]}\|)}
 S_{i+1,j}^{\,\sigma+\gamma+\delta}\(M_{(i+1..j]}\)
 \\
+\sum_{\gamma+\sigma=\epsilon}(-1)^{\sigma(\1+\|M\|)+\1+\delta(\gamma+\1+\|M\setminus\{\odd{i{+}1}\}\|)}S_{i,j}^{\,\gamma+\delta}\(M\setminus\left\{\odd{i{+}1}\right\}\)H^\sigma_i.
\end{align*}
Considering separately the cases $\gamma=\0,\sigma=\epsilon$ and $\gamma=\1,\sigma=\epsilon+\1$,
we see that the middle sum is zero. Noting that $M\setminus\{\odd{i{+}1}\}=M_{(i+1..j]}$ and $\|M_{(i+1..j]}\|=\|M\|+\1$ and using the new parameters $\gamma':=\sigma$, $\sigma':=\gamma+\delta$ in the first sum and
$\gamma':=\gamma+\delta$, $\sigma':=\sigma$ in the last sum, we get
\begin{align*}
\sum_{\gamma'+\sigma'=\epsilon+\delta}
({-}1)^{(\sigma'+\delta)(\epsilon+\|M\|)+\gamma'\sigma'}
S_{i+1,j}^{\,\gamma'}\(M_{(i+1..j]}\)(H_i^{\sigma'}-(-1)^{\delta(\sigma'+\delta)}H_{i+1}^{\sigma'})\\
+\sum_{\gamma'+\sigma'=\epsilon+\delta}(-1)^{\sigma'(\1+\|M\|)+\1+\delta(\gamma'+\delta+\|M\|)}S_{i,j}^{\,\gamma'}\(M\setminus\left\{\odd{i{+}1}\right\}\)H^{\sigma'}_i.
\end{align*}
which gives the required formula~\ref{lemma:ops:6:iii}.

\smallskip

{\it Case~2: $i+1\in M$.} By~(\ref{S4}), Lemma~\ref{lemma:ops:2}(\ref{lemma:ops:2:iv})
and part~\ref{lemma:ops:6:ii} of the inductive hypothesis, we get
\begin{align*}
E^\delta_iS_{i,j}^{\,\epsilon}(M)\=&\sum_{\gamma+\sigma=\epsilon}(-1)^{\gamma(\1+\epsilon+\|M_{(i+1..j]}\|)}\cond_{\delta=\gamma}B(i,i+1)S_{i+1,j}^{\,\sigma}\(M_{(i+1..j]}\)
\\
&+(-1)^{\1+\delta(\epsilon+\1+\|M\setminus\{i{+}1\}\|)}S_{i+1,j}^{\,\epsilon+\delta}(M\setminus\{i{+}1\})C(i,i{+}1)
\\
=&(-1)^{\delta(\1+\epsilon+\|M_{(i+1..j]}\|)}B(i,i+1)S_{i+1,j}^{\,\epsilon+\delta}\(M_{(i+1..j]}\)
\\
&+(-1)^{\1+\delta(\epsilon+\1+\|M\setminus\{i{+}1\}\|)}S_{i+1,j}^{\,\epsilon+\delta}(M\setminus\{i{+}1\})C(i,i{+}1).
\end{align*}
The last expression is zero, as $M_{(i+1..j]}=M\setminus\{i{+}1\}$ and  $S_{i+1,j}^{\,\epsilon+\delta}\(M_{(i+1..j]}\)$ has weight $-\alpha(i+1,j)$. This proves \ref{lemma:ops:6:i}.

\smallskip

{\it Case~3: $\odd{i{+}2}=\min M<j$.} By~(\ref{S5}),
parts~\ref{lemma:ops:6:ii} and~\ref{lemma:ops:6:iii} of the inductive hypothesis and~(\ref{S3}), we get
\begin{align*}
E^\delta_iS_{i,j}^{\,\epsilon}(M)\=\sum_{\gamma+\sigma=\epsilon}(-1)^{\gamma(\1+\epsilon+\|M_{(i+2..j]}\|)+\1+\delta\gamma}
 S_{i+1,i+2}^{\,\gamma+\delta}\(\left\{\odd{i{+}2}\right\}\)S_{i+2,j}^{\,\sigma}\(M_{(i{+}2..j]}\)
\\
+\sum_{\gamma+\sigma=\epsilon+\delta}(-1)^{(\epsilon+\gamma)(\1+\|M_{\odd{\vphantom{1}i{+}2}\mapsto\odd{i{+}1}}\|)+\1+\delta\epsilon}
S_{i+1,j}^{\,\gamma}\(M_{\odd{\vphantom{1}i{+}2}\mapsto\odd{i{+}1}}{\setminus}\left\{\overline{i{+}1}\right\}\)H^{\,\sigma}_{i+1}
\\
=\sum_{\gamma'+\sigma=\epsilon+\delta}(-1)^{(\gamma'+\delta)(\epsilon+\|M\|)+\1+\delta(\gamma'+\delta)}
 S_{i+1,i+2}^{\,\gamma'}\(\left\{\odd{i{+}2}\right\}\)S_{i+2,j}^{\,\sigma}\(M_{(i{+}2..j]}\)
\\
+\sum_{\gamma+\sigma=\epsilon+\delta}(-1)^{(\epsilon+\gamma)(\1+\|M\|)+\1+\de\epsilon}
S_{i+1,j}^{\,\gamma}\(M{\setminus}\left\{\overline{i{+}2}\right\}\)H^{\,\sigma}_{i+1}
\\
=(-1)^{\1+\delta(\epsilon+\1+\|M\|)}S_{i+1,j}^{\,\epsilon+\delta}(M),
\end{align*}
as required.

\smallskip

{\it Case~4: $i{+}2<\bar m=\min M<j$.}
By~(\ref{S5}) and part~\ref{lemma:ops:6:ii} of the inductive hypothesis, we have
\begin{align*}
E^\delta_iS_{i,j}^{\,\epsilon}(M)\=\sum_{\gamma+\sigma=\epsilon}(-1)^{\gamma(\1+\epsilon+\|M_{(m..j]}\|)+\1+\delta\gamma}
S_{i+1,m}^{\,\gamma+\delta}\(\left\{\bar m\right\}\)S_{m,j}^{\,\sigma}\(M_{(m..j]}\)
\\
+(-1)^{\1+\delta(\epsilon+\1+\|M_{\odd{\vphantom{1}m}\mapsto\odd{m{-}1}}\|)}S_{i+1,j}^{\,\epsilon+\delta}\(M_{\odd{\vphantom{1}m}\mapsto\odd{m{-}1}}\)
\\
=\sum_{\gamma'+\sigma=\epsilon+\delta}(-1)^{(\gamma'+\delta)(\epsilon+\|M\|)+\1+\delta(\gamma'+\delta)}\,S_{i+1,m}^{\,\gamma'}\(\left\{\bar m\right\}\)S_{m,j}^{\,\sigma}\(M_{(m..j]}\)
\\
+(-1)^{\1+\delta(\epsilon+\1+\|M\|)}S_{i+1,j}^{\,\epsilon+\delta}\(M_{\odd{\vphantom{1}m}\mapsto\odd{m{-}1}}\)
\ =\ (-1)^{\1+\delta(\epsilon+\1+\|M\|)}S_{i+1,j}^{\,\epsilon+\delta}(M),
\end{align*}
as required.
\smallskip

{\it Case~5: $i{+}2=\min M<j$.} This case is similar to Case 3.

\smallskip

{\it Case~6: $i+2<m=\min M<j$.} This  case is similar to Case 4.
\end{proof}

Now we consider $E^\delta_lS^\epsilon_{i,j}(M)$ for $i<l<j-1$:

\begin{lemma}\label{lemma:ops:7}
Let $i<l<j-1$. Modulo $I^+_l$, we have:
\begin{enumerate}
\item\label{lemma:ops:7:i} 
If $M$ does not contain $l,\ \bar l$ or if $M$ contains $l+1$, then $E^\delta_lS^\epsilon_{i,j}(M) \=0$;
\item\label{lemma:ops:7:ii} 
If $M$ contains either $\bar l$ or $l$, and $M$ does not contain $\odd{l{+}1}$, $l{+}1$, then
\begin{align*}
E^\delta_lS^\epsilon_{i,j}(M) &\=\displaystyle\sum_{\gamma+\sigma=\epsilon+\delta}(-1)^{\1+(\delta+\gamma)(\epsilon+\1+\|M_{(l+1..j]}\|)} S^\gamma_{i,l}\bigl(M_{(i..l]}\bigr)S^{\,\sigma}_{l+1,j}\bigl(M_{(l+1..j]}\bigr);
\end{align*}
\item\label{lemma:ops:7:iii}
If $M$ contains either $\bar l$ or $l$ and $M$ contains $\odd{l{+}1}$, then
\begin{align*}
E^\delta_lS^\epsilon_{i,j}(M) \=\hspace{-2 mm}\sum_{\gamma+\sigma+\tau=\epsilon+\delta}
\hspace{-5 mm} (-1)^{\1+(\delta+\gamma)\epsilon+(\epsilon+\sigma)\|M_{(l+1..j]}\|} S_{i,l}^{\,\gamma}(M_{(i..l]})S^{\,\sigma}_{l+1,j}(M_{(l+1..j]})H_{l+1}^\tau.
\end{align*}
\end{enumerate}
\end{lemma}
\begin{proof}
Induction on  $\height M$. The base cases $M=\{j\}$ and $M=\{\bar\jmath\}$ follow from Lemmas~\ref{lemma:ops:2} and~\ref{lemma:ops:2'}.
Let $M$ be distinct from $\{j\}$ and $\{\bar\jmath\}$.

We are going consider all possible combinations of the following cases:
$$
\begin{array}{rl}
\text{I.} & \odd{i{+}1}\in M.  \\
\text{II.} & i+1\in M.\\
\text{III.} & i+1<\bar m=\min M.\\
\text{IV.} & i+1<m=\min M.
\end{array}
\quad
\begin{array}{rl}
\text{a.} & l< \min M-1.  \\
\text{b.} & l\sim \min M-1.\\
\text{c.} & l\sim \min M.\\
\text{d.} & l>\min M.
\end{array}
$$
as well as the cases
\begin{enumerate}
\item[(i)] $M$ does not contain $l,\bar l$ or $M$ contains $l+1$;
\item[(ii)] $M$ contains either $\bar l$ or $l$, and $M$ does not contain $\odd{l{+}1}$ and $l{+}1$;
\item[(iii)] $M$ contains either $\bar l$ or $l$, and $M$ contains $\odd{l{+}1}$;
\end{enumerate}
which correspond to (i),(ii),(iii) in the assumptions of the lemma. Of course, not all possible combinations of these cases can happen. The following picture will help the reader to see which cases are possible and navigate the proof.

\vspace{1 cm}

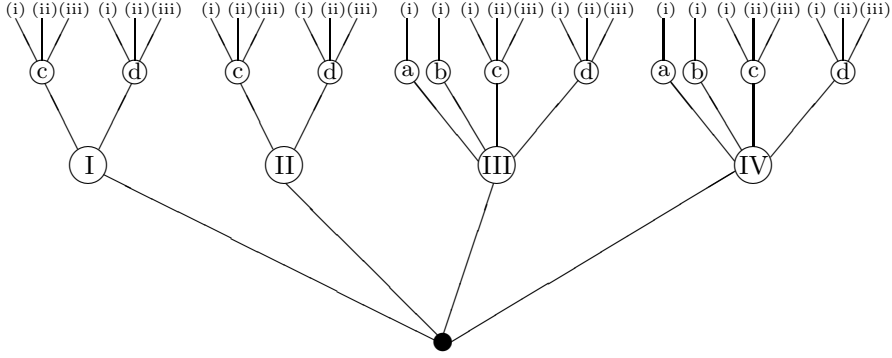
\setlength{\unitlength}{1pt}

\def\smallspear{
\put(10,4.3){\line(0,1){16}}
\put(10,0){\circle{8.5}}
\put(7,22){\tiny (i)}
}

\def\smalltrident{
\put(10,0){\circle{8.5}}
\put(8,4.1){\line(-1,2){8}}
\put(10,4.3){\line(0,1){16}}
\put(12,4.1){\line(1,2){8}}
\put(-4,22){\tiny (i)}
\put(6,22){\tiny (ii)}
\put(16,22){\tiny (iii)}
}

\def\bident{
\put(0,0){\smalltrident}
\put(35,0){\smalltrident}
\put(8,-2.3){\small c}
\put(42.5,-2.5){\small d}
\put(27.5,-35){\circle{14}}
\put(23.4,-29){\line(-1,2){12.3}}
\put(31.6,-29){\line(1,2){12.3}}
}

\def\quadrudent{
\put(0,0)\smallspear
\put(12,0)\smallspear
\put(34,0)\smalltrident
\put(68,0)\smalltrident
\put(8,-1.6){\small a}
\put(20.2,-2.5){\small b}
\put(42,-2){\small c}
\put(75.5,-2.5){\small d}
\put(44,-35){\circle{14}}
\put(39.6,-29.3){\line(-3,5){15.2}}
\put(44,-28){\line(0,1){23.7}}
\put(36.9,-33.8){\line(-4,5){24.1}}
\put(50.6,-32){\line(5,6){24.1}}
}

\begin{picture}(0,0)
\put(0,0){\bident}
\put(74,0){\bident}
\put(138,0){\quadrudent}
\put(235,0){\quadrudent}
\put(26.1,-38.5){I}
\put(98,-38.5){II}
\put(176.4,-38.4){III}
\put(273.6,-38.4){IV}
\put(161.5,-102){\circle*{7}}
\put(161.5,-101.5){\line(-1,1){59.3}}
\put(161.5,-100){\line(1,3){19.3}}
\put(161.5,-103){\line(-2,1){128}}
\put(161.5,-104){\line(5,3){110.7}}
\end{picture}

\vspace{3.9 cm}






{\it  Case~{\rm Ic(i)}}.
Note that in this case we must have $l+1=i+2\in M$.
Then $E_l^\delta S_{i,j}^{\,\epsilon}(M)\=0$ by~(\ref{S3}),
Corollary~\ref{corollary:ops:2},
and part~(\ref{lemma:ops:7:i})
of the inductive hypothesis.

\smallskip

{\it Case}\, Ic(ii).
By~(\ref{S3}), part~(\ref{lemma:ops:7:i}) of the inductive hypothesis and Lemma~\ref{lemma:ops:6}\ref{lemma:ops:6:ii},
\begin{align*}
E_l^{\,\delta}S_{i,j}^\epsilon(M)\=\sum_{\gamma+\sigma=\epsilon}
({-}1)^{\gamma(\1+\epsilon+\|M_{(i+1..j]}\|)+\delta\gamma}
 S^{\,\gamma}_{i,i+1}\(\left\{\odd{i{+}1}\right\}\)E_{i+1}^{\,\delta}S_{i+1,j}^{\,\sigma}\(M_{(i+1..j]}\)
 \\
\=\sum_{\gamma+\sigma=\epsilon}
({-}1)^{\gamma(\epsilon+\|M\|)+\delta\gamma+\1+\delta(\sigma+\1+\|M_{(i+1..j]}\|)}
S^{\,\gamma}_{i,i+1}\(\left\{\odd{i{+}1}\right\}\)S_{i+2,j}^{\,\sigma+\delta}\(M_{(i+1..j]}\).
\end{align*}
Applying the substitution $\sigma':=\sigma+\delta$, we obtain the desired formula~(\ref{lemma:ops:7:ii}).

\smallskip

{\it Case}\, Ic(iii).
By~(\ref{S3}), part~(\ref{lemma:ops:7:i}) of the inductive hypothesis and Lemma~\ref{lemma:ops:6}\ref{lemma:ops:6:iii},
\begin{align*}
E_l^{\,\delta}S_{i,j}^\epsilon(M)=\sum_{\gamma+\rho=\epsilon}
({-}1)^{\gamma(\1+\epsilon+\|M_{(i+1..j]}\|)+\delta\gamma}
S^{\,\gamma}_{i,i+1}\(\left\{\odd{i{+}1}\right\}\)E_{i+1}^{\,\delta}S_{i+1,j}^{\,\rho}\(M_{(i+1..j]}\)\\
\=\sum_{\gamma+\rho=\epsilon}
({-}1)^{\gamma(\1+\epsilon+\|M_{(i+1..j]}\|)+\delta\gamma}S^{\,\gamma}_{i,i+1}\(\left\{\odd{i{+}1}\right\}\)
\\
\times\sum_{\sigma+\tau=\rho+\delta}(-1)^{(\rho+\sigma)(\1+\|M_{(i+1..j]}\|)+\1+\delta\rho} S^{\,\sigma}_{i+2,j}\(M_{(i+1..j]}{\setminus}\left\{\overline{i{+}2}\right\}\)H^{\,\tau}_{i+2}
\\
=\sum_{\gamma+\sigma+\tau=\epsilon+\delta}(-1)^{\gamma(\1+\epsilon+\|M_{(i+1..j]}\|)+\delta\gamma+(\tau+\delta)(\1+\|M_{(i+1..j]}\|)+\1+\delta(\sigma+\tau+\delta)}
\\
\times S^{\,\gamma}_{i,i+1}\(\left\{\odd{i{+}1}\right\}\)S^{\,\sigma}_{i+2,j}\(M_{(i+2..j]}\)H^{\,\tau}_{i+2},
\end{align*}
which the desired formula~(\ref{lemma:ops:7:iii}).

{\it Case} Id(i): Note that $M_{(i+1..j]}=M\setminus\left\{\odd{i{+}1}\right\}$ does not contain $l,\bar l$
if $l,\bar l\notin M$, and $M_{(i+1..j]}=M\setminus\left\{\odd{i{+}1}\right\}$ contains $l+1$ if $l+1\in M$. Now $E_l^\delta S_{i,j}^{\,\epsilon}(M)\=0$ by~(\ref{S3}) and part~(\ref{lemma:ops:7:i})
of the inductive hypothesis.

{\it Case} Id(ii).
By~(\ref{S3}) and part~(\ref{lemma:ops:7:ii}) of the inductive hypothesis, we obtain
\begin{align*}
E_l^{\,\delta}S_{i,j}^\epsilon(M)=
\sum_{\rho+\zeta=\epsilon}
({-}1)^{\rho(\1+\epsilon+\|M_{(i+1..j]}\|)+\delta\rho}
 S^{\,\rho}_{i,i+1}\(\left\{\odd{i{+}1}\right\}\)E_l^{\,\delta}S_{i+1,j}^{\,\zeta}\(M_{(i+1..j]}\)
\end{align*}
\begin{align*}
+\sum_{\pi+\xi=\epsilon}(-1)^{\xi(\1+\|M\|)}E_l^{\,\delta}S_{i,j}^{\,\pi}\(M\setminus\left\{\odd{i{+}1}\right\}\)H^\xi_i
\\
\=\sum_{\rho+\zeta=\epsilon}
({-}1)^{\rho(\1+\epsilon+\|M_{(i+1..j]}\|)+\delta\rho}S^{\,\rho}_{i,i+1}\(\left\{\odd{i{+}1}\right\}\)
\hspace{-2 mm}\sum_{\xi+\sigma=\zeta+\delta}(-1)^{\1+(\delta+\xi)(\zeta+\1+\|M_{(l+1..j]}\|)}
\\
\times S^\xi_{i+1,l}\bigl(M_{(i+1..l]}\bigr)S^{\,\sigma}_{l+1,j}\bigl(M_{(l+1..j]}\bigr)
\\
+\sum_{\pi+\xi=\epsilon}(-1)^{\xi(\1+\|M\|)}
\sum_{\rho+\sigma=\pi+\delta}(-1)^{\1+(\delta+\rho)(\pi+\1+\|(M{\setminus}\{\odd{i{+}1}\})_{(l+1..j]}\|)}
\\
\times S^\rho_{i,l}\bigl((M{\setminus}\{\odd{i{+}1}\})_{(i..l]}\bigr)
S^{\,\sigma}_{l+1,j}\bigl((M{\setminus}\{\odd{i{+}1}\}))_{(l+1..j]}\bigr)H^\xi_i
\\
=\sum_{\rho+\xi+\sigma=\epsilon+\delta}({-}1)^{\rho(\1+\epsilon+\|M_{(i+1..j]}\|)+\delta\rho+\1+(\delta+\xi)(\xi+\sigma+\delta+\1+\|M_{(l+1..j]}\|)}
\\
\times S^{\,\rho}_{i,i+1}\(\left\{\odd{i{+}1}\right\}\)
S^\xi_{i+1,l}\bigl(M_{(i+1..l]}\bigr)S^{\,\sigma}_{l+1,j}\bigl(M_{(l+1..j]})
\\
+
\sum_{\rho+\xi+\sigma=\epsilon+\delta}
(-1)^{\xi(\!\1+\|M\|)+\1+(\delta+\rho)(\!\rho+\sigma+\delta+\1+\|M_{(l+1..j]}\|)}
\\
\times S^\rho_{i,l}(M_{(i..l]}\!\!\setminus\!\{\odd{i{+}1}\})S^{\sigma}_{l+1,j}(M_{(l+1..j]})H^\xi_i
\\
=\sum_{\gamma+\sigma=\epsilon+\delta}\sum_{\rho+\xi=\gamma}\biggl(
({-}1)^{\rho(\1+\epsilon+\|M_{(i+1..j]}\|)+\delta\rho+\1+(\delta+\xi)(\xi+\sigma+\delta+\1+\|M_{(l+1..j]}\|)}
\\
\times S^{\,\rho}_{i,i+1}\(\left\{\odd{i{+}1}\right\}\)
S^\xi_{i+1,l}\bigl(M_{(i+1..l]}\bigr)S^{\,\sigma}_{l+1,j}\bigl(M_{(l+1..j]})
\\
+(-1)^{\xi(\1+\|M\|)+\1+(\delta+\rho)(\sigma+\|M_{(l+1..j]}\|)+\xi\sigma}
S^\rho_{i,l}\bigl(M_{(i..l]}\setminus\{\odd{i{+}1}\}\bigr)H^\xi_iS^{\,\sigma}_{l+1,j}\bigl(M_{(l+1..j]}\bigr)
\biggr)
\\
=\sum_{\gamma+\sigma=\epsilon+\delta}(-1)^{\1+(\delta+\gamma)(\epsilon+\1+\|M_{(l+1..j]}\|)}
\Biggl[\sum_{\rho+\xi=\gamma}(-1)^{\rho(\1+\gamma+\|M_{(i+1..l]}\|)}
\\
\times S^{\,\rho}_{i,i+1}\(\left\{\odd{i{+}1}\right\}\)S^\xi_{i+1,l}\bigl(M_{(i+1..l]}\bigr)
\\
+\sum_{\rho+\xi=\gamma}(-1)^{\xi(\1+\|M_{(i..l]}\|)}S^\rho_{i,l}\bigl(M_{(i..l]}\setminus\{\odd{i{+}1}\}\bigr)H^\xi_i\Biggr]S^{\,\sigma}_{l+1,j}\bigl(M_{(l+1..j]}\bigr).
\end{align*}
By~(\ref{S3}), the expression in the big square brackets equals  $S_{i,l}^{\,\gamma}(M_{(i..l]})$. This gives the desired formula~(\ref{lemma:ops:7:ii}).

{\it Case} Id(iii).
By~(\ref{S3}), we get for $E_l^{\,\delta}S_{i,j}^\epsilon(M)$:
\begin{align*}
\sum_{\rho+\zeta=\epsilon}
({-}1)^{\rho(\1+\epsilon+\|M_{(i+1..j]}\|)+\delta\rho}S^{\,\rho}_{i,i+1}\(\left\{\odd{i{+}1}\right\}\)E_l^{\,\delta}S_{i+1,j}^{\,\zeta}\(M_{(i+1..j]}\)\\
+\sum_{\pi+\xi=\epsilon}(-1)^{\xi(\1+\|M\|)}E_l^{\delta}S_{i,j}^{\,\pi}\(M\setminus\left\{\odd{i{+}1}\right\}\)H^\xi_i.
\end{align*}

Applying part~(\ref{lemma:ops:7:iii}) of the inductive hypothesis, we get
\begin{align*}
\sum_{\rho+\zeta=\epsilon}
({-}1)^{\rho(\1+\epsilon+\|M_{(i+1..j]}\|)+\delta\rho}S^{\,\rho}_{i,i+1}\(\left\{\odd{i{+}1}\right\}\)
\\
\times\sum_{\xi+\sigma+\tau=\zeta+\delta}(-1)^{\1+(\delta+\xi)\zeta+(\zeta+\sigma)\|M_{(l+1..j]}\|}S_{i+1,l}^{\,\xi}\(M_{(i+1..l]}\)S^{\,\sigma}_{l+1,j}\(M_{(l+1..j]}\)H_{l+1}^\tau
\end{align*}
\begin{align*}
+\sum_{\pi+\xi=\epsilon}(-1)^{\xi(\1+\|M\|)}
\sum_{\rho+\sigma+\tau=\pi+\delta}(-1)^{\1+(\delta+\rho)\pi+(\pi+\sigma)\|(M\setminus\{\overline{i+1}\})_{(l+1..j]}\|}
\\
\times S_{i,l}^{\,\rho}\((M{\setminus}{\{\odd{i{+}1}\}})_{(i..l]}\)S^{\,\sigma}_{l+1,j}\((M{\setminus}{\{\odd{i{+}1}\}})_{(l+1..j]}\)H_{l+1}^\tau H^\xi_i
\\
=\sum_{\rho+\xi+\sigma+\tau=\epsilon+\delta}
({-}1)^{\rho(\1+\epsilon+\|M_{(i+1..j]}\|)+\delta\rho+\1+(\delta+\xi)(\xi+\sigma+\tau+\delta)+(\xi+\tau+\delta)\|M_{(l+1..j]}\|}
\\
\times S^{\,\rho}_{i,i+1}\(\left\{\odd{i{+}1}\right\}\) S_{i+1,l}^{\,\xi}\(M_{(i+1..l]}\)S^{\,\sigma}_{l+1,j}\(M_{(l+1..j]}\)H_{l+1}^\tau
\\
+\sum_{\rho+\xi+\sigma+\tau=\epsilon+\delta}(-1)^{\xi(\1+\|M\|)+\1+(\delta+\rho)(\rho+\sigma+\tau+\delta)+(\rho+\tau+\delta)\|M_{(l+1..j]}\|}
\\
\times S_{i,l}^{\,\rho}\(M_{(i..l]}\setminus{\{\odd{i{+}1}\}}\)S^{\,\sigma}_{l+1,j}\(M_{(l+1..j]}\)H_{l+1}^\tau H^\xi_i.
\\
=\sum_{\gamma+\sigma+\tau=\epsilon+\delta}\;\sum_{\rho+\xi=\gamma}
\bigg(
({-}1)^{\rho(\1+\epsilon+\|M_{(i+1..j]}\|)+\delta\rho+\1+(\delta+\xi)(\xi+\sigma+\tau+\delta)+(\xi+\tau+\delta)\|M_{(l+1..j]}\|}
\\
\times S^{\,\rho}_{i,i+1}\(\left\{\odd{i{+}1}\right\}\) S_{i+1,l}^{\,\xi}\(M_{(i+1..l]}\)S^{\,\sigma}_{l+1,j}\(M_{(l+1..j]}\)H_{l+1}^\tau
\\
+(-1)^{\xi(\1+\|M\|)+\1+(\delta+\rho)(\rho+\sigma+\tau+\delta)+(\rho+\tau+\delta)\|M_{(l+1..j]}\|+\xi(\sigma+\tau)}
\\
\times S_{i,l}^{\,\rho}\(M_{(i..l]}\setminus{\{\odd{i{+}1}\}}+\xi(\sigma+\tau)\)H^\xi_iS^{\,\sigma}_{l+1,j}\(M_{(l+1..j]}\)H_{l+1}^\tau.
\bigg)
\\
=\sum_{\gamma+\sigma+\tau=\epsilon+\delta}(-1)^{\1+(\delta+\gamma)\epsilon+(\epsilon+\sigma)\|M_{(l+1..j]}\|}
\\
\times \Biggl[\sum_{\rho+\xi=\gamma}({-}1)^{\rho(\1+\gamma+\|M_{(i+1..l]}\|)}S^{\,\rho}_{i,i+1}\!\!\(\left\{\odd{i{+}1}\right\}\)\!S_{i+1,l}^{\,\xi}\!\!\(M_{(i+1..l]}\)
\\
+\sum_{\rho+\xi=\gamma}(-1)^{\xi(\1+\|M_{(i..l]}\|)}S_{i,l}^{\,\rho}\(M_{(i..l]}\setminus{\{\odd{i{+}1}\}}\)H^\xi_i\Biggr]
S^{\,\sigma}_{l+1,j}\(M_{(l+1..j]}\)H_{l+1}^\tau.
\end{align*}
By~(\ref{S3}), the expression in the big square brackets equals $S_{i,l}^{\,\gamma}(M_{(i..l]})$. This gives the desired formula~(\ref{lemma:ops:7:iii}).

{\it  Case~{\rm IIc(i)}} is similar to case~Ic(i).

{\it Case} IIc(ii). By~(\ref{S4}), part~(\ref{lemma:ops:7:i}) of the inductive hypothesis and Lemmas~\ref{lemma:ops:1}
and~\ref{lemma:ops:6}\ref{lemma:ops:6:ii}, we obtain
\begin{align*}
E^\delta_lS_{i,j}^\epsilon(M)=\sum_{\gamma+\sigma=\epsilon}
(-1)^{\gamma(\1+\epsilon+\|M_{(i+1..j]}\|)+\delta\gamma}S^{\,\gamma}_{i,i+1}\(\left\{i{+}1\right\}\)E^\delta_lS_{i+1,j}^{\,\sigma}\(M_{(i+1..j]}\)\\
+\sum_{\gamma+\sigma=\epsilon}(-1)^{\gamma(\1+\epsilon+\|M_{(i+1..j]}\|)}
\sum_{\rho+\tau=\gamma+\delta}(-1)^{\1+(\delta+\rho)\gamma}
F_{i,j}^{\rho}E^\tau_{i+1}S_{i+1,j}^{\,\sigma}\(M_{(i+1..j]}\)
\\
=
\hspace{-2mm}
\sum_{\gamma+\sigma=\epsilon}
\hspace{-2mm}
(-1)^{\gamma(\1+\epsilon+\|M_{(i+1..j]}\|)+\delta\gamma+\1+\delta(\sigma+\1+\|M_{(i+1..j]}\|)}
S^{\,\gamma}_{i,i+1}\(\left\{i{+}1\right\}\)
S_{i+2,j}^{\,\sigma+\delta}\(M_{(i+1..j]}\)
\\
+\sum_{\gamma+\sigma=\epsilon}(-1)^{\gamma(\1+\epsilon+\|M_{(i+1..j]}\|)}\\
\times\sum_{\rho+\tau=\gamma+\delta}(-1)^{\1+(\delta+\rho)\gamma+\1+\tau(\sigma+\1+\|M_{(i+1..j]}\|)}F_{i,j}^{\rho}E^\tau_{i+1}S_{i+1,j}^{\,\sigma}\(M_{(i+1..j]}\)
\\
=\sum_{\gamma+\sigma'=\epsilon+\delta}
(-1)^{\gamma(\1+\epsilon+\|M_{(i+1..j]}\|)+\delta\gamma+\1+\delta(\sigma'+\delta+\1+\|M_{(i+1..j]}\|)}
\end{align*}
\begin{align*}
\times
S^{\,\gamma}_{i,i+1}\(\left\{i{+}1\right\}\)S_{i+2,j}^{\,\sigma'}\(M_{(i+1..j]}\)
\\
+
\hspace{-3mm}\sum_{\rho+\tau+\sigma=\epsilon+\delta}
\hspace{-5.5mm}
(-1)^{(\rho+\tau+\delta)(\1+\epsilon+\|M\|)+\1+(\delta+\rho)(\rho+\tau+\delta)+\1+\tau(\sigma+\1+\|M\|)}
\!F_{i,j}^{\rho}S_{i+1,j}^{\sigma+\tau}(M_{(i+1..j]}).
\end{align*}
The last summand is easily checked to equal zero, while the first summand yields the required
formula~(\ref{lemma:ops:7:ii}).

{\it Case} IIc(iii). By~(\ref{S4}), part~(\ref{lemma:ops:7:i}) of the inductive hypothesis and Lemmas~\ref{lemma:ops:1} and~\ref{lemma:ops:2}(\ref{lemma:ops:2:i}),
we obtain
\begin{align*}
E_l^\delta S_{i,j}^\epsilon(M)=\sum_{\gamma+\rho=\epsilon}
(-1)^{\gamma(\1+\epsilon+\|M_{(i+1..j]}\|)+\delta\gamma} S^{\,\gamma}_{i,i+1}\(\left\{i{+}1\right\}\)E_{i+1}^\delta S_{i+1,j}^{\,\rho}\(M_{(i+1..j]}\)\\
+\sum_{\gamma+\rho=\epsilon}(-1)^{\gamma(\1+\epsilon+\|M_{(i+1..j]}\|)}
\sum_{\zeta+\tau=\gamma+\delta}(-1)^{\1+(\delta+\zeta)\gamma}F_{i,i+1}^\zeta E_{i+1}^\tau S_{i+1,j}^{\,\rho}\(M_{(i+1..j]}\)
\\
=\sum_{\gamma+\rho=\epsilon}
(-1)^{\gamma(\1+\epsilon+\|M_{(i+1..j]}\|)+\delta\gamma}S^{\,\gamma}_{i,i+1}\(\left\{i{+}1\right\}\)
\\
\times
\sum_{\sigma+\tau=\rho+\delta}(-1)^{(\rho+\sigma)(\1+\|M_{(i+1..j]}\|)+\1+\delta\rho}
S^\sigma_{i+2,j}(M_{(i+1..j]}\setminus\{\overline{i{+}2}\})H^\tau_{i+2}\\
+\sum_{\gamma+\rho=\epsilon}(-1)^{\gamma(\1+\epsilon+\|M_{(i+1..j]}\|)}\sum_{\zeta+\tau=\gamma+\delta}(-1)^{\1+(\delta+\zeta)\gamma}F_{i,i+1}^\zeta\\
\times\sum_{\xi+\pi=\rho+\tau}(-1)^{(\rho+\xi)(\1+\|M_{(i+1..j]}\|)+\1+\tau\rho}
S^\xi_{i+2,j}(M_{(i+1..j]}\setminus\{\overline{i{+}2}\})H^\pi_{i+2}\\
=\sum_{\gamma+\sigma+\tau=\epsilon+\delta}(-1)^{\gamma(\1+\epsilon+\|M_{(i+1..j]}\|)+\delta\gamma+(\sigma+\tau+\delta+\sigma)(\1+\|M_{(i+1..j]}\|)}\\
\times(-1)^{\1+\delta(\sigma+\tau+\delta)}
S^{\,\gamma}_{i,i+1}\(\left\{i{+}1\right\}\)S^\sigma_{i+2,j}(M_{(i+1..j]}\setminus\{\overline{i{+}2}\})H^\tau_{i+2}\\
+\sum_{\gamma+\rho=\epsilon}(-1)^{\gamma(\1+\epsilon+\|M_{(i+1..j]}\|)}\sum_{\zeta+\xi+\pi=\epsilon+\delta}(-1)^{\1+(\delta+\zeta)\gamma+(\rho+\xi)(\1+\|M_{(i+1..j]}\|)}\\
\times(-1)^{\1+(\xi+\pi+\rho)\rho}
F_{i,i+1}^\zeta S^\xi_{i+2,j}(M_{(i+1..j]}\setminus\{\overline{i{+}2}\})H^\pi_{i+2}.
\end{align*}
The last summand contains two independent summations. Swapping them, we easily check that this summand equals zero.
The first summand is easily checked to give the required formula (ii).



{\it Cases} IId are similar to the corresponding cases Id.


{\it Case} IIIa(i). By~(\ref{S5}), Lemma~\ref{lemma:ops:2'}(\ref{lemma:ops:2':ii}) and part~(\ref{lemma:ops:7:i}) of the inductive hypothesis,
we get $E_l^\delta S_{i,j}^{\,\epsilon}(M)\=0$.

{\it\quad Case} IIIb(i). By~(\ref{S5}), Lemma~\ref{lemma:ops:2'}(\ref{lemma:ops:2':iii}) and part~(\ref{lemma:ops:7:ii}) of the inductive hypothesis, we get
\begin{align*}
E_l^{\,\delta}S_{i,j}^\epsilon(M)\=
\sum_{\gamma+\sigma=\epsilon}(-1)^{\gamma(\1+\epsilon+\|M_{(m..j]}\|)} S_{i,m-1}^{\,\gamma+\delta}\(\left\{\odd{m{-}1}\right\}\)S_{m,j}^{\,\sigma}\(M_{(m..j]}\)\\
+\sum_{\gamma+\sigma=\epsilon+\delta}(-1)^{\1+(\delta+\gamma)(\epsilon+\1+\|(M_{\odd{\vphantom{1}m}\mapsto\odd{m{-}1}})_{(m..j]}\|)}
\\
\qquad\times S^\gamma_{i,m-1}\bigl((M_{\odd{\vphantom{1}m}\mapsto\odd{m{-}1}})_{(i..m-1]}\bigr)S^{\,\sigma}_{m,j}\bigl((M_{\odd{\vphantom{1}m}\mapsto\odd{m{-}1}})_{(m..j]}\bigr)
\end{align*}
Applying the substitution $\gamma:=\gamma+\delta$ in the first sum, and taking into account that $(M_{\odd{\vphantom{1}m}\mapsto\odd{m{-}1}})_{(m..j]}=M_{(m..j]}$ and
$(M_{\odd{\vphantom{1}m}\mapsto\odd{m{-}1}})_{(i..m-1]}=\left\{\odd{m{-}1}\right\}$, it is easy to see that the last expression is zero.

\smallskip

{\it Case}~IIIc(i). In this case we must have $m{+}1\in M$. Note that $l+1\in 
M_{\odd{\vphantom{1}m}\mapsto\odd{m{-}1}}$.
So $E_l^\delta S_{i,j}^{\,\epsilon}(M)\=0$ by~(\ref{S5}), Corollary~\ref{corollary:ops:2}, Lemma~\ref{lemma:ops:6}\ref{lemma:ops:6:i} and part~(\ref{lemma:ops:7:i}) of the inductive hypothesis.

\smallskip

{\it\quad Case}~IIIc(ii).  
By the formula (\ref{S5}), part~(\ref{lemma:ops:7:i}) of the inductive hypothesis, and
Lemma~\ref{lemma:ops:6}\ref{lemma:ops:6:ii}, we get that $E_l^{\,\delta}S_{i,j}^\epsilon(M)$ is $\=$ to
\begin{align*}
\sum_{\gamma+\sigma=\epsilon}
({-}1)^{\gamma(\1+\epsilon+\|M_{(m..j]}\|)+\delta\gamma+\1+\delta(\sigma+\1+\|M_{(m..j]}\|)}S^{\,\gamma}_{i,m}\(\{\bar m\}\)S_{m+1,j}^{\,\sigma+\delta}\(M_{(m..j]}\)
\\
=\sum_{\gamma+\sigma'=\epsilon+\delta}({-}1)^{\gamma(\1+\epsilon+\|M_{(m+1..j]}\|)+\delta\gamma+\1+\delta(\sigma'+\delta+\1+\|M_{(m+1..j]}\|)}
\\
\qquad\qquad\times
S^{\,\gamma}_{i,m}\(\{\bar m\}\)S_{m+1,j}^{\,\sigma'}\(M_{(m+1..j]}\),
\end{align*}
which yields the required formula~(\ref{lemma:ops:7:ii}).

\smallskip

{\it Case}~IIIc(iii).
We have by~(\ref{S5}), part~(\ref{lemma:ops:7:i}) of the inductive hypothesis,
and Lemma~\ref{lemma:ops:6}\ref{lemma:ops:6:iii} that
$E_l^{\delta}S_{i,j}^\epsilon(M)$ is $\=$ to
\begin{align*}
\sum_{\gamma+\rho=\epsilon}
({-}1)^{\gamma(\1+\epsilon+\|M_{(m..j]}\|)+\delta\gamma}
S^{\,\gamma}_{i,m}\(\left\{\bar m\right\}\)
\\
\times\sum_{\sigma+\tau=\rho+\delta}(-1)^{(\rho+\sigma)(\1+\|M_{(m..j]}\|)+\1+\delta\rho} S^{\,\sigma}_{m+1,j}\(M_{(m..j]}{\setminus}\left\{\overline{m{+}1}\right\}\)H^{\,\tau}_{m+1}
\\
=\sum_{\gamma+\sigma+\tau=\epsilon+\delta}(-1)^{\gamma(\1+\epsilon+\|M_{(m..j]}\|)+\delta\gamma+(\tau+\delta)(\1+\|M_{(m..j]}\|)+\1+\delta(\sigma+\tau+\delta)}
\\
\times
S^{\,\gamma}_{i,i+1}\(\left\{\bar m\right\}\)S^{\,\sigma}_{m+1,j}\(M_{(m+1..j]}\)H^{\,\tau}_{m+1}.
\end{align*}
which yields the required formula~(\ref{lemma:ops:7:iii}).

\smallskip

{\it\quad Case}~IIId(i). Note that $l$ and $\bar l$ do not belong to the sets
$M_{(m..j]}$ and $M_{\odd{\vphantom{1}m}\mapsto\odd{m{-}1}}$ if $l,\bar l\notin M$
and $l+1$ belongs to these sets if $l+1\in M$.
Then $E_l^\delta S_{i,j}^{\,\epsilon}(M)\=0$ by~(\ref{S5}) and part~(\ref{lemma:ops:7:i}) of the inductive hypothesis.

\smallskip

{\it Case}~IIId(ii).
By~(\ref{S5}) and part~(\ref{lemma:ops:7:ii}) of the inductive hypothesis, we get
\begin{align*}
E_l^{\,\delta}S_{i,j}^\epsilon(M)\=\sum_{\rho+\zeta=\epsilon}
({-}1)^{\rho(\1+\epsilon+\| M_{(m..j]} \|)+\delta\rho}
 S^{\,\rho}_{i,m}\(\left\{\bar m\right\}\)E_l^{\,\delta}S_{m,j}^{\,\zeta}\(M_{(m..j]}\)\\
+\sum_{\gamma+\sigma=\epsilon+\delta}(-1)^{\1+(\delta+\gamma)(\epsilon+\1+\|(M_{\odd{\vphantom{1}m}\mapsto\odd{m{-}1}})_{(l+1..j]}\|)}
\\
\times
S^\gamma_{i,l}\bigl((M_{\odd{\vphantom{1}m}\mapsto\odd{m{-}1}})_{(i..l]}\bigr)S^{\,\sigma}_{l+1,j}\bigl((M_{\odd{\vphantom{1}m}\mapsto\odd{m{-}1}})_{(l+1..j]}\bigr)
\\
\=\sum_{\rho+\zeta=\epsilon}
({-}1)^{\rho(\1+\epsilon+\|M_{(m..j]}\|)+\delta\rho}S^{\,\rho}_{i,m}\(\left\{\bar m\right\}\)
\\
\times\sum_{\xi+\sigma=\zeta+\delta}(-1)^{\1+(\delta+\xi)(\zeta+\1+\|M_{(l+1..j]}\|)}
S^\xi_{m,l}\bigl(M_{(m..l]}\bigr)S^{\,\sigma}_{l+1,j}\bigl(M_{(l+1..j]}\bigr)
\\
+\sum_{\gamma+\sigma=\epsilon+\delta}(-1)^{\1+(\delta+\gamma)(\epsilon+\1+\|M_{(l+1..j]}\|)}
 S^\gamma_{i,l}\bigl((M_{(i..l]})_{\odd{\vphantom{1}m}\mapsto\odd{m{-}1}}\bigr)S^{\,\sigma}_{l+1,j}\bigl(M_{(l+1..j]}\bigr).
\end{align*}
Introducing 
$\gamma:=\rho+\xi$ and rearranging the first summand as in case~I.d.(ii), we get
\begin{align*}
\sum_{\gamma+\sigma=\epsilon+\delta}
\hspace{-3.5mm}
(-1)^{\1+(\delta+\gamma)(\epsilon+\1+\|M_{(l+1..j]}\|)}\Biggl[\sum_{\rho+\xi=\gamma}
\hspace{-2mm}
(-1)^{\rho(\1+\gamma+\|M_{(m..l]}\|)}\!S^{\rho}_{i,m}(\{\bar m\})S^\xi_{m,l}\bigl(M_{(m..l]}\bigr)
\\
+S^\gamma_{i,l}\bigl((M_{(i..l]})_{\odd{\vphantom{1}m}\mapsto\odd{m{-}1}}\bigr)\Biggl]S^{\,\sigma}_{l+1,j}\bigl(M_{(l+1..j]}\bigr).
\end{align*}
By~(\ref{S5}), the expression in the big square brackets equals $S_{i,l}^{\,\gamma}(M_{(i..l]})$.

\smallskip

{\it Case}~IIId(iii)
By~(\ref{S5}) and part~(\ref{lemma:ops:7:iii}) of the inductive hypothesis, we get
\begin{align*}
E_l^{\,\delta}S_{i,j}^\epsilon(M)\=\sum_{\rho+\zeta=\epsilon}
({-}1)^{\rho(\1+\epsilon+\|M_{(m..j]}\|)+\delta\rho}S^{\,\rho}_{i,m}\(\left\{\bar m\right\}\)E_l^{\,\delta}S_{m,j}^{\,\zeta}\(M_{(m..j]}\)
\\
+\sum_{\gamma+\sigma+\tau=\epsilon+\delta}(-1)^{\1+(\delta+\gamma)\epsilon+(\epsilon+\sigma)\|(M_{\odd{\vphantom{1}m}\mapsto\odd{m{-}1}})_{(l+1..j]}\|}
\\
\times S_{i,l}^{\,\gamma}\((M_{\odd{\vphantom{1}m}\mapsto\odd{m{-}1}})_{(i..l]}\)S^{\,\sigma}_{l+1,j}\((M_{\odd{\vphantom{1}m}\mapsto\odd{m{-}1}})_{(l+1..j]}\)H_{l+1}^\tau\\
\=\sum_{\rho+\zeta=\epsilon}
({-}1)^{\rho(\1+\epsilon+\|M_{(m..j]}\|)+\delta\rho}S^{\,\rho}_{i,m}\(\left\{\bar m\right\}\)
\\
\times\sum_{\xi+\sigma+\tau=\zeta+\delta}(-1)^{\1+(\delta+\xi)\zeta+(\zeta+\sigma)\|M_{(l+1..j]}\|}
 S_{m,l}^{\,\xi}\(M_{(m..l]}\)S^{\,\sigma}_{l+1,j}\(M_{(l+1..j]}\)H_{l+1}^\tau
\\
+\hspace{-3.5mm}
\sum_{\gamma+\sigma+\tau=\epsilon+\delta}
\hspace{-3.5mm}
(-1)^{\1+(\delta+\gamma)\epsilon+(\epsilon+\sigma)\|M_{(l+1..j]}\|}
S_{i,l}^{\,\gamma}\((M_{(i..l]})_{\odd{\vphantom{1}m}\mapsto\odd{m{-}1}}\)S^{\,\sigma}_{l+1,j}\(M_{(l+1..j]}\)H_{l+1}^\tau.
\end{align*}
Introducing $\gamma:=\rho+\xi$ and rearranging the first summand as in~I.d.(iii), we get
\begin{align*}
\sum_{\gamma+\sigma+\tau=\epsilon+\delta}
\hspace{-5.5mm}
(-1)^{\1+(\delta+\gamma)\epsilon+(\epsilon+\sigma)\|M_{(l+1..j]}\|}
\!\Biggl[\sum_{\rho+\xi=\gamma}
\hspace{-2mm}
({-}1)^{\rho(\1+\gamma+\|M_{(m..l]}\|)}\!S^{\,\rho}_{i,m}(\!\{\bar m\}\!)S_{m,l}^{\,\xi}(M_{(m..l]})\\
+S_{i,l}^{\,\gamma}\((M_{(i..l]})_{\odd{\vphantom{1}m}\mapsto\odd{m{-}1}}\)
\Biggl]S^{\,\sigma}_{l+1,j}\(M_{(l+1..j]}\)H_{l+1}^\tau.
\end{align*}
By~(\ref{S5}), the expression in the big square brackets equals $S_{i,l}^{\,\gamma}(M_{(i..l]})$.



{\it\quad Case}~IV is similar to cases II and III.
\end{proof}

\section{Some coefficients}\label{some_coefficients}
By Proposition~\ref{proposition:intro:1}, any element $F\in U^{\leq 0}_\Z(n)^{-\alpha(i,j)}$ can be written as a sum of the following elements:
\begin{equation}\label{eq:rev:1}
 F^{\,\epsilon_0}_{i,a_1} F^{\,\epsilon_1}_{a_1,a_2}\cdots F^{\,\epsilon_m}_{a_m,j}\,H_{i,\,a_1,\ldots,a_m,j}^{\,\epsilon_0,\ldots,\epsilon_m},
\end{equation}
where $H_{i,\,a_1,\ldots,a_m,j}^{\,\epsilon_0,\ldots,\epsilon_m}\in U^0_\Z(n)$,
$\epsilon_0,\ldots,\epsilon_m\in\{\0,\1\}$, and $i<a_1<\cdots<a_m<j$. We set $\cf_{i,j}^{\,\delta}(F):=H_{i,j}^{\,\delta}$, the ``$U^0_\Z(n)$-coefficient'' of $F_{i,j}^\de$.


\begin{lemma}\label{lemma:rev:1}
Let $M$ be a signed $(i..j]$-set all elements of which are even except $\bar\jmath$,
and $\epsilon,\delta\in\{\0,\1\}$. Then
$
\cf_{i,j}^{\,\delta}\bigl( S_{i,j}^{\,\epsilon}(M)\bigr)=\cond_{\epsilon=\delta}\prod\nolimits_{t\in M\setminus\{\bar\jmath\}} C(i,t).
$
\end{lemma}
\begin{proof} We apply induction on $\height M$. If $M=\{\bar\jmath\}$
the result is clear by~(\ref{S1}).
Let now $M\ne\{\bar\jmath\}$. We set $m:=\min M$.
If $m=i+1$, then by~(\ref{S4}) and the inductive hypothesis, we get
\begin{align*}
\cf_{i,j}^{\,\delta}( S_{i,j}^{\,\epsilon}(M))
=&\cf_{i,j}^{\,\delta}\bigl( S_{i,j}^{\,\epsilon}(M\setminus\{i+1\}\bigr) C(i,i+1)
\\
=&\cond_{\epsilon=\delta} C(i,i+1)\prod\nolimits_{t\in M\setminus\{i+1,\bar\jmath\}} C(i,t)
\!=\!\cond_{\epsilon=\delta}\prod\nolimits_{t\in M\setminus\{\bar\jmath\}}\! C(i,t).
\end{align*}

If $m>i+1$, then by~(\ref{S6}) and the inductive hypothesis, $\cf_{i,j}^{\,\delta}\bigl( S_{i,j}^{\,\epsilon}(M)\bigr)$ equals
\begin{align*}
&\cf_{i,j}^{\,\delta}\bigl( S_{i,j}^{\,\epsilon}(M_{m\mapsto m-1})\bigr)+\cf_{i,j}^{\,\delta}\bigl( S_{i,j}^{\,\epsilon}(M\setminus\{m\})\bigr)\, C(m{-}1,m)
\\
=&\cond_{\epsilon=\delta}\prod\nolimits_{t\in M_{m\mapsto m-1}\setminus\{\bar\jmath\}} C(i,t)+ C(m{-}1,m)\,\cond_{\epsilon=\delta}\prod\nolimits_{t\in M\setminus\{m,\bar\jmath\}} C(i,t)
\\
= &C(i,m{-}1)\,\cond_{\epsilon=\delta}\prod\nolimits_{t\in M\setminus\{m,\bar\jmath\}} C(i,t)+ C(m{-}1,m)\,\cond_{\epsilon=\delta}\prod\nolimits_{t\in M\setminus\{m,\bar\jmath\}} C(i,t)
\\
=& C(i,m)\,\cond_{\epsilon=\delta}\prod\nolimits_{t\in M\setminus\{m,\bar\jmath\}} C(i,t)=\cond_{\epsilon=\delta}\prod\nolimits_{t\in M\setminus\{\bar\jmath\}} C(i,t),
\end{align*}
as desired.
\end{proof}

\begin{lemma}\label{lemma:rev:1.5}
Let $j\in M\subset(i..j]$ and $\epsilon,\delta\in\{\0,\1\}$. Then
$
\cf_{i,j}^{\,\delta}\bigl( S_{i,j}^{\,\epsilon}(M)\bigr)
=((-1)^{\epsilon+\delta}H_i^{\epsilon+\delta}+(-1)^{\epsilon\delta}H_j^{\epsilon+\delta})\prod\nolimits_{t\in M\setminus\{j\}}C(i,t).
$
\end{lemma}
\begin{proof}
We apply induction on $\height M$. If $M=\{j\}$ the result is clear by~\mbox{(\ref{S2})}.
Otherwise set $m:=\min M$.
If $m=i+1$, then by~(\ref{S4}) and the inductive hypothesis, we get
\begin{align*}
\cf_{i,j}^{\,\delta}( S_{i,j}^{\,\epsilon}(M))
=\cf_{i,j}^{\,\delta}\bigl( S_{i,j}^{\,\epsilon}(M\setminus\{i+1\}\bigr) C(i,i+1)
\\
=((-1)^{(\epsilon+\delta)}H_i^{\epsilon+\delta}+(-1)^{\epsilon\delta}H_j^{\epsilon+\delta})C(i,i+1)\prod\nolimits_{t\in M\setminus\{i+1,j\}}C(i,t)
\\
=((-1)^{(\epsilon+\delta)}H_i^{\epsilon+\delta}+(-1)^{\epsilon\delta}H_j^{\epsilon+\delta})\prod\nolimits_{t\in M\setminus\{j\}}C(i,t).
\end{align*}
Denote $A:= (-1)^{(\epsilon+\delta)}H_i^{\epsilon+\delta}+(-1)^{\epsilon\delta}H_j^{\epsilon+\delta}$.
If $m>i+1$, then by~(\ref{S6}) and the inductive hypothesis,
$\cf_{i,j}^{\,\delta}\bigl( S_{i,j}^{\,\epsilon}(M)\bigr)$ equals
\begin{align*}
\cf_{i,j}^{\,\delta}\bigl( S_{i,j}^{\,\epsilon}(M_{m\mapsto m-1})\bigr)+\cf_{i,j}^{\,\delta}\bigl( S_{i,j}^{\,\epsilon}(M\setminus\{m\})\bigr)\, C(m{-}1,m)
\\
=A\prod\nolimits_{t\in M_{m\mapsto m-1}\setminus\{j\}}C(i,t)
+AC(m{-}1,m)\prod\nolimits_{t\in M\setminus\{m,j\}}C(i,t)\\
=A\Big(C(i,m-1)\prod\nolimits_{t\in M\setminus\{m,j\}}C(i,t)+C(m{-}1,m)\prod\nolimits_{t\in M\setminus\{m,j\}}C(i,t)\Big)
\\
=AC(i,m)\prod\nolimits_{t\in M\setminus\{m,j\}}C(i,t)
=A\prod\nolimits_{t\in M\setminus\{j\}}C(i,t)
\end{align*}
as desired.
\end{proof}

\begin{lemma}\label{lemma:main:2'}
Let $1\le i<j<l\le n$ and $F\in U^{\leq 0}_\Z(n)^{-\alpha(i,j)}$, and
$\delta,\sigma\in\{\0,\1\}$. Then $\cf_{i,l}^\delta( F^{\,\sigma}_{j,l}F)=\cf_{i,j}^{\delta+\sigma}(F)$.
\end{lemma}
\begin{proof} Multiplying~(\ref{eq:rev:1}) by $ F^{\,\sigma}_{j,l}$ on the left, we get
\begin{align*}
& F^{\,\sigma}_{j,l}\, F^{\,\epsilon_0}_{i,a_1}\cdots F^{\,\epsilon_m}_{a_m,j}\,H_{i,\,a_1,\ldots,a_m,j}^{\,\epsilon_0,\ldots,\epsilon_m}
 \\
=&(-1)^{\sigma(\epsilon_0+\cdots+\epsilon_{m-1})} F^{\,\epsilon_0}_{i,a_1}\cdots F^{\,\epsilon_{m-1}}_{a_{m-1},a_m}\, F^{\,\sigma}_{j,l} F^{\,\epsilon_m}_{a_m,j}\,H_{i,\,a_1,\ldots,a_m,j}^{\,\epsilon_0,\ldots,\epsilon_m}\\
=&(-1)^{\sigma(\epsilon_0+\cdots+\epsilon_m)} F^{\,\epsilon_0}_{i,a_1}\cdots F^{\,\epsilon_{m-1}}_{a_{m-1},a_m}\, F^{\,\epsilon_m}_{a_m,j}\, F^{\,\sigma}_{j,l}\,H_{i,\,a_1,\ldots,a_m,j}^{\,\epsilon_0,\ldots,\epsilon_m}
\\
&+(-1)^{\sigma(\epsilon_0+\cdots+\epsilon_{m-1})} F^{\,\epsilon_0}_{i,a_1}\cdots F^{\,\epsilon_{m-1}}_{a_{m-1},a_m} F^{\,\epsilon_m+\sigma}_{a_m,l}\,H_{i,\,a_1,\ldots,a_m,j}^{\,\epsilon_0,\ldots,\epsilon_m}.
\end{align*}
The result follows.
\end{proof}

\chapter{Some polynomials}

In this chapter, we work with the polynomial ring $\R$ over $\Z$ in the variables
$\{x_i,y_i\suchthat i\in\Z\}$. It will be convenient to consider $\R$
embedded into its fraction field to be able to divide arbitrary polynomials of $\R$ by nonzero polynomials of $\R$.

\section{Operators $\sigma^k_{i,j}$}
\label{operators sigmaMij and relations between them}
For integers $a<b$ and $k$, we denote by $\sigma_{a,b}^k$ the $\Z$-algebra homomorphism of $\R$ such that:
$$
\sigma_{a,b}^k(x_t)=\left\{
{\arraycolsep=2pt
\begin{array}{ll}
x_t+x_a-x_b\, &\mbox{ if }t\ge k;\\[6pt]
x_t&\mbox{ otherwise},
\end{array}}
\right.\quad
\sigma_{a,b}^k(y_t)=\left\{
{\arraycolsep=2pt
\begin{array}{ll}
y_t+x_a-x_b\, &\mbox{ if }t\ge k;\\[6pt]
y_t&\mbox{ otherwise}.
\end{array}}
\right.\label{sigmaabk}
$$

The definition immediately implies:

\begin{proposition}\label{proposition:compol:1}
For any $f\in\R$, we have
$
\dfrac{(\id-\sigma_{a,b}^k)f}{x_a-x_b}\in \R.
$
\end{proposition}

\begin{lemma}\label{lemma:compol:1}
Let $a<b<c$ and $e,h\in\{0,1\}$. Then
$$
\sigma_{a,b}^{b+e}\sigma_{a,c}^{c+h}=\sigma_{b,c}^{c+h}\sigma_{a,b}^{b+e}.
$$
\end{lemma}
\begin{proof} Let $z_t$ be either $x_t$ or $y_t$.
Note that $a<b+e\le c\le c+h$.


{\it Case 1: $t<b+e$}. We have
$
\sigma_{a,b}^{b+e}\sigma_{a,c}^{c+h}z_t=z_t=\sigma_{b,c}^{c+h}\sigma_{a,b}^{b+e}z_t.
$


{\it Case 2: $b+e\le t<c+h$}. We have
$$
\begin{array}{l}
\sigma_{a,b}^{b+e}\sigma_{a,c}^{c+h}z_t=\sigma_{a,b}^{b+e}z_t=z_t+x_a-x_b,\\[6pt]
\sigma_{b,c}^{c+h}\sigma_{a,b}^{b+e}z_t=\sigma_{b,c}^{c+h}(z_t+x_a-x_b)=z_t+x_a-x_b.
\end{array}
$$


{\it Case 3: $c+h\le t$}. We have
$$
\begin{array}{l}
\sigma_{a,b}^{b+e}\sigma_{a,c}^{c+h}z_t=\sigma_{a,b}^{b+e}(z_t+x_a-x_c)=z_t+(x_a-x_b)+x_a-(x_c+(x_a-x_b)), 
\end{array}
$$
which is the same as
$
\sigma_{b,c}^{c+h}\sigma_{a,b}^{b+e}z_t=\sigma_{b,c}^{c+h}(z_t+x_a-x_b)=z_t+(x_b-x_c)+x_a-x_b$. 
\end{proof}

\begin{lemma}\label{lemma:compol:2}
Let $a<b\le c<d$ and $e,h\in\{0,1\}$ be such that $b+e\le c$.
Then  $\sigma_{a,b}^{b+e}$ and $\sigma_{c,d}^{d+h}$ commute.
\end{lemma}
\begin{proof} Let $z_t$ be $x_t$ or $y_t$. 
Note that $b+e\le c<d\le d+h$.


{\it Case 1: $t<b+e$}. We have
$
\sigma_{a,b}^{b+e}\sigma_{c,d}^{d+h}z_t=z_t=\sigma_{c,d}^{d+h}\sigma_{a,b}^{b+e}z_t.
$


{\it Case 2: $b+e\le t<d+h$}. We have
$$
\begin{array}{l}
\sigma_{a,b}^{b+e}\sigma_{c,d}^{d+h}z_t=\sigma_{a,b}^{b+e}z_t=z_t+x_a-x_b=\sigma_{c,d}^{d+h}(z_t+x_a-x_b)=\sigma_{c,d}^{d+h}\sigma_{a,b}^{b+e}z_t.
\end{array}
$$


{\it Case 3: $d+h\le t$}. We have
\begin{align*}
\sigma_{a,b}^{b+e}\sigma_{c,d}^{d+h}z_t=\sigma_{a,b}^{b+e}(z_t+x_c-x_d)=z_t+(x_a-x_b)+x_c
+(x_a-x_b)
\\
-(x_d+(x_a-x_b))=z_t+x_c-x_d+x_a-x_b,
\end{align*}
which is the same as
$\sigma_{c,d}^{d+h}\sigma_{a,b}^{b+e}z_t=\sigma_{c,d}^{d+h}(z_t+x_a-x_b)=z_t+x_c-x_d+x_a-x_b.
$
\end{proof}

\section{Polynomials $f_{i,j}^{D,l}(S)$ }\label{Polynomials_f}

Let $D\subset\Z$. For $t\in\Z$, we denote $\max_t D:=\max (D\cap (-\infty..t))$.
For integers $i\le j$  we set
$$
D_t^i:={\max}_t (D\cup\{i\}),
$$
and
$$
u_{i,j}^D=\prod_{t=i+1}^j\(x_{D_t^i}-y_t\).
$$
For example, $u_{i,i}^D=1$,
$$
{\arraycolsep=1pt
\begin{array}{rcl}
u_{1,5}^\emptyset&=&(x_1-y_2)\,(x_1-y_3)\,(x_1-y_4)\,(x_1-y_5),\\
u_{1,5}^{\{3\}}&=&(x_1-y_2)\,(x_1-y_3)\,(x_3-y_4)\,(x_3-y_5).
\end{array}}
$$

Let in addition $l$ be a function on $(i..j]$ taking values $0$ or $1$.
We define the polynomials $f_{i,j}^{D,l}(S)$ for any subset $S\subset(i..j]$ by the following inductive rules:
\medskip
{
\renewcommand{\labelenumi}{{\rm \theenumi}}
\renewcommand{\theenumi}{{$\rm(\alph{enumi}$)}}
\begin{enumerate}\label{fijDl}
\item\label{f:a} $f_{i,j}^{D,l}(\emptyset)=u_{i,j}^D$;\\
\item\label{f:b} $\displaystyle f_{i,j}^{D,l}(S)
=\frac{(\id-\sigma^{s+l(s)}_{D_s^i,s})f_{i,j}^{D,l}(S\setminus\{s\})}
{x_{D_s^i}-x_s}$\,,
      for $S\ne\emptyset$ and $s=\min S$.
\end{enumerate}}
\smallskip
By Proposition~\ref{proposition:compol:1}, the fractions $f_{i,j}^{D,l}(S)$ defined by the above rules are polynomials of $\R$,
which can depend only on the variables $x_i,\ldots,x_j,y_{i+1},\ldots,y_j$.

Keep the notation as above fixed and recall the notion of an {\em end} from Section~\ref{notation}.

\begin{lemma}\label{lemma:compol:4}
Let 
$R$ be an end of $S$ and $\phi:R\to(i..j]$ be an injection such that $\phi(t)\ge t+l(t)$
for all $t\in R$.
Suppose additionally that $l|_{R\cap D}$ is identically $0$. We denote by $\I$ the ideal of $\R$ generated by the polynomials
\renewcommand{\labelenumi}{{\rm \theenumi}}
\renewcommand{\theenumi}{{\rm(\arabic{enumi})}}
\begin{enumerate}
\itemsep=2pt
\item\label{gen:type:1} $\{x_{D_{\phi(t)}^i}-y_{\phi(t)}\mid t\in R,\ \bigl[t{+}l(t)..\phi(t)\bigr)\cap D\ne\emptyset\}$ and
\item\label{gen:type:2} $\{x_t-y_{\phi(t)}\mid t\in R\setminus\{j\},\ \bigl[t{+}l(t)..\phi(t)\bigr)\cap D=\emptyset\}$.
\end{enumerate}
\noindent
Then, modulo $\I$, we have
$$
f_{i,j}^{D,l}(S)\=\left\{
{\arraycolsep=1pt
\begin{array}{ll}
0&\text{ if }\bigcup_{t\in R}\bigl[t{+}l(t)..\phi(t)\bigr)\cap D\ne\emptyset\,;\\
{\displaystyle\prod_{t\in(i..j]\setminus\phi(S)}}(x_{D_t^i}-y_t)
&\text{ if }R=S\ \text{and}\ \bigcup_{t\in S}\bigl[t{+}l(t)..\phi(t)\bigr)\cap D=\emptyset\,.
\end{array}}
\right.
$$
\end{lemma}
\begin{proof} Induction on $|S|$. If $S=\emptyset$, the result follows from the definition \ref{f:a} of the polynomial
$f_{i,j}^{D,l}(\emptyset)$.
Now let $S\ne\emptyset$ and $s:=\min S$.
Let $s_0:=D_s^i$. Then
$$
f_{i,j}^{D,l}(S)
=\frac{\bigl(\id-\sigma^{s+l(s)}_{s_0,s}\bigr)f_{i,j}^{D,l}(S\setminus\{s\})}{x_{s_0}-x_s}\,,
$$
by definition of $f_{i,j}^{D,l}(S)$.

Consider the ideal $\I^{\,\prime}$ generated by
\renewcommand{\labelenumi}{{\rm \theenumi}}
\renewcommand{\theenumi}{{$\rm(\arabic{enumi}'$)}}
\begin{enumerate}
\item\label{gen:type:1'} $\{x_{D_{\phi(t)}^i}-y_{\phi(t)}\mid t\in R\setminus\{s\},\ \bigl[t{+}l(t)..\phi(t)\bigr)\cap D\ne\emptyset\}$ and
\item\label{gen:type:2'} $\{x_t-y_{\phi(t)}\mid t\in R\setminus\{s,j\},\ \bigl[t{+}l(t)..\phi(t)\bigr)\cap D=\emptyset\}$.
\end{enumerate}
In other words,  $\I^{\,\prime}$ is defined similarly to $\I$ but using $R\setminus\{s\}$ instead of $R$.
Clearly, $\I^{\,\prime}\subset\I$. Note also that $R\setminus\{s\}$ is an end of $S\setminus\{s\}$.

Since $\phi$ is an injection,
\cite[Proposition~6.18]{Shchigolev_Rectangular_low_level_case}
applies to $\I^{\,\prime}$ if we choose an order on the generators $\{x_a,y_a\suchthat a\in\Z\}$ of $\R$ so that $y_a$'s precede $x_b$'s. We conclude that $\I'$ is a prime ideal and that an element $X=c_0+\sum_a c_ay_a+\sum_bd_bx_b$, where $c_a,d_b\in\Z$, belongs to $\I'$ if and only if $X$ is a $\Z$-linear combination of the generators (1$'$) and (2$'$). In particular, $x_{s_0}-x_s\notin\I^{\prime}$. Hence
\begin{equation}\label{Proposition6.18}
g\cdot(x_{s_0}-x_s)\in\I'\quad\text{implies}\quad g\in\I'\qquad(g\in \R).
\end{equation}


We claim that $\I^{\,\prime}$ is closed under the action of $\sigma^{s+l(s)}_{s_0,s}$.
Indeed, consider first a generator
$x_{D_{\phi(t)}^i}-y_{\phi(t)}$ of type~\ref{gen:type:1'}.
We have $t\in R\setminus\{s\}$ and $\bigl[t+l(t)..\phi(t)\bigr)\cap D\ne\emptyset$.
Take any $d\in \bigl[t+l(t)..\phi(t)\bigr)\cap D$.
Noting that $s<t$, we get
$$
s+l(s)\le t+l(t)\le d\le D_{\phi(t)}^i<\phi(t).
$$
Hence
$$
\sigma^{s+l(s)}_{s_0,s}(x_{D_{\phi(t)}^i}-y_{\phi(t)})=x_{D_{\phi(t)}^i}+x_{s_0}-x_s-(y_{\phi(t)}+x_{s_0}-x_s)=x_{D_{\phi(t)}^i}-y_{\phi(t)}.
$$
Now consider a generator $x_t-y_{\phi(t)}$ of type~\ref{gen:type:2'}. As $s<t$, we get
$
s+l(s)\le t\le\phi(t),
$
and so again we obtain that $\sigma^{s+l(s)}_{s_0,s}$
acts on $x_t-y_{\phi(t)}$ identically.

To complete the proof we now consider three cases.

\smallskip

{\it Case 1: $\bigcup_{t\in R\setminus\{s\}}\bigl[t{+}l(t)..\phi(t)\bigr)\cap D\ne\emptyset$.}
By the inductive hypothesis, applied to the set $S\setminus\{s\}$ and its end $R\setminus\{s\}$,
we get $f_{i,j}^{D,l}(S\setminus\{s\})\in\I^{\,\prime}$. Therefore
$$
f_{i,j}^{D,l}(S)\cdot(x_{s_0}-x_s)=
\bigl(\id-\sigma^{s+l(s)}_{s_0,s}\bigr)f_{i,j}^{D,l}(S\setminus\{s\})\in\I^{\,\prime}.
$$
By (\ref{Proposition6.18}), $f_{i,j}^{D,l}(S)\in\I^{\prime}\subset\I$.

\smallskip

{\it Case 2:} $\bigcup_{t\in R\setminus\{s\}}\bigl[t{+}l(t)..\phi(t)\bigr)\cap D=\emptyset$ but
$\bigcup_{t\in R}\bigl[t{+}l(t)..\phi(t)\bigr)\cap D\ne\emptyset$.
In this case $s\in R$ and so $S=R$.
By the inductive hypothesis, applied to the set $S\setminus\{s\}$ and its end $R\setminus\{s\}$, 
we have
$
f_{i,j}^{D,l}(S\setminus\{s\})=h+P
$
for some $h\in\I^{\,\prime}$, where
$$
P:=\prod_{t\in(i..j]\setminus\phi(S)}(x_{D_t^i}-y_t).
$$
Since $\phi$ is an injection, we have
\begin{equation}\label{eq:compol:1}
\begin{split}
f_{i,j}^{D,l}(S)\cdot(x_{s_0}-x_s)&=
(\id-\sigma^{s+l(s)}_{s_0,s})f_{i,j}^{D,l}(S\setminus\{s\})
\\
&=h'+(\id-\sigma^{s+l(s)}_{s_0,s})
\displaystyle\prod_{t\in(i..j]
\setminus\phi(S\setminus\{s\})}(x_{D_t^i}-y_t)
\\
\displaystyle
&=h'+(\id-\sigma^{s+l(s)}_{s_0,s})
\bigl[
\bigl(x_{D_{\phi(s)}^i}-y_{\phi(s)}\bigr)\cdot P\bigr],
\end{split}
\end{equation}

\noindent
where $h'=(\id-\sigma^{s+l(s)}_{s_0,s})h$. As $\I^{\prime}$ is invariant under
$\sigma^{s+l(s)}_{s_0,s}$, we have $h'\in\I^{\,\prime}$.

Note that $[s{+}l(s)..\phi(s))\cap D\ne\emptyset$. Let $d\in [s{+}l(s)..\phi(s))\cap D\ne\emptyset$.
Then
$\phi(s)>\max_{\phi(s)}D\cup\{i\}\ge d\ge s+l(s)$.
Therefore, $\sigma^{s+l(s)}_{s_0,s}(x_{D_{\phi(s)}^i}-y_{\phi(s)})=x_{D_{\phi(s)}^i}-y_{\phi(s)}$
and we can rewrite~(\ref{eq:compol:1}) as
$
h'+(x_{D_{\phi(s)}^i}-y_{\phi(s)}\bigr)\cdot
(\id-\sigma^{s+l(s)}_{s_0,s})P.
$
Hence
$$
\bigg(
f_{i,j}^{D,l}(S)-\bigl(x_{D_{\phi(s)}D^i}-y_{\phi(s)}\bigr)
\frac{\bigl(\id-\sigma^{s+l(s)}_{s_0,s}\bigr)P}{x_{s_0}-x_s} \bigg)\cdot(x_{s_0}-x_s)=h'\in\I^{\,\prime}.
$$
By Proposition~\ref{proposition:compol:1} the fraction in the big brackets belongs to $\R$.
Now, by (\ref{Proposition6.18}),
$$
f_{i,j}^{D,l}(S)-\bigl(x_{D_{\phi(s)}^i}-y_{\phi(s)}\bigr)
\frac{\bigl(\id-\sigma^{s+l(s)}_{s_0,s}\bigr)P}{x_{s_0}-x_s}\in\I^{\,\prime}.
$$
Thus $f_{i,j}^{D,l}(S)\=0\pmod\I$, since $x_{D_{\phi(s)}^i}-y_{\phi(s)}\in\I$
(as a generator of type~\ref{gen:type:1}) and $\I^{\,\prime}\subset\I$.

\smallskip

{\it Case 3:} $R=S$ and $\bigcup_{t\in S}\bigl[t{+}l(t)..\phi(t)\bigr)\cap D=\emptyset$.
We claim that
\begin{equation}\label{eq:compol:2}
s_0=D_s^i =D_{\phi(s)}^i.
\end{equation}
Indeed, we have
$D_{\phi(s)}^i=D_{s+l(s)}^i$,
since $[s{+}l(s)..\phi(s)\bigr)\cap(D\cup\{i\})=\emptyset$.
If $l(s)=0$ then $D_{s+l(s)}^i=s_0$ and (\ref{eq:compol:2}) follows.
If  $l(s)=1$, then $s\notin D$ (recall that $l|_{S\cap D}=0$) and, therefore,
$s\notin D\cup\{i\}$. Hence we again obtain $D_{s+l(s)}^i=s_0$.

Since $s_0<s+l(s)\le\phi(s)$, we have
$$
\sigma^{s+l(s)}_{s_0,s}\bigl(x_{s_0}-y_{\phi(s)}\bigr)
=x_{s_0}-(y_{\phi(s)}+x_{s_0}-x_s)=
x_s-y_{\phi(s)}.
$$

Note that in the present case the formula~(\ref{eq:compol:1}) is again true. Applying~(\ref{eq:compol:2}),
we can rewrite this formula in the following form:
\begin{align*}
f_{i,j}^{D,l}(S)\cdot(x_{s_0}-x_s)
=&h'+\bigl(\id-\sigma^{s+l(s)}_{s_0,s}\bigr)
\big[
\bigl(x_{s_0}-y_{\phi(s)}\bigr)\cdot P
\big]\\
=&h'+\bigl(x_{s_0}-y_{\phi(s)}\bigr)\cdot P
-\bigl(x_s-y_{\phi(s)}\bigr)\cdot\sigma^{s+l(s)}_{s_0,s}P\\
=&h'+\bigl(x_s-y_{\phi(s)}\bigr)\cdot\big[P-\sigma^{s+l(s)}_{s_0,s}P\big]
+(x_{s_0}-x_s)\cdot P.
\end{align*}
Hence
$$
\bigg(
f_{i,j}^{D,l}(S)-P-\bigl(x_s-y_{\phi(s)}\bigr)\frac{\bigl(\id-\sigma^{s+l(s)}_{s_0,s}\bigr)P}{x_{s_0}-x_s}
\bigg)\cdot(x_{s_0}-x_s)=h'\in\I^{\,\prime}.
$$
By Proposition~\ref{proposition:compol:1}, the fraction in the big brackets belongs to $\R$.
Now, by (\ref{Proposition6.18}),
\begin{equation}\label{eq:compol:1.25}
f_{i,j}^{D,l}(S)-P-\bigl(x_s-y_{\phi(s)}\bigr)\frac{\bigl(\id-\sigma^{s+l(s)}_{s_0,s}\bigr)P}{x_{s_0}-x_s}\in\I^{\,\prime}.
\end{equation}
If $s<j$ then $x_s-y_{\phi(s)}\in\I$ as a generator of type~\ref{gen:type:2}, and we are done.

Consider now the case $s=j$. We have $R=S=\{j\}$, $\phi(j)=j$, $l(j)=0$ and
\mbox{$P=\prod_{t\in(i..j)}(x_{D_t^i}{-}y_t)$.}
Hence $\sigma^{s+l(s)}_{s_0,s}=\sigma^{j}_{s_0,j}$ acts identically on $P$,
whence
$$
\frac{\bigl(\id-\sigma^{s+l(s)}_{s_0,s}\bigr)P}{x_{s_0}-x_s}=0,
$$
and we are done.
\end{proof}

\begin{lemma}\label{lemma:compol:4.5} If $a<b\le c<d$, $e=0,1$ and $b+e\le c$,
then $\sigma_{a,b}^{b+e}(f_{c,d}^{D,l}(S))=f_{c,d}^{D,l}(S)$.
\end{lemma}
\begin{proof}
We apply induction on $|S|$. By definition, we have
$$
f_{c,d}^{D,l}(\emptyset)=\prod\nolimits_{t=c+1}^d\(x_{D_t^c}-y_t\).
$$
The operator $\sigma_{a,b}^{b+e}$ acts identically on $f_{c,d}^{D,l}(\emptyset)$,
since it does so on each factor of this product.
Now let $S\ne\emptyset$ and set $s:=\min S$ and $s_0:=D_s^c$. Applying the inductive hypothesis and
Lemma~\ref{lemma:compol:2}, we get
\begin{align*}
(x_{s_0}-x_s)\sigma_{a,b}^{b+e}f_{c,d}^{D,l}(S)=
\sigma_{a,b}^{b+e}\big[(x_{s_0}-x_s)f_{c,d}^{D,l}(S)\big]
\!=\sigma_{a,b}^{b+e}(\id-\sigma^{s+l(s)}_{s_0,s})f_{c,d}^{D,l}(S\setminus\{s\})\\
=\!(\id-\sigma^{s+l(s)}_{s_0,s})\sigma_{a,b}^{b+e}\,f_{c,d}^{D,l}(S\setminus\{s\})
\!=\!(\id-\sigma^{s+l(s)}_{s_0,s})f_{b,c}^{D,l}(S\setminus\{s\})
\!=\!(x_{s_0}-x_s)\,f_{c,d}^{D,l}(S).
\end{align*}
Cancelling out $x_{s_0}-x_s$, we get $\sigma_{a,b}^{b+e}f_{b,c}^{D,l}(S)=f_{b,c}^{D,l}(S)$.
\end{proof}

\section{Polynomials $g^{(1)}_{i,j}(S)$}\label{polinomialsg1}
Here and below, we sometimes denote by $0$ and $1$ the constant functions on some set (clear from the context) taking the values $0$ or $1$, respectively.
Let $i<j$ and $S\subset(i..j]$. Define
$$g^{(1)}_{i,j}(S)=f_{i,j}^{\emptyset,1}(S).$$
%
%
The defining relations~\ref{f:a} and~\ref{f:b} of Section~\ref{Polynomials_f}
in our special case become:

\renewcommand{\labelenumi}{{\rm \theenumi}}
\renewcommand{\theenumi}{{$\rm(\alph{enumi}$\,-1)}}
\begin{enumerate}
\item\label{g-1:a} $g_{i,j}^{(1)}(\emptyset)=u_{i,j}^\emptyset=(x_i-y_{i+1})(x_i-y_{i+2})\cdots(x_i-y_j)$;\\
\item\label{g-1:b} $\displaystyle g_{i,j}^{(1)}(S)
=\frac{(\id-\sigma^{s+1}_{i,s})g_{i,j}^{(1)}(S\setminus\{s\})}
{x_i-x_s}$,
      for  $S\ne\emptyset$ and $s=\min S$.
\end{enumerate}

\begin{lemma}\label{lemma:compol:6}
$g^{(1)}_{i,j}\bigl((i..j)\bigr)=x_i-y_{i+1}$.
\end{lemma}
\begin{proof} We apply downward induction on $k=j\ldots,i+1$ to prove:
\begin{equation}\label{eq:compol:3}
g^{(1)}_{i,j}\bigl([k..j)\bigr)=(x_i-y_{i+1})\cdots(x_i-y_k).
\end{equation}
The induction base is clear from~\ref{g-1:a}.
Let~(\ref{eq:compol:3}) holds and $k>i+1$. By~\ref{g-1:b},
\begin{align*}
g^{(1)}_{i,j}\bigl([k-1..j)\bigr)=
&\frac{\bigl(\id-\sigma^k_{i,k-1}\bigr)g^{(1)}_{i,j}\bigl([k..j)\bigr)}
     {x_i-x_{k-1}}\\
=&\frac{(x_i-y_{i+1})\cdots(x_i-y_k)-(x_i-y_{i+1})\cdots(x_i-y_{k-1})(x_{k-1}-y_k)}
{x_i-x_{k-1}}
\\
=&(x_i-y_{i+1})\cdots(x_i-y_{k-1}).
\end{align*}
The required formula follows from~(\ref{eq:compol:3}) for $k=i+1$.
\end{proof}

\begin{lemma}
\label{lemma:compol:7}
Let $i+1<j$ and $S\subset(i{+}1..j)$. Then
$$
(x_i-y_{i+1})\,g^{(1)}_{i+1,j}(S)+(x_i-x_{i+1})\,g_{i,j}^{(1)}(\{i{+}1\}\cup S)=g_{i,j}^{(1)}(S).
$$
\end{lemma}
\begin{proof} By~\ref{g-1:b}, the left hand side of the above formula equals
$$
\begin{array}{l}
(x_i-y_{i+1})\,g^{(1)}_{i+1,j}(S)+g_{i,j}^{(1)}(S)-\sigma^{i+2}_{i,i+1}g_{i,j}^{(1)}(S).
\end{array}
$$
Thus it suffices to prove that
\begin{equation}\label{eq:compol:4}
\sigma^{i+2}_{i,i+1}g_{i,j}^{(1)}(S)=(x_i-y_{i+1})\,g^{(1)}_{i+1,j}(S)
\end{equation}
for any $S\subset(i{+}1..j)$.

If $S=\emptyset$, then by~\ref{g-1:a}, we get
\begin{align*}
\sigma^{i+2}_{i,i+1}g_{i,j}^{(1)}(\emptyset)=&\sigma^{i+2}_{i,i+1}\bigl[(x_i-y_{i+1})\cdots (x_i-y_j)\bigr]\\
=&(x_i-y_{i+1})(x_{i+1}-y_{i+2})\cdots (x_{i+1}-y_j)=(x_i-y_{i+1})\,g_{i+1,j}^{(1)}(\emptyset).
\end{align*}

Now let $S\ne\emptyset$. Set $s=\min S$. Clearly $s\ge i+2$.
Applying~\ref{g-1:b}, the inductive hypothesis and Lemma~\ref{lemma:compol:1}, we get
\begin{align*}
(x_{i+1}-x_s)\sigma^{i+2}_{i,i+1}g_{i,j}^{(1)}(S)
=\sigma^{i+2}_{i,i+1}\big[(x_i-x_s)\,g_{i,j}^{(1)}(S)\big]
=\sigma^{i+2}_{i,i+1}(\id-\sigma^{s+1}_{i,s})g_{i,j}^{(1)}(S\setminus\{s\})
\\
=(\id-\sigma^{s+1}_{i+1,s})\sigma^{i+2}_{i,i+1}\,g_{i,j}^{(1)}(S\setminus\{s\})
=(\id-\sigma^{s+1}_{i+1,s})\big[(x_i-y_{i+1})\,g^{(1)}_{i+1,j}(S\setminus\{s\})\big]\\
=(x_i-y_{i+1})\cdot\bigl(\id-\sigma^{s+1}_{i+1,s}\bigr)\,g^{(1)}_{i+1,j}(S\setminus\{s\})
=(x_i-y_{i+1})\,(x_{i+1}-x_s)\,g^{(1)}_{i+1,j}(S).
\end{align*}
Cancelling out $(x_{i+1}-x_s)$, we obtain the required relation~(\ref{eq:compol:4}).
\end{proof}

\begin{lemma}\label{lemma:compol:8}
Let $i+1<m<j$ and $S\subset(m..j)$. Then
\begin{align*}
&(x_i-y_{i+1})\,g_{m,j}^{(1)}(S)+g_{i,j}^{(1)}\bigl((i..m{-}1)\cup\{m\}\cup S\bigr)
\\
+
&(x_{m-1}-x_m)\,g^{(1)}_{i,j}\bigl((i..m]\cup S\bigr)
=g^{(1)}_{i,j}\bigl((i..m)\cup S\bigr).
\end{align*}
\end{lemma}
\begin{proof} We apply downward induction on $k=m-1\ldots,i+1$ to prove
\begin{align*}
&(x_i-y_{i+1})\cdots(x_i-y_k)\,g_{m,j}^{(1)}(S)+g_{i,j}^{(1)}\bigl([k..m{-}1)\cup\{m\}\cup S\bigr)
\\
+&(x_{m-1}-x_m)\,g^{(1)}_{i,j}\bigl([k..m]\cup S\bigr)=g^{(1)}_{i,j}\bigl([k..m)\cup S\bigr).
\end{align*}
Denote by $\Phi^S_k$ the difference of the left-hand side and the right-hand side of this formula.
The claim of the current lemma will follow from $\Phi^S_{i+1}=0$.

The induction base is the case $k=m-1$. By~\ref{g-1:b}, $(x_i-x_{m-1})\Phi^S_{m-1}$ equals
\begin{align*}
&(x_i-x_{m-1})(x_i-y_{i+1})\cdots(x_i-y_{m-1})\,g_{m,j}^{(1)}(S)
+(x_i-x_{m-1})\,g_{i,j}^{(1)}\bigl(\{m\}\cup S\bigr)\\
&+(x_i-x_{m-1})(x_{m-1}-x_m)\,g^{(1)}_{i,j}\bigl(\{m{-}1,m\}\cup S\bigr)
-(x_i-x_{m-1})\,g^{(1)}_{i,j}\bigl(\{m{-}1\}\cup S\bigr)\\
=&(x_i-x_{m-1})(x_i-y_{i+1})\cdots(x_i-y_{m-1})\,g_{m,j}^{(1)}(S)
+(x_i-x_{m-1})\,g_{i,j}^{(1)}\bigl(\{m\}\cup S\bigr)\\
&+(x_{m-1}-x_m)\,\bigl(\id-\sigma_{i,m-1}^m\bigr)\,g^{(1)}_{i,j}\bigl(\{m\}\cup S\bigr)
-\(\id-\sigma_{i,m-1}^m\)\,g_{i,j}^{(1)}(S)\\
=&(x_i-x_{m-1})(x_i-y_{i+1})\cdots(x_i-y_{m-1})\,g_{m,j}^{(1)}(S)+(x_i-x_m)\,g_{i,j}^{(1)}\bigl(\{m\}\cup S\bigr)\\
&-(x_{m-1}-x_m)\cdot\sigma_{i,m-1}^m\,g^{(1)}_{i,j}\bigl(\{m\}\cup S\bigr)
-\bigl(\id-\sigma_{i,m-1}^m\bigr)\,g_{i,j}^{(1)}(S).
\end{align*}
Note that $\sigma_{i,m-1}^m(x_i-x_m)=x_i-(x_m+x_i-x_{m-1})=x_{m-1}-x_m$.
Applying this equality and~\ref{g-1:b}, we get
\begin{align*}
&(x_i-x_{m-1})\cdot(x_i-y_{i+1})\cdots(x_i-y_{m-1})\,g_{m,j}^{(1)}(S)
+\bigl(\id-\sigma_{i,m}^{m+1}\bigr)\,g_{i,j}^{(1)}(S)
\\
&-\sigma_{i,m-1}^m\Bigl[(x_i-x_m)\,g^{(1)}_{i,j}\bigl(\{m\}\cup S\bigr)\Bigr]
-\(\id-\sigma_{i,m-1}^m\)g_{i,j}^{(1)}(S)
\\
=&(x_i-x_{m-1})(x_i-y_{i+1})\cdots(x_i-y_{m-1})\,g_{m,j}^{(1)}(S)
+\bigl(\id-\sigma_{i,m}^{m+1}\bigr)\,g_{i,j}^{(1)}(S)
\\
&-\sigma_{i,m-1}^m\bigl(\id-\sigma_{i,m}^{m+1}\bigr)g^{(1)}_{i,j}(S)
-\(\id-\sigma_{i,m-1}^m\)g_{i,j}^{(1)}(S)
\\
=&(x_i-x_{m-1})(x_i-y_{i+1})\cdots(x_i-y_{m-1})\,g_{m,j}^{(1)}(S)-\sigma_{i,m}^{m+1}\,g_{i,j}^{(1)}(S)
\\
&+\sigma_{i,m-1}^m\sigma_{i,m}^{m+1}\,g_{i,j}^{(1)}(S).
\end{align*}
Denote the last expression by $\Psi^S$ and prove by induction on $|S|$ that $\Psi^S=0$
for any $S\subset(m..j)$. By~\ref{g-1:a}, we get
\begin{align*}
\Psi^\emptyset&=(x_i-x_{m-1})\cdot(x_i-y_{i+1})\cdots(x_i-y_{m-1})\cdot(x_m-y_{m+1})\cdots(x_m-y_j)\\
&-\sigma_{i,m}^{m+1}\Bigl[(x_i-y_{i+1})\cdots(x_i-y_j)\Bigr]
+\sigma_{i,m-1}^m\sigma_{i,m}^{m+1}\Bigl[(x_i-y_{i+1})\cdots(x_i-y_j)\Bigr]
\\
=&(x_i-x_{m-1})\cdot(x_i-y_{i+1})\cdots(x_i-y_{m-1})\cdot(x_m-y_{m+1})\cdots(x_m-y_j)
\\
&-(x_i-y_{i+1})\cdots(x_i-y_m)\cdot(x_m-y_{m+1})\cdots(x_m-y_j)\\
&+\sigma_{i,m-1}^m\Bigl[(x_i-y_{i+1})\cdots(x_i-y_m)\cdot(x_m-y_{m+1})\cdots(x_m-y_j)\Bigr]
\\
=&(x_i-x_{m-1})\cdot(x_i-y_{i+1})\cdots(x_i-y_{m-1})\cdot(x_m-y_{m+1})\cdots(x_m-y_j)
\\
&-(x_i-y_{i+1})\cdots(x_i-y_m)\cdot(x_m-y_{m+1})\cdots(x_m-y_j)
\\
&+(x_i-y_{i+1})\cdots(x_i-y_{m-1})\cdot(x_{m-1}-y_m)\cdot(x_m-y_{m+1})\cdots(x_m-y_j)
=0
\end{align*}
Now let $S\ne\emptyset$ and $s:=\min S$. Recall that $m<s<j$. Applying~\ref{g-1:b},
Lemmas~\ref{lemma:compol:1} and~\ref{lemma:compol:2} and the inductive hypothesis, we get that $(x_m-x_s)\Psi^S$ equals
\begin{align*}
&(x_i-x_{m-1})\cdot(x_i-y_{i+1})\cdots(x_i-y_{m-1})\cdot(x_m-x_s)\,g_{m,j}^{(1)}(S)\\
&-\sigma_{i,m}^{m+1}\big[(x_i-x_s)\,g_{i,j}^{(1)}(S)\big]
+\sigma_{i,m-1}^m\sigma_{i,m}^{m+1}\big[(x_i-x_s)\,g_{i,j}^{(1)}(S)\big]
\\
=&(x_i-x_{m-1})\cdot(x_i-y_{i+1})\cdots(x_i-y_{m-1})\cdot(\id-\sigma_{m,s}^{s+1})\,g_{m,j}^{(1)}(S\setminus\{s\})
\\
&-\sigma_{i,m}^{m+1}(\id-\sigma_{i,s}^{s+1})\,g_{i,j}^{(1)}(S\setminus\{s\})
+\sigma_{i,m-1}^m\sigma_{i,m}^{m+1}(\id-\sigma_{i,s}^{s+1})\,g_{i,j}^{(1)}(S\setminus\{s\})
\\
=&(\id-\sigma_{m,s}^{s+1})\bigl[(x_i-x_{m-1})\cdot(x_i-y_{i+1})\cdots(x_i-y_{m-1})\,g_{m,j}^{(1)}(S\setminus\{s\})\bigr]
\\
&-(\id-\sigma_{m,s}^{s+1})\sigma_{i,m}^{m+1}g_{i,j}^{(1)}(S\setminus\{s\})
+(\id-\sigma_{m,s}^{s+1})\sigma_{i,m-1}^m\sigma_{i,m}^{m+1}\,g_{i,j}^{(1)}(S\setminus\{s\})
\\
=&(\id-\sigma_{m,s}^{s+1})\Psi^{S\setminus\{s\}}=0.
\end{align*}
Hence $\Psi^S=0$. We have now proved that $\Phi^S_{m-1}=0$.

Now let $i<k<m-1$ and by the undictive assumption we have $\Phi^S_{k+1}=0$.
By definition, $\sigma^{k+1}_{i,k}$ acts identically on $x_i-y_t$ for $t=i+1,\ldots,k$. By Lemma~\ref{lemma:compol:4.5}, it also acts identically on $g_{m,j}^{(1)}(S)$.
Moreover, $\sigma^{k+1}_{i,k}(x_i-y_{k+1})=x_k-y_{k+1}$. So
\begin{align*}
&\bigl(\id-\sigma^{k+1}_{i,k}\bigr)\bigl[(x_i-y_{i+1})\cdots(x_i-y_k)\cdot(x_i-y_{k+1})\,g_{m,j}^{(1)}(S)\bigr]
\\
=&(x_i-y_{i+1})\cdots(x_i-y_k)\cdot(x_i-y_{k+1})\,g_{m,j}^{(1)}(S)
\\
&-(x_i-y_{i+1})\cdots(x_i-y_k)\cdot(x_k-y_{k+1})\,g_{m,j}^{(1)}(S)
\\
=&(x_i-x_k)\cdot(x_i-y_{i+1})\cdots(x_i-y_k)\,g_{m,j}^{(1)}(S).
\end{align*}
Hence usingy~\ref{g-1:b}, we get
\begin{align*}
(x_i-x_k)\,\Phi^S_k=&(x_i-x_k)\cdot(x_i-y_{i+1})\cdots(x_i-y_k)\,g_{m,j}^{(1)}(S)
\\
&+\bigl(\id-\sigma_{i,k}^{k+1}\bigr)\,g_{i,j}^{(1)}\bigl([k{+}1..m{-}1)\cup\{m\}\cup S\bigr)
\\
&+(x_{m-1}-x_m)\cdot\bigl(\id-\sigma_{i,k}^{k+1}\bigr)\,g^{(1)}_{i,j}\bigl([k+1..m]\cup S\bigr)
\\
&-\bigl(\id-\sigma_{i,k}^{k+1}\bigr)g^{(1)}_{i,j}\bigl([k{+}1..m)\cup S\bigr)
=\bigl(\id-\sigma_{i,k}^{k+1}\bigr)\Phi^{S\setminus\{s\}}_k=0.
\end{align*}
So $\Phi^S_k=0$, and the inductive step is complete. 
\end{proof}

\section{Polynomials $g^{(2)}_{i,k,q,j}(S)$}\label{polinomialsg2}
Let $i\le k\le q\le j$, $i<q$, and $S\subset (i..j]$. Set
$$
g^{(2)}_{i,k,q,j}(S)=f_{i,j}^{\{k\},l^{(2)}_{i,k,q,j}}(S).$$
where $l^{(2)}_{i,k,q,j}:(i..j]\to\{0,1\}$ is the following function:
      $$\label{l2}
      l^{(2)}_{i,k,q,j}(t)=\left\{
                             {\arraycolsep=0pt
                             \begin{array}{l}
                             1\mbox{ if }i<t<k\mbox{ or }q<t<j;\\[3pt]
                             0\mbox{ if }k\le t\le q\mbox{ or }t=j.
                             \end{array}}
        \right.
      $$
The defining relations~\ref{f:a} and~\ref{f:b} of Section~\ref{Polynomials_f} in our special case become:
\renewcommand{\labelenumi}{{\rm \theenumi}}
\renewcommand{\theenumi}{{$\rm(\alph{enumi}$\,-2)}}
\begin{enumerate}
\item\label{g-2:a} $g^{(2)}_{i,k,q,j}(\emptyset)=u_{i,j}^{\{k\}}=(x_i-y_{i+1})\cdots(x_i-y_k)\cdot(x_k-y_{k+1})\cdots(x_k-y_j)$;\\
\item\label{g-2:b} $\displaystyle g^{(2)}_{i,k,q,j}(S)=\frac{\bigl(\id-\sigma^{s+l^{(2)}_{i,k,q,j}}_{\{k\}_s^i,s}\bigr)g^{(2)}_{i,k,q,j}(S\setminus\{s\})}{x_{\{k\}_s^i}-x_s}$,
      for $S\ne\emptyset$ and $s=\min S$.
\end{enumerate}

\begin{lemma}\label{lemma:compol:9}
$g^{(2)}_{i,i,j,j}\bigl((i..j]\bigr)=1$ and $g^{(2)}_{i,i+1,j,j}\bigl((i..j]\bigr)=1$.
\end{lemma}
\begin{proof}
Let $k\in[i..j]$. Apply downward induction on $s=j, \dots, \max(i,k-1)$ to prove:
$$
g^{(2)}_{i,k,j,j}\bigl((s..j]\bigr)=\prod_{t=i+1}^s(x_{\{k\}_t^i}-y_t).
$$
By~\ref{g-2:a}, this formula is true for $s=j$.
Let $\max(i,k-1)\le s<j$. Note that $k \le s+1\le j$. So $l^{(2)}_{i,k,j,j}(s+1)=0$.
Denote $u:=\{k\}_{s+1}^i$.
By~\ref{g-2:b} and the inductive hypothesis, $g^{(2)}_{i,k,j,j}\bigl((s..j]\bigr)$ equals
\begin{align*}
\frac{(\id-\sigma_{u,s+1}^{s+1})g^{(2)}_{i,k,j,j}((s{+}1..j])}{x_u-x_{s+1}}
=\frac{(\id-\sigma_{u,s+1}^{s+1})\prod\nolimits_{t=i+1}^{s+1}(x_{\{k\}_t^i}-y_t)}{x_u-x_{s+1}}\\
=\frac{\prod\nolimits_{t=i+1}^s(x_{\{k\}_t^i}-y_t)\,(\id-\sigma_{u,s+1}^{s+1})(x_u-x_{s+1})}{x_u-x_{s+1}}
=\prod_{t=i+1}^s(x_{\{k\}_t^i}-y_t),
\end{align*}
using  $\sigma_{u,s+1}^{s+1}(x_u-x_{s+1})=0$.
The required formulas now follow from the special cases $s=i$, $k=i$ or $s=i$, $k=i+1$.
\end{proof}

\begin{lemma}\label{lemma:compol:10}
Let $i+1<j$ and $S\subset(i{+}1..j)$. Then
{\renewcommand{\labelenumi}{{\rm \theenumi}}
\renewcommand{\theenumi}{{\rm(\alph{enumi})}}
\begin{enumerate}
\item\label{lemma:compol:10:part:a} $g^{(1)}_{i+1,j}(S)+g^{(1)}_{i,j}\bigl(\{i{+}1\}\cup S\bigr)=g^{(2)}_{i,i,i+1,j}\bigl(\{i{+}1\}\cup S\bigr)${\rm;}\\
\item\label{lemma:compol:10:part:b} $g^{(1)}_{i+1,j}(S)=g^{(2)}_{i,i+1,i+1,j}\bigl(\{i{+}1\}\cup S\bigr)$.
\end{enumerate}}
\end{lemma}
\begin{proof}\ref{lemma:compol:10:part:a}
Note that
$g^{(2)}_{i,i,i+1,j}(S)=g^{(1)}_{i,j}(S)$ for any $S\subset(i+1..j)$.
This fact follows directly from the inductive definitions~\ref{g-1:a},~\ref{g-1:b},~\ref{g-2:a},~\ref{g-2:b} and the equalities $u_{i,j}^{\{i\}}=u_{i,j}^\emptyset$,
$\{i\}_s^i=i$ and $l^{(2)}_{i,i,i+1,j}(s)=1$ for any $s\in(i{+}1..j)$.

Now, by~\ref{g-1:b} and~\ref{g-2:b}, the difference of the left-hand side and the right-hand side
of~\ref{lemma:compol:10:part:a} multiplied by $(x_i-x_{i+1})$ equals
\begin{align*}
&(x_i-x_{i+1})\,g^{(1)}_{i+1,j}(S)+\bigl(\id-\sigma_{i,i+1}^{i+2}\bigr)g^{(1)}_{i,j}(S)
-\bigl(\id-\sigma_{i,i+1}^{i+1}\bigr)g^{(2)}_{i,i,i+1,j}(S)\\
=&(x_i-x_{i+1})\,g^{(1)}_{i+1,j}(S)-\sigma_{i,i+1}^{i+2}\,g^{(1)}_{i,j}(S)
+\sigma_{i,i+1}^{i+1}\,g^{(1)}_{i,j}(S).
\end{align*}
Denote the last expression by $\Phi^S$ and prove by induction on $|S|$ that $\Phi^S=0$
for all $S\subset(i{+}1..j)$. By~\ref{g-1:a}, $\Psi^\emptyset$ equals
\begin{align*}
&(x_i-x_{i+1})\cdot(x_{i+1}-y_{i+2})\cdots(x_{i+1}-y_j)
\\
&-\sigma_{i,i+1}^{i+2}\bigl[(x_i-y_{i+1})\cdots(x_i-y_j)\bigr]
+\sigma_{i,i+1}^{i+1}\bigl[(x_i-y_{i+1})\cdots(x_i-y_j)\bigr]
\\
=&(x_i-x_{i+1})\cdot(x_{i+1}-y_{i+2})\cdots(x_{i+1}-y_j)
\\
&-(x_i-y_{i+1})\cdot(x_{i+1}-y_{i+2})\cdots(x_{i+1}-y_j)
+(x_{i+1}-y_{i+1})\cdots(x_{i+1}-y_j)=0.
\end{align*}

Now let $S\ne\emptyset$ and set $s:=\min S$.
By~\ref{g-1:b}, Lemma~\ref{lemma:compol:1} and the inductive hypothesis, $(x_{i+1}-x_s)\,\Phi^S$ equals
\begin{align*}
&(x_i-x_{i+1})\,(\id-\sigma_{i+1,s}^{s+1})\,g^{(1)}_{i+1,j}(S\setminus\{s\})
\\
&-\sigma_{i,i+1}^{i+2}\bigl[(x_i-x_s)\,g^{(1)}_{i,j}(S)\bigr]
+\sigma_{i,i+1}^{i+1}\bigl[(x_i-x_s)g^{(1)}_{i,j}(S)\bigr]
\\
=&(x_i-x_{i+1})\,(\id-\sigma_{i+1,s}^{s+1})\,g^{(1)}_{i+1,j}(S\setminus\{s\})
-\sigma_{i,i+1}^{i+2}(\id-\sigma_{i,s}^{s+1})g^{(1)}_{i,j}(S\setminus\{s\})
\\
&+\sigma_{i,i+1}^{i+1}(\id-\sigma_{i,s}^{s+1})g^{(1)}_{i,j}(S\setminus\{s\})=(\id-\sigma_{i+1,s}^{s+1})\Phi^{S\setminus\{s\}}=0.
\end{align*}
Therefore $\Phi^S=0$.

\smallskip

\ref{lemma:compol:10:part:b} By~\ref{g-2:b}, the difference of the left-hand side and the right-hand side
of~\ref{lemma:compol:10:part:b} multiplied by $(x_i-x_{i+1})$ equals
$$
(x_i-x_{i+1})\,g^{(1)}_{i+1,j}(S)-\bigl(\id-\sigma_{i,i+1}^{i+1}\bigr)g^{(2)}_{i,i+1,i+1,j}(S).
$$
Denote the last expression by $\Psi^S$ and prove by induction on $|S|$ that $\Psi^S=0$ for all  $S\subset(i+1..j)$. By~\ref{g-1:a} and~\ref{g-2:a}, we get
\begin{align*}
\Psi^\emptyset=&(x_i-x_{i+1})\cdot(x_{i+1}-y_{i+2})\cdots(x_{i+1}-y_j)
\\
&-\bigl(\id-\sigma_{i,i+1}^{i+1}\bigr)\bigl[(x_i-y_{i+1})\cdot(x_{i+1}-y_{i+2})\cdots(x_{i+1}-y_j)\bigr]
\\
=&(x_i-x_{i+1})\cdot(x_{i+1}-y_{i+2})\cdots(x_{i+1}-y_j)
\\
&-(x_i-y_{i+1})\cdot(x_{i+1}-y_{i+2})\cdots(x_{i+1}-y_j)
\\
&+(x_{i+1}-y_{i+1})\cdot(x_{i+1}-y_{i+2})\cdots(x_{i+1}-y_j)
=0.
\end{align*}

Now let $S\ne\emptyset$ and $s:=\min S$. We have $i+1<s<j$. So, by Lemma~\ref{lemma:compol:2},
formulas~\ref{g-1:b} and~\ref{g-2:b} and the inductive hypothesis, $(x_{i+1}-x_s)\,\Psi^S$ equals
\begin{align*}
&(x_i{-}x_{i+1}){\cdot}(\id-\sigma_{i+1,s}^{s+1})g^{(1)}_{i+1,j}(S\setminus\{s\})
{-}(\id-\sigma_{i,i+1}^{i+1})(\id-\sigma_{i+1,s}^{s+1})g^{(2)}_{i,i+1,i+1,j}(S\setminus\{s\})
\\
=&(\id-\sigma_{i+1,s}^{s+1})\bigl[(x_i-x_{i+1})\,g^{(1)}_{i+1,j}(S\setminus\{s\})
-(\id-\sigma_{i,i+1}^{i+1})g^{(2)}_{i,i+1,i+1,j}(S\setminus\{s\})\bigr]
\\
=&(\id-\sigma_{i+1,s}^{s+1})\Psi^{S\setminus\{s\}}=0.
\end{align*}
Therefore $\Psi^S=0$.
\end{proof}

\begin{lemma}\label{lemma:compol:11}
Let $i+1<q\le j$, $k\ge i+2$, and $S\subset(i{+}1..j]$. Then:
{\renewcommand{\labelenumi}{{\rm \theenumi}}
\renewcommand{\theenumi}{{\rm(\alph{enumi})}}
\begin{enumerate}
\item\label{lemma:compol:11:part:a} $(x_{i+1}-y_{i+1})\,g^{(2)}_{i+1,i+1,q,j}(S)+(x_i-x_{i+1})\,g^{(2)}_{i,i,q,j}\bigl(\{i{+}1\}\cup S\bigr)=g^{(2)}_{i,i,q,j}(S)${\rm;}\\[-3pt]
\item\label{lemma:compol:11:part:b} $g^{(2)}_{i+1,i+1,q,j}(S)=g^{(2)}_{i,i+1,q,j}(\{i{+}1\}\cup S)${\rm;}\\[-3pt]
\item\label{lemma:compol:11:part:c} $(x_i-y_{i+1})\,g^{(2)}_{i+1,i+2,q,j}(S)=g^{(2)}_{i,i+2,q,j}(S)$ if \,$i+2\in S${\rm;}\\[-3pt]
\item\label{lemma:compol:11:part:d} $(x_i-y_{i+1})\,g^{(2)}_{i+1,k,q,j}(S)+(x_i-x_{i+1})\,g^{(2)}_{i,k,q,j}\bigl(\{i{+}1\}\cup S\bigr)=g^{(2)}_{i,k,q,j}(S)$.
\end{enumerate}}
\end{lemma}
\begin{proof}\ref{lemma:compol:11:part:a}
Denote by $\Phi^S$ the difference of the right-hand side and left-hand side of~\ref{lemma:compol:11:part:a}.
Applying~\ref{g-2:b}, we get
\begin{align*}
\Phi^S&=g^{(2)}_{i,i,q,j}(S)-(x_{i+1}-y_{i+1})\,g^{(2)}_{i+1,i+1,q,j}(S)-\bigl(\id-\sigma_{i,i+1}^{i+1}\bigr)\,g^{(2)}_{i,i,q,j}(S)\\
&=\sigma_{i,i+1}^{i+1}g^{(2)}_{i,i,q,j}(S)-(x_{i+1}-y_{i+1})\,g^{(2)}_{i+1,i+1,q,j}(S).
\end{align*}
We prove $\Phi^S=0$ by induction on $|S|$.
By~\ref{g-2:a}, we get
$$
\Phi^\emptyset=\sigma_{i,i+1}^{i+1}\bigl[(x_i-y_{i+1})\cdots(x_i-y_j)\bigr]-(x_{i+1}-y_{i+1})\cdot(x_{i+1}-y_{i+2})\cdots(x_{i+1}-y_j)=0.
$$
Let $S\ne\emptyset$ and $s:=\min S$. Note that
$l^{(2)}_{i,i,q,j}(s)=l^{(2)}_{i+1,i+1,q,j}(s)$. Denote this number by $h$.
By~\ref{g-2:b}, Lemma~\ref{lemma:compol:1} and the inductive hypothesis,
$(x_{i+1}-x_s)\,\Phi^S$ equals
\begin{align*}
&\sigma_{i,i+1}^{i+1}\bigl[(x_i-x_s) g^{(2)}_{i,i,q,j}(S)\bigr]
-(x_{i+1}-y_{i+1}) (x_{i+1}-x_s)g^{(2)}_{i+1,i+1,q,j}(S)
\\
=&\sigma_{i,i+1}^{i+1}(\id-\sigma_{i,s}^{s+h})g^{(2)}_{i,i,q,j}(S\setminus\{s\})
-(x_{i+1}-y_{i+1})(\id-\sigma_{i+1,s}^{s+h})g^{(2)}_{i+1,i+1,q,j}(S\setminus\{s\})
\\
=&(\id-\sigma_{i+1,s}^{s+h})\sigma_{i,i+1}^{i+1}g^{(2)}_{i,i,q,j}(S\setminus\{s\})
-(\id-\sigma_{i+1,s}^{s+h})\bigl[(x_{i+1}-y_{i+1})g^{(2)}_{i+1,i+1,q,j}(S\setminus\{s\})\bigr]
\\
=&\bigl(\id-\sigma_{i+1,s}^{s+h}\bigr)\Phi^{S\setminus\{s\}}=0.
\end{align*}
Therefore $\Phi^S=0$.

\smallskip

\ref{lemma:compol:11:part:b} We denote by $\Phi^S$ the difference of the right-hand side and the left-hand side
of~\ref{lemma:compol:11:part:b} multiplied by $(x_i-x_{i+1})$.
By~\ref{g-2:b}, we have
$$
\Phi^S=\bigl(\id-\sigma_{i,i+1}^{i+1}\bigr)g^{(2)}_{i,i+1,q,j}(S)-(x_i-x_{i+1})\,g^{(2)}_{i+1,i+1,q,j}(S).
$$
We prove $\Phi^S=0$ by induction on $|S|$.
By~\ref{g-2:a}, $\Phi^\emptyset$ equals
\begin{align*}
&(\id-\sigma_{i,i+1}^{i+1})\bigl[(x_i-y_{i+1})\cdot(x_{i+1}-y_{i+2})\cdots(x_{i+1}-y_j)\bigr]
\\
&-(x_i{-}x_{i+1})\cdot(x_{i+1}-y_{i+2})\cdots(x_{i+1}-y_j)\\
=&(x_i-y_{i+1})\cdot(x_{i+1}-y_{i+2})\cdots(x_{i+1}-y_j)
\\
&-(x_{i+1}-y_{i+1})\cdot(x_{i+1}-y_{i+2})\cdots(x_{i+1}-y_j)
\\
&-(x_i-x_{i+1})\cdot(x_{i+1}-y_{i+2})\cdots(x_{i+1}-y_j)=0.
\end{align*}
Let $S\ne\emptyset$ and $s:=\min S$.
Note that $l^{(2)}_{i,i+1,q,j}(s)=l^{(2)}_{i+1,i+1,q,j}(s):=h$.
By~\ref{g-2:b}, Lemma~\ref{lemma:compol:2} and the inductive hypothesis, $(x_{i+1}{-}x_s)\,\Phi^S$ equals
\begin{align*}
&(\id-\sigma_{i,i+1}^{i+1})\bigl[(x_{i+1}{-}x_s)\,g^{(2)}_{i,i+1,q,j}(S)\bigr]
-(x_i{-}x_{i+1})\,(x_{i+1}{-}x_s)\,g^{(2)}_{i+1,i+1,q,j}(S)
\\
=&(\id-\sigma_{i,i+1}^{i+1})(\id-\sigma_{i+1,s}^{s+h})g^{(2)}_{i,i+1,q,j}(S\setminus\{s\})
\\
&-(x_i-x_{i+1})(\id-\sigma_{i+1,s}^{s+h})g^{(2)}_{i+1,i+1,q,j}(S\setminus\{s\})
\\
=&(\id-\sigma_{i+1,s}^{s+h})(\id-\sigma_{i,i+1}^{i+1})g^{(2)}_{i,i+1,q,j}(S\setminus\{s\})
\\
&-(\id-\sigma_{i+1,s}^{s+h})\bigl[(x_i-x_{i+1}) g^{(2)}_{i+1,i+1,q,j}(S\setminus\{s\})\bigr]
=(\id-\sigma_{i+1,s}^{s+h})\Phi^{S\setminus\{s\}}=0.
\end{align*}

\smallskip

\ref{lemma:compol:11:part:c} Set $T:=S\setminus\{i+2\}$.
Denote by $\Phi^T$ the difference of the left-hand side and the right-hand side of~\ref{lemma:compol:11:part:c} multiplied by $(x_i-x_{i+2})\,(x_{i+1}-x_{i+2})$.
By~\ref{g-2:b},
\begin{align*}
\Phi^T=&(x_i-y_{i+1})\,(x_i-x_{i+2})\,(\id-\sigma_{i+1,i+2}^{i+2})\,g^{(2)}_{i+1,i+2,q,j}(T)
\\
&-(x_{i+1}-x_{i+2})\,(\id-\sigma_{i,i+2}^{i+2})\,g^{(2)}_{i,i+2,q,j}(T).
\end{align*}
We prove $\Phi^T=0$ by induction on $|T|$.
By~\ref{g-2:a}, we get
\begin{align*}
&(x_i-y_{i+1})(x_i-x_{i+2})(\id-\sigma_{i+1,i+2}^{i+2})\bigl[(x_{i+1}-y_{i+2})\cdot(x_{i+2}-y_{i+3})\cdots(x_{i+2}-y_j)\bigr]
\\
&-(x_{i+1}-x_{i+2})(\id-\sigma_{i,i+2}^{i+2})\bigl[(x_i-y_{i+1})(x_i-y_{i+2})\cdot(x_{i+2}-y_{i+3})\cdots(x_{i+2}-y_j)\bigr]
\\
=&(x_i-y_{i+1})(x_i-x_{i+2})(x_{i+1}-y_{i+2})\cdot(x_{i+2}-y_{i+3})\cdots(x_{i+2}-y_j)
\\
&-(x_i-y_{i+1})(x_i-x_{i+2})\,(x_{i+2}-y_{i+2})\cdot(x_{i+2}-y_{i+3})\cdots(x_{i+2}-y_j)
\\
&-(x_{i+1}-x_{i+2})(x_i-y_{i+1})\,(x_i-y_{i+2})\cdot(x_{i+2}-y_{i+3})\cdots(x_{i+2}-y_j)
\\
&+(x_{i+1}-x_{i+2})\,(x_i-y_{i+1})\,(x_{i+2}-y_{i+2})\cdot(x_{i+2}-y_{i+3})\cdots(x_{i+2}-y_j).
\end{align*}
So
\begin{align*}
\Phi^\emptyset=&(x_i-y_{i+1})\,(x_i-x_{i+2})\,(x_{i+1}-x_{i+2})\cdot(x_{i+2}-y_{i+3})\cdots(x_{i+2}-y_j)
\\
&-(x_{i+1}-x_{i+2})\,(x_i-y_{i+1})\,(x_i-x_{i+2})\cdot(x_{i+2}-y_{i+3})\cdots(x_{i+2}-y_j)=0.
\end{align*}
Let $T{\ne}\emptyset$ and $t:=\min T$. Note that $l^{(2)}_{i+1,i+2,q,j}(t)=l^{(2)}_{i,i+2,q,j}(t)=:h$. By~\ref{g-2:b}, Lemma~\ref{lemma:compol:2} and the inductive hypothesis, $(x_{i+2}-x_t)\Phi^T$ equals
\begin{align*}
&(x_i-y_{i+1})(x_i-x_{i+2})(\id-\sigma_{i+1,i+2}^{i+2})\bigl[(x_{i+2}-x_t) g^{(2)}_{i+1,i+2,q,j}(T)\bigr]
\\
&-(x_{i+1}-x_{i+2})(\id-\sigma_{i,i+2}^{i+2})\bigl[(x_{i+2}-x_t)\,g^{(2)}_{i,i+2,q,j}(T)\bigr]
\\
=&(x_i-y_{i+1})\,(x_i-x_{i+2})\,\bigl(\id-\sigma_{i+1,i+2}^{i+2}\bigr)\bigl(\id-\sigma_{i+2,t}^{t+h}\bigr)g^{(2)}_{i+1,i+2,q,j}\bigl(T\setminus\{t\}\bigr)
\\
&-(x_{i+1}-x_{i+2})\bigl(\id-\sigma_{i,i+2}^{i+2}\bigr)\bigl(\id-\sigma_{i+2,t}^{t+h}\bigr)g^{(2)}_{i,i+2,q,j}\bigl(T\setminus\{t\}\bigr)
\\
=&(\id-\sigma_{i+2,t}^{t+h})\bigl[(x_i-y_{i+1})\,(x_i-x_{i+2})\,(\id-\sigma_{i+1,i+2}^{i+2})g^{(2)}_{i+1,i+2,q,j}(T\setminus\{t\})
\\
&-(x_{i+1}-x_{i+2})\,(\id-\sigma_{i,i+2}^{i+2})g^{(2)}_{i,i+2,q,j}(T\setminus\{t\})\bigr]=(\id-\sigma_{i+2,t}^{t+h})\Phi^{T\setminus\{t\}}=0.
\end{align*}

\smallskip

\ref{lemma:compol:11:part:d} Denote by $\Phi^S$ the difference of the right-hand side and the left-hand side
of~\ref{lemma:compol:11:part:d}. By~\ref{g-2:b},
\begin{align*}
\Phi^S=&g^{(2)}_{i,k,q,j}(S)-(x_i-y_{i+1})\,g^{(2)}_{i+1,k,q,j}(S)-(\id-\sigma_{i,i+1}^{i+2})g^{(2)}_{i,k,q,j}(S)
\\
=&\sigma_{i,i+1}^{i+2}g^{(2)}_{i,k,q,j}(S)-(x_i-y_{i+1})\,g^{(2)}_{i+1,k,q,j}(S).
\end{align*}
We prove $\Phi^S=0$ by induction on $|S|$. By~\ref{g-2:a},
\begin{align*}
\Phi^\emptyset=&\sigma_{i,i+1}^{i+2}\bigl[(x_i-y_{i+1})\cdots(x_i-y_k)\cdot(x_k-y_{k+1})\cdots(x_k-y_j)\bigr]
\\
&-(x_i-y_{i+1})\cdot(x_{i+1}-y_{i+2})\cdots(x_{i+1}-y_k)\cdot(x_k-y_{k+1})\cdots(x_k-y_j)=0.
\end{align*}
Let $S\ne\emptyset$ and $s:=\min S$. Note that $l^{(2)}_{i,k,q,j}(s)=l^{(2)}_{i+1,k,q,j}(s)=:h$. Moreover, we set $s_0:=\{k\}_s^i$ and $s'_0:=\{k\}_s^{i+1}$.
If $s>k$ then $s_0=s'_0=k$. By Lemma~\ref{lemma:compol:2},
$$
\begin{array}{l}
\sigma_{i,i+1}^{i+2}(x_{s_0}-x_s)=\sigma_{i,i+1}^{i+2}(x_k-x_s)=x_k-x_s=x_{s'_0}-x_s,\\[6pt]
\sigma_{i,i+1}^{i+2}\sigma_{s_0,s}^{s+h}=\sigma_{i,i+1}^{i+2}\sigma_{k,s}^{s+h}=\sigma_{k,s}^{s+h}\sigma_{i,i+1}^{i+2}=\sigma_{{s'_0},s}^{s+h}\sigma_{i,i+1}^{i+2}.
\end{array}
$$
If $s\le k$ then $s_0=i$ and $s'_0=i+1$. By Lemma~\ref{lemma:compol:1}, we get
$$
\begin{array}{l}
\sigma_{i,i+1}^{i+2}(x_{s_0}-x_s)=\sigma_{i,i+1}^{i+2}(x_i-x_s)=x_{i+1}-x_s=x_{s'_0}-x_s,\\[6pt]
\sigma_{i,i+1}^{i+2}\sigma_{s_0,s}^{s+h}=\sigma_{i,i+1}^{i+2}\sigma_{i,s}^{s+h}=\sigma_{i+1,s}^{s+h}\sigma_{i,i+1}^{i+2}=\sigma_{s'_0,s}^{s+h}\sigma_{i,i+1}^{i+2}.
\end{array}
$$
Thus in either case, we have
$$
\sigma_{i,i+1}^{i+2}(x_{s_0}-x_s)=x_{s'_0}-x_s, \quad
\sigma_{i,i+1}^{i+2}\sigma_{s_0,s}^{s+h}=\sigma_{s'_0,s}^{s+h}\sigma_{i,i+1}^{i+2}.
$$
By these formulas,~\ref{g-2:b} and the inductive hypothesis, $(x_{s'_0}-x_s)\,\Phi^S$ equals
\begin{align*}
&\sigma_{i,i+1}^{i+2}\bigl[(x_{s_0}-x_s)\,g^{(2)}_{i,k,q,j}(S)\bigr]-(x_i-y_{i+1})\,(x_{s'_0}-x_s)\,g^{(2)}_{i+1,k,q,j}(S)\\
=&\sigma_{i,i+1}^{i+2}(\id-\sigma_{s_0,s}^{s+h}) g^{(2)}_{i,k,q,j}(S\setminus\{s\})-(x_i-y_{i+1})(\id-\sigma_{s'_0,s}^{s+h}) g^{(2)}_{i+1,k,q,j}(S\setminus\{s\})
\\
=&\bigl(\id-\sigma_{s'_0,s}^{s+h}\bigr)\bigl[\sigma_{i,i+1}^{i+2}\,g^{(2)}_{i,k,q,j}\bigl(S\setminus\{s\}\bigr)-(x_i-y_{i+1})\,g^{(2)}_{i+1,k,q,j}\bigl(S\setminus\{s\}\bigr)\bigr]
\\
=&\bigl(\id-\sigma_{s'_0,s}^{s+h}\bigr)\Phi^{S\setminus\{s\}}=0.
\end{align*}
Therefore $\Phi^S=0$.
\end{proof}

\begin{lemma}\label{lemma:compol:12}
Let $i+1<q<j$ and $S\subset(q..j)$. Then:
{\renewcommand{\labelenumi}{{\rm \theenumi}}
\renewcommand{\theenumi}{{\rm(\alph{enumi})}}
\begin{enumerate}
\item\label{lemma:compol:12:part:a} $g_{q,j}^{(1)}(S)+g^{(2)}_{i,i,q-1,j}((i..q]\cup S)=g^{(2)}_{i,i,q,j}((i..q]\cup S)${\rm ;}\\
\item\label{lemma:compol:12:part:b} $g_{q,j}^{(1)}(S)+g^{(2)}_{i,i+1,q-1,j}((i..q]\cup S)=g^{(2)}_{i,i+1,q,j}((i..q]\cup S)$.
\end{enumerate}}
\end{lemma}
\begin{proof} Let $i':=i$ in case~\ref{lemma:compol:12:part:a} and $i':=i+1$ in case~\ref{lemma:compol:12:part:b}.
In both cases, we have $i'<q$. Note that
$g^{(2)}_{i,i',q-1,j}(S)=g^{(2)}_{i,i',q,j}(S)$ for any $S\subset(q..j)$.
This equality follows from~\ref{g-2:b} and the equality $l^{(2)}_{i,i',q-1,j}(s)=l^{(2)}_{i,i',q,j}(s)=1$
for any $s\in(q..j)$.

We apply downward induction on $k=q\ldots,i+1$ to prove the equality
$$
u_{i,k-1}^{\{i'\}}
\,g_{q,j}^{(1)}(S)+g^{(2)}_{i,i',q-1,j}\bigl([k..q]\cup S\bigr)=g^{(2)}_{i,i',q,j}\bigl([k..q]\cup S\bigr)
$$
Let $\Phi^S_k$ denote the difference of the left-hand side and the right-hand side in this formula.
The required formulas will follow from $\Phi_{i+1}^S=0$.

At first, we consider the case $k=q$. By~\ref{g-2:b}, $(x_{i'}-x_q)\Phi^S_q$ equals
\begin{align*}
&(x_{i'}-x_q)u_{i,q-1}^{\{i'\}}g_{q,j}^{(1)}(S)+(\id-\sigma_{i',q}^{q+1})g^{(2)}_{i,i',q-1,j}(S)
-(\id-\sigma_{i',q}^q)g^{(2)}_{i,i',q,j}(S)
\\
=&(x_{i'}-x_q)\,u_{i,q-1}^{\{i'\}}\,g_{q,j}^{(1)}(S)-\sigma_{i',q}^{q+1}g^{(2)}_{i,i',q,j}(S)+\sigma_{i',q}^q\,g^{(2)}_{i,i',q,j}(S)=:\Psi^S.
\end{align*}
We prove that $\Psi^S=0$ 
by induction on $|S|$.
By~\ref{g-1:a} and~\ref{g-2:a}, $\Psi^\emptyset$ equals
\begin{align*}
&(x_{i'}-x_q)u_{i,q-1}^{\{i'\}}u_{q,j}^\emptyset-\sigma_{i',q}^{q+1}u_{i,j}^{\{i'\}}+\sigma_{i',q}^qu_{i,j}^{\{i'\}}
\\
=&(x_{i'}-x_q) u_{i,q-1}^{\{i'\}}u_{q,j}^\emptyset-\sigma_{i',q}^{q+1}\bigl[u_{i,q-1}^{\{i'\}}\cdot(x_{i'}-y_q)\cdots(x_{i'}-y_j)\bigr]
\\
&+\sigma_{i',q}^q\bigl[u_{i,q-1}^{\{i'\}}\cdot(x_{i'}-y_q)\cdots(x_{i'}-y_j)\bigr]
\\
=&(x_{i'}-x_q)\,u_{i,q-1}^{\{i'\}}u_{q,j}^\emptyset-u_{i,q-1}^{\{i'\}}(x_{i'}-y_q)u_{q,j}^\emptyset
+u_{i,q-1}^{\{i'\}}(x_q-y_q)u_{q,j}^\emptyset
\\
=&u_{i,q-1}^{\{i'\}}u_{q,j}^\emptyset\,\bigl[(x_{i'}-x_q)-(x_{i'}-y_q)+(x_q-y_q)\bigr]=0.
\end{align*}
Let $S\ne\emptyset$ and $s:=\min S$. Note that $q+1\le s<j$.
By~\ref{g-1:b} and~\ref{g-2:b}, Lemma~\ref{lemma:compol:1}
and the inductive hypothesis, $(x_q-x_s)\Psi^S$ equals
\begin{align*}
&(x_{i'}-x_q)u_{i,q-1}^{\{i'\}}(\id-\sigma_{q,s}^{s+1})g_{q,j}^{(1)}(S)-\sigma_{i',q}^{q+1}\bigl[(x_{i'}-x_s)g^{(2)}_{i,i',q,j}(S)\bigr]
\\
&+\sigma_{i',q}^q\bigl[(x_{i'}-x_s)g^{(2)}_{i,i',q,j}(S)\bigr]
\\
=&\bigl(\id-\sigma_{q,s}^{s+1}\bigr)\bigl[(x_{i'}-x_q)\,u_{i,q-1}^{\{i'\}}\,g_{q,j}^{(1)}\bigl(S\setminus\{s\}\bigr)\bigr]
\\
&-\sigma_{i',q}^{q+1}\bigl(\id-\sigma_{i',s}^{s+1}\bigr)\,g^{(2)}_{i,i',q,j}\bigl(S\setminus\{s\}\bigr)
+\sigma_{i',q}^q\bigl(\id-\sigma_{i',s}^{s+1}\bigr)\,g^{(2)}_{i,i',q,j}\bigl(S\setminus\{s\}\bigr)
\\
=&\bigl(\id-\sigma_{q,s}^{s+1}\bigr)\bigl[(x_{i'}-x_q)\,u_{i,q-1}^{\{i'\}}\,g_{q,j}^{(1)}\bigl(S\setminus\{s\}\bigr)
\\
&-\sigma_{i',q}^{q+1}\,g^{(2)}_{i,i',q,j}\bigl(S\setminus\{s\}\bigr)+\sigma_{i',q}^q\,g^{(2)}_{i,i',q,j}\bigl(S\setminus\{s\}\bigr)\bigr]
=\bigl(\id-\sigma_{q,s}^{s+1}\bigr)\Psi^{S\setminus\{s\}}=0.
\end{align*}
Therefore $\Psi^S=0$. Thus we have proved $\Phi^S_q=0$.

Now let $i<k<q$. By the inductive assumption, $\Phi^S_{k+1}=0$.
We set $k_0:=\max_k\{i,i'\}$.
Note that $\sigma^k_{k_0,k}$ acts identically on $u_{i,k-1}^{\{i'\}}$ by definition and on $g_{q,j}^{(1)}(S)$ by Lemma~\ref{lemma:compol:4.5}.
On the other hand, $\sigma^k_{k_0,k}(x_{k_0}-y_k)=x_k-y_k$. Therefore, we get
\begin{align*}
\bigl(\id-\sigma_{k_0,k}^k\bigr)\bigl[u_{i,k}^{\{i'\}}g_{q,j}^{(1)}(S)\bigr]
=\bigl(\id-\sigma_{k_0,k}^k\bigr)\bigl[u_{i,k-1}^{\{i'\}}\,(x_{k_0}-y_k)\,g_{q,j}^{(1)}(S)\bigr]
\\
=u_{i,k-1}^{\{i'\}}(x_{k_0}-y_k)\,g_{q,j}^{(1)}(S)-u_{i,k-1}^{\{i'\}}(x_k-y_k)\,g_{q,j}^{(1)}(S)
=(x_{k_0}-x_k)u_{i,k-1}^{\{i'\}}\,g_{q,j}^{(1)}(S).
\end{align*}
By this formula and~\ref{g-2:b}, $(x_{k_0}-x_k)\,\Phi_k^S$ equals
\begin{align*}
&(x_{k_0}-x_k)u_{i,k-1}^{\{i'\}}\,g_{q,j}^{(1)}(S)
+\bigl(\id-\sigma_{k_0,k}^k\bigr)\,g^{(2)}_{i,i',q-1,j}\bigl([k{+}1..q]\cup S\bigr)
\\
&-\bigl(\id-\sigma_{k_0,k}^k\bigr)\,g^{(2)}_{i,i',q,j}\bigl([k{+}1..q]\cup S\bigr)
\\
=&\bigl(\id-\sigma_{k_0,k}^k\bigr)\bigl[u_{i,k}^{\{i'\}}g_{q,j}^{(1)}(S)
+g^{(2)}_{i,i',q-1,j}\bigl([k{+}1..q]\cup S\bigr)
-g^{(2)}_{i,i',q,j}\bigl([k{+}1..q]\cup S\bigr)\bigr]
\\
=&\bigl(\id-\sigma_{i,k}^k\bigr)\Phi_{k+1}^S=0.
\end{align*}
Therefore $\Phi_k^S=0$. This completes the inductive step.
\end{proof}

\begin{lemma}\label{lemma:compol:13}
Let $i+1<m<q\le j$ and $S\subset(m..j]$. Then
{
\renewcommand{\labelenumi}{{\rm \theenumi}}
\renewcommand{\theenumi}{{\rm(\alph{enumi})}}
\begin{enumerate}
\item\label{lemma:compol:13:part:a}$(x_m-y_m)g^{(2)}_{m,m,q,j}(S)+g^{(2)}_{i,i,q,j}\bigl((i..m{-}1)\cup\{m\}\cup S\bigr)$\\
                                  ${}$\hspace{0pt}$+(x_{m-1}-x_m)\,g^{(2)}_{i,i,q,j}\bigl((i..m]\cup S\bigr)=g^{(2)}_{i,i,q,j}\bigl((i..m)\cup S\bigr)${\rm;}\\[-2pt]
\item\label{lemma:compol:13:part:b}$(x_m-y_m)\,g^{(2)}_{m,m,q,j}(S)+\cond_{i+1<m-1}\,g^{(2)}_{i,i+1,q,j}\bigl((i..m{-}1)\cup\{m\}\cup S\bigr)$\\[4pt]
                                  ${}$\hspace{0pt}$+(x_{m-1}-x_m)\,g^{(2)}_{i,i+1,q,j}\bigl((i..m]\cup S\bigr)=g^{(2)}_{i,i+1,q,j}\bigl((i..m)\cup S\bigr)${\rm;}\\
\item\label{lemma:compol:13:part:c}$(x_i-y_{i+1})\,g^{(2)}_{m,m,q,j}(S)=g^{(2)}_{i,m,q,j}\bigl((i..m{-}1)\cup\{m\}\cup S\bigr)${\rm;}\\
\item\label{lemma:compol:13:part:d}$(x_i-y_{i+1})\,g^{(2)}_{m,m+1,q,j}(S)=g^{(2)}_{i,m+1,q,j}\bigl((i..m)\cup S\bigr)$ if $m{+}1\in S${\rm;}\\
\item\label{lemma:compol:13:part:e}$(x_i-y_{i+1})\,g^{(2)}_{m,k,q,j}(S)+g^{(2)}_{i,k,q,j}\bigl((i..m-1)\cup\{m\}\cup S\bigr)$\\
                                  ${}$\hspace{0pt}$+(x_{m-1}-x_m)\,g^{(2)}_{i,k,q,j}\bigl((i..m]\cup S\bigr)=g^{(2)}_{i,k,q,j}\bigl((i..m)\cup S\bigr)$\\[6pt]
                                  for all $k=m+2,\ldots,q$.
\end{enumerate}}
\end{lemma}
\begin{proof}\ref{lemma:compol:13:part:a},~\ref{lemma:compol:13:part:b}. Let $i':=i$
in case~\ref{lemma:compol:13:part:a} and $i':=i+1$ in case~\ref{lemma:compol:13:part:b}.
In both cases, we have $i'<m$. We prove the following equality by downward induction on $k=m-1,\dots,i+1$:
\begin{align*}
&(x_m-y_m)\,u_{i,k-1}^{\{i'\}}\,g^{(2)}_{m,m,q,j}(S)+
\cond_{i+1<m-1\text{ or }i'=i}\,g^{(2)}_{i,i',q,j}\bigl([k..m{-}1)\cup\{m\}\cup S\bigr)
\\
&+(x_{m-1}-x_m)\,g^{(2)}_{i,i',q,j}\bigl([k..m]\cup S\bigr)\ =\ g^{(2)}_{i,i',q,j}\bigl([k..m)\cup S\bigr).
\end{align*}
Let $\Phi^S_k$ denote the difference of the left-hand side and the right-hand side of this formula.
The assertions of parts~\ref{lemma:compol:13:part:a} and~\ref{lemma:compol:13:part:b}
will follow from $\Phi^S_{i+1}=0$.

\smallskip

{\it Case~1:} $m=i+2$ and $i'=i+1$. In this case, the only value of $k$ we need to consider is $i+1$.
By~\ref{g-2:b}, $(x_i-x_{i+1})\,\Phi^S_{i+1}$ equals
\begin{align*}
&(x_i-x_{i+1})\,(x_{i+2}-y_{i+2})\,g^{(2)}_{i+2,i+2,q,j}(S)
\\
&+(x_{i+1}-x_{i+2})\,(\id-\sigma_{i,i+1}^{i+1})g^{(2)}_{i,i+1,q,j}(\{i{+}2\}\cup S)
-(\id-\sigma_{i,i+1}^{i+1})g^{(2)}_{i,i+1,q,j}(S)
\\
=&(x_i-x_{i+1})\,(x_{i+2}-y_{i+2})\,g^{(2)}_{i+2,i+2,q,j}(S)
\\
&+\bigl(\id-\sigma_{i,i+1}^{i+1}\bigr)\bigl[(x_{i+1}-x_{i+2})\,g^{(2)}_{i,i+1,q,j}\bigl(\{i{+}2\}\cup S\bigr)\bigr]-\bigl(\id-\sigma_{i,i+1}^{i+1}\bigr)g^{(2)}_{i,i+1,q,j}(S)
\\
=&(x_i-x_{i+1})\,(x_{i+2}-y_{i+2})\,g^{(2)}_{i+2,i+2,q,j}(S)
\\
&+\bigl(\id-\sigma_{i,i+1}^{i+1}\bigr)\bigl(\id-\sigma_{i+1,i+2}^{i+2}\bigr)\,g^{(2)}_{i,i+1,q,j}(S)-\bigl(\id-\sigma_{i,i+1}^{i+1}\bigr)g^{(2)}_{i,i+1,q,j}(S)
\\
=&(x_i-x_{i+1})\,(x_{i+2}-y_{i+2})\,g^{(2)}_{i+2,i+2,q,j}(S)-\bigl(\id-\sigma_{i,i+1}^{i+1}\bigr)\sigma_{i+1,i+2}^{i+2}\,g^{(2)}_{i,i+1,q,j}(S)=:\Psi^S.
\end{align*}
We 
prove by induction on $|S|$ that $\Psi^S=0$.
By~\ref{g-2:a}, $\Psi^\emptyset$ equals
\begin{align*}
&(x_i-x_{i+1})\,(x_{i+2}-y_{i+2})\cdot(x_{i+2}-y_{i+3})\cdots(x_{i+2}-y_j)\\
&-\bigl(\id-\sigma_{i,i+1}^{i+1}\bigr)\sigma_{i+1,i+2}^{i+2}\bigl[(x_i-y_{i+1})\cdot(x_{i+1}-y_{i+2})\cdots(x_{i+1}-y_j)\bigr]
\\
=&(x_i-x_{i+1})\cdot(x_{i+2}-y_{i+2})\,\cdots(x_{i+2}-y_j)
\\&-\bigl(\id-\sigma_{i,i+1}^{i+1}\bigr)\bigl[(x_i-y_{i+1})\cdot(x_{i+2}-y_{i+2})\cdots(x_{i+2}-y_j)\bigr]
\\
=&(x_i-x_{i+1})(x_{i+2}-y_{i+2})\cdots(x_{i+2}-y_j)-(x_i-y_{i+1})(x_{i+2}-y_{i+2})\cdots(x_{i+2}-y_j)
\\
&+(x_{i+1}-y_{i+1})(x_{i+2}-y_{i+2})\cdots(x_{i+2}-y_j)=0.
\end{align*}
Let $S\ne\emptyset$ set $s:=\min S$.
We have $l^{(2)}_{i+2,i+2,q,j}(s)=l^{(2)}_{i,i+1,q,j}(s)=:h$.
By~\ref{g-2:b}, Lemmas~\ref{lemma:compol:1} and~\ref{lemma:compol:2}
and the inductive hypothesis, $(x_{i+2}-x_s)\,\Psi^S$ equals
\begin{align*}
&(x_i-x_{i+1})\,(x_{i+2}-y_{i+2})\,\bigl(\id-\sigma_{i+2,s}^{s+h}\bigr)g^{(2)}_{i+2,i+2,q,j}(S\setminus\{s\})
\\
&-\bigl(\id-\sigma_{i,i+1}^{i+1}\bigr)\sigma_{i+1,i+2}^{i+2}\bigl[(x_{i+1}-x_s)\,g^{(2)}_{i,i+1,q,j}(S)\bigr]
\\
=&\bigl(\id-\sigma_{i+2,s}^{s+h}\bigr)\bigl[(x_i-x_{i+1})\,(x_{i+2}-y_{i+2})\,g^{(2)}_{i+2,i+2,q,j}\bigl(S\setminus\{s\}\bigr)\bigr]
\\
&-\bigl(\id-\sigma_{i,i+1}^{i+1}\bigr)\sigma_{i+1,i+2}^{i+2}\bigl(\id-\sigma_{i+1,s}^{s+h}\bigr)\,g^{(2)}_{i,i+1,q,j}\bigl(S\setminus\{s\}\bigr)
\\
=&\bigl(\id-\sigma_{i+2,s}^{s+h}\bigr)\bigl[(x_i-x_{i+1})\,(x_{i+2}-y_{i+2})\,g^{(2)}_{i+2,i+2,q,j}\bigl(S\setminus\{s\}\bigr)\bigr]
\\
&-\bigl(\id-\sigma_{i+2,s}^{s+h}\bigr)\bigl(\id-\sigma_{i,i+1}^{i+1}\bigr)\sigma_{i+1,i+2}^{i+2}\,g^{(2)}_{i,i+1,q,j}\bigl(S\setminus\{s\}\bigr)
=\bigl(\id-\sigma_{i+2,s}^{s+h}\bigr)\Psi^{S\setminus\{s\}}=0.
\end{align*}

We have proved that $\Phi^S_{i+1}=0$ for any $S$ as in this lemma.
Thus we have proved part~\ref{lemma:compol:13:part:b} in the case $m=i+2$.

{\it Case 2:} $m>i+2$ or $i'=i$. In this case $i'<m-1$. By~\ref{g-2:b}, $(x_{i'}-x_{m-1})\,\Phi^S_{m-1}$ equals
\begin{align*}
&(x_{i'}-x_{m-1})\,(x_m-y_m)\,u_{i,m-2}^{\{i'\}}\,g^{(2)}_{m,m,q,j}(S)+(x_{i'}-x_{m-1})\,g^{(2)}_{i,i',q,j}\bigl(\{m\}\cup S\bigr)
\\
&+(x_{m-1}-x_m)\bigl(\id-\sigma_{i',m-1}^{m-1}\bigr)g^{(2)}_{i,i',q,j}\bigl(\{m\}\cup S\bigr)-\bigl(\id-\sigma_{i',m-1}^{m-1}\bigr)g^{(2)}_{i,i',q,j}(S)
\\
=&\bigl[(x_{i'}-x_{m-1})+(x_{m-1}-x_m)\bigr]\,g^{(2)}_{i,i',q,j}\bigl(\{m\}\cup S\bigr)
\\
&-\sigma_{i',m-1}^{m-1}\bigl[(x_{i'}-x_m)\,g^{(2)}_{i,i',q,j}\bigl(\{m\}\cup S\bigr)\bigr]
\\
&+(x_{i'}-x_{m-1})\,(x_m-y_m)\,u_{i,m-2}^{\{i'\}}\,g^{(2)}_{m,m,q,j}(S)-\bigl(\id-\sigma_{i',m-1}^{m-1}\bigr)\,g^{(2)}_{i,i',q,j}(S),
\end{align*}
\begin{align*}
=&\bigl(\id-\sigma_{i',m}^m\bigr)g^{(2)}_{i,i',q,j}(S)-\sigma_{i',m-1}^{m-1}\bigl(\id-\sigma_{i',m}^m\bigr)\,g^{(2)}_{i,i',q,j}(S)
\\
&+(x_{i'}-x_{m-1})\,(x_m-y_m)\,u_{i,m-2}^{\{i'\}}\,g^{(2)}_{m,m,q,j}(S)
-\bigl(\id-\sigma_{i',m-1}^{m-1}\bigr)g^{(2)}_{i,i',q,j}(S)
\\
=&(x_{i'}-x_{m-1})\,(x_m-y_m)\,u_{i,m-2}^{\{i'\}}\,g^{(2)}_{m,m,q,j}(S)-\bigl(\id-\sigma_{i',m-1}^{m-1}\bigr)\sigma_{i',m}^m\,g^{(2)}_{i,i',q,j}(S)=:\Psi^S.
\end{align*}
We
prove by induction on $|S|$ that $\Psi^S=0$.
If $m>i+2$ then by~\ref{g-2:a}, $\Psi^\emptyset$ equals
\begin{align*}
&(x_{i'}-x_{m-1})(x_m-y_m)\cdot(x_i-y_{i+1})\cdot(x_{i'}-y_{i+2})\cdots(x_{i'}-y_{m-2})
\\
&\times (x_m-y_{m+1})\cdots(x_m-y_j)
\\
&-(\id-\sigma_{i',m-1}^{m-1})\sigma_{i',m}^m\bigl[(x_i-y_{i+1})\cdot(x_{i'}-y_{i+2})\cdots(x_{i'}-y_j)\bigr]
\\
=&(x_{i'}-x_{m-1})\,(x_i-y_{i+1})\cdot(x_{i'}-y_{i+2})\cdots(x_{i'}-y_{m-2})\cdot(x_m-y_m)\cdots(x_m-y_j)
\\
&-{\hspace{-1 mm}(\id-\sigma_{i',m-1}^{m-1})[(x_i-y_{i+1})(x_{i'}-y_{i+2})\cdots(x_{i'}-y_{m-1})(x_m-y_m)\cdots(x_m-y_j)]}
\\
=&(x_{i'}-x_{m-1})\,(x_i-y_{i+1})(x_{i'}-y_{i+2})\cdots(x_{i'}-y_{m-2})(x_m-y_m)\cdots(x_m-y_j)
\\
&-(x_i-y_{i+1})(x_{i'}-y_{i+2})\cdots(x_{i'}-y_{m-1})(x_m-y_m)\cdots(x_m-y_j)
\\
&+(x_i-y_{i+1})(x_{i'}-y_{i+2})\cdots(x_{i'}-y_{m-2})(x_{m-1}-y_{m-1})(x_m-y_m)\cdots(x_m-y_j),
\end{align*}
which equals $0$.
If $m=i+2$ then $i'=i$. By~\ref{g-2:a}, $\Psi^\emptyset$ equals
\begin{align*}
&(x_i-x_{i+1})\,(x_{i+2}-y_{i+2})\cdot(x_{i+2}-y_{i+3})\cdots(x_{i+2}-y_j)
\\
&-\bigl(\id-\sigma_{i,i+1}^{i+1}\bigr)\sigma_{i,i+2}^{i+2}\bigl[(x_i-y_{i+1})\cdots(x_i-y_j)\bigr]
\\
=&(x_i-x_{i+1})\cdot(x_{i+2}-y_{i+2})\cdots(x_{i+2}-y_j)
\\
&-\bigl(\id-\sigma_{i,i+1}^{i+1}\bigr)\bigl[(x_i-y_{i+1})\cdot(x_{i+2}-y_{i+2})\cdots(x_{i+2}-y_j)\bigr]
\\
=&(x_i-x_{i+1})\cdot(x_{i+2}-y_{i+2})\cdots(x_{i+2}-y_j)
\\
&-(x_i-y_{i+1})\cdot(x_{i+2}-y_{i+2})\cdots(x_{i+2}-y_j)
\\
&+(x_{i+1}-y_{i+1})\cdot(x_{i+2}-y_{i+2})\cdots(x_{i+2}-y_j)=0
\end{align*}
Let $S\ne\emptyset$ and $s:=\min S$. We have $s>m>i'$ and thus
$l^{(2)}_{m,m,q,j}(s)=l^{(2)}_{i,i',q,j}(s)=:h$.
By~\ref{g-2:b}, Lemmas~\ref{lemma:compol:1} and~\ref{lemma:compol:2}
and the inductive hypothesis, $(x_m-x_s)\,\Psi^S$ equals
\begin{align*}
&(x_{i'}-x_{m-1})\,(x_m-y_m)\,u_{i,m-2}^{\{i'\}}\,\bigl(\id-\sigma_{m,s}^{s+h}\bigr)\,g^{(2)}_{m,m,q,j}\bigl(S\setminus\{s\}\bigr)\\
&-\bigl(\id-\sigma_{i',m-1}^{m-1}\bigr)\sigma_{i',m}^m\bigl[(x_{i'}-x_s)\,g^{(2)}_{i,i',q,j}(S)\bigr]
\\
=&\bigl(\id-\sigma_{m,s}^{s+h}\bigr)\bigl[(x_{i'}-x_{m-1})\,(x_m-y_m)\,u_{i,m-2}^{\{i'\}}\,g^{(2)}_{m,m,q,j}\bigl(S\setminus\{s\}\bigr)\bigr]
\\
&-\bigl(\id-\sigma_{i',m-1}^{m-1}\bigr)\sigma_{i',m}^m\bigl(\id-\sigma_{i',s}^{s+h}\bigr)\,g^{(2)}_{i,i',q,j}\bigl(S\setminus\{s\}\bigr)
\\
=&\bigl(\id-\sigma_{m,s}^{s+h}\bigr)\bigl[(x_m-y_m)\,(x_{i'}-x_{m-1})\,u_{i,m-2}^{\{i'\}}\,g^{(2)}_{m,m,q,j}\bigl(S\setminus\{s\}\bigr)\bigr]
\\
&-\bigl(\id-\sigma_{m,s}^{s+h}\bigr)\bigl(\id-\sigma_{i',m-1}^{m-1}\bigr)\sigma_{i',m}^m\,g^{(2)}_{i,i',q,j}\bigl(S\setminus\{s\}\bigr)=\bigl(\id-\sigma_{m,s}^{s+h}\bigr)\Psi^{S\setminus\{s\}}=0.
\end{align*}
Therefore $\Phi^S_{m-1}=0$

Now suppose that $i<k<m-1$ and we already proved $\Phi^S_{k+1}=0$ for any $S\subset(m..j]$.
We set $k_0:=\max_k\{i,i'\}$.
It is easy to see that $\sigma^k_{k_0,k}$ acts identically on $u_{i,k-1}^{\{i'\}}$, $x_m-y_m$ and $x_{m-1}-x_m$.
This operator also acts identically on $g^{(2)}_{m,m,q,j}(S)$ by Lemma~\ref{lemma:compol:4.5}.
On the other hand, $\sigma^k_{k_0,k}(x_{k_0}-y_k)=x_k-y_k$. Therefore,
\begin{align*}
&\bigl(\id-\sigma_{k_0,k}^{k}\bigr)\bigl[u_{i,k}^{\{i'\}}\,g^{(2)}_{m,m,q,j}(S)\bigr]
\\
=&(x_{k_0}-y_k)\,u_{i,k-1}^{\{i'\}}\,g^{(2)}_{m,m,q,j}(S)
-(x_k-y_k)\,u_{i,k-1}^{\{i'\}}\,g^{(2)}_{m,m,q,j}(S)\\
=&(x_{k_0}-x_k)\,u_{i,k-1}^{\{i'\}}\,g^{(2)}_{m,m,q,j}(S).
\end{align*}
By~\ref{g-2:b}, $(x_{k_0}-x_k)\,\Phi_k^S$ equals
\begin{align*}
&(x_m\!-\!y_m)(x_{k_0}\!-\!x_k)u_{i,k-1}^{\{i'\}}g^{(2)}_{m,m,q,j}(S)
+\bigl(\id-\sigma_{k_0,k}^{k}\bigr)g^{(2)}_{i,i',q,j}\bigl([k{+}1..m{-}1)\!\cup\!\{m\}\!\cup\! S\bigr)
\\
&+\!(x_{m-1}-x_m)\bigl(\id\!-\sigma_{k_0,k}^{k}\bigr)g^{(2)}_{i,i',q,j}\bigl([k{+}1..m]\!\cup\! S\bigr)
\!-\!\bigl(\id-\sigma_{k_0,k}^{k}\bigr)g^{(2)}_{i,i',q,j}\bigl([k{+}1..m)\!\cup\! S\bigr)
\\
&=\bigl(\id-\sigma_{k_0,k}^{k}\bigr)\bigl[(x_m-y_m)\,u_{i,k}^{\{i'\}}\,g^{(2)}_{m,m,q,j}(S)
+(x_{m-1}-x_m)\,g^{(2)}_{i,i',q,j}\bigl([k+1..m]\cup S\bigr)
\\
&+g^{(2)}_{i,i',q,j}\bigl([k{+}1..m{-}1)\cup\{m\}\cup S\bigr)-\bigl(\id-\sigma_{k_0,k}^{k}\bigr)g^{(2)}_{i,i',q,j}\bigl([k+1..m)\cup S\bigr)\bigr]
\\
&=\bigl(\id-\sigma_{k_0,k}^{k}\bigr)\Phi_{k+1}^S=0.
\end{align*}
Therefore $\Phi_k^S=0$. This completes the inductive step. 

\smallskip

\ref{lemma:compol:13:part:c} We apply downward induction on $k=m-1,\dots,i+1$ to prove
$$
(x_i-y_{i+1})\cdots(x_i-y_k)\,g^{(2)}_{m,m,q,j}(S)=g^{(2)}_{i,m,q,j}\bigl([k..m{-}1)\cup\{m\}\cup S\bigr)
$$
Let $\Phi^S_k$ denote the difference of the left-hand side and the right-hand side of this formula.
The required equality will follow from $\Phi^S_{i+1}=0$.

First, we consider the case $k=m-1$. By~\ref{g-2:b}, $(x_i-x_m)\,\Phi^S_{m-1}$ equals
$$
(x_i-x_m)\cdot(x_i-y_{i+1})\cdots(x_i-y_{m-1})\cdot g^{(2)}_{m,m,q,j}(S)-\bigl(\id-\sigma_{i,m}^m\bigr)g^{(2)}_{i,m,q,j}(S)=:\Psi^S.
$$
We 
prove by induction on $|S|$ that $\Psi^S=0$. 
By~\ref{g-2:a}, $\Psi^\emptyset$ equals
\begin{align*}
(x_i-x_m)\cdot(x_i-y_{i+1})\cdots(x_i-y_{m-1})\cdot(x_m-y_{m+1})\cdots(x_m-y_j)\\
-\bigl(\id-\sigma_{i,m}^m\bigr)\bigl[(x_i-y_{i+1})\cdots(x_i-y_m)\cdot(x_m-y_{m+1})\cdots(x_m-y_j)\bigr]
\\
=(x_i-x_m)\cdot(x_i-y_{i+1})\cdots(x_i-y_{m-1})\cdot(x_m-y_{m+1})\cdots(x_m-y_j)
\\
-(x_i-y_{i+1})\cdots(x_i-y_m)\cdot(x_m-y_{m+1})\cdots(x_m-y_j)
\\
+(x_i-y_{i+1})\cdots(x_i-y_{m-1})\cdot(x_m-y_m)\cdots(x_m-y_j)=0.
\end{align*}

Let $S\ne\emptyset$ and $s:=\min S$. We have $m<s$ and, therefore,
$l^{(2)}_{m,m,q,j}(s)=l^{(2)}_{i,m,q,j}(s)=:h$.
By~\ref{g-2:b}, Lemma~\ref{lemma:compol:2} and the inductive hypothesis, $(x_m-x_s)\,\Psi^S$ equals
\begin{align*}
&(x_i-x_m)\cdot(x_i-y_{i+1})\cdots(x_i-y_{m-1})\cdot\bigl(\id-\sigma_{m,s}^{s+h}\bigr)g^{(2)}_{m,m,q,j}\bigl(S\setminus\{s\}\bigr)
\\
&-\bigl(\id-\sigma_{i,m}^m\bigr)\bigl[(x_m-x_s)\,g^{(2)}_{i,m,q,j}(S)\bigr]
\\
=&\bigl(\id-\sigma_{m,s}^{s+h}\bigr)\bigl[(x_i-x_m)\cdot(x_i-y_{i+1})\cdots(x_i-y_{m-1})\cdot g^{(2)}_{m,m,q,j}\bigl(S\setminus\{s\}\bigr)\bigr]
\\
&-\bigl(\id-\sigma_{i,m}^m\bigr)\bigl(\id-\sigma_{m,s}^{s+h}\bigr)g^{(2)}_{i,m,q,j}\bigl(S\setminus\{s\}\bigr)
\\
=&\bigl(\id-\sigma_{m,s}^{s+h}\bigr)\bigl[(x_i-x_m)\cdot(x_i-y_{i+1})\cdots(x_i-y_{m-1})\cdot g^{(2)}_{m,m,q,j}\bigl(S\setminus\{s\}\bigr)
\\
&-\bigl(\id-\sigma_{i,m}^m\bigr)g^{(2)}_{i,m,q,j}\bigl(S\setminus\{s\}\bigr)\bigr]=\bigl(\id-\sigma_{m,s}^{s+h}\bigr)\Psi^{S\setminus\{s\}}=0.
\end{align*}
Therefore $\Psi^S=0$.

Now let $i<k<m-1$. By the inductive assumption, $\Phi^S_{k+1}=0$. 
By definition, $\sigma^{k+1}_{i,k}$ acts identically on any polynomial $x_i-y_t$, where $t=i+1,\ldots,k$,
and also on $g^{(2)}_{m,m,q,j}(S)$ by Lemma~\ref{lemma:compol:4.5}.
Therefore
\begin{align*}
&\bigl(\id-\sigma_{i,k}^{k+1}\bigr)\bigl[(x_i-y_{i+1})\cdots(x_i-y_{k+1})\cdot g^{(2)}_{m,m,q,j}(S)\bigr]
\\
=
&(x_i-y_{i+1})\cdots(x_i-y_{k+1})\,g^{(2)}_{m,m,q,j}(S)
\\
&-(x_i-y_{i+1})\cdots(x_i-y_k)\cdot(x_k-y_{k+1})\cdot g^{(2)}_{m,m,q,j}(S)
\\
=&(x_i-x_k)\cdot(x_i-y_{i+1})\cdots(x_i-y_k)\cdot g^{(2)}_{m,m,q,j}(S).
\end{align*}
Applying this formula and~\ref{g-2:b}, we get that $(x_i-x_k)\,\Phi_k^S$ equals
\begin{align*}
&(x_i-x_k)\cdot(x_i-y_{i+1})\cdots(x_i-y_k)\cdot g^{(2)}_{m,m,q,j}(S)
\\
&-\bigl(\id-\sigma_{i,k}^{k+1}\bigr)g^{(2)}_{i,m,q,j}\bigl([k{+}1..m{-}1)\cup\{m\}\cup S\bigr)
\\
=&\bigl(\id-\sigma_{i,k}^{k+1}\bigr)\bigl[(x_i-y_{i+1})\cdots(x_i-y_{k+1})\,g^{(2)}_{m,m,q,j}(S)
\\
&-g^{(2)}_{i,m,q,j}\bigl([k+1..m{-}1)\cup\{m\}\cup S\bigr)\bigr]
=\bigl(\id-\sigma_{i,k}^{k+1}\bigr)\Phi_{k+1}^S=0.
\end{align*}
So $\Phi_k^S=0$. This completes the inductive step.

\smallskip

\ref{lemma:compol:13:part:d} We have $m+1\in S$. Set $T:=S\setminus\{m+1\}$.
We apply downward induction on $k=m,\dots,i+1$ to prove the following equality:
$$
(x_i-y_{i+1})\cdots(x_i-y_k)\,g^{(2)}_{m,m+1,q,j}(\{m{+}1\}\cup T)=g^{(2)}_{i,m+1,q,j}([k..m)\cup\{m{+}1\}\cup T)
$$
for any $T\subset(m{+}1..j]$. Let $\Phi_k^T$ denote the difference of the left-hand side and the right-hand side of
this formula.Part (d) will follow from $\Phi_{i+1}^T=0$. By~\ref{g-2:b},
\begin{align*}
&(x_i-x_{m+1})\,(x_m-x_{m+1})\,\Phi_m^T
\\
=&(x_i-x_{m+1})\cdot(x_i-y_{i+1})\cdots(x_i-y_m)\cdot\bigl(\id-\sigma_{m,m+1}^{m+1}\bigr)g^{(2)}_{m,m+1,q,j}(T)
\\
&-(x_m-x_{m+1})\bigl(\id-\sigma_{i,m+1}^{m+1}\bigr)g^{(2)}_{i,m+1,q,j}(T)=:\Psi^T.
\end{align*}
We 
prove by induction on $|T|$ that $\Psi^T=0$. 
By~\ref{g-2:a}, $\Psi^\emptyset$ equals
\begin{align*}
&(x_i{-}x_{m+1})\cdot(x_i{-}y_{i+1})\cdots(x_i{-}y_m)\\
&\times
\bigl(\id-\sigma_{m,m+1}^{m+1}\bigr)\bigr[(x_m{-}y_{m+1})\cdot(x_{m+1}{-}y_{m+2})\cdots(x_{m+1}{-}y_j)\bigr]
\\
&-(x_m-x_{m+1})\bigl(\id-\sigma_{i,m+1}^{m+1}\bigr)\bigr[(x_i{-}y_{i+1})\cdots(x_i{-}y_{m+1})
\\
&\times(x_{m+1}{-}y_{m+2})\cdots(x_{m+1}-y_j)\bigr]
\\
=&(x_i-x_{m+1})\cdot(x_i-y_{i+1})\cdots(x_i-y_m)\cdot(x_m{-}y_{m+1})
\\
&\times(x_{m+1}{-}y_{m+2})\cdots(x_{m+1}-y_j)
\\
&-(x_i-x_{m+1})\cdot(x_i-y_{i+1})\cdots(x_i-y_m)\cdot(x_{m+1}{-}y_{m+1})
\\
&\times(x_{m+1}{-}y_{m+2})\cdots(x_{m+1}-y_j)
\\
&-(x_m-x_{m+1})\cdot(x_i-y_{i+1})\cdots(x_i-y_{m+1})\cdot(x_{m+1}{-}y_{m+2})\cdots(x_{m+1}-y_j)
\\
&+(x_m-x_{m+1})\cdot(x_i-y_{i+1})\cdots(x_i-y_m)\cdot(x_{m+1}-y_{m+1})
\\
&\times (x_{m+1}{-}y_{m+2})\cdots(x_{m+1}-y_j),
\end{align*}
which equals  $0$.

Now let $T\ne\emptyset$ and set $t:=\min T$. We have $t>m+1$ and therefore $l^{(2)}_{m,m+1,q,j}(t)=l^{(2)}_{i,m+1,q,j}(t)=:h$.
By~\ref{g-2:b}, Lemma~\ref{lemma:compol:2}, and the inductive hypothesis, $(x_{m+1}-x_t)\,\Psi^T$ equals
\begin{align*}
&(x_i-x_{m+1})\cdot(x_i-y_{i+1})\cdots(x_i-y_m)\cdot\bigl(\id-\sigma_{m,m+1}^{m+1}\bigr)\bigl[(x_{m+1}-x_t)\,g^{(2)}_{m,m+1,q,j}(T)\bigr]
\\
&-(x_m-x_{m+1})\bigl(\id-\sigma_{i,m+1}^{m+1}\bigr)\bigl[(x_{m+1}-x_t)\,g^{(2)}_{i,m+1,q,j}(T)\bigr]
\\
=&(x_i{-}x_{m+1})\cdot(x_i{-}y_{i+1})\cdots(x_i{-}y_m)(\id-\sigma_{m,m+1}^{m+1})(\id-\sigma_{m+1,t}^{t+h})g^{(2)}_{m,m+1,q,j}(T\setminus\{t\})
\\
&-(x_m-x_{m+1})\cdot\bigl(\id-\sigma_{i,m+1}^{m+1}\bigr)\bigl(\id-\sigma_{m+1,t}^{t+h}\bigr)\,g^{(2)}_{i,m+1,q,j}(T\setminus\{t\})
\\
=&\bigl(\id-\sigma_{m+1,t}^{t+h}\bigr)\Psi^{T\setminus\{t\}}=0.
\end{align*}
So $\Psi^T=0$, and we have proved $\Phi_m^T=0$ for any $T\subset(m{+}1..j]$.

Now let $i<k<m$. By the inductive hypothesis, $\Phi^T_{k+1}=0$ for any $T\subset(m{+}1..j]$. By definition,
$\sigma_{i,k}^{k+1}$ acts identically on any polynomial $x_i-y_r$, where $r=i{+}1,\ldots,k$, and also on
$g^{(2)}_{m,m+1,q,j}\bigl(\{m{+}1\}\cup T\bigr)$ by Lemma~\ref{lemma:compol:4.5}.
On the other hand, $\sigma_{i,k}^{k+1}(x_i-y_{k+1})=x_k-y_{k+1}$. So
\begin{align*}
&\bigl(\id-\sigma_{i,k}^{k+1}\bigr)\bigl[(x_i-y_{i+1})\cdots(x_i-y_{k+1})\cdot g^{(2)}_{m,m+1,q,j}\bigl(\{m{+}1\}\cup T\bigr)\bigr]
\\
=&(x_i-y_{i+1})\cdots(x_i-y_{k+1})\cdot g^{(2)}_{m,m+1,q,j}\bigl(\{m{+}1\}\cup T\bigr)
\\
&-(x_i-y_{i+1})\cdots(x_i-y_k)\cdot(x_k-y_{k+1})\,g^{(2)}_{m,m+1,q,j}\bigl(\{m{+}1\}\cup T\bigr)
\\
=&(x_i-x_k)\cdot(x_i-y_{i+1})\cdots(x_i-y_k)\cdot g^{(2)}_{m,m+1,q,j}\bigl(\{m{+}1\}\cup T\bigr).
\end{align*}
Hence by~\ref{g-2:b}, $(x_i-x_k)\,\Phi^T_k$ equals
\begin{align*}
&(x_i-x_k)\cdot(x_i-y_{i+1})\cdots(x_i-y_k)\cdot g^{(2)}_{m,m+1,q,j}\bigl(\{m+1\}\cup T\bigr)
\\
&-\bigl(\id-\sigma_{i,k}^{k+1}\bigr)\,g^{(2)}_{i,m+1,q,j}\bigl([k+1..m)\cup\{m+1\}\cup T\bigr)
\\
=&\bigl(\id-\sigma_{i,k}^{k+1}\bigr)\bigl[(x_i-y_{i+1})\cdots(x_i-y_{k+1})\cdot g^{(2)}_{m,m+1,q,j}\bigl(\{m+1\}\cup T\bigr)
\\
&-g^{(2)}_{i,m+1,q,j}\bigl([k+1..m)\cup\{m+1\}\cup T\bigr)\bigr]=\bigl(\id-\sigma_{i,k}^{k+1}\bigr)\Phi^T_{k+1}=0.
\end{align*}
So $\Phi^T_k=0$. This completes the inductive step.

\smallskip

\ref{lemma:compol:13:part:e} We apply downward induction on $r=m-1,\ldots,i+1$ to prove \begin{align*}
&(x_i-y_{i+1})\cdots(x_i-y_r)\cdot g^{(2)}_{m,k,q,j}(S)+g^{(2)}_{i,k,q,j}\bigl([r..m-1)\cup\{m\}\cup S\bigr)\\
&+(x_{m-1}-x_m)\,g^{(2)}_{i,k,q,j}\bigl([r..m]\cup S\bigr)\ =\ g^{(2)}_{i,k,q,j}\bigl([r..m)\cup S\bigr).
\end{align*}
Denote by $\Phi^S_r$ the difference of the left-hand side and the right-hand side of this formula.
Part (e) will follow from $\Phi^S_{i+1}=0$.
By~\ref{g-2:b}, $(x_i-x_{m-1})\,\Phi^S_{m-1}$ equals
\begin{align*}
&(x_i-x_{m-1})\cdot(x_i-y_{i+1})\cdots(x_i-y_{m-1})\cdot g^{(2)}_{m,k,q,j}(S)
\\
&+(x_i{-}x_{m-1})\,g^{(2)}_{i,k,q,j}\bigl(\{m\}\cup S\bigr)+(x_{m-1}{-}x_m)\,\bigl(\id{-}\sigma^m_{i,m-1}\bigr)\,g^{(2)}_{i,k,q,j}\bigl(\{m\}\cup S\bigr)
\\
&-\bigl(\id{-}\sigma^m_{i,m-1}\bigr)\,g^{(2)}_{i,k,q,j}(S)
\\
=&(x_i{-}x_{m-1})\cdot(x_i{-}y_{i+1})\cdots(x_i{-}y_{m-1})\cdot g^{(2)}_{m,k,q,j}(S)
\\
&+\bigl[(x_i{-}x_{m-1})+(x_{m-1}{-}x_m)\bigr]g^{(2)}_{i,k,q,j}\bigl(\{m\}\cup S\bigr)
\\
&-\sigma^m_{i,m-1}\bigl[(x_i-x_m)\,g^{(2)}_{i,k,q,j}\bigl(\{m\}\cup S\bigr)\bigr]-\bigl(\id-\sigma^m_{i,m-1}\bigr)\,g^{(2)}_{i,k,q,j}(S)
\\
=&(x_i-x_{m-1})\cdot(x_i-y_{i+1})\cdots(x_i-y_{m-1})\cdot g^{(2)}_{m,k,q,j}(S)+\bigl(\id-\sigma_{i,m}^{m+1}\bigr)\,g^{(2)}_{i,k,q,j}(S)
\\
&-\sigma^m_{i,m-1}\bigl(\id-\sigma_{i,m}^{m+1}\bigr)\,g^{(2)}_{i,k,q,j}(S)-\bigl(\id-\sigma^m_{i,m-1}\bigr)\,g^{(2)}_{i,k,q,j}(S)
\\
=&(x_i-x_{m-1})(x_i-y_{i+1})\cdots(x_i-y_{m-1}) g^{(2)}_{m,k,q,j}(S)-(\id-\sigma_{i,m-1}^m)\sigma_{i,m}^{m+1}g^{(2)}_{i,k,q,j}(S).
\end{align*}
We denote the last expression by $\Psi^S$ and prove by induction on $|S|$ that $\Psi^S=0$. 
By~\ref{g-2:a}, $\Psi^\emptyset$ equals
\begin{align*}
&(x_i-x_{m-1})\cdot(x_i-y_{i+1})\cdots(x_i-y_{m-1})\cdot(x_m-y_{m+1})\cdots(x_m-y_k)
\\
&\times (x_k-y_{k+1})\cdots(x_k-y_j)
\\
&-\bigl(\id-\sigma_{i,m-1}^m\bigr)\sigma_{i,m}^{m+1}\bigl[(x_i-y_{i+1})\cdots(x_i-y_k)\cdot(x_k-y_{k+1})\cdots(x_k-y_j)\bigr]
\\
=&(x_i{-}x_{m-1})\cdot(x_i{-}y_{i+1})\cdots(x_i{-}y_{m-1})
\\
&\times (x_m-y_{m+1})\cdots(x_m-y_k)\cdot(x_k-y_{k+1})\cdots(x_k-y_j)
\\
&-\bigl(\id-\sigma_{i,m-1}^m\bigr)\bigl[(x_i-y_{i+1})\cdots(x_i-y_m)
\\
&\times (x_m-y_{m+1})\cdots(x_m-y_k)\cdot(x_k-y_{k+1})\cdots(x_k-y_j)\bigr]
\\
=&(x_i{-}x_{m-1})\cdot(x_i{-}y_{i+1})\cdots(x_i{-}y_{m-1})
\\
&\times(x_m-y_{m+1})\cdots(x_m-y_k)\cdot(x_k-y_{k+1})\cdots(x_k-y_j)
\\
&-(x_i-y_{i+1})\cdots(x_i-y_m)\cdot(x_m-y_{m+1})\cdots(x_m-y_k)\cdot(x_k-y_{k+1})\cdots(x_k-y_j)
\\
&+(x_i-y_{i+1})\cdots(x_i-y_{m-1})\cdot(x_{m-1}-y_m)\cdot(x_m-y_{m+1})\cdots(x_m-y_k)
\\
&\times (x_k-y_{k+1})\cdots(x_k-y_j)=0.
\end{align*}

Now let $S\ne\emptyset$ and $s:=\min S$. Since $s>m$, we have $l^{(2)}_{m,k,q,j}(s)=l^{(2)}_{i,k,q,j}(s)=:h$. Moreover, set $s_0:=\max_s\{i,k\}$, $s'_0:=\max_s\{m,k\}$.
Take any polynomial $g\in\R$.

If $s>k$ then we have $s_0=s'_0=k$. Therefore, applying Lemma~\ref{lemma:compol:2}, we get
\begin{align*}
&\bigl(\id-\sigma_{i,m-1}^m\bigr)\sigma_{i,m}^{m+1}\bigl[(x_{s_0}-x_s)\,g\bigr]=\bigl(\id-\sigma_{i,m-1}^m\bigr)\sigma_{i,m}^{m+1}\bigl[(x_k-x_s)\,g\bigr]\\
=&(x_k-x_s)\bigl(\id-\sigma_{i,m-1}^m\bigr)\sigma_{i,m}^{m+1}\,g=(x_{s'_0}-x_s)\,\bigl(\id-\sigma_{i,m-1}^m\bigr)\sigma_{i,m}^{m+1}\,g,\\[6pt]
&\bigl(\id{-}\sigma_{i,m-1}^m\bigr)\sigma_{i,m}^{m+1}\sigma_{s_0,s}^{s+h}=\bigl(\id{-}\sigma_{i,m-1}^m\bigr)\sigma_{i,m}^{m+1}\sigma_{k,s}^{s+h}\\
=&\sigma_{k,s}^{s+h}\bigl(\id{-}\sigma_{i,m-1}^m\bigr)\sigma_{i,m}^{m+1}=\sigma_{s'_0,s}^{s+h}\bigl(\id-\sigma_{i,m-1}^m\bigr)\sigma_{i,m}^{m+1}.
\end{align*}
If $s\le k$ then $s_0=i$ and $s'_0=m$. So, applying Lemmas~\ref{lemma:compol:1} and~\ref{lemma:compol:2},
we get
\begin{align*}
&\bigl(\id-\sigma_{i,m-1}^m\bigr)\sigma_{i,m}^{m+1}\bigl[(x_{s_0}-x_s)\,g\bigr]=\bigl(\id-\sigma_{i,m-1}^m\bigr)\sigma_{i,m}^{m+1}\bigl[(x_i-x_s)\,g\bigr]
\\
=&\bigl(\id-\sigma_{i,m-1}^m\bigr)\bigl[(x_m-x_s)\,\sigma_{i,m}^{m+1}g\bigr]\\
=&(x_m-x_s)\,\bigl(\id-\sigma_{i,m-1}^m\bigr)\sigma_{i,m}^{m+1}g=(x_{s'_0}-x_s)\,\bigl(\id-\sigma_{i,m-1}^m\bigr)\sigma_{i,m}^{m+1}g,\\[6pt]
&\bigl(\id{-}\sigma_{i,m-1}^m\bigr)\sigma_{i,m}^{m+1}\sigma_{s_0,s}^{s+h}=\bigl(\id-\sigma_{i,m-1}^m\bigr)\sigma_{i,m}^{m+1}\sigma_{i,s}^{s+h}\\
=&\bigl(\id-\sigma_{i,m-1}^m\bigr)\sigma_{m,s}^{s+h}\sigma_{i,m}^{m+1}=\sigma_{m,s}^{s+h}\bigl(\id-\sigma_{i,m-1}^m\bigr)\sigma_{i,m}^{m+1}=\sigma_{s'_0,s}^{s+h}\bigl(\id-\sigma_{i,m-1}^m\bigr)\sigma_{i,m}^{m+1}.
\end{align*}
In either case, we have
\begin{align*}
&\bigl(\id-\sigma_{i,m-1}^m\bigr)\sigma_{i,m}^{m+1}\bigl[(x_{s_0}-x_s)\,g\bigr]=(x_{s'_0}-x_s)\,\bigl(\id-\sigma_{i,m-1}^m\bigr)\sigma_{i,m}^{m+1}\,g,\\
&\bigl(\id{-}\sigma_{i,m-1}^m\bigr)\sigma_{i,m}^{m+1}\sigma_{s_0,s}^{s+h}=\sigma_{s'_0,s}^{s+h}\bigl(\id-\sigma_{i,m-1}^m\bigr)\sigma_{i,m}^{m+1}.
\end{align*}
By these formulas,~\ref{g-2:b} and the inductive hypothesis, $(x_{s'_0}-x_s)\,\Psi^S$ equals
\begin{align*}
&(x_i{-}x_{m-1})\cdot(x_i{-}y_{i+1})\cdots(x_i{-}y_{m-1})\cdot\bigl(\id-\sigma^{s+h}_{s'_0,s}\bigr)g^{(2)}_{m,k,q,j}\bigl(S\setminus\{s\}\bigr)\\
&-\bigl(\id-\sigma_{i,m-1}^m\bigr)\sigma_{i,m}^{m+1}\,\bigl(\id-\sigma^{s+h}_{s_0,s}\bigr)g^{(2)}_{i,k,q,j}\bigl(S\setminus\{s\}\bigr)
=\bigl(\id-\sigma^{s+h}_{s'_0,s}\bigr)\Psi^{S\setminus\{s\}}=0.
\end{align*}
Hence $\Psi^S=0$. We have proved $\Phi_{m-1}^S=0$
for any $S\subset(m..j]$.

Now let $i<r<m-1$. By the inductive assumption, $\Phi^S_{r+1}=0$ for any $S\subset(m..j]$.
By definition, $\sigma^{r+1}_{i,r}$ acts identically on any polynomial $x_i-y_t$,
where $t=i+1,\ldots,r$, on $x_{m-1}-x_m$, and also on $g^{(2)}_{m,k,q,j}(S)$ by Lemma~\ref{lemma:compol:4.5}.
On the other hand, we have $\sigma^{r+1}_{i,r}(x_i-y_{r+1})=x_r-y_{r+1}$. So we get
\begin{align*}
&\bigl(\id-\sigma_{i,r}^{r+1}\bigr)\bigl[(x_i-y_{i+1})\cdots(x_i-y_{r+1})\cdot g^{(2)}_{m,k,q,j}(S)\bigr]
\\
=&(x_i-y_{i+1})\cdots(x_i-y_{r+1})\,g^{(2)}_{m,k,q,j}(S)
\\
&-(x_i-y_{i+1})\cdots(x_i-y_r)\cdot(x_r-y_{r+1})\,g^{(2)}_{m,k,q,j}(S)
\\
=&(x_i-x_r)\cdot(x_i-y_{i+1})\cdots(x_i-y_r)\cdot g^{(2)}_{m,k,q,j}(S).
\end{align*}
Applying these formulas and relation~\ref{g-2:b}, we get that $(x_i-x_r)\,\Phi_r^S$ equals
\begin{align*}
&(x_i-x_r)\cdot(x_i-y_{i+1})\cdots(x_i-y_r)\,g^{(2)}_{m,k,q,j}(S)\\
&+\bigl(\id-\sigma_{i,r}^{r+1}\bigr)g^{(2)}_{i,k,q,j}\bigl([r{+}1..m{-}1)\cup\{m\}\cup S\bigr)
\\
&+(x_{m-1}-x_m)\cdot\bigl(\id-\sigma_{i,r}^{r+1}\bigr)\,g^{(2)}_{i,k,q,j}\bigl([r{+}1..m]\cup S\bigr)
\\
&-\bigl(\id-\sigma_{i,r}^{r+1}\bigr)\,g^{(2)}_{i,k,q,j}\bigl([r+1..m)\cup S\bigr)
\\
=&\bigl(\id-\sigma_{i,r}^{r+1}\bigr)\bigl[(x_i-y_{i+1})\cdots(x_i-y_{r+1})\cdot g^{(2)}_{m,k,q,j}(S)
\\
&+g^{(2)}_{i,k,q,j}\bigl([r{+}1..m{-}1)\cup\{m\}\cup S\bigr)
+(x_{m-1}-x_m)\,g^{(2)}_{i,k,q,j}\bigl([r{+}1..m]\cup S\bigr)
\\
&-g^{(2)}_{i,k,q,j}\bigl([r{+}1..m)\cup S\bigr)\bigr]
=\bigl(\id-\sigma_{i,r}^{r+1}\bigr)\Phi_{r+1}^S=0.
\end{align*}
So $\Phi_r^S=0$. This completes the inductive step.
\end{proof}

\chapter{Raising coefficients}\label{Q(n):raising coefficients}

Let $1\le i<j\le n$, $\epsilon\in\{\0,\1\}$, $M$ be a signed $(i..j]$-set containing either $\bar\jmath$ or $j$
and let $\delta:[i..j)\to\{\0,\1\}$ be a function. By the triangular decomposition $U_\Z(n)=U_\Z^-(n)U_\Z^0(n)U_\Z^+(n)$,
there exists a unique $P_{i,j}^{\epsilon,\delta}(M)\in U_\Z^0(n)$ such that
\begin{equation}\label{eq:rcoeff:1}
E^{\delta_i}_i\cdots E^{\delta_{j-1}}_{j-1}S_{i,j}^\epsilon(M)\=P_{i,j}^{\epsilon,\delta}(M)\,\Bigl({\rm mod}{I_{[i..j)}^+}\Bigr)\,.
\end{equation}
We refer to $P_{i,j}^{\epsilon,\delta}(M)$ as a {\em raising coefficient}.
The aim of this section is to calculate $P_{i,j}^{\epsilon,\delta}(M)$ for all $M$ as above
with at most one odd element---this is all we need for
Chapters~\ref{ConstructingUmathbbF(n-1)-primitive vectors} and~\ref{main_results}.

It turns out (see Lemmas~\ref{lemma:rcoeff:1} and~\ref{lemma:rcoeff:2} below)
that the raising coefficients 
can be expressed via polynomials $g^{(1)}_{i,j}(S)$ and $g^{(2)}_{i,k,q,j}(S)$ introduced in Sections~\ref{polinomialsg1} and \ref{polinomialsg2}.
To do it, we need a ring homomorphism $\llbracket\,\rrbracket:\R\to U^0_\Z(n)$ such that
$$
\llbracket x_i\rrbracket=H_i(H_i-1),\quad \llbracket y_i\rrbracket=(H_i+1)H_i \qquad(1\leq i\leq n),
$$
(other variables will not appear, so for them $\llbracket\,\rrbracket$ can be defined in an arbitrary  way).
Thus, for example, $\llbracket x_i-x_j\rrbracket=C(i,j)$ and $\llbracket x_i-y_j\rrbracket=B(i,j)$.

We denote by $\chi_i^z$ the function $:\{i\}\to\{z\}$ on the one element set $\{i\}$.
If $f$ and $g$ are functions on disjoint sets $S$ and $T$ respectively, then
$f\cup g$ denotes the function on $S\cup T$ taking value $f(x)$ at $x\in S$ and value $g(y)$ at $y\in T$.

\section{Inductive formulas}

Applying Lemma~\ref{lemma:ops:2'}(\ref{lemma:ops:2':iii}),(\ref{lemma:ops:2':iv}), we get
\begin{align*}
E^{\delta_i}_i\cdots E^{\delta_{j-1}}_{j-1}S_{i,j}^\epsilon(\{\odd \jmath\})
\=(-1)^{\epsilon\delta_{j-1}}E^{\delta_i}_i\cdots E^{\delta_{j-2}}_{j-2}S_{i,j-1}^{\epsilon+\delta_{j-1}}(\{\odd{j{-}1}\})\\
\=E^{\delta_i}_i\cdots E^{\delta_{j-3}}_{j-3}S_{i,j-2}^{\epsilon+\delta_{j-1}+\delta_{j-2}}(\{\odd{j{-}2}\})
\=\cdots\=E^{\delta_i}_iS_{i,i+1}^{\epsilon+\delta_{j-1}+\cdots+\delta_{i+1}}(\{\odd{i{+}1}\})
\\
\=H^{\epsilon+\delta_{j-1}+\cdots+\delta_i}_i-(-1)^{\delta_i(\epsilon+\delta_{j-1}+\cdots+\delta_{i+1})}H^{\epsilon+\delta_{j-1}+\cdots+\delta_i}_{i+1}\bigl({\rm mod}{I_{[i..j)}^+}\bigr).
\end{align*}
Thus:
$$
P^{\,\epsilon,\delta}_{i,j}\(\{\odd \jmath\}\)=
H_i^{\epsilon+\sum\delta}-(-1)^{\delta_i(\epsilon+\sum\delta|_{(i..j)})}H_{i+1}^{\epsilon+\sum\delta}.
\leqno{\text{(\bf P-1)}}
$$

Applying Lemma~\ref{lemma:ops:2}(\ref{lemma:ops:2:iii}),(\ref{lemma:ops:2:iv}), we get
\begin{align*}
E^{\delta_i}_i\cdots E^{\delta_{j-1}}_{j-1}S_{i,j}^{\epsilon}(\{j\})
\=E^{\delta_i}_i\cdots E^{\delta_{j-2}}_{j-2}S_{i,j-1}^{\epsilon+\delta_{j-1}}(\{j{-}1\})
\\
\=E^{\delta_i}_i\cdots E^{\delta_{j-3}}_{j-3}S_{i,j-2}^{\epsilon+\delta_{j-1}+\delta_{j-2}}(\{j{-}2\})
\=\cdots\=E^{\delta_i}_iS_{i,i+1}^{\epsilon+\delta_{j-1}+\cdots+\delta_{i+1}}(\{i{+}1\})\\
\=\cond_{\delta_i=\epsilon+\delta_{j-1}+\cdots+\delta_{i+1}}B(i,i+1)\,\bigl({\rm mod}{I_{[i..j)}^+}\bigr).
\end{align*}
Substituting $\delta_i+\cdots+\delta_{j-1}$ for $\epsilon$ in the exponent of $-1$, we get
\begin{equation*}
P^{\,\epsilon,\delta}_{i,j}\(\{j\}\)=\cond_{\sum\delta=\epsilon}\,B(i,i+1).
\leqno{\text{(\bf P\,-2)}}
\end{equation*}

Let $\min M=\odd{i+1}<j$. Then by~(\ref{S3}), we have modulo
$I_{[i..j)}^+$:
\begin{align*}
P_{i,j}^{\epsilon,\delta}(M)\=
\sum_{\gamma+\sigma=\epsilon}
({-}1)^{\gamma(\1+\epsilon+\|M_{(i+1..j]}\|)+\gamma\sum\delta|_{(i..j)}}
E^{\delta_i}_i
S^{\,\gamma}_{i,i+1}\(\left\{\odd{i{+}1}\right\}\)
\\
\times
E^{\delta_{i+1}}_{i+1}\cdots E^{\delta_{j-1}}_{j-1}\,S_{i+1,j}^{\,\sigma}\(M_{(i+1..j]}\)
+\sum_{\gamma+\sigma=\epsilon}(-1)^{\sigma(\1+\|M\|)}P_{i,j}^{\,\gamma,\delta}\(M\setminus\left\{\odd{i{+}1}\right\}\)H^\sigma_i
\\
\=\sum_{\gamma+\sigma=\epsilon}
({-}1)^{\gamma(\1+\epsilon+\|M_{(i+1..j]}\|)+\gamma\sum\delta|_{(i..j)}}P_{i,i+1}^{\gamma,\delta|_{\{i\}}}\(\left\{\odd{i{+}1}\right\}\) P_{i+1,j}^{\sigma,\delta|_{[i+1..j)}}\!\(M_{(i+1..j]}\)
\\
+\sum_{\gamma+\sigma=\epsilon}(-1)^{\sigma(\1+\|M\|)}P_{i,j}^{\,\gamma,\delta}\(M\setminus\left\{\odd{i{+}1}\right\}\)H^\sigma_i.
\end{align*}
Thus, if $\min M=\odd{i+1}<j$, then we have
$$
\begin{array}{l}
\displaystyle
P_{i,j}^{\epsilon,\delta}(M)=\sum_{\gamma+\sigma=\epsilon}
({-}1)^{\gamma(\1+\epsilon+\|M_{(i+1..j]}\|)+\gamma\sum\delta|_{(i..j)}} P_{i,i+1}^{\gamma,\delta|_{\{i\}}}(\{\odd{i{+}1}\})
\\ \times P_{i+1,j}^{\sigma,\delta|_{(i..j)}}(M_{(i+1..j]})
\displaystyle+\sum_{\gamma+\sigma=\epsilon}(-1)^{\sigma(\1+\|M\|)}P_{i,j}^{\,\gamma,\delta}\(M\setminus\left\{\odd{i{+}1}\right\}\)H^\sigma_i.
\end{array}
\leqno{\text{(\bf P-3)}}
$$

Let $\min M=i+1<j$. By (\ref{S4}),({\bf P-2}), and Lemma~\ref{lemma:ops:1},
modulo $I_{[i..j)}^+$, we have
\begin{align*}
P_{i,j}^{\epsilon,\delta}(M)\=
\sum_{\gamma+\sigma=\epsilon}
(-1)^{\gamma(\1+\epsilon+\|M_{(i+1..j]}\|)+\gamma(\delta_{i+2}+\cdots+\delta_{j-1})}E_i^{\delta_i}E_{i+1}^{\delta_{i+1}}S^{\gamma}_{i,i+1}(\{i{+}1\})
\\
\times E_{i+2}^{\delta_{i+2}}\cdots E_{j-1}^{\delta_{j-1}}S_{i+1,j}^{\sigma}(M_{(i+1..j]})
+P_{i,j}^{\epsilon,\delta}(M\setminus\{i{+}1\})C(i,i{+}1)
\\
\=\sum_{\gamma+\sigma=\epsilon}
(-1)^{\gamma(\1+\epsilon+\|M_{(i+1..j]}\|)+\gamma\sum\delta|_{(i..j)}}E_i^{\delta_i}S^{\gamma}_{i,i+1}(\{i{+}1\})P_{i+1,j}^{\sigma,\delta|_{(i..j)}}(M_{(i+1..j]})
\\
+\hspace{-1 mm}\sum_{\gamma+\sigma=\epsilon}
\hspace{-2mm}(-1)^{\gamma(\1+\epsilon+\|M_{(i+1..j]}\|)+\gamma\sum\delta|_{(i+1..j)}}E_i^{\delta_i}
{\hspace{-1 mm}\sum_{\xi+\tau=\gamma+\delta_{i+1}}{\hspace {-5 mm}(-1)^{\1+(\delta_{i+1}+\xi)\gamma}S^{\xi}_{i,i+1}(\{\odd{i{+}1}\})}}
\\
\times E_{i+1}^{\,\tau}E_{i+2}^{\delta_{i+2}}\cdots E_{j-1}^{\delta_{j-1}}S_{i+1,j}^{\,\sigma}(M_{(i+1..j]})
+P_{i,j}^{\epsilon,\delta}(M\setminus\{i{+}1\})\,C(i,i{+}1)
\\
\=\sum_{\gamma+\sigma=\epsilon}
(-1)^{\gamma(\1+\epsilon+\|M_{(i+1..j]}\|)+\gamma\sum\delta|_{(i..j)}}
\cond_{\delta_i=\gamma}B(i,i+1)P_{i+1,j}^{\sigma,\delta|_{(i..j)}}(M_{(i+1..j]})
\\
-\sum_{\xi+\tau+\sigma=\epsilon+\delta_{i+1}}
(-1)^{(\xi+\tau+\delta_{i+1})\(\1+\epsilon+\|M_{(i+1..j]}\|+\sum\delta|_{(i..j)}+\xi\)}
\\
\times P_{i,i+1}^{\xi,\delta|_{\{i\}}}(\{\odd{i{+}1}\})P_{i+1,j}^{\sigma,\chi_{i+1}^\tau\cup\delta|_{(i+1..j)}}(M_{(i+1..j]})
+P_{i,j}^{\epsilon,\delta}(M\setminus\{i{+}1\})C(i,i{+}1).
\end{align*}
Thus, if $\min M=i+1<j$, then we have
$$
{\arraycolsep=0pt
\begin{array}{l}
\displaystyle
P_{i,j}^{\epsilon,\delta}(M)\!=\!
(-1)^{\delta_i(\1+\epsilon+\|M_{(i+1..j]}\|)+\delta_i\sum\delta|_{(i..j)}}
\!B(i,\!i\!+\!1)P_{i+1,j}^{\epsilon+\delta_i,\delta|_{(i..j)}}(M_{(i+1..j]})
\\
\displaystyle
-\sum_{\xi+\tau+\sigma=\epsilon+\delta_{i+1}}
\hspace{-1mm}(-1)^{(\xi+\tau+\delta_{i+1})\(\1+\epsilon+\|M_{(i+1..j]}\|+\sum\delta|_{(i..j)}+\xi\)}
P_{i,i+1}^{\xi,\delta|_{\{i\}}}(\{\odd{i{+}1}\})
\\
\times
 P_{i+1,j}^{\sigma,\chi_{i+1}^\tau\cup\delta|_{(i+1..j)}}(M_{(i+1..j]})
 +P_{i,j}^{\epsilon,\delta}(M\setminus\{i{+}1\})C(i,i{+}1).
\end{array}}
\leqno{\text{(\bf P-4)}}
$$

\smallskip

Let $i+1<\min M=\bar m<j$. Then by~(\ref{S5}), we have modulo $I_{[i..j)}^+$:
\begin{align*}
P_{i,j}^{\epsilon,\delta}(M)\=\sum_{\gamma+\sigma=\epsilon}(-1)^{\gamma(\1+\epsilon+\|M_{(m..j]}\|)+\gamma\sum\delta|_{[m..j)}}E_i^{\delta_i}\cdots E_{m-1}^{\delta_{m-1}}S_{i,m}^{\,\gamma}\(\left\{\bar m\right\}\)
\\
\times E_m^{\delta_m}\cdots E_{j-1}^{\delta_{j-1}}S_{m,j}^{\,\sigma}\(M_{(m..j]}\)
+P_{i,j}^{\,\epsilon,\delta}\(M_{\odd{\vphantom{1}m}\mapsto\odd{m{-}1}}\).
\end{align*}
Thus, if $i+1<\min M=\bar m<j$, then  we have
$$
\begin{array}{r}
\displaystyle
P_{i,j}^{\epsilon,\delta}(M)=
\sum_{\gamma+\sigma=\epsilon}(-1)^{\gamma(\1+\epsilon+\|M_{(m..j]}\|\sum\delta|_{[m..j)})}P_{i,m}^{\gamma,\delta|_{[i..m)}}(\{\bar m\})
\\
\times P_{m,j}^{\sigma,\delta_{[m..j)}}(M_{(m..j]})
+P_{i,j}^{\epsilon,\delta}(M_{\odd{\vphantom{1}m}\mapsto\odd{m{-}1}}).
\end{array}
\leqno{\text{(\bf P\,-5)}}
$$

\smallskip

Let $i+1<\min M=m<j$. By (\ref{S6}), we have modulo $I_{[i..j)}^+$:
\begin{align*}
P_{i,j}^{\epsilon,\delta}(M)\=\sum_{\gamma+\sigma=\epsilon}
(-1)^{\gamma(\1+\epsilon+\|M_{(m..j]}\|)+\gamma(\delta_{m+1}+\cdots+\delta_{j-1})}\,E_i^{\delta_i}\cdots E_m^{\delta_m}S^{\,\gamma}_{i,m}(\{m\})
\\
\times E_{m+1}^{\delta_{m+1}}\cdots E_{j-1}^{\delta_{j-1}}S_{m,j}^{\sigma}(M_{(m..j]})
+P_{i,j}^{\epsilon,\delta}(M_{m\mapsto m{-}1})+P_{i,j}^{\epsilon,\delta}(M\setminus\{m\})C(m{-}1,m)
\end{align*}
Moreover, by Lemma~\ref{lemma:ops:1} and ({\bf P-2}), the first sum is equal to
\begin{align*}
\sum_{\gamma+\sigma=\epsilon}
(-1)^{\gamma(\1+\epsilon+\|M_{(m..j]}\|)+\gamma\sum\delta|_{[m..j)}}E_i^{\delta_i}\cdots E_{m-1}^{\delta_{m-1}}S^{\,\gamma}_{i,m}(\{m\})
\\
\times
E_m^{\delta_m}\cdots E_{j-1}^{\delta_{j-1}}S_{m,j}^{\,\sigma}(M_{(m..j]})
+\sum_{\gamma+\sigma=\epsilon}
(-1)^{\gamma(\1+\epsilon+\|M_{(m..j]}\|)+\gamma\sum\delta|_{(m..j)}}\,E_i^{\delta_i}\cdots E_{m-1}^{\delta_{m-1}}
\\
\times\!\!\sum_{\xi+\tau=\gamma+\delta_m}\!(-1)^{1+(\delta_m+\xi)\gamma}S_{i,m}^{\,\xi}(\{\bar m\})
E_m^\tau E_{m+1}^{\delta_{m+1}}\cdots E_{j-1}^{\delta_{j-1}}S_{m,j}^{\,\sigma}\(M_{(m..j]}\)
\\
=\sum_{\gamma+\sigma=\epsilon}
(-1)^{\gamma(\1+\epsilon+\|M_{(m..j]}\|+\sum\delta|_{[m..j)})}
\cond_{\sum\delta|_{[i..m)}=\gamma}B(i,i+1)
P_{m,j}^{\sigma,\delta|_{[m..j)}}(M_{(m..j]})
\\
-
\hspace{-2mm}
\sum_{\gamma+\sigma=\epsilon}
\hspace{-2mm}
(-1)^{\gamma(\1+\epsilon+\|M_{(m..j]}\|+\sum\delta|_{[m..j)})}
\hspace{-5mm}
\sum_{\xi+\tau=\gamma+\delta_m}
\hspace{-5mm}
(-1)^{\xi\gamma}
P_{i,m}^{\xi,\delta|_{[i..m)}}(\{\bar m\})
P_{m,j}^{\sigma,\chi_m^\tau\cup\delta|_{(m..j)}}(M_{(m..j]})
\\
=(-1)^{\sum\delta|_{[i..m)}(\1+\epsilon+\|M_{(m..j]}\|+\sum\delta|_{[m..j)})}
B(i,i+1)P_{m,j}^{\epsilon+\sum\delta|_{[i..m)},\delta|_{[m..j)}}(M_{(m..j]})
\\
-
\hspace{-3.4 mm}
\sum_{\xi+\tau+\sigma=\epsilon+\delta_m}
\hspace{-6.8 mm}
(-1)^{(\xi+\tau+\delta_m)(\1+\epsilon+\|M_{(m..j]}\|+\sum\delta|_{[m..j)}+\xi)}
\!P_{i,m}^{\xi,\delta|_{[i..m)}}\!(\!\{\bar m\}\!)
P_{m,j}^{\sigma,\chi_m^\tau\cup\delta|_{(m..j)}}\!(\!M_{(m..j]}).
\end{align*}

Thus, if $i+1<\min M=m<j$, then we have
$$
\begin{array}{r}
\displaystyle
P_{i,j}^{\epsilon,\delta}(M)
=(-1)^{\sum\delta|_{[i..m)}(\1+\epsilon+\|M_{(m..j]}\|+\sum\delta|_{[m..j)})}
B(i,i+1)
\\
\times P_{m,j}^{\epsilon+\sum\delta|_{[i..m)},\delta|_{[m..j)}}(M_{(m..j]})
\\
\displaystyle-\sum_{\xi+\tau+\sigma=\epsilon+\delta_m}
(-1)^{(\xi+\tau+\delta_m)(\1+\epsilon+\|M_{(m..j]}\|+\sum\delta|_{[m..j)}+\xi)}\\
\times
P_{i,m}^{\xi,\delta|_{[i..m)}}(\{\bar m\})
P_{m,j}^{\sigma,\chi_m^\tau\cup\,\delta|_{(m..j)}}(M_{(m..j]})
\\
+P_{i,j}^{\,\epsilon,\delta}(M_{m\mapsto m{-}1})+P_{i,j}^{\,\epsilon,\delta}(M\setminus\{m\})\,C(m{-}1,m).
\end{array}
\leqno{\text{(\bf P-6)}}
$$

\section{The case of signed sets with only even elements}

\begin{lemma}\label{lemma:rcoeff:1}
If $M$ consists of even elements then we have
$$
P_{i,j}^{\epsilon,\delta}(M)=\cond_{\epsilon=\sum\delta}\ll g^{(1)}_{i,j}\bigl((i..j)\setminus M\bigr)\rr.
$$
\end{lemma}
\begin{proof} Induction on $\height M$. In the base case $M=\{j\}$, by Lemma~\ref{lemma:compol:6}, we have
$$
\left\llbracket g^{(1)}_{i,j}\bigl((i..j)\setminus M\bigr)\right\rrbracket=\left\llbracket g^{(1)}_{i,j}\bigl((i..j)\bigr)\right\rrbracket=\ll x_i-y_{i+1}\rr=B(i,i+1).
$$
To obtain the required result, it remains to apply~({\bf P-2}).

Now suppose that $M\ne\{j\}$. Set $m:=\min M$. Then $i<m<j$.


{\it Case 1: $m=i+1$}. By~{(\bf P\,-4)} and the inductive hypothesis, $P_{i,j}^{\epsilon,\delta}(M)$ equals
\begin{align*}
(-1)^{\delta_i(\1+\epsilon+\sum\delta|_{(i..j)})}
B(i,i+1)\,
\cond_{\epsilon+\delta_i=\sum\delta|_{(i..j)}}\ll g^{(1)}_{i+1,j}\bigl((i{+}1..j)\setminus M_{(i+1..j]}\bigr)\rr
\\
-
\hspace{-3.3mm}\sum_{\xi+\tau+\sigma=\epsilon+\delta_{i+1}}
\hspace{-7.7mm}
(-1)^{(\xi+\tau+\delta_{i+1})(\1+\epsilon+\|M_{(i+1..j]}\|+\sum\delta|_{(i..j)}+\xi)}\! P_{i,i+1}^{\xi,\delta|_{\{i\}}}\!(\!\{\odd{i{+}1}\}\!)
\cond_{\sigma=\sum\chi_{i+1}^\tau\cup\,\delta|_{(i+1..j)}}
\\
\times
\ll g^{(1)}_{i+1,j}\bigl((i{+}1..j)\setminus M_{(i+1..j]}\bigr)\rr
+\cond_{\epsilon=\sum\delta}\ll g^{(1)}_{i,j}\bigl((i..j)\setminus(M\setminus\{i{+}1\})\bigr)\rr C(i,i{+}1).
\end{align*}
In the middle term, we can assume that $\xi=\epsilon+\sum\delta|_{(i..j)}$ and sum over $\tau$ and $\sigma$ such that
$\tau+\sigma=\sum\delta|_{(i+1..j)}$.
We get
\begin{align*}
\cond_{\epsilon=\sum\delta}\,B(i,i+1)\,\ll g^{(1)}_{i+1,j}\bigl((i{+}1..j)\setminus M_{(i+1..j]}\bigr)\rr
\\
-\sum_{\tau+\sigma=\sum\delta|_{(i+1..j)}}
\hspace{-8mm}
(-1)^{\epsilon+\sum\delta|_{(i+1..j)}+\tau} P_{i,i+1}^{\epsilon+\sum\delta|_{(i..j)},\delta_{i}}\!\(\left\{\odd{i{+}1}\right\}\)
\ll g^{(1)}_{i+1,j}\bigl((i{+}1..j)\setminus M_{(i+1..j]}\bigr)\rr
\\
+\cond_{\epsilon=\sum\delta}\ll g^{(1)}_{i,j}\bigl((i..j)\setminus(M\setminus\{i{+}1\})\bigr)\rr C(i,i{+}1).
\end{align*}
The middle term is zero, and by Lemma~\ref{lemma:compol:7} applied to $S=(i{+}1..j)\setminus M$,
we get
$$
\cond_{\epsilon=\sum\delta}\ll g_{i,j}^{(1)}((i{+}1..j)\setminus M)\rr
=\cond_{\epsilon=\sum\delta}\ll g_{i,j}^{(1)}\((i..j)\setminus M\)\rr.
$$


{\it Case 2: $m>i+1$}.  By~({\bf P-6}) and the inductive hypothesis, $P_{i,j}^{\epsilon,\delta}(M)$ equals
\begin{align*}
(-1)^{\sum\delta|_{[i..m)}\(\1+\epsilon\)+\sum\delta|_{[i..m)}\sum\delta|_{[m..j)}}
B(i,i+1)
\\
\times
\cond_{\epsilon+\sum\delta|_{[i..m)}=\sum\delta|_{[m..j)}}\ll g^{(1)}_{m,j}\bigl((m..j)\setminus M_{(m..j]}\bigr)\rr
\\
-
\sum_{\xi+\tau+\sigma=\epsilon+\delta_m}
(-1)^{(\xi+\tau+\delta_m)(\1+\epsilon+\|M_{(m..j]}\|+\sum\delta|_{[m..j)}+\xi)}
P_{i,m}^{\xi,\delta|_{[i..m)}}(\{\bar m\})
\\
\times
\cond_{\sigma=\sum\chi_m^\tau\cup\,\delta|_{(m..j)}}
\!\!\ll g^{(1)}_{m,j}\bigl((m..j){\setminus}M_{(m..j]}\bigr)\rr
+\cond_{\epsilon=\sum\delta}\ll g^{(1)}_{i,j}\bigl((i..j)\setminus M_{m\mapsto m{-}1}\bigr)\rr
\\
+\cond_{\epsilon=\sum\delta}\ll g^{(1)}_{i,j}\bigl((i..j)\setminus(M\setminus\{m\})\bigr)\rr C(m{-}1,m).
\end{align*}
In the second summand, we can assume that $\xi=\epsilon+\sum\delta|_{[m..j)}$ and sum over $\tau$ and $\sigma$ such that
$\tau+\sigma=\sum\delta|_{(m..j)}$. We get
\begin{align*}
P_{i,j}^{\epsilon,\delta}(M)
=\cond_{\epsilon=\sum\delta}
B(i,i+1)
\ll g^{(1)}_{m,j}\bigl((m..j)\setminus M_{(m..j]}\bigr)\rr
\\
-\sum_{\tau+\sigma=\sum\delta|_{(m..j)}}(-1)^{\epsilon+\sum\delta|_{(m..j)}+\tau}
 P_{i,m}^{\epsilon+\sum\delta|_{[m..j)},\delta|_{[i..m)}}(\{\bar m\})\ll g^{(1)}_{m,j}\bigl((m..j){\setminus}M_{(m..j]}\bigr)\rr
\\
+\cond_{\epsilon=\sum\delta} \Big(\ll g^{(1)}_{i,j}\bigl((i..j)\setminus M_{m\mapsto m{-}1}\bigr)\rr
+\ll g^{(1)}_{i,j}\bigl((i..j)\setminus(M\setminus\{m\})\bigr)\rr C(m{-}1,m)\Big).
\end{align*}
The second term is again zero as in the previous case. Thus by Lemma~\ref{lemma:compol:8} applied to $S=(m..j)\setminus M$,
\begin{align*}
P_{i,j}^{\epsilon,\delta}(M)
=\cond_{\epsilon=\sum\delta}\Bigl\llbracket(x_i-y_{i+1})\,g^{(1)}_{m,j}\bigl((m..j)\setminus M\bigr)
\\
+g^{(1)}_{i,j}\bigl((i..m-1)\!\cup\!\{m\}\!\cup\! ((m..j)\setminus M)\bigr)
+(x_{m-1}-x_m)\,g^{(1)}_{i,j}\bigl((i..m]\cup((m..j)\setminus M)\bigr)\Bigr\rrbracket
\\
=\cond_{\epsilon=\sum\delta}\ll g_{i,j}^{(1)}((i..m)\cup((m..j)\setminus M))\rr
=\cond_{\epsilon=\sum\delta}\ll g_{i,j}^{(1)}((i..j)\setminus M))\rr,
\end{align*}
as required.
\end{proof}

\section{The case of signed sets with one odd element}

\begin{lemma}\label{lemma:rcoeff:2}
If $\bar q\in M$ is the only odd element of $M$, then setting $X(i,q,M):=\{k\in[i..q]\setminus M\mid k{-}1\in M\cup\{i{-}1,i\}\}$, we have
\begin{align*}
P_{i,j}^{\epsilon,\delta}(M)=
\sum_{k\in X(i,q,M)}(-1)^{\cscript_{k>i}+\(\1+\epsilon+\sum\delta\)\sum\delta|_{[i..k)}}\ll g_{i,k,q,j}^{(2)}\bigl((i..j]\setminus M\bigr)\rr H_k^{\epsilon+\sum\delta}.
\end{align*}
\end{lemma}

Before we prove this lemma, we establish the following auxiliary fact.

\begin{proposition}\label{proposition:rcoeff:1}
Let $N$ be a signed $(m..j]$-set, containing either $j$ or $\bar\jmath$, and all of whose elements are even except $\bar q$.
Let $\epsilon,\xi\in\{\0,\1\}$ and $\delta:[m..j)\to\{\0,\1\}$ be a function. Suppose that Lemma~\ref{lemma:rcoeff:2} holds for $P_{m,j}^{\epsilon',\delta'}(N)$ for all $\epsilon'$ and $\delta'$. Then:
\begin{align*}
\sum_{\tau+\sigma=\epsilon+\delta_m+\xi}(-1)^{(\xi+\tau+\delta_m)(\epsilon+\sum\delta|_{[m..j)}+\xi)}P_{m,j}^{\sigma,\chi_m^\tau\cup\,\delta|_{(m..j)}}(N)\\
=2\cond_{\xi=\epsilon+\sum\delta|_{[m..j)}}\ll g_{m,m,q,j}^{(2)}\bigl((m..j]\setminus N\bigr)\rr H_m.
\end{align*}
\end{proposition}
\begin{proof}
We substitute for $P_{m,j}^{\sigma,\chi_m^\tau\cup\,\delta|_{(m..j)}}(N)$ in the left hand side its expression given by  Lemma~\ref{lemma:rcoeff:2}.
This expression is a sum over $k\in X(m,q,N)$. If $k>m$ we get the following contribution:
\begin{align*}
\sum_{\tau+\sigma=\epsilon+\delta_m+\xi}
\hspace{-2mm}
(-1)^{(\xi+\tau+\delta_m)(\epsilon+\sum\delta|_{[m..j)}+\xi)+\1+\(\1+\sigma+\sum\chi_m^\tau\cup\,\delta|_{(m..j)}\)\(\sum(\chi_m^\tau\cup\,\delta|_{(m..j)})|_{[m..k)}\)}
\\
\times\ll g_{m,k,q,j}^{(2)}\bigl((m..j]\setminus N\bigr)\rr H_k^{\sigma+\sum\chi_m^\tau\cup\,\delta|_{(m..j)}}
\\
=\sum_{\tau+\sigma=\epsilon+\delta_m+\xi}
\hspace{-2mm}
(-1)^{(\xi+\tau+\delta_m)(\epsilon+\sum\delta|_{[m..j)}+\xi)+\1+\(\1+\sigma+\tau+\sum\,\delta|_{(m..j)}\)\(\tau+\sum\delta|_{(m..k)}\)}
\\
\times\ll g_{m,k,q,j}^{(2)}\bigl((m..j]\setminus N\bigr)\rr H_k^{\sigma+\tau+\sum\delta|_{(m..j)}}
\\
=\sum_{\tau+\sigma=\epsilon+\delta_m+\xi}
\hspace{-2mm}
(-1)^{(\xi+\tau+\delta_m)(\epsilon+\sum\delta|_{[m..j)}+\xi)+\1+\(\1+\epsilon+\xi+\sum\,\delta|_{[m..j)}\)\(\tau+\sum\delta|_{(m..k)}\)}
\\
\times\ll g_{m,k,q,j}^{(2)}\bigl((m..j]\setminus N\bigr)\rr H_k^{\epsilon+\xi+\sum\delta|_{[m..j)}}.
\end{align*}
It is elementary to check that this sum equals zero.

If $k=m$, we get the following contribution:
\begin{align*}
\sum_{\tau+\sigma=\epsilon+\delta_m+\xi}(-1)^{(\xi+\tau+\delta_m)(\epsilon+\sum\delta|_{[m..j)}+\xi)}\ll g_{m,m,q,j}^{(2)}\bigl((m..j]\setminus N\bigr)\rr H_m^{\epsilon+\xi+\sum\delta|_{[m..j)}}\\
=2\cond_{\sum\delta|_{[m..j)}+\epsilon+\xi=\0}\ll g_{m,m,q,j}^{(2)}\bigl((m..j]\setminus N\bigr)\rr H_m.
\end{align*}


\end{proof}

\begin{proof}[Proof of Lemma~\ref{lemma:rcoeff:2}] We apply induction on $\height M$.

\smallskip

{\it Case 1: $M=\{\bar\jmath\}$.} Then $q=j$, and the required result comes from Lemma~\ref{lemma:compol:9} and~({\bf P-1}).

\smallskip

{\it Case 2: $q=i+1<j$.} In this case $j\in M$. By~({\bf P-3}),~({\bf P-1}) and Lemmas~\ref{lemma:rcoeff:1},~\ref{lemma:compol:10} for $S=(i{+}1..j)\setminus M$,
we get that $P_{i,j}^{\epsilon,\delta}(M)$ equals
\begin{align*}
\sum_{\gamma+\sigma=\epsilon}
\hspace{-2mm}
({-}1)^{\gamma(\1+\epsilon+\sum\delta|_{(i..j)})} P_{i,i+1}^{\gamma,\delta|_{\{i\}}}(\{\odd{i{+}1}\})
\cond_{\sigma=\sum\delta|_{[i+1..j)}}\!\ll g^{(1)}_{i+1,j}\bigl((i{+}1..j)\!\setminus\! M_{(i+1..j]}\bigr)\rr
\end{align*}\begin{align*}
+\sum_{\gamma+\sigma=\epsilon}
\cond_{\gamma=\sum\delta}\ll g^{(1)}_{i,j}\bigl((i..j)\setminus(M\setminus\{\odd{i{+}1}\})\bigr)\rr H^\sigma_i
\\
=({-}1)^{(\epsilon+\sum\delta|_{(i..j)})(\1+\epsilon+\sum\delta|_{(i..j)})} P_{i,i+1}^{\epsilon+\sum\delta|_{(i..j)},\delta|_{\{i\}}}\(\left\{\odd{i{+}1}\right\}\)\ll g^{(1)}_{i+1,j}\bigl((i{+}1..j)\setminus M\bigr)\rr
\\
+\ll g^{(1)}_{i,j}\Bigl((i..j)\setminus M\Bigr)\rr H^{\epsilon+\sum\delta}_i
\\
=\Bigl(H_i^{\epsilon+\sum\delta}-(-1)^{\delta_i\(\epsilon+\sum\delta|_{(i..j)}\)}H_{i+1}^{\epsilon+\sum\delta}\Bigr)\ll g^{(1)}_{i+1,j}\bigl((i{+}1..j)\setminus M\bigr)\rr
\\
+\ll g^{(1)}_{i,j}\Bigl(\{i+1\}\cup\bigl((i+1..j)\setminus M\bigr)\Bigr)\rr H^{\epsilon+\sum\delta}_i
\\
=\Bigl\llbracket g^{(1)}_{i+1,j}\bigl((i{+}1..j)\setminus M\bigr)
+g^{(1)}_{i,j}\Bigl(\{i{+}1\}\cup\bigl((i{+}1..j)\setminus M\bigr)\Bigr)\Bigl\rrbracket H^{\epsilon+\sum\delta}_i
\\
-(-1)^{\delta_i\(\epsilon+\sum\delta|_{(i..j)}\)}\ll g^{(1)}_{i+1,j}\bigl((i{+}1..j)\setminus M\bigr)\rr H_{i+1}^{\epsilon+\sum\delta}
\\
=\ll g_{i,i,i+1,j}^{(2)}\Bigl(\{i{+}1\}\cup\bigl((i{+}1..j\bigr)\setminus M\Bigr)\rr H^{\epsilon+\sum\delta}_i
\\
+(-1)^{\1+\delta_i\(\1+\epsilon+\sum\delta\)}\ll g^{(2)}_{i,i+1,i+1,j}\Bigl(\{i{+}1\}\cup\bigl((i{+}1..j)\setminus M\bigr)\Bigr)\rr H_{i+1}^{\epsilon+\sum\delta}
\\
=\!\!\ll g_{i,i,i+1,j}^{(2)}((i..j]\setminus M)\rr\!\! H^{\epsilon+\sum\delta}_i
\!+\!(-1)^{\1+\delta_i(\1+\epsilon+\sum\delta)}\!\!\ll g^{(2)}_{i,i+1,i+1,j}((i..j]\setminus M)\rr\! H_{i+1}^{\epsilon+\sum\delta}.
\end{align*}

\smallskip

{\it Case~3: $i+1=\min M<j$.} In this case $q>i+1$. By~({\bf P-4}), we have
\begin{align*}
P_{i,j}^{\epsilon,\delta}(M)=(-1)^{\delta_i\epsilon+\delta_i\sum\delta|_{(i..j)}}
B(i,i+1)P_{i+1,j}^{\epsilon+\delta_i,\delta|_{(i..j)}}(M_{(i+1..j]})
\\
-
\hspace{-2mm}
\sum_{\xi\in\{\0,\1\}}
\hspace{-3mm}
P_{i,i+1}^{\xi,\delta|_{\{i\}}}(\{\odd{i{+}1}\})
\hspace{-2mm}
\sum_{\tau+\sigma=\epsilon+\delta_{i+1}+\xi}
\hspace{-6mm}
(-1)^{(\xi+\tau+\delta_{i+1})(\epsilon+\sum\delta|_{(i..j)}+\xi)}
P_{i+1,j}^{\sigma,\chi_{i+1}^\tau\cup\delta|_{(i+1..j)}}\!(M_{(i+1..j]})
\\
+P_{i,j}^{\epsilon,\delta}(M_{(i+1..j]}) C(i,i{+}1).
\end{align*}
By Proposition~\ref{proposition:rcoeff:1} with $m=i+1$ and ({\bf P-1}), the second summand equals
\begin{align*}
-2P_{i,i+1}^{\epsilon+\sum\delta|_{[i+1..j)},\delta_i}\(\left\{\odd{i{+}1}\right\}\)
\ll g_{i+1,i+1,q,j}^{(2)}\bigl((i+1..j]\setminus M_{(i+1..j]}\bigr)\rr H_{i+1}
\\
=\!-2\big(H_i^{\epsilon+\sum\delta}\!\!-\!(-1)^{(\epsilon+\sum\delta|_{[i+1..j)})\delta_i}H_{i+1}^{\epsilon+\sum\delta}\big)
\!\ll g_{i+1,i+1,q,j}^{(2)}((i+1..j]\!\setminus\! M_{(i+1..j]})\rr\! H_{i+1}.
\end{align*}

Now, using the inductive hypothesis, we gather the ``coefficients'' of
$H_i^{\epsilon+\sum\delta}$,
$H_{i+1}^{\epsilon+\sum\delta}$, $H_{i+2}^{\epsilon+\sum\delta}$ and $H_k^{\epsilon+\sum\delta}$ for $k>i+2$.

The coefficient of $H_i^{\epsilon+\sum\delta}$ is
\begin{align*}
-2H_{i+1}\ll g_{i+1,i+1,q,j}^{(2)}\bigl((i{+}1..j]{\setminus}M_{(i+1..j]}\bigr)\rr
+\ll g_{i,i,q,j}^{(2)}\bigl((i..j]{\setminus}M_{(i+1..j]}\bigr)\rr C(i,i{+}1)
\\
=\Bigl\llbracket(x_{i+1}-y_{i+1})g_{i+1,i+1,q,j}^{(2)}((i{+}1..j]{\setminus}M)\\
+(x_i-x_{i+1})g_{i,i,q,j}^{(2)}(\{i{+}1\}\cup((i{+}1..j]{\setminus}M))\Bigr\rrbracket
=\Bigl\llbracket g_{i,i,q,j}^{(2)}\bigl((i{+}1..j]{\setminus}M\bigr)\Bigr\rrbracket,
\end{align*}
where the last equality comes from Lemma~\ref{lemma:compol:11}\ref{lemma:compol:11:part:a} with $S=(i{+}1..j]\setminus M$.

Using Lemma~\ref{lemma:compol:11}\ref{lemma:compol:11:part:b} in the last equality below, we see that the coefficient of $H_{i+1}^{\epsilon+\sum\delta}$ is:
\begin{align*}
(-1)^{\delta_i\epsilon+\delta_i\sum\delta|_{(i..j)}} \Big( B(i,i+1)
\ll g_{i+1,i+1,q,j}^{(2)}\bigl((i{+}1..j]\setminus M_{(i+1..j]}\bigr)\rr
\\
+2
H_{i+1}\!\ll g_{i+1,i+1,q,j}^{(2)}\bigl((i{+}1..j]\setminus M_{(i+1..j]}\bigr)\rr
\!-C(i,i{+}1)
\ll g_{i,i+1,q,j}^{(2)}\bigl((i..j]\setminus\! M_{(i+1..j]}\bigr)\rr\!\Big)
\end{align*}
\begin{align*}
=(-1)^{\delta_i(\epsilon+\sum\delta|_{(i..j)})}C(i,i{+}1)
\\
\times \ll g_{i+1,i+1,q,j}^{(2)}((i{+}1..j]{\setminus}M)
-g_{i,i+1,q,j}^{(2)}(\{i{+}1\}\!\cup((i{+}1..j]{\setminus}M))\rr=0.
\end{align*}

Using Lemma~\ref{lemma:compol:11}\ref{lemma:compol:11:part:c}, we see that the coefficient of $H_{i+2}^{\epsilon+\sum\delta}$ is:
\begin{align*}
\cond_{i+2\notin M}(-1)^{\delta_i\epsilon+\delta_i\sum\delta|_{(i..j)}+\1+(\1+\epsilon+\sum\delta)\delta_{i+1}}\ll (x_i-y_{i+1})
 g_{i+1,i+2,q,j}^{(2)}\bigl((i{+}1..j]\setminus M\bigr)\rr
\\
=\cond_{i+2\notin M}(-1)^{\1+\(\1+\epsilon+\sum\delta\)(\delta_i+\delta_{i+1})}\!
\ll g_{i,i+2,q,j}^{(2)}\bigl((i{+}1..j]{\setminus}M\bigr)\rr\!.
\end{align*}

Finally, using Lemma~\ref{lemma:compol:11}\ref{lemma:compol:11:part:d}, we see that the coefficient of
$H_k^{\epsilon+\sum\delta}$ (where $k>i+2$) is:
\begin{align*}
\sum_{\tiny\begin{array}{c}k\in(i{+}2..q]\!\setminus\! M\\k{-}1\in M\end{array}}
\hspace{-8mm}
(-1)^{\delta_i(\epsilon+\sum\delta|_{(i..j)})+\1+\(\1+\epsilon+\sum\delta\)\sum\delta|_{(i..k)}}\!\Bigl\llbracket (x_i\!-\!y_{i+1})
 \,g_{i+1,k,q,j}^{(2)}\bigl((i{+}1..j]\!\setminus\! M\bigr)\Bigr\rrbracket
\\
+
\sum_{\tiny\begin{array}{c}k\in(i{+}2..q]\setminus M\\k{-}1\in M\end{array}}
\hspace{-8mm}
(-1)^{\1+\(\1+\epsilon+\sum\delta\)\sum\delta|_{[i..k)}}\Bigl\llbracket (x_i-x_{i+1})g_{i,k,q,j}^{(2)}\Bigl(\{i{+}1\}\cup\bigl((i{+}1..j]\setminus M\bigr)\Bigr)\Bigr\rrbracket
\\
=
\sum_{\tiny\begin{array}{c}k\in(i{+}2..q]\setminus M\\k{-}1\in M\end{array}}
(-1)^{\1+\(\1+\epsilon+\sum\delta\)\sum\delta|_{[i..k)}}\\
\times
\Bigl\llbracket(x_i{-}y_{i+1})\,g_{i+1,k,q,j}^{(2)}\bigl((i{+}1..j]\setminus M\bigr)
+(x_i{-}x_{i+1})\,g_{i,k,q,j}^{(2)}\Bigl(\{i{+}1\}\cup\bigl((i{+}1..j]\setminus M\bigr)\Bigr)
\Bigr\rrbracket
\\
=\sum_{\tiny\begin{array}{c}k\in(i{+}2..q]\setminus M\\k{-}1\in M\end{array}}(-1)^{\1+\(\1+\epsilon+\sum\delta\)\sum\delta|_{[i..k)}}\Bigl\llbracket
g_{i,k,q,j}^{(2)}\bigl((i{+}1..j]\setminus M\bigr)\Bigr\rrbracket.
\end{align*}
Summarizing, we have
\begin{align*}
P_{i,j}^{\epsilon,\delta}(M)&=\Bigl\llbracket g_{i,i,q,j}^{(2)}\bigl((i{+}1..j]{\setminus}M\bigr)\Bigr\rrbracket H_i^{\epsilon+\sum\delta}
\\
&+\cond_{i+2\notin M}(-1)^{\1+\(\1+\epsilon+\sum\delta\)(\delta_i+\delta_{i+1})}\ll g_{i,i+2,q,j}^{(2)}\bigl((i{+}1..j]{\setminus}M\bigr)\rr H_{i+2}^{\epsilon+\sum\delta}
\\
&+\!\sum_{\tiny\begin{array}{c}k\in(i{+}2..q]\setminus M\\k{-}1\in M\end{array}}(-1)^{\1+\(\1+\epsilon+\sum\delta\)\sum\delta|_{[i..k)}}\Bigl\llbracket
g_{i,k,q,j}^{(2)}\bigl((i{+}1..j]\setminus M\bigr)\Bigr\rrbracket H_k^{\epsilon+\sum\delta}.
\end{align*}
As $(i{+}1..j]\setminus M=(i..j]\setminus M$, we get the required result.

\smallskip

{\it Case~4: $i+1<\bar q=\min M<j$.} In this case $j\in M$. By~{(\bf P\,-5)}, Lemma~\ref{lemma:rcoeff:1}
and the inductive hypothesis, we get
\begin{align*}
P_{i,j}^{\epsilon,\delta}(M)=
\hspace{-2mm}
\sum_{\gamma+\sigma=\epsilon}
\hspace{-2mm}
(-1)^{\gamma(\1+\epsilon+\sum\delta|_{[q..j)})}P_{i,q}^{\gamma,\delta|_{[i..q)}}\!(\{\bar q\})
\cond_{\sigma=\sum\delta_{[q..j)}}\!\ll g_{q,j}^{(1)}\bigl((q..j)\setminus M_{(q..j]}\bigr)\rr
\\
+
\hspace{-2mm}
\sum_{\tiny
\begin{array}{c}
k\in[i..q{-}1]\setminus M_{\bar{\vphantom{1}q}\mapsto\odd{q{-}1}}
\\
 k{-}1\in M_{\bar{\vphantom{1}q}\mapsto\odd{q{-}1}}\cup\{i{-}1,i\}
 \end{array}
 }
 \hspace{-14mm}
(-1)^{\cscript_{k>i}\cdot\1+\(\1+\epsilon+\sum\delta\)\sum\delta|_{[i..k)}}
\ll g_{i,k,q-1,j}^{(2)}\bigl((i..j]\setminus M_{\bar{\vphantom{1}q}\mapsto\odd{q{-}1}}\bigr)\rr H_k^{\epsilon+\sum\delta}.
\end{align*}
In the last sum, the summation parameter $k$ can take only two values $i$ and $i+1$,
as $k{-}1\in M_{\bar{\vphantom{1}q}\mapsto\odd{q{-}1}}\cup\{i{-}1,i\}$ holds only for these values.
Hence, by Lemma~\ref{lemma:compol:12} applied for $S=(q..j]\setminus M$ and ({\bf P-1}), $P_{i,j}^{\epsilon,\delta}(M)$ equals
\begin{align*}
(-1)^{(\epsilon+\sum\delta|_{[q..j)})(\1+\epsilon+\sum\delta|_{[q..j)})} P_{i,q}^{\epsilon+\sum\delta_{[q..j)},\delta|_{[i..q)}}\!\(\left\{\bar q\right\}\)
\ll g_{q,j}^{(1)}\bigl((q..j)\setminus M_{(q..j]}\bigr)\rr
\\
+\ll g_{i,i,q-1,j}^{(2)}\bigl((i..j]\setminus M_{\bar{\vphantom{1}q}\mapsto\odd{q{-}1}}\bigr)\rr H_i^{\epsilon+\sum\delta}
\\
+(-1)^{\1+\(\1+\epsilon+\sum\delta\)\delta_i}\ll g_{i,i+1,q-1,j}^{(2)}\bigl((i..j]\setminus M_{\bar{\vphantom{1}q}\mapsto\odd{q{-}1}}\bigr)\rr H_{i+1}^{\epsilon+\sum\delta}
\\
=\Bigl(H_i^{\epsilon+\sum\delta}-(-1)^{\delta_i(\epsilon+\sum\delta|_{(i..j)})}H_{i+1}^{\epsilon+\sum\delta}\Bigr)\ll g_{q,j}^{(1)}\bigl((q..j)\setminus M\bigr)\rr
\\
+\!\ll g_{i,i,q-1,j}^{(2)}\bigl((i..j]\!\setminus\! M\bigr)\!\rr\! H_i^{\epsilon+\sum\delta}
\!\!
+\!(-1)^{\1+\(\1+\epsilon+\sum\delta\)\delta_i}
\!\ll g_{i,i+1,q-1,j}^{(2)}\bigl((i..j]\!\setminus\! M\bigr)\rr\! H_{i+1}^{\epsilon+\sum\delta}
\\
=
\ll
g_{q,j}^{(1)}\bigl((q..j)\setminus M\bigr)+g_{i,i,q-1,j}^{(2)}\bigl((i..j)\setminus M\bigr)
\rr
H_i^{\epsilon+\sum\delta}
\\
+(-1)^{\1+\(\1+\epsilon+\sum\delta\)\delta_i}
\ll
g_{q,j}^{(1)}\bigl((q..j)\setminus M\bigr)+g_{i,i+1,q-1,j}^{(2)}\bigl((i..j)\setminus M\bigr)
\rr
H_{i+1}^{\epsilon+\sum\delta}
\\
=
\ll
g_{i,i,q,j}^{(2)}\bigl((i..j]\setminus M\bigr)
\rr
H_i^{\epsilon+\sum\delta}
+(-1)^{\1+\(\1+\epsilon+\sum\delta\)\delta_i}
\ll
g_{i,i+1,q,j}^{(2)}((i..j]\setminus M)
\rr
H_{i+1}^{\epsilon+\sum\delta}.
\end{align*}

{\it Case 5: $i+1<\min M<q\le j$.} We set $m:=\min M$.  By~({\bf P-6}) and~({\bf P-1}), we get
\begin{align*}
P_{i,j}^{\epsilon,\delta}(M)
=(-1)^{\sum\delta|_{[i..m)}(\epsilon+\sum\delta|_{[m..j)})}
B(i,i+1)
\,P_{m,j}^{\epsilon+\sum\delta|_{[i..m)},\delta|_{[m..j)}}\(M_{(m..j]}\)
\\
-
\hspace{-2 mm}
\sum_{\xi\in\{\0,\1\}}
\hspace{-2 mm}
P_{i,m}^{\xi,\delta|_{[i..m)}}(\{\bar m\})
\hspace{-2 mm}
\sum_{\tau+\sigma=\epsilon+\delta_m+\xi}
\hspace{-2 mm}
(-1)^{(\xi+\tau+\delta_m)(\epsilon+\sum\delta|_{[m..j)}+\xi)}
P_{m,j}^{\sigma,\chi_m^\tau\cup\,\delta|_{(m..j)}}\!\(M_{(m..j]}\)
\\
+P_{i,j}^{\,\epsilon,\delta}(M_{m\mapsto m{-}1})+P_{i,j}^{\,\epsilon,\delta}(M\setminus\{m\})\,C(m{-}1,m)
\end{align*}

By Proposition~\ref{proposition:rcoeff:1}, the middle summand equals
\begin{align*}
-2P_{i,m}^{\epsilon+\sum\delta|_{[m..j)},\delta|_{[i..m)}}(\{\bar m\})
\ll g_{m,m,q,j}^{(2)}\bigl((m..j]\setminus M_{(m..j]}\bigr)\rr H_m
\\
=-2\Bigl(H_i^{\epsilon+\sum\delta}-(-1)^{\delta_i\(\epsilon+\sum\delta|_{(i..j)}\)}H_{i+1}^{\epsilon+\sum\delta}\Bigr)
\ll g_{m,m,q,j}^{(2)}\bigl((m..j]\setminus M_{(m..j]}\bigr)\rr H_m
\\
=-2
\Bigl(H_i^{\epsilon+\sum\delta}-(-1)^{\delta_i\(\epsilon+\sum\delta|_{(i..j)}\)}H_{i+1}^{\epsilon+\sum\delta}\Bigr)
\ll g_{m,m,q,j}^{(2)}\bigl((m..j]\setminus M_{(m..j]}\bigr)\rr H_m.
\end{align*}

Now, using the inductive hypothesis, we gather the ``coefficients'' of
$H_i^{\epsilon+\sum\delta}$,
$H_{i+1}^{\epsilon+\sum\delta}$, $H_m^{\epsilon+\sum\delta}$,
$H_{m+1}^{\epsilon+\sum\delta}$ and $H_k^{\epsilon+\sum\delta}$ for $k>m+1$.

The coefficient of $H_i^{\epsilon+\sum\delta}$ is
\begin{align*}
P_{i,j}^{\epsilon,\delta}(M)
=-2H_m\ll g_{m,m,q,j}^{(2)}\bigl((m..j]\setminus M_{(m..j]}\bigr)\rr
+\ll g_{i,i,q,j}^{(2)}\bigl((i..j]\setminus M_{m\mapsto m{-}1}\bigr)\rr
\\
+C(m{-}1,m)\ll g_{i,i,q,j}^{(2)}\Bigl((i..j]\setminus\bigl(M\setminus\{m\}\bigr)\Bigr)\rr\\
=\Bigl\llbracket(x_m-y_m)\,g_{m,m,q,j}^{(2)}\bigl((m..j]\setminus M\bigr)
+g_{i,i,q,j}^{(2)}\Bigl((i..m{-}1)\cup\{m\}\cup\bigl((m..j]\setminus M\bigr)\Bigr)\\
+(x_{m-1}-x_m)\,g_{i,i,q,j}^{(2)}\Bigl((i..m]\cup\bigl((m..j]\setminus M\bigr)\Bigr)\Bigr\rrbracket
=\Bigl\llbracket g^{(2)}_{i,i,q,j}\Bigl((i..m)\cup\bigl((m..j]\setminus M\bigr)\Bigr)\Bigr\rrbracket,
\end{align*}
where the last equality comes from Lemma~\ref{lemma:compol:13}\ref{lemma:compol:13:part:a}.

Using Lemma~\ref{lemma:compol:13}\ref{lemma:compol:13:part:b} in the last equality below,
we see that the coefficient of $H_{i+1}^{\epsilon+\sum\delta}$ is:
\begin{align*}
(-1)^{\1+\(\1+\epsilon+\sum\delta\)\delta_i}\Bigl(
-2\,H_m\ll g_{m,m,q,j}^{(2)}\bigl((m..j]{\setminus}M_{(m..j]}\bigr)\rr
\\
+\cond_{i+1<m-1}\!\ll g^{(2)}_{i,i+1,q,j}\bigl((i..j]{\setminus}M_{m\mapsto m{-}1}\bigr)\rr
\\
+C(m{-}1,m)\ll g^{(2)}_{i,i+1,q,j}\Bigl((i..j]\setminus\bigl(M\setminus\{m\}\bigr)\Bigr)\rr\Bigr)\\
=(-1)^{\1+\(\1+\epsilon+\sum\delta\)\delta_i}\Bigl\llbracket(x_m-y_m)\,g_{m,m,q,j}^{(2)}\bigl((m..j]{\setminus}M\bigr)
\\
+\cond_{i+1<m-1}g^{(2)}_{i,i+1,q,j}\Bigl((i..m{-}1)\cup\{m\}\cup\bigl((m..j]{\setminus}M\bigr)\Bigr)
\\
+(x_{m-1}-x_m)\,g^{(2)}_{i,i+1,q,j}\Bigl((i..m]\cup\bigl((m..j]\setminus M\bigr)\Bigr)\Bigr\rrbracket\\
=(-1)^{\1+\(\1+\epsilon+\sum\delta\)\delta_i}
\ll g^{(2)}_{i,i+1,q,j}\Bigl((i..m)\cup\bigl((m..j]\setminus M\bigr)\Bigr)\rr.
\end{align*}

Using Lemma~\ref{lemma:compol:13}\ref{lemma:compol:13:part:c} in the last equality below,
we see that the coefficient of $H_m^{\epsilon+\sum\delta}$ is:
\begin{align*}
(-1)^{\sum\delta|_{[i..m)}\epsilon+\sum\delta|_{[i..m)}\sum\delta|_{[m..j)}}
B(i,i+1)\ll g^{(2)}_{m,m,q,j}\bigl((m..j]\setminus M_{(m..j]}\bigr)\rr
\\
+(-1)^{\1+(\1+\epsilon+\sum\delta)\sum\delta|_{[i..m)}}\ll g_{i,m,q,j}^{(2)}\bigl((i..j]\setminus M_{m\mapsto m{-}1}\bigr)\rr
\\
=(-1)^{\1+(\1+\epsilon+\sum\delta)\sum\delta|_{[i..m)}}
\Bigl\llbracket-(x_i-y_{i+1})\,g^{(2)}_{m,m,q,j}\bigl((m..j]\setminus M\bigr)
\\
+g_{i,m,q,j}^{(2)}
\Bigl((i..m{-}1)\cup\{m\}\cup\bigl((m..j]{\setminus}M\bigr)\Bigr)\Bigr\rrbracket=0.
\end{align*}

Using Lemma~\ref{lemma:compol:13}\ref{lemma:compol:13:part:d} in the last equality below,
we see that the coefficient of $H_{m+1}^{\epsilon+\sum\delta}$ is:

\begin{align*}
\cond_{\hspace{-.1mm}m+1\notin M}
(-1)^{\sum\delta|_{[i..m)}\epsilon+\sum\delta|_{[i..m)}\sum\delta|_{[m..j)}+\1+(\1+\epsilon+\sum\delta)\delta_m}
B(i,i+1)
\\
\times
\ll g_{m,m+1,q,j}^{(2)}((m..j]\setminus M_{(m..j]})\rr
\\
=\cond_{m+1\notin M}(-1)^{\1+(\1+\epsilon+\sum\delta)\sum\delta|_{[i..m]}}
\ll(x_i-y_{i+1})\,g_{m,m+1,q,j}^{(2)}((m..j]\setminus M)\rr
\\
=\cond_{m+1\notin M}(-1)^{\1+(\1+\epsilon+\sum\delta)\sum\delta|_{[i..m+1)}}
\ll g_{m,m+1,q,j}^{(2)}\Bigl((i..m)\cup\bigl((m..j]\setminus M\bigr)\Bigr)\rr.
\end{align*}

By Lemma~\ref{lemma:compol:13}\ref{lemma:compol:13:part:e},
we see that for $k>m+1$, the coefficient of $H_k^{\epsilon+\sum\delta}$ is:
\begin{align*}
(-1)^{\sum\delta|_{[i..m)}\epsilon+\sum\delta|_{[i..m)}\sum\delta|_{[m..j)}}\,
B(i,i+1)
\\
\times
\sum_
{
\tiny
\begin{array}{c}
k\in(m{+}1..q]\setminus M_{(m..j]}\\ k{-}1\in M_{(m..j]}\cup\{m{-}1,m\}
\end{array}
}
\hspace{-15mm}(-1)^{\1+(\1+\epsilon+\sum\delta)\sum\delta|_{[m..k)}}\!\!\ll g_{m,k,q,j}^{(2)}\bigl((m..j]{\setminus}M_{(m..j]}\bigr)\!\rr
\\
+\sum_
{\tiny
\begin{array}{c}
k\in(m{+}1..q]\setminus M_{m\mapsto m{-}1}\\ k{-}1\in M_{m\mapsto m{-}1}\cup\{i{-}1,i\}
\end{array}
}
\hspace{-15mm}
(-1)^{\1+\(\1+\epsilon+\sum\delta\)\sum\delta|_{[i..k)}}\ll g_{i,k,q,j}^{(2)}\bigl((i..j]\setminus M_{m\mapsto m{-}1}\bigr)\rr
\\
+C(m-1,m)
\hspace{-8mm}
\sum_{
\tiny
\begin{array}{c}
k\in(m{+}1..q]\setminus\bigl(M\setminus\{m\}\bigr)\\ k{-}1\in\bigl(M\setminus\{m\}\bigr)\cup\{i{-}1,i\}
\end{array}}
\hspace{-15mm}(-1)^{\1+\(\1+\epsilon+\sum\delta\)\sum\delta|_{[i..k)}}\ll g_{i,k,q,j}^{(2)}\Bigl((i..j]\setminus\bigl(M\setminus\{m\}\bigr)\Bigr)\rr
\end{align*}
\begin{align*}
=
\sum_
{
\tiny
\begin{array}{c}
k\in(m{+}1..q]\setminus M\\ k{-}1\in M
\end{array}
}
(-1)^{\1+(\1+\epsilon+\sum\delta)\sum\delta|_{[i..k)}}\\
\times\Bigl\llbracket
(x_i-y_{i+1})\,g_{m,k,q,j}^{(2)}\bigl((m..j]{\setminus}M\bigr)
+g_{i,k,q,j}^{(2)}\Bigl((i..m{-}1)\cup\{m\}\cup\bigl((m..j]\setminus M\bigr)\Bigr)
\\
+(x_{m-1}-x_m)\,g_{i,k,q,j}^{(2)}\Bigl((i..m]\cup\bigl((m..j]\setminus M\bigr)\Bigr)
\Bigr\rrbracket
\\
=
\sum_
{
\tiny
\begin{array}{c}
k\in(m{+}1..q]\setminus M\\ k{-}1\in M
\end{array}
}
(-1)^{\1+(\1+\epsilon+\sum\delta)\sum\delta|_{[i..k)}}
\Bigl\llbracket
g^{(2)}_{i,k,q,j}\Bigl((i..m)\cup\bigl((m..j]\setminus M\bigr)\Bigr)
\Bigr\rrbracket.
\end{align*}

Note that $(i..m)\cup\bigl((m..j]\setminus M\bigr)=(i..j]\setminus M$, since $m\in M$. Hence and from the above formulas
we get
\begin{align*}
P_{i,j}^{\epsilon,\delta}(M)=\ll g_{i,i,q,j}^{(2)}\bigl((i..j]\setminus M\bigr)\rr H_i^{\epsilon+\sum\delta}
\\
+(-1)^{\1+\(\1+\epsilon+\sum\delta\)\delta_i}\ll g^{(2)}_{i,i+1,q,j}\bigl((i..j]\setminus M\bigr)\rr H_{i+1}^{\epsilon+\sum\delta}
\\
+\cond_{m+1\notin M}(-1)^{\1+\(\1+\epsilon+\sum\delta\)\sum\delta|_{[i..m{+}1)}}\!\ll g_{i,m+1,q,j}^{(2)}\bigl((i..j]\setminus M\bigr)\rr\! H_{m+1}^{\epsilon+\sum\delta}
\\
+\sum_{\tiny
\begin{array}{c}
k\in(m{+}1..q]\setminus M\\
k{-}1\in M
\end{array}
}
\hspace{-2 mm}(-1)^{\1+\(\1+\epsilon+\sum\delta\)\sum\delta|_{[i..k)}}\Bigl\llbracket g_{i,k,q,j}^{(2)}\bigl((i..j]\setminus M\bigr)\Bigr\rrbracket H_k^{\epsilon+\sum\delta},
\end{align*}
which is exactly the required formula.
\end{proof}

\chapter{Combinatorics of signature sequences}\label{ChComb}

\section{Marked signature sequences}

A {\em signature sequence} is a finite sequence with entries $+$ or $-$.
A {\em marked signature sequence} is a finite sequence with entries of the form
$+_i$ or $-_i$
with $i\in\Z$. We refer to ``$+_i$'' as a ``$+$'' marked with $i$ and similarly for $-$.
Here is an example of a marked signature sequence:
$+_1-_2$.
When convenient, we will ignore marks and consider a marked signature sequence as just a signature sequence.
Given several marked signature sequences $w_1,\dots,w_n$ we denote by $w_1\dots w_n$ the marked signature sequence obtained by the concatenation of
$w_1,\dots,w_n$.
We denote the empty sequence by $\emptyset$.

Our signature sequences will usually arise from the following set up.
Let $I$ be a finite subset of $\Z$ and $u$ be a map from $I$ to the set of signature sequences (usually each $u_i$ will be a one-element or a two-element sequence).
If $J=\{j_1<\cdots<j_m\}$
is a subset of $I$, we denote by $\prod_{i\in J} u_i$ the {\em marked}\, signature sequence $u_{j_1}\cdots u_{j_m}$, where we mark the entries of each $u_{k}$ with $k$. We will also abbreviate $\prod u=\prod_{i\in I} u_i$.


Let $u$ be a marked signature sequence.
The {\it reduction} of $u$, denoted $[u]$, is the
marked signature sequence obtained from $u$ by
successively erasing all subsequences $-+$ (whatever the marks are).
We point out that this definition is different from the one used in the Introduction (where we were erasing $+-$'s instead). The reason for this discrepancy is that in the Introduction it was convenient to read off the signature sequence of a partition $(\la_1,\la_2,\dots)$ going from bottom left to top right, while in the main body of the paper we are going to read off the signature sequence of a weight $(\la_1,\dots,\la_n)$ starting with $\la_1$ and continuing onto $\la_n$.

We will use the well-known fact that the reduction is well defined, that is, it is independent of the order in which one erases subsequences $-+$:

\begin{proposition}\label{proposition:pm:1}
The reduction of any marked signature sequence is well defined.
\end{proposition}

\begin{corollary}\label{corollary:pm:1}
For arbitrary marked signature sequences $u,v$, we have $[uv]=\bigl[[u]v\bigr]=\bigl[u[v]\bigr]=\bigl[[u][v]\bigr]$.
\end{corollary}

It is also clear that the reduction $[u]$ is always a sequence of $+$'s (possibly empty) followed by a sequence of $-$'s (possibly empty):

\begin{proposition}\label{proposition:pm:2}
If $u$ contains $a$ symbols $+$ and $b$ symbols $-$ (with whatever marks), then $[u]=+^s-^r$ (with some marks), where $s-r=a-b$.
\end{proposition}

\begin{corollary}
\label{corollary:pm:2}
Let $I$ be a finite subset of $\Z$ and
$u:I\to\{\emptyset,--,+-,++\}$ with 
$\bigl[\prod u\bigr]=+^s-^r$. Then $s-r\=0\pmod 2$.
\end{corollary}


\begin{definition}\label{definition:pm:1}
{\rm
Let $I$ be a finite subset of $\Z$.
A {\em flow} on $I$  is a set of pairs $\Gamma=\{(a_1,b_1),\ldots,(a_N,b_N)\}$ satisfying the conditions \ref{definition:pm:1:1}-\ref{definition:pm:1:3} below:
{\leftmargini=30pt
\renewcommand{\labelenumi}{{\rm \theenumi}}
\renewcommand{\theenumi}{{\rm(\arabic{enumi})}}
\begin{enumerate}
\item\label{definition:pm:1:1}   $a_1,b_1,\ldots,a_N,b_N$ belong to $I$;
\item\label{definition:pm:1:2}  $a_1,\ldots,a_N$ are all distinct and $b_1,\ldots,b_N$ are also all distinct;
\item\label{definition:pm:1:3} $a_k<b_k$ for any $k=1,\ldots,N$.
\end{enumerate}}
\noindent If, instead, $\Gamma$ satisfies~\ref{definition:pm:1:1},~\ref{definition:pm:1:2} and \ref{definition:pm:1:3'}  below, then we call it a {\em weak flow}.
{\leftmargini=30pt
\renewcommand{\labelenumi}{{\rm \theenumi}}
\renewcommand{\theenumi}{{\rm(\arabic{enumi}${}^\prime$)}}
\begin{enumerate}
\setcounter{enumi}{2}
\item\label{definition:pm:1:3'} $a_k\le b_k$ for any $k=1,\ldots,N$,
\end{enumerate}}
\noindent
}
\end{definition}

By a {\it graph} we mean a set $\Gamma$ of pairs of integers, called the {\it edges} of $\Gamma$. If $E=(s,t)$ is an edge, then $s$ is its {\it source}
and $t$ is its {\it target}. We also say that $E$ {\it begins} at $s$ and {\it ends} at $t$. All of this applies to flows which are examples of graphs.

\begin{definition}\label{definition:pm:2}
{\rm
Let $u$ be a map from $I$ to the set of signature sequences.
A weak flow $\Gamma$ on $I$ is {\em coherent with $u$} if the following conditions \ref{definition:pm:2:4} and \ref{definition:pm:2:5} hold:
{\leftmargini=30pt
\renewcommand{\labelenumi}{{\rm \theenumi}}
\renewcommand{\theenumi}{{\rm(\arabic{enumi})}}
\begin{enumerate}
\setcounter{enumi}{3}
\item\label{definition:pm:2:4} if an edge of $\Gamma$ begins at $a$, then $u_a$ contains $-${\rm;}
\item\label{definition:pm:2:5} if an edge of $\Gamma$ ends at $b$, then $u_b$ contains $+$.
\end{enumerate}}
\noindent
\noindent
Moreover, $\Gamma$ is {\em fully coherent with $u$} if, in addition to \ref{definition:pm:2:4} and  \ref{definition:pm:2:5}, the following condition \ref{definition:pm:2:6} holds:
{\leftmargini=30pt
\renewcommand{\labelenumi}{{\rm \theenumi}}
\renewcommand{\theenumi}{{\rm(\arabic{enumi})}}
\begin{enumerate}
\setcounter{enumi}{5}
\item\label{definition:pm:2:6} for any $b\in I$ such that $u_b$ contains $+$,
there exists an edge of $\Gamma$ ending at $b$.
\end{enumerate}}
\noindent
Finally, an element $c\in I$ is a {\em bud} of $\Gamma$ with respect to $u$ if:
{\leftmargini=30pt
\renewcommand{\labelenumi}{{\rm \theenumi}}
\renewcommand{\theenumi}{{\rm(\arabic{enumi})}}
\begin{enumerate}
\setcounter{enumi}{6}
\item\label{definition:pm:2:7} no edge of $\Gamma$ begins at $c${\rm;}
\item\label{definition:pm:2:8} $u_c$ contains $-$.
\end{enumerate}}
}
\end{definition}

From now on until the end of the section $I$ is a finite subset of $\Z$.

\begin{lemma}\label{lemma:pm:1}
Let
$u:I\to\{\emptyset,-,+\}$ be 
such that
$\bigl[\prod u\bigr]=-^m$. Then there exists a flow on $I$ fully coherent with $u$ and having exactly $m$ buds with respect to~$u$.
\end{lemma}
\begin{proof}
By assumption, $\emptyset$'s and pairs $-+$ can be `erased' from $\prod u$ so that the sequence $-^m$ is left. Let the $i$th pair we have `erased' be $-_{a_i}+_{b_i}$. Then $\{(a_1,b_1),\dots,(a_N,b_N)\}$ is the desired flow.
\end{proof}

\begin{lemma}\label{lemma:pm:2}
Let 
$u:I\to\{\emptyset,-,+\}$ be 
such that $\bigl[\prod u\bigr]$
contains at least one $+$.
Then there exist a beginning $J$ of $I$ such that $\bigl[\prod_{i\in J} u_i\bigr]=+$ and a flow $\Gamma$ on $J$
coherent but not fully coherent with $u|_J$ and having no buds on $J$.
\end{lemma}
\begin{proof}
Induction on $|I|$. If $|I|=1$ then $u$ takes value $+$ at the only point of $I$, and we can set $J:=I$, $\Gamma:=\emptyset$.
Let $|I|>1$. Set $e:=\max I$, $E:=I\setminus\{e\}$, and write $\left[\prod\nolimits_{i\in E}u_i\right]=+^s-^r$.
By Corollary~\ref{corollary:pm:1},
\begin{equation}\label{equation:pm:2}
\textstyle
-^m=\bigl[\prod u\bigr]=
\bigl[\bigl[\prod\nolimits_{i\in E}u_i\bigr]u_e\bigr]=\bigl[(+^s-^r)\,u_e\bigr].
\end{equation}
If $s>0$, we can apply the inductive hypothesis to the restriction $u|_E$ to obtain
a suitable beginning $J$ of $E$ and a suitable flow $\Gamma$ on $J$.
Let
$s=0$. As $\bigl[\prod u\bigr]$
contains a `$+$', we must have $r=0$ and $u_e=+$. By Lemma~\ref{lemma:pm:1}, there exists a flow $\Gamma$ on $E$
fully coherent with $u|_E$ having no buds on $E$. It remains to set $J:=I$.
\end{proof}



\begin{lemma}\label{lemma:pm:3}
Let 
$u:I\to\{\emptyset,--,+-,++\}$
be such that $\bigl[\prod u\bigr]=-^m$ for some $m\geq 0$.
Then there exists a flow on $I$ fully coherent with $u$ and having exactly $m/2$ buds with respect to $u$.
\end{lemma}
\begin{proof} By Corollary~\ref{corollary:pm:2}, $m$ is even. We apply induction
on $|I|$. The case $I=\emptyset$ is clear. Let $|I|>0$. Set $e:=\max I$,
$E:=I\setminus\{e\}$, and $\left[\prod\nolimits_{i\in E}u_i\right]=+^s-^r$.
By Corollary~\ref{corollary:pm:1},
\begin{equation}\label{equation:pm:2.5}
\textstyle
-^m=\bigl[\prod u\bigr]=
\bigl[\bigl[\prod\nolimits_{i\in E}u_i\bigr]u_e\bigr]=\bigl[(+^s-^r)\,u_e\bigr].
\end{equation}

{\it Case 0: $u_e=\emptyset$}. The required flow for $u$ is the same as for $u|_E$.

{\it Case 1: $u_e=--$}. By~(\ref{equation:pm:2.5}), we have
$
-^m=[(+^s-^r)--]=+^s-^{r+2}.
$
So $s=0$, $r=m-2$, and the inductive hypothesis applies to the restriction $u|_E$.
Let $\Gamma$ be a flow on $E$ fully coherent with $u|_E$ and having on $E$ exactly $(m-2)/2$ buds.
Clearly, $\Gamma$ is fully coherent with $u$. Moreover, $e$ is its extra bud with respect to $u$. So $\Gamma$ has exactly $(m-2)/2+1=m/2$ buds.

{\it Case 2: $u_e=+-$}. By~(\ref{equation:pm:2.5}), we have $-^m=[(+^s-^r)+-]$.
So $s=0$ and $r=m>0$. So the inductive hypothesis applies to $u|_E$.
Let $\Gamma'$ be a flow on $E$ fully coherent with $u|_E$ having on $E$ exactly $m/2$ buds.
Since $m>0$, we obtain that $m/2>0$ and that $\Gamma'$ has at least one bud on $E$.
We denote it by $d$. Then $\Gamma:=\Gamma'\cup\{(d,e)\}$ is a flow fully coherent
with $u$. The flows $\Gamma$ and $\Gamma'$ both have $m/2$ buds, since $d$ is not a bud of $\Gamma$,
while $e$, on the contrary, is.

{\it Case 3: $u_e=++$}. By~(\ref{equation:pm:2.5}), we obtain
$$
-^m=[(+^s-^r)++]=
\left\{
{\arraycolsep=1pt
\begin{array}{ll}
+^{s+2-r}   & \mbox{ if } r<2;\\[1pt]
+^s-^{r-2}  & \mbox{ if } r\ge2.
\end{array}}
\right.
$$
Hence $s=0$, $r\ge2$, $-^m=-^{r-2}$, and $r=m+2$.
So the inductive hypothesis applies to $u|_E$.
Let $\Gamma'$ be a flow on $E$ fully coherent with $u|_E$ having on $E$ exactly $(m+2)/2$ buds.
Since $(m+2)/2>0$, the flow $\Gamma'$ has at least one bud on $E$. We denote it by $d$.
Then $\Gamma:=\Gamma'\cup\{(d,e)\}$ is a flow fully coherent with $u$.
Moreover, $\Gamma$ has one bud less than $\Gamma'$, since $d$ is not a bud of $\Gamma$.
\end{proof}

\begin{lemma}\label{lemma:pm:3.5}
Let 
$u:I\to\{\emptyset,--,+-,++\}$ be 
such that
$\bigl[\prod u\bigr]=-^m$ with $m>0$. Then there exists an index $a\in I$ such that
\begin{enumerate}
\item[{\rm (i)}] $u_a=--${\rm;}
\item[{\rm (ii)}] $\bigl[\prod_{i\in I\cap(a..+\infty)} u_i\bigr]$ equals either $\emptyset$ or $+-${\rm;}
\item[{\rm (iii)}] $\bigl[\prod_{i\in I\cap(-\infty..a]} u_i\bigr]=-^m$.
\end{enumerate}
\end{lemma}
\begin{proof} Induction on $|I|$. If $|I|=1$ we can take for $a$ the only element of $I$.
Let $|I|>1$. Set $e:=\max I$, $E:=I\setminus\{e\}$, and $\left[\prod\nolimits_{i\in E}u_i\right]=+^s-^r$.
By Corollary~\ref{corollary:pm:1},
\begin{equation}\label{equation:pm:2.625}
\textstyle
-^m=\bigl[\prod u\bigr]=
\bigl[\bigl[\prod\nolimits_{i\in E}u_i\bigr]u_e\bigr]=[(+^s-^r)\,u_e].
\end{equation}

{\it Case 0: $u_e=\emptyset$}. The required index for $u$ is the same as for $u|_E$.

{\it Case 1: $u_e=--$}. In this case, we can take  $a:=e$.

{\it Case 2: $u_e=+-$}. By~(\ref{equation:pm:2.625}), we get
$$
-^m=[(+^s-^r)+-]=\left\{
{\arraycolsep=0pt
\begin{array}{l}
+^{s+1}-\text{ if }r=0;\\[6pt]
+^s-^r\text{ if }r>0.
\end{array}}
\right.
$$
So $s=0$ and $r=m>0$. In particular, we can apply the inductive hypothesis to the restriction $u|_E$
to obtain an index $a\in E$ such that:
\begin{itemize}
\itemsep=4pt
\item $u_a=--${\rm;}
\item $\bigl[\prod_{i\in E\cap(a..+\infty)} u_i\bigr]$ equals either $\emptyset$ or $+-${\rm;}
\item $\bigl[\prod_{i\in E\cap(-\infty..a]} u_i\bigr]=-^m$.
\end{itemize}
If $\bigl[\prod_{i\in E\cap(a..+\infty)} u_i\bigr]=\emptyset$ then
$
\textstyle
\bigl[\prod_{i\in I\cap(a..+\infty)} u_i\bigr]=\bigl[\bigl[\prod\nolimits_{i\in E\cap(a..+\infty)} u_i\bigr]+-\bigr]=+-,
$
and if $\bigl[\prod_{i\in E\cap(a..+\infty)} u_i\bigr]=+-$ then still
$$
\textstyle
\bigl[\prod_{i\in I\cap(a..+\infty)} u_i\bigr]=\bigl[\bigl[\prod\nolimits_{i\in E\cap(a..+\infty)} u_i\bigr]+-\bigr]=[+-+-]=+-.
$$
On the other hand,
$
\textstyle
\bigl[\prod_{i\in I\cap(-\infty..a]} u_i\bigr]=
\bigl[\prod_{i\in E\cap(-\infty..a]} u_i\bigr]=-^m.
$

{\it Case 3: $u_e=++$}. We have
$$
-^m=[(+^s-^r)++]=\left\{
{\arraycolsep=0pt
\begin{array}{l}
+^{s+2-r}\text{ if }r<2;\\[6pt]
+^s-^{r-2}\text{ if }r\ge2.
\end{array}}
\right.
$$
So $s=0$ and $r=m+2$.
The inductive hypothesis applied to $u|_E$ yields $b\in E$ with
\begin{itemize}
\itemsep=4pt
\item $u_b=--${\rm;}
\item $\bigl[\prod_{i\in E\cap(b..+\infty)} u_i\bigr]$ equals either $\emptyset$ or $+-${\rm;}
\end{itemize}
By Corollary~\ref{corollary:pm:1}, 
$
\textstyle
\bigl[\prod_{i\in I\cap[b..+\infty)}u_i\bigr]=
\bigl[--\bigl[\prod_{i\in E\cap(b..+\infty)}u_i\bigr]++\bigr]=\emptyset,
$
so
$$
\textstyle
-^m=\bigl[\prod u\bigr]=\Bigl[\bigl[\prod_{i\in I\cap(-\infty..b)}u_i\bigr]\bigl[\prod_{i\in I\cap[b..+\infty)}u_i\bigr]\Bigr]=\bigl[\prod_{i\in I\cap(-\infty..b)}u_i\bigr].
$$
Therefore, we can apply the inductive hypothesis to the restriction $u|_{I\cap(-\infty..b)}$ and
find an index $a\in I\cap(-\infty..b)$ satisfying the following properties:
\begin{itemize}
\itemsep=4pt
\item $u_a=--${\rm;}
\item $\bigl[\prod_{i\in I\cap(-\infty..b)\cap(a..+\infty)} u_i\bigr]=\bigl[\prod_{i\in I\cap(a..b)} u_i\bigr]$ equals either $\emptyset$ or $+-${\rm;}
\item $\bigl[\prod_{i\in I\cap(-\infty..b)\cap(-\infty..a]} u_i\bigr]=\bigl[\prod_{i\in I\cap(-\infty..a]} u_i\bigr]=-^m$.
\end{itemize}
\noindent
Finally, we have
$$
\textstyle
\bigl[\prod_{i\in I\cap(a..+\infty)} u_i\bigr]=\Bigl[\bigl[\prod_{i\in I\cap(a..b)} u_i\bigr]\bigl[\prod_{i\in I\cap[b..+\infty)} u_i\bigr]\Bigr]=\bigl[\prod_{i\in I\cap(a..b)} u_i\bigr],
$$
which is $\emptyset$ or $+-$, as required.
\end{proof}

\begin{lemma}\label{lemma:pm:4}
Let 
$u:I\to\{\emptyset,--,+-,++\}$ be
such that $\bigl[\prod u\bigr]{=}+-^m$.
Then there exists an index $a\in I$ such that $u_a=+-$ and $\bigl[\prod_{i\in I\cap(-\infty..a)} u_i\bigr]=\emptyset$.
\end{lemma}
\begin{proof}
Induction on $|I|$. If $|I|=1$ we can take $a$ to be the only element of $I$.
Let $|I|>1$. Set $e:=\max I$, $E:=I\setminus\{e\}$, and $\left[\prod\nolimits_{i\in E}u_i\right]=+^s-^r$.
By Corollary~\ref{corollary:pm:1},
\begin{equation}\label{equation:pm:3}
\textstyle
+-^m=\bigl[\prod u\bigr]=
\bigl[\bigl[\prod\nolimits_{i\in E}u_i\bigr]u_e\bigr]
=\bigl[(+^s-^r)u_e\bigr].
\end{equation}
So~\mbox{$s\le1$.}
If $s=1$ we can apply the inductive hypothesis to the restriction $u|_E$.
Let $s=0$. Then the cases $u_e=\emptyset$ and $u_e=--$ are clearly impossible.
Moreover, if
$u_e=++$,
then by~(\ref{equation:pm:3}), we have
$
+-^m=[-^r++]
$.
By Corollary~\ref{corollary:pm:2}, $m$ is odd, so $m>0$, and we get a contradiction.
Finally, let $u_e=+-$. By~(\ref{equation:pm:3}), $+-^m=[-^r+-]$.
Hence $r=0$, and we can set $a:=e$.
\end{proof}

\begin{lemma}\label{lemma:pm:5}
Let $u:I\to\{\emptyset,--,+-,++\}$ be such that
$\bigl[\prod u\bigr]=+-^m$. Then there exist indices $a_1<\cdots<a_h$, with $h>0$, belonging to $I$ such that
\begin{enumerate}
\item[{\rm (i)}] $u_{a_1}=\cdots=u_{a_h}=+-${\rm;}
\item[{\rm (ii)}] $\bigl[\prod_{i\in I\cap(a_{k-1}..a_k)}u_i\bigr]=\emptyset$ for all $k=1,\ldots,h$, where $a_0:=-\infty${\rm;}
\item[{\rm (iii)}] $\bigl[\prod_{i\in I\cap(a_h..+\infty)}u_i\bigr]=-^{m-1}$.
\end{enumerate}
\end{lemma}
\begin{proof}
Induction on $|I|$. By Lemma~\ref{lemma:pm:4}, there is $a\in I$ such that $u_a=+-$,
and $\bigl[\prod_{i\in I\cap(-\infty..a)} u_i\bigr]=\emptyset$. Set $J:=I\cap(a..{+}\infty)$ and let
$\left[\prod\nolimits_{i\in J}u_i\right]=+^s-^r$.
By Corollary~\ref{corollary:pm:1},
$$
+-^m=\Bigl[\bigl[{\textstyle\prod\nolimits}_{i\in I\cap(-\infty..a)} u_i\bigr]\bigl[+-\bigr]\bigl[{\textstyle\prod\nolimits}_{i\in J}u_i\bigr]\Bigr]
=\bigl[+-(+^s-^r)\bigr]=\left\{
{\arraycolsep=1pt
\begin{array}{ll}
+-^{r+1}&\text{ if }s=0;\\
+^s-^r  &\text{ if }s>0.
\end{array}}
\right.
$$
If $s=0$ then $r=m-1$, and we can take $h:=1$ and $a_1:=a$.

Now let $s>0$. Then $s=1$ and $r=m$. By the inductive hypothesis applied to $u|_J$,
there are indices $b_1<\ldots<b_q$ belonging to $J$ such that
\begin{itemize}
\item $u_{b_1}=\cdots=u_{b_q}=+-${\rm;}\\[-6pt]
\item $\bigl[\prod_{i\in J\cap(b_{k-1}..b_k)}u_i\bigr]=\emptyset$ for any $k=1,\ldots,q$, where $b_0=a${\rm;}\\[-6pt]
\item $\bigl[\prod_{i\in J\cap(b_q..+\infty)}u_i\bigr]=-^{m-1}$.
\end{itemize}
\noindent
Now we can take $h:=q+1$, $a_1:=a$ and $a_k:=b_{k-1}$ for $k=2,\ldots,h$.
\end{proof}

\begin{definition}\label{definition:pm:3}
{\rm The set~$\{a_1<\cdots<a_h\}$ as in Lemma~\ref{lemma:pm:5} is called a {\em section} of $u$.
By Lemma~\ref{lemma:pm:3}, for all $k=1,\ldots,h+1$, there exists a graph $\Gamma_k$ fully coherent with
$u|_{I\cap(a_{k-1}..a_k)}$, where $a_0=-\infty$ and $a_{h+1}=+\infty$. The graph
$
\{(a_1,a_1),\ldots,(a_h,a_h)\}\cup\,\bigcup_{k=1}^{h+1}\Gamma_k
$
is called a {\em resolution} of $u$.
}
\end{definition}

By definition:

\begin{proposition}\label{proposition:A}
A resolution of a map $u:I\to\{\emptyset,--,+-,++\}$ is a weak flow
but not a flow
on $I$ fully coherent with $u$.
\end{proposition}

\begin{lemma}\label{lemma:pm:6}
Let 
$u:I\to\{\emptyset,--,+-,++\}$ be
such that $\bigl[\prod u\bigr]$ contains more that one $+$.
Then there exists a beginning $J$ of $I$ such that $\bigl[\prod_{i\in J}u_i\bigr]=++$ and a flow $\Gamma$ on $J$
coherent, but not fully coherent, with $u|_J$, and having no buds on $J$.
\end{lemma}
\begin{proof}
Induction on $|I|$. If $|I|=1$, then $u$ takes value $++$ at the only point of $I$, and we can take  $J:=I$ and $\Gamma:=\emptyset$.
Let now $|I|>1$. Set $e:=\max I$, $E:=I\setminus\{e\}$,
and let $\left[\prod\nolimits_{i\in E}u_i\right]=+^s-^r$.
By Corollary~\ref{corollary:pm:1}, we have
\begin{equation}\label{equation:pm:4}
\textstyle
\bigl[\prod u\bigr]
=\bigl[\bigl[\prod\nolimits_{i\in E}u_i\bigr]u_e\bigr]
=\bigl[(+^s-^r)u_e\bigr].
\end{equation}

If $s>1$, we can apply the inductive hypothesis to $u|_E$.
Consider the case $s=0$. The only possibility for $\bigl[\prod u\bigr]$ to contain more than one $+$ is the following:
$u_e=++$ and $\bigl[\prod\nolimits_{i\in E}u_i\bigr]=\emptyset$.
By Lemma~\ref{lemma:pm:3}, there exists a flow $\Gamma$ on $E$ fully coherent with $u|_E$ having no buds on $E$.
The same flow $\Gamma$ is coherent, but not fully coherent, with $u$ (as none of its edges ends in $e$),
and $\Gamma$ has no buds on~$I$. So we can take $J:=I$.

Finally, we consider the case $s=1$. By Corollary~\ref{corollary:pm:2}, $r$ is odd, in particular  $r>0$.
The case where $u_e=\emptyset$, $u_e=--$ and $u_e=+-$ are impossible, since then $\bigl[\prod u\bigr]$ contain at most one $+$.
So $u_e=++$.
By~(\ref{equation:pm:4}), we get
$$
\textstyle
\bigl[\prod u\bigr]=\bigl[+-^r++\bigr]=
\left\{
{\arraycolsep=0pt
\begin{array}{ll}
++      &\text{ if }r=1;\\
+-^{r-2}&\text{ if }r>1.
\end{array}}
\right.
$$
Hence $r=1$.
Let $\{a_1<\cdots<a_h\}$ be a section of $u|_E$.
By Lemma~\ref{lemma:pm:3}, for \mbox{$k=1,\ldots,h+1$} there exists a flow $\Gamma_k$
fully coherent with $u|_{I\cap(a_{k-1}..a_k)}$ and having no buds on $I\cap(a_{k-1}..a_k)$ (with $a_0:=-\infty$, $a_{h+1}:=+\infty$). Then the flow
$
\Gamma=
\{(a_1,a_2),\ldots,(a_{h-1},a_h),(a_h,e)\}\cup\bigcup_{k=1}^{h+1}\Gamma_k
$
is coherent, but not fully coherent, with $u$ (as none of its edges ends at $a_1$) and has no buds on $I$. It remains to set $J:=I$.
\end{proof}

\section{Normal and good indices}\label{SNormGood}
For any
weight $\lm=(\lm_1,\ldots,\lm_n)\in X(n)$, we define the map
$r_{\mathbf 0}(\lambda):[1..n]\to\{\emptyset,--,+-,++\}$
as follows:
$$
r_{\mathbf 0}(\lambda)_k:=\left\{ {\arraycolsep=0pt
\begin{array}{cl}
--&\mbox{ if }\res_p\lm_k=\mathbf0\mbox{ and }\res_p(\lm_k{-}1)=\mathbf0\,\bigl(\mbox{i.e.}\,\lm_k\=1\!\!\!\!\!\pmod p\bigr);\\
+-&\mbox{ if }\res_p(\lm_k{+}1)=\mathbf0\mbox{ and }\res_p\lm_k=\mathbf0\,\bigl(\mbox{i.e.}\,\lm_k\=0\!\!\!\!\!\pmod p\bigr);\\
++&\mbox{ if }\res_p(\lm_k{+}2)=\mathbf 0\mbox{ and }\res_p(\lm_k{+}1)\=\mathbf 0\,\bigl(\mbox{i.e.}\,\lm_k\=-1\!\!\!\!\!\pmod p\bigr);\\
\emptyset&\mbox{ otherwise}.
\end{array}}
\right.
$$
For any residue $\beta\in\Z/p\Z$ not equal to $\mathbf0$, we define the map
$r_\beta(\lambda):[1..n]\to\{-,+,\emptyset\}$ by the following rule:
$$
r_\beta(\lambda)_k:=\left\{ {\arraycolsep=0pt
\begin{array}{cl}
-&\mbox{ if }\res_p\lm_k=\beta;\\
+&\mbox{ if }\res_p(\lm_k{+}1)=\beta;\\
\emptyset&\mbox{ otherwise}.
\end{array}}
\right.
$$

We record the following obvious observation.

\begin{proposition}\label{proposition:rbeta}
Let $\lm\in X(n)$, $\beta\in \Z/p\Z$, and $1\leq k\leq n$.
Then
{
\renewcommand{\labelenumi}{{\rm \theenumi}}
\renewcommand{\theenumi}{{\rm(\arabic{enumi})}}
\begin{enumerate}
\item\label{proposition:rbeta:part:1} $r_\beta(\lambda)_k$ contains $-$ if and only if $\res_p\lm_k=\beta$.
\item\label{proposition:rbeta:part:2} $r_\beta(\lambda)_k$ contains $+$ if and only if $\res_p(\lm_k+1)=\beta$.
\end{enumerate}}
\end{proposition}

Recall that, according to our agreement, the entries of $\prod_{k\in J}r_\beta(\lambda)_k$,  for $J\subseteq [1..n]$, coming from the signature sequence $r_\be(\la)_k$ are marked with $k$.

\begin{definition}\label{definition:constr:4tens}
{\rm Let $\lm\in X(n)$ and $1\le i\leq n$. Set $\beta:=\res_p\lm_i$.
The index $i$ is called {\em tensor $\lm$-normal} if the reduction $\bigl[\prod_{i\le k\le n}r_\beta(\lambda)_k\bigr]$
contains the symbol $-_i$.
}
\end{definition}

Note that $n$ is always a tensor $\la$-normal index for $\la\in X(n)$.

\begin{definition}\label{definition:constr:4}
{\rm Let $\lm\in X(n)$ and $1\le i<n$. Set $\beta:=\res_p\lm_i$.
The index $i$ is called {\em $\lm$-normal} if the reduction $\bigl[\prod_{i\le k<n}r_\beta(\lambda)_k\bigr]$
contains the symbol $-_i$ and the following two conditions do {\em not} both hold:
(1)
$\bigl[\prod_{i<k<n}r_\beta(\lambda)_k\bigr]=\emptyset$ and (2) $\lm_i\=\lm_n\=0\pmod p$.
}
\end{definition}

\begin{definition}\label{definition:main:5tens}
{\rm Let $\lm\in X(n)$ and $1\le i\le n$. The index $i$ is called {\em tensor $\lm$-good} if it is tensor $\lm$-normal and there is no other
tensor $\lm$-normal index $h$ such that $\res_p\lm_h=\res_p\lm_i$ and $h<i$.
}
\end{definition}

\begin{definition}\label{definition:main:5}
{\rm Let $\lm\in X(n)$ and $1\le i<n$. The index $i$ is called {\em $\lm$-good} if it is  $\lm$-normal and there is no other $\lm$-normal index $h$ such that $\res_p\lm_h=\res_p\lm_i$ and $h<i$.
}
\end{definition}

To determine (tensor) normal and good nodes of a fixed residue $\beta$, we actually need to do less work
than is suggested by the definitions above.   Corollary~\ref{corollary:main:1} below shows that to determine normal and good indices, it suffices to calculate just the reduction  $[\prod_{1\le k<n} r_\beta(\lm)_k]$, and to calculate tensor normal and tensor good indices, it suffices to calculate just the reduction  $[\prod_{1\le k\le n} r_\beta(\lm)_k]$. Recall that we abbreviate $\prod r_\beta(\lm):=\prod_{1\le k\le n} r_\beta(\lm)_k$.

\begin{lemma}\label{lemma:main:2} Let $\lm\in X(n)$.
\begin{enumerate}
\item[{\rm (i)}] If $1\le i<n$ and $\beta:=\res_p\lm_i$, then 
$\bigl[\prod_{i\le k<n} r_\beta(\lm)_k\bigr]$ contains $-_i$ if and only if
 $\bigl[\prod_{1\le k<n} r_\beta(\lm)_k\bigr]$ contains $-_i$.
\item[{\rm (ii)}] If $1\le i\le n$ and $\beta:=\res_p\lm_i$, then 
$\bigl[\prod_{i\le k\le n} r_\beta(\lm)_k\bigr]$ contains $-_i$ if and only if
 $\bigl[\prod r_\beta(\lm)]$ contains $-_i$.
\end{enumerate}
\end{lemma}
\begin{proof}
We prove (i), (ii) being similar. By Corollary~\ref{corollary:pm:1}, we get
$$
\textstyle
\bigl[
\prod_{1\le k<n} r_\beta(\lm)_k
\bigr]=
\Bigl[
\prod_{1\le k<i} r_\beta(\lm)_k
\bigl[
\prod_{i\le k<n} r_\beta(\lm)_k
\bigr]
\Bigr].
$$
So, if $\bigl[\prod_{i\le k<n} r_\beta(\lm)_k\bigr]$ does not contain $-_i$, then
$\bigl[\prod_{1\le k<n} r_\beta(\lm)_k\bigr]$ also does not contain $-_i$.
Conversely, if
$\bigl[\prod_{i\le k<n} r_\beta(\lm)_k\bigr]$ contains $-_i$, then this symbol can not be erased as we reduce $\prod_{1\le k<i} r_\beta(\lm)_k
\bigl[
\prod_{i\le k<n} r_\beta(\lm)_k
\bigr]$.
\end{proof}

\begin{lemma}\label{lemma:main:3} Let $\lm\in X(n)$ and $1\le i<n$. Set $\beta:=\res_p\lm_i$.
Then $\bigl[\prod_{i<k<n} r_\beta(\lm)_k\bigr]$ is empty if and only if $-_i$ is the last symbol of
$\bigl[\prod_{1\le k<n} r_\beta(\lm)_k\bigr]$.
\end{lemma}
\begin{proof}
By Corollary~\ref{corollary:pm:1}, we get
$$
\textstyle
\bigl[
\prod_{1\le k<n} r_\beta(\lm)_k
\bigr]=
\Bigl[\Bigl(
\prod_{1\le k<i} r_\beta(\lm)_k\Bigr)\,
r_\beta(\lm)_i
\,
\bigl[
\prod_{i<k<n} r_\beta(\lm)_k
\bigr]
\Bigr].
$$
Therefore, if $\bigl[\prod_{i<k<n} r_\beta(\lm)_k\bigr]=\emptyset$, then $-_i$ is the last symbol of
$\bigl[\prod_{1\le k<n} r_\beta(\lm)_k\bigr]$, since this symbol is contained in $r_\beta(\lm)_i$
and can not be further erased.

Now suppose that $\bigl[\prod_{i<k<n}r_\beta(\lm)_k\bigr]\ne\emptyset$.
If $\bigl[\prod_{i<k<n}r_\beta(\lm)_k\bigr]$ contains at least two $+$'s,
then all symbols $-_i$ contained in $r_\beta(\lm)_i$ get erased during reduction, and
 so $\bigl[\prod_{1\le k<n} r_\beta(\lm)_k\bigr]$ does not contain $-_i$.

Let $\bigl[\prod_{i<k<n}r_\beta(\lm)_k\bigr]$ contain exactly one $+$.
If $\beta\ne\mathbf0$ then $r_\beta(\lm)_i=-_i$, and $-_i$ gets erased during reduction.
On the other hand, if $\beta=\mathbf0$ then by Corollary~\ref{corollary:pm:2}, the reduction
$\bigl[\prod_{i<k<n}r_\beta(\lm)_k\bigr]$ contains at lest one symbol $-_q$ with $q>i$.
This symbol cannot be erased and so $-_i$ is not the last symbol of $\bigl[\prod_{1\le k<n} r_\beta(\lm)_k\bigr]$.

Finally, if $\bigl[\prod_{i<k<n}r_\beta(\lm)_k\bigr]$ does not contain any $+$'s,
then it contains $-_q$ with $q>i$ at the end, and and so again $-_i$ is not the last symbol of $\bigl[\prod_{1\le k<n} r_\beta(\lm)_k\bigr]$.
\end{proof}

\begin{corollary}\label{corollary:main:1}
Let $\lm\in X(n)$.
\begin{enumerate}
\item[{\rm (i)}] Let $1\le i<n$ and $\beta:=\res_p\lm_i$. Then
$i$ is $\lm$-normal if and only if $\bigl[\prod_{1\le k<n} r_\beta(\lm)_k\bigr]$ contains $-_i$, and $-_i$ is not its last symbol
if $\lm_i\=\lm_n\=0\pmod p$.
\item[{\rm (ii)}] Let $1\le i\le n$ and $\beta:=\res_p\lm_i$. Then
$i$ is tensor $\lm$-normal if and only if $\bigl[\prod r_\beta(\lm)\bigr]$ contains $-_i$.
\end{enumerate}
\end{corollary}

In the case where $\lm_n\=0\pmod p$ we can simplify Corollary~\ref{corollary:main:1} as follows:

\begin{corollary}\label{corollary:main:2}
Let $\lm\in X(n)$, $\lm_n\=0\pmod p$, $1\le i<n$, and $\beta:=\res_p\lm_i$.
Then $i$ is $\lm$-normal if and only if
$\bigl[\prod r_\beta(\lm)\bigr]$
contains $-_i$ if and only if $i$ is tensor $\la$-normal.
\end{corollary}
\begin{proof}
The second ``if-and-only-if'' comes from Corollary~\ref{corollary:main:1}(ii). We now prove the first one. If $\beta\ne\mathbf0$ then we get
$\bigl[\prod r_\beta(\lm)\bigr]=\bigl[\prod_{1\le k<n} r_\beta(\lm)_k\bigr]$
and the condition~$\lm_i\=\lm_n\=0\pmod p$ is not satisfied, so the result follows from Corollary~\ref{corollary:main:1}.
Now let $\beta=\mathbf0$.

Let $i$ be $\lm$-normal.
By Corollary~\ref{corollary:main:1}, $\bigl[\prod_{1\le k<n}r_\beta(\lambda)_k\bigr]$
contains $-_i$. By Corollary~\ref{corollary:pm:1}, we have
$$
\textstyle
\bigl[\prod r_\beta(\lm)\bigr]=
\bigl[[\prod_{1\le k<n} r_\beta(\lm)_k]+_n-_n\bigr].
$$
If $-_i$ is not the last symbol of $\bigl[\prod_{1\le k<n}r_\beta(\lambda)_k\bigr]$,
then there is another $-$ after it, and so $-_i$ is not erased together with $+_n$.
So $\bigl[\prod r_\beta(\lm)\bigr]$ contains $-_i$.
It remains to consider the case where $\bigl[\prod_{1\le k<n}r_\beta(\lambda)_k\bigr]$
contains exactly one $-_i$ and this $-_i$  appears in the end. By Lemma~\ref{lemma:main:3},
$\bigl[\prod_{i<k<n}r_\beta(\lm)_k\bigr]=\emptyset$.
Hence by Corollary~\ref{corollary:pm:1},
$$
\textstyle
\bigl[\prod_{1\le k<n}r_\beta(\lm)_k\bigr]=
\Bigl[\bigl[\prod_{1\le k\le i}r_\beta(\lm)_k\bigr]
\bigl[\prod_{i<k<n}r_\beta(\lm)_k\bigr]\Bigr]
=\big[\prod_{1\le k\le i}r_\beta(\lm)_k\big].
$$
It now follows that  $r_\beta(\lm)_i=+_i-_i$, i.e. $\lm_i\=0\pmod p$.
Then we have $\la_i\=\lm_n\=0\pmod p$ and $\bigl[\prod_{i<k<n}r_\beta(\lm)_k\bigr]=\emptyset$, which
contradicts the
$\lm$-normalty of
$i$.

Now let $i$ be not $\lm$-normal. If
$\bigl[\prod_{i\le k<n} r_\beta(\lambda)_k\bigr]$
does not contain $-_i$, then
$
\textstyle
\bigl[\prod r_\beta(\lm)\bigr]
=\bigl[[\prod_{1\le k<i} r_\beta(\lm)_k][\prod_{i\le k<n} r_\beta(\lm)_k]+_n-_n\bigr]
$
also does not contain $-_i$.
On the other hand, if $\bigl[\prod_{i<k<n}r_\beta(\lambda)_k\bigr]=\emptyset$ and $\lm_i\=\lm_n\=0\pmod p$, then
\begin{align*}
\textstyle
&\bigl[\prod r_\beta(\lm)\bigr]
=\bigl[\bigl[\prod_{1\le k<i} r_\beta(\lm)_k\bigr]+_i-_i\bigl[\prod_{i<k<n} r_\beta(\lm)_k\bigr]+_n-_n\bigr]\\
\textstyle
=&\bigl[\bigl[\prod_{1\le k<i} r_\beta(\lm)_k\bigr]+_i-_i+_n-_n\bigr]=\bigl[[\prod_{1\le k<i} r_\beta(\lm)_k]+_i-_n\bigr]
\end{align*}
again does not contain $-_i$.
\end{proof}

\section{Tensor conormal and tensor cogood indices}\label{SSTensConCog}

We now introduce the notion dual to that of the tensor $\la$-normal index (cf. Corollary~\ref{corollary:main:1}(ii)).

\begin{definition}\label{DConorm}
{\rm Let $\lm\in X(n)$ and $1\le i\le n$. Set $\beta:=\res_p(\lm_i+1)$.
The index $i$ is called {\em tensor $\lm$-conormal} if
$\bigl[\prod r_\beta(\lm)\bigr]$ contains $+_i$.}
\end{definition}

Note that $1$ is always a tensor $\lm$-conormal index.

\begin{definition}\label{DGood}
{\rm
Let $\lm\in X(n)$ and $1\le i\le n$. The index $i$ is called {\em tensor $\lm$-cogood} if it is tensor $\lm$-conormal and there is no other
tensor $\lm$-conormal index $h$ such that $\res_p\lm_h=\res_p\lm_i$ and $h>i$.
}
\end{definition}

We fix $n$ and consider the map $w_0:[1..n]\to[1..n]$ given by $w_0i=n+1-i$.
For a marked signature sequence $u$,
we denote by $-w_0u$ the sequence obtained from $u$ by substitutions $-_i\mapsto+_{w_0i}$ and $+_i\mapsto-_{w_0i}$
for each $i$ and rewriting the resulting sequence in the reverse order.
Note that the definition of $-w_0u$ depends on $n$.

\begin{lemma}\label{lemma:-w_0u}
If $v$ is obtained from $u$ by erasing some subsequences
$-+$
then $-w_0v$ is obtained from $-w_0u$ by erasing similar subsequences.
In particular, we have $[-w_0u]=-w_0[u]$.
\end{lemma}
\begin{proof}
The first statement is obvious. To prove the second one, take $v=[u]$.
Then $-w_0[u]$ is obtained from $-w_0u$ by erasing some
subsequences of the form $-_i+_j$. However the sequence $-w_0[u]$ can not be further reduced and
thus is the reduction of $-w_0u$.
\end{proof}

\begin{lemma}\label{lemma:r_beta(-w_0lm)}
For any $\lm\in X(n)$ and $\beta\in\Z/p\Z$, we have $\prod r_\beta(-w_0\lm)=-w_0\Bigl(\prod r_\beta(\lm)\Bigr)$.
\end{lemma}
\begin{proof}
Follows from the definitions.
\end{proof}

\begin{corollary}\label{corollary:conormal}
Let $\la\in X(n)$ and $1\leq i\leq n$. Then $i$ is tensor $\lm$-conormal
if and only if the index $w_0i$ is tensor $-w_0\la$-normal.
\end{corollary}
\begin{proof}
Combining Lemmas~\ref{lemma:-w_0u} and~\ref{lemma:r_beta(-w_0lm)},
we get
$$
\textstyle
\bigl[\prod r_\beta(-w_0\lm)\bigr]=-w_0\bigl[\prod r_\beta(\lm)\bigr],
$$
and it remains to apply Definition~\ref{DConorm} and Corollary~\ref{corollary:main:1}(ii).
\end{proof}

\begin{corollary}\label{CCogood}
Let $\la\in X(n)$ and $1\leq i\leq n$. Then $i$ is tensor $\lm$-good
if and only if $w_0i$ is $-w_0\la$-cogood.
\end{corollary}
\begin{proof}
Let $\beta=\res_p\la_i$. By definition, $i$ is tensor $\la$-good if and only if $i$ is the smallest tensor $\la$-normal index of residue $\beta$. By Corollary~\ref{corollary:conormal}, this is equivalent to $w_0i$ being the largest   $-w_0\la$-conormal index of residue $\beta$. 
\end{proof}

\begin{lemma} \label{LGoodTopNormal}
Let $\la\in X(n)$ and $1\le i\le n$. Then:
\begin{enumerate}
\item[{\rm (i)}] $i$ is tensor $\lm$-good if and only if it is tensor $\la$-normal and tensor $(\la-\eps_i)$-conormal.
\item[{\rm (ii)}] $i$ is tensor $\lm$-cogood if and only if it is tensor $\la$-conormal and tensor $(\la+\eps_i)$-normal.
\end{enumerate}
\end{lemma}
\begin{proof}
(i) Let $i$ be a tensor $\la$-good index and $\beta:=\res_p\lm_i$. By definition, $i$ is the smallest among the $\la$-normal indices $j$ such that $\res_p\lm_j=\beta$. Let $j>i$ be a $\la$-normal index with $\res_p\lm_j=\beta$. We need to prove the following two claims:

\begin{enumerate}
\item[{\rm I.}] $j$ is not tensor $(\la-\eps_j)$-conormal;
\item[{\rm II.}] $i$ is tensor $(\la-\eps_i)$-conormal.
\end{enumerate}

First, we prove Claim I.
Observe that $\prod r_\beta(\lm-\epsilon_j)$ is obtained from $\prod r_\beta(\lm)$
by replacing the first symbol $-_j$ by $+_j$. If   $\beta\ne\mathbf0$, then $\bigl[\prod_{i<k<j}r_\beta(\lm)_k\bigr]=-^m$
for some $m\ge0$, and
$$
\textstyle
\bigl[\prod r_\beta(\lm-\epsilon_j)\bigr]=
\bigl[\bigl[\prod_{1\le k<i} r_\beta(\lm)_k\bigr]-_i-^m+_j\bigl[\prod_{j<k\le n} r_\beta(\lm)_k\bigr]\bigr].
$$
Clearly, the sequence in the right-hand side does not contain $+_j$, i.e.
$j$ is not tensor
$\lm-\epsilon_j$-conormal.

Now suppose that $\beta=\mathbf0$. If the reduction $\bigl[\prod_{i<k<j}r_\beta(\lm)_k\bigr]$ contains
more than one sign $-$, then $j$ is not tensor
$\lm-\epsilon_j$-conormal exactly as above. On the other hand,
if this reduction contained more than one sign $+$, then $i$ could not be tensor $\lm$-normal. By Corollary~\ref{corollary:pm:2}, we only have to consider the following two cases.

{\it Case 1:} $\bigl[\prod_{i<k<j}r_\beta(\lm)_k\bigr]=\emptyset$.
We have
$$
\textstyle
\bigl[\prod r_\beta(\lm-\epsilon_j)\bigr]=
\bigl[\bigl[\prod_{1\le k<i} r_\beta(\lm)_k\bigr]
[r_\beta(\lm)_ir_\beta(\lm-\epsilon_j)_j]
\bigl[\prod_{j<k\le n} r_\beta(\lm)_k\bigr]\bigr].
$$
The only chance for at least one $+_j$ to survive is $r_\beta(\lm)_i=+_i-_i$ and $r_\beta(\lm-\epsilon_j)_j=+_j+_j$.
Hence $r_\beta(\lm)_j=+_j-_j$ and
\begin{align*}
\textstyle
\bigl[\prod r_\beta(\lm)\bigr]=
\bigl[\bigl[\prod_{1\le k<i} r_\beta(\lm)_k\bigr]
[+_i-_i+_j-_j]
\bigl[\prod_{j<k\le n} r_\beta(\lm)_k\bigr]\bigr],
\end{align*}
in which case $-_i$ will not survive.
This contradicts the tensor $\lm$-normality of $i$.

{\it Case 2:} $\bigl[\prod_{i<k<j}r_\beta(\lm)_k\bigr]=+-$. We have
$$
\textstyle
\bigl[\prod r_\beta(\lm)\bigr]=
\bigl[\bigl[\prod_{1\le k<i} r_\beta(\lm)_k\bigr]
[r_\beta(\lm)_i+-]
\bigl[\prod_{j\le k\le n} r_\beta(\lm)_k\bigr]\bigr].
$$
As $i$ is tensor $\lm$-normal we must have $r_\beta(\lm)_i=-_i-_i$. Now the sequence
\begin{align*}
\textstyle
\bigl[\prod r_\beta(\lm-\epsilon_j)\bigr]=
\bigl[\bigl[\prod_{1\le k<i} r_\beta(\lm)_k\bigr]
[-_i-_i+-r_\beta(\lm-\epsilon_j)_j]
\bigl[\prod_{j<k\le n} r_\beta(\lm)_k\bigr]\bigr]\\
\textstyle
=\bigl[\bigl[\prod_{1\le k<i} r_\beta(\lm)_k\bigr]
[--r_\beta(\lm-\epsilon_j)_j]
\bigl[\prod_{j<k\le n} r_\beta(\lm)_k\bigr]\bigr].
\end{align*}
does not contain $+_j$, i.e. $j$ is not tensor $(\lm-\eps_j)$-conormal.

Now we prove Claim II. Note that $\bigl[\prod_{i\le k\le n}r_\beta(\lm)_k\bigr]$ is a sequence of pluses followed by  minuses, containing $-_i$. So the only plus it can contain is~$+_i$.

{\it Case a:} $\bigl[\prod_{i\le k\le n}r_\beta(\lm)_k\bigr]$ does not contain pluses.
We always have
$
\textstyle
\bigl[\prod r_\beta(\lm)\bigr]=
\bigl[\bigl[\prod_{1\le k<i} r_\beta(\lm)_k\bigr]
\bigl[\prod_{i\le k\le n} r_\beta(\lm)_k\bigr]\bigr]
$,
so any symbol $-_t$ occurring in
$\bigl[\prod_{1\le k<i} r_\beta(\lm)_k\bigr]$
occurs also in $\bigl[\prod r_\beta(\lm)\bigr]$. This $t$ is then tensor $\lm$-normal
of residue $\beta$, which is a contradiction as $t<i$.
Hence $\bigl[\prod_{1\le k<i} r_\beta(\lm)_k\bigr]=+^m$ for some $m\ge0$.
Now we get
$$
\textstyle
\bigl[\prod r_\beta(\lm-\epsilon_i)\bigr]=
\bigl[
\bigl[+^m r_\beta(\lm-\epsilon_i)_i\bigr]
\bigl[\prod_{i<k\le n} r_\beta(\lm)_k\bigr]\bigr].
$$
Therefore the symbol $+_i$ contained in $r_\beta(\lm-\epsilon_i)_i$ survives.

{\it Case b:} $\bigl[\prod_{i\le k\le n}r_\beta(\lm)_k\bigr]$ contains exactly one $+$.
In this case, $\beta=\mathbf 0$ and $r_{\mathbf 0}(\lm)_i=+_i-_i$.
If $\bigl[\prod_{1\le k<i} r_\beta(\lm)_k\bigr]$ contains at least two minuses, then similarly to Case~a, one of them survives in~$\bigl[\prod r_\beta(\lm)\bigr]$,
which contradicts the minimality of $i$. Hence $\bigl[\prod_{1\le k<i} r_\beta(\lm)_k\bigr]=+^m-^s$ for $s\le1$.
Now we get
$$
\textstyle
\bigl[\prod r_\beta(\lm-\epsilon_i)\bigr]=
\bigl[
\bigl[+^m-^s+_i+_i\bigr]
\bigl[\prod_{i<k\le n} r_\beta(\lm)_k\bigr]\bigr].
$$
Therefore the rightmost symbol $+_i$ contained in $r_\beta(\lm-\epsilon_i)_i=+_i+_i$ again survives.

(ii) By Corollary~\ref{CCogood}, $i$ is tensor $\la$-cogood if and only if
$w_0i$ is tensor $-w_0\lm$-good.
Now (i) implies (ii) using Corollary~\ref{corollary:conormal}.
\end{proof}

\begin{corollary}\label{CGoodCogood}
Let $\la\in X(n)$ and $1\le i\le n$. Then $i$ is tensor $\la$-good if and only if $i$ is tensor $(\la-\eps_i)$-cogood.
\end{corollary}

\section{Removable and addable nodes for dominant $p$-strict weights}\label{addableandremovablenodes}
Let $\lm\in X(n)$. We identify $\la$ and its ``Young diagram'':
$$
\lm:=\{(i,j)\in\Z^2\suchthat1\le i\le n, j\le \lm_i\}.
$$
Elements of $\Z^2$ are called {\it nodes}.
The {\it residue} of a node $(i,j)$ is $\res_p (i,j):=j(j-1)+p\,\Z$. Thus we will speak of nodes of $\la$, will remove nodes of $\la$, add nodes to $\la$, and so on.

Let $\beta\in\Z/p\Z$ and $\lm\in X^+_p(n)$ be a dominant $p$-strict weight. A node $A=(i,j)$ is called {\it $\beta$-removable} for $\lm$ if $\res_p A=\beta$ and one of the following two conditions holds:
\begin{itemize}
\itemsep=4pt
\item[(R1)]
$A\in\lm$ and $\lm\setminus\{A\}\in X^+_p(n)$;
\item[(R2)] the node $B=(i,j+1)$ immediately to the right of $A$ belongs to $\lm$, $\res_p B=\res_p A$,
            $\lm\setminus\{B\}\in X^+_p(n)$ and $\lm\setminus\{A,B\}\in X^+_p(n)$.
\end{itemize}
A node $B=(i,j)$ is called {\it $\beta$-addable} for $\lm$
if $\res_p B=\beta$ and one of the following two conditions holds:
\begin{itemize}
\itemsep=4pt
\item[(A1)] $B\notin\lm$ and $\lm\cup\{B\}\in X^+_p(n)$;
\item[(A2)] the node $A=(i,j-1)$ immediately to the left of $B$ does not belong to $\lm$, $\res_p B=\res_p A$,
            $\lm\cup\{A\}\in X^+_p(n)$ and $\lm\cup\{A,B\}\in X^+_p(n)$.
\end{itemize}
Of course, $(i,j)$ can be $\beta$-removable or $\beta$-addable for $\lm\in X^+_p(n)$
only if  $1\le i\le n$.

We introduce the following order on $\Z^2$: $(i,j)<(i',j')$ if and only if either
$i<i'$ or $i'=i$ and $j>j'$.
Consider now all $\beta$-removable and $\beta$-addable nodes of $\lm\in X^+_p(n)$ for a fixed $\beta$. Reading these nodes
in the increasing order 
and  assigning $-_i$ to $\beta$-removable node of the form $(i,j)$ and $+_i$ to a $\beta$-addable node of the form $(i,j)$, we get the {\it $\beta$-signature} of $\lm$.
The {\it reduced $\beta$-signature} of $\lm$ is the sequence obtained from the $\beta$-signature of $\lm$
by erasing all possible subsequences of the form $-+$, i.e. the reduced $\beta$-signature of $\lm$ is the reduction of the $\beta$-signature of $\lm$.

\begin{lemma}\label{LNormRemovAdd}
 Let $\lm\in X^+_p(n)$. Then the reduced $\beta$-signature of $\lm$ equals
              $\bigl[\prod r_\beta(\lambda)\bigr]$.
\end{lemma}
\begin{proof} We apply induction on $n$. If $n=1$ then the $\beta$-signature of $\lm=(\lm_1)$ equals
$r_\beta(\lm)_1$. Therefore the reduced  $\beta$-signature of $\lm$ equals
$[r_\beta(\lm)_1]=\bigl[\prod r_\beta(\lambda)\bigr]$.

Now let $n>1$ and set $\bar\lm:=(\lm_1,\ldots,\lm_{n-1})$.
We first investigate what happens to the set of $\beta$-removable and $\beta$-addable nodes when we pass from
$\bar\lm$ to $\lm$. Consider the strip $S=\{n\}\times(-\infty..\lm_n]$.
It can be considered as the diagram of $(\lm_n)$ shifted to row $n$. We can consider its $\beta$-removable and $\beta$-addable nodes.
More precisely, a node $A=(n,j)$ is  $\beta$-removable for $S$ if $\res_p j=\beta$ and
one of the following conditions holds: (R1$^\prime$) $j=\la_n$;
(R2$^\prime$) $j=\la_n-1$ and $\res_p (j+1)=\beta$.
The set of $\beta$-removable nodes of $S$ is denoted by $\Rem_S$.
A node $B=(n,j)$ is $\beta$-addable for $S$ if $\res_p j=\beta$
one of the following conditions holds:
(A1$^\prime$) $j=\lm_n+1$; (A2$^\prime$) $j=\lm_n+2$ and $\res_p(j-1)=\beta$.
The set of $\beta$-addable nodes of $S$ is denoted by $\Add_S$.
Note that he $\beta$-signature of $S$ equals $r_\beta(\lm)_n$.

Let $\Rem$ and $\overline\Rem$ denote the sets of $\beta$-removable nodes of $\lm$ and $\bar\lm$
respectively.
Also set $\overline\Rem_{<n-1}:=\{(i,j)\in\overline\Rem\mid i<n-1\}$.
Let $\Add$ and $\overline\Add$ denote the sets of $\beta$-addable nodes of $\lm$ and
$\bar\lm$ respectively.
One can easily verify the following formulas:
$$
\Rem=\left\{
{\arraycolsep=0pt
\begin{array}{ll}
\overline\Rem\cup\Rem_S&\text{ if }\lm_n<\lm_{n-1}-2;\\[3pt]
\overline\Rem\cup\Rem_S&\text{ if }\lm_n=\lm_{n-1}-2\\[3pt]
&\hfill\text{ and }\lm_{n-1}\not\equiv1\!\!\!\!\pmod p\text{ or }\beta\ne\mathbf0;\\[3pt]
\overline\Rem_{<n-1}\cup\{(n{-}1,\lm_{n-1})\}&\text{ if }\lm_n=\lm_{n-1}-2\\[3pt]
&\hfill\text{ and }\lm_{n-1}\equiv1\!\!\!\!\pmod p\text{ and }\beta=\mathbf0;\\[3pt]
\overline\Rem_{<n-1}&\text{ if }\lm_n=\lm_{n-1}-1\\[3pt]
&\hfill\text{ and }\res_p\lm_{n-1}=\beta\ne\mathbf0;\\[3pt]
\overline\Rem\cup\Rem_S&\text{ if }\lm_n=\lm_{n-1}-1\text{ and }\res_p\lm_{n-1}\ne\beta;\\[3pt]
\overline\Rem_{<n-1}\cup\{(n{-}1,\lm_{n-1}),(n,\lm_n)\}&\text{ if }\lm_n=\lm_{n-1}-1\\[3pt]
&\hfill\text{ and }\lm_{n-1}\equiv1\!\!\!\!\pmod p\text{ and }\beta=\mathbf0;\\[3pt]
\overline\Rem_{<n-1}&\text{ if }\lm_n=\lm_{n-1}-1\\[3pt]
&\hfill\text{ and }\lm_{n-1}\equiv0\!\!\!\!\pmod p\text{ and }\beta=\mathbf0;\\[3pt]
\overline\Rem\cup\Rem_S&\text{ if }\lm_n=\lm_{n-1}\text{ and }\beta\ne\mathbf0;\\[3pt]
\overline\Rem_{<n-1}\cup\{(n,\lm_n)\}&\text{ if }\lm_n=\lm_{n-1}\text{ and }\beta=\mathbf0.
\end{array}}
\right.
$$

$$
\Add=\left\{
{\arraycolsep=0pt
\begin{array}{ll}
\overline\Add\cup \Add_S&\text{ if }\lm_n<\lm_{n-1}-2;\\[3pt]
\overline\Add\cup \Add_S&\text{ if }\lm_n=\lm_{n-1}-2\text{ and }\lm_n\not\equiv-1\!\!\!\!\pmod p\text{ or }\beta\ne\mathbf0;\\[3pt]
\overline\Add\cup\{(n,\lm_n+1)\}&\text{ if }\lm_n=\lm_{n-1}-2\text{ and }\lm_n\equiv-1\!\!\!\!\pmod p\text{ and }\beta=\mathbf0;\\[3pt]
\overline\Add&\text{ if }\lm_n=\lm_{n-1}-1\text{ and }\res_p\lm_{n-1}=\beta\ne\mathbf0;\\[3pt]
\overline\Add\cup \Add_S&\text{ if }\lm_n=\lm_{n-1}-1\text{ and }\res_p\lm_{n-1}\ne\beta;\\[3pt]
\overline\Add&\text{ if }\lm_n=\lm_{n-1}-1\text{ and }\lm_{n-1}\equiv1\!\!\!\!\pmod p\text{ and }\beta=\mathbf0;\\[3pt]
\overline\Add\cup\{(n,\lm_n+1)\}&\text{ if }\lm_n=\lm_{n-1}-1\text{ and }\lm_{n-1}\equiv0\!\!\!\!\pmod p\text{ and }\beta=\mathbf0;\\[3pt]
\overline\Add\cup \Add_S&\text{ if }\lm_n=\lm_{n-1}\text{ and }\beta\ne\mathbf0;\\[3pt]
\overline\Add&\text{ if }\lm_n=\lm_{n-1}\text{ and }\beta=\mathbf0.
\end{array}}
\right.
$$

Finally, let $U$ and $\bar U$ denote the $\beta$-signatures of $\lm$ and $\bar\lm$ respectively.
By the inductive hypothesis, we have $\bigl[\prod r_\beta\bigl(\bar\lm\bigr)\bigr]=\bigl[\bar U\bigr]$.
Hence by Corollary~\ref{corollary:pm:1},
\begin{equation}\label{equation:main:5}
\textstyle
\bigl[\prod r_\beta(\lm)\bigr]=\Bigl[\bigl[\prod r_\beta\bigl(\bar\lm\bigr)\bigr]r_\beta(\lm)_n\Bigr]=\Bigl[\bigl[\bar U\bigr]r_\beta(\lm)_n\Bigr]
=\bigl[\bar Ur_\beta(\lm)_n\bigr].
\end{equation}

Further, we have $\bar U=\bar U'\bar U''$,
where $\bar U''$ consists if all symbols $-_{n-1}$ occurring in $\bar U$ and $\bar U'$ consists of symbols $-_k$ coming from the nodes of $\overline\Rem_{<n-1}$ and symbols $+_l$ coming from the nodes of $\overline\Add$. Now we consider several cases.

{\it Case~1:} $\lm_m<\lm_{m-1}{-}2$. We have $U=\bar Ur_\beta(\lm)_n$. Therefore by~(\ref{equation:main:5}),
we get $\bigl[\prod r_\beta(\lm)\bigr]=[U]$, as required.

{\it Case~2:} $\lm_n=\lm_{n-1}-2$ and $\lm_{n-1}\not\equiv1\!\pmod p$ or $\beta\ne\mathbf0$.
We again have $U=\bar Ur_\beta(\lm)_n$. So this case is similar to Case~1.

{\it Case~3: $\lm_n=\lm_{n-1}-2$ and $\lm_{n-1}\equiv1\!\pmod p$ and $\beta=\mathbf0$}.
We have $\bar U''=-_{n-1}-_{n-1}$,
$r_\mathbf 0(\lm)_n=+_n+_n$ and $U=\bar U'-_{n-1}+_n$.
Therefore by~(\ref{equation:main:5}), we get
\vspace{-1mm}
$$
\textstyle
\Bigl[\prod r_\beta(\lm)\Bigr]=\bigl[\bar U'\bar U''r_\beta(\lm)_n\bigr]=\bigl[\bar U'-_{n-1}-_{n-1}+_n+_n\bigr]
=\bigl[\bar U'-_{n-1}+_n\bigr]=[U].
$$

\vspace{-1mm}
{\it Case~4: $\lm_n=\lm_{n-1}-1$ and $\res_p\lm_{n-1}=\beta\ne\mathbf0$}.
In this case, $\bar U''=-_{n-1}$, $r_\beta(\lm)_n=+_n$ and $\bar U'=U$.
By~(\ref{equation:main:5}), we get
\vspace{-1mm}
$$
\textstyle
\Bigl[\prod r_\beta(\lm)\Bigr]=\bigl[\bar U'\bar U''r_\beta(\lm)_n\bigr]=\bigl[\bar U'-_{n-1}+_n\bigr]=\bigl[\bar U'\bigr]=[U].
$$

\vspace{-1mm}
{\it Case~5: $\lm_n=\lm_{n-1}-1$ and $\res_p\lm_{n-1}\ne\beta$}.
We have $U=\bar Ur_\beta(\lm)_n$. Therefore, this case is similar to case~1.

{\it Case~6: $\lm_n=\lm_{n-1}-1$ and $\lm_{n-1}\equiv1\!\pmod p$ and $\beta=\mathbf0$}.
In this case, $\bar U''=-_{n-1}-_{n-1}$, $r_\beta(\lm)_n=+_n-_n$ and $U=\bar U'-_{n-1}-_n$.
By~(\ref{equation:main:5}), we get
\vspace{-1mm}
$$
\textstyle
\Bigl[\prod r_\beta(\lm)\Bigr]=\bigl[\bar U'\bar U''r_\beta(\lm)_n\bigr]=\bigl[\bar U'-_{n-1}-_{n-1}+_n-_n\bigr]=\bigl[\bar U'-_{n-1}-_n\bigr]=[U].
$$

\vspace{-1mm}
{\it Case~7: $\lm_n=\lm_{n-1}-1$ and $\lm_{n-1}\equiv0\!\pmod p$ and $\beta=\mathbf0$}.
In this case, $\bar U''=-_{n-1}$, $r_\beta(\lm)_n=+_n+_n$ and $U=\bar U'+_{n-1}$.
By~(\ref{equation:main:5}), we get
\vspace{-1mm}
$$
\textstyle
\Bigl[\prod r_\beta(\lm)\Bigr]=\bigl[\bar U'\bar U''r_\beta(\lm)_n\bigr]
=\bigl[\bar U'-_{n-1}+_n+_n\bigr]=\bigl[\bar U'+_n\bigr]=[U].
$$

\vspace{-1mm}

{\it Case~8: $\lm_n=\lm_{n-1}$ and $\beta\ne\mathbf0$}.  We have $U=\bar Ur_\beta(\lm)_n$.
Therefore, this case is similar to case~1.

{\it Case~9: $\lm_n=\lm_{n-1}$ and $\beta=\mathbf0$}.
In this case, $\bar U''=-_{n-1}$, $r_\beta(\lm)_n=+_n-_n$ and $U=\bar U'-_{n}$.
By~(\ref{equation:main:5}), we get
$$
\textstyle
\Bigl[\prod r_\beta(\lm)\Bigr]=\bigl[\bar U'\bar U''r_\beta(\lm)_n\bigr]
=\bigl[\bar U'-_{n-1}+_n-_n\bigr]=\bigl[\bar U'-_n\bigr]=[U],
$$
\vspace{-1mm}
as required.
\end{proof}

\chapter{Constructing $U(n-1)$-primitive vectors}\label{ConstructingUmathbbF(n-1)-primitive vectors}

In this chapter, we consider constructions of $U(n-1)$-primitive vectors of weights $\lm-\alpha(i,n)$
in the irreducible $U(n)$-module $L(\lm)$. There will be six such constructions: the first three
(Theorems~\ref{theorem:constr:1},~\ref{theorem:constr:3} and~\ref{theorem:constr:2})
produce $U(n-1)$-primitive vectors in $L(\lm)$ from a nonzero highest weight vector of $L(\lm)$;
the other three (see Theorems~\ref{theorem:socle:0.5},~\ref{theorem:socle:1} and~\ref{theorem:socle:2})
allow us to `extend' a nonzero $U(n-1)$-primitive vector $v\in L(\lm)$ of weight $\lm-\alpha(i,n)$ to
a nonzero $U(n{-}1)$-primitive vector  $w\in U(i)v$ of weight $\lm-\alpha(h,n)$, where $1\le h<i<n$ and $\res_p\lm_h=\res_p\lm_i$.

In what follows, we follow our usual agreement and denote again by $X$ the element $X\otimes 1\in U(n)=U_\Z(n)\otimes\mathbb F$ for $X\in U_\Z(n)$.
Thus we have various
$\ll x\rr, P_{i,j}^{\delta,\epsilon}(M)\in U^0(n)$ and $S^\epsilon_{i,j}(M)\in U(n)$.

\section{Construction: case $\bigl[\prod_{i<k\le n} r_\beta(\lm)_k\bigr]=-^m$.}\label{method:1}
This construction uses only signed sets containing even elements.

\begin{lemma}\label{lemma:constr:1}
Let $1\le i<j$ and $v$ be a primitive vector of weight $\lm\in X(j)$ in an arbitrary $U(j)$-supermodule. Pick $\epsilon\in\{\0,\1\}$ and $M\subset(i..j]$ such that $j\in M$.
%
\makeatletter
\renewcommand{\p@enumii}{}
\makeatother
\renewcommand{\labelenumi}{{\rm \theenumi}}
\renewcommand{\theenumi}{{\rm(\roman{enumi})}}
\begin{enumerate}
\item\label{lemma:constr:1:part:i}
Suppose that $\psi:(i..j]\setminus M\to(i..j]$ is an injection such that
\renewcommand{\labelenumii}{{\rm \theenumii}}
\renewcommand{\theenumii}{{\rm(\alph{enumii})}}
\begin{enumerate}
\item\label{lemma:constr:1:condition:a} $\res_p\lm_t=\res_p\bigl(\lm_{\psi(t)}+1\bigr)$ for all $t\in(i..j]\setminus M$;\\[-10pt]
\item\label{lemma:constr:1:condition:b} $\psi(t)>t$ for all $t\in(i..j]\setminus M$.
\end{enumerate}
\noindent
Then for any function $\delta:[i..j)\to\{\0,\1\}$ we have
\begin{align*}
E^{\delta_i}_i\cdots
E^{\delta_{j-1}}_{j-1} S_{i,j}^{\,\epsilon}(M)v
=\cond_{\sum\delta=\epsilon}
\prod
_{t\in(i..j]\setminus\psi((i..j]\setminus M)}\bigl(\res_p\lm_i-\res_p(\lm_t+1)\bigr)v;
\end{align*}
\item\label{lemma:constr:1:part:ii} Let $i\le l<j-1$, $\delta\in\{\0,\1\}$, and 
either $l+1\in M$ or $ S^{\,\sigma}_{l+1,j}(M_{(l+1..j]})v=0$ for $\sigma=\0,\1$.
           Then $E_l^\delta S_{i,j}^{\,\epsilon}(M)\,v=0$.
\end{enumerate}
\end{lemma}
\begin{proof}\ref{lemma:constr:1:part:i}
By Lemma~\ref{lemma:rcoeff:1}, we have
\begin{equation}\label{eq:constr:1}
E^{\delta_i}_i\cdots E^{\delta_{j-1}}_{j-1} S_{i,j}^{\,\epsilon}(M)\,v=
 P_{i,j}^{\delta,\epsilon}(M)\,v=
\cond_{\sum\delta=\epsilon}\Bigl\llbracket g^{(1)}_{i,j}\bigl((i..j]{\setminus}M\bigr)\Bigr\rrbracket \,v.
\end{equation}
By Lemma~\ref{lemma:compol:4} with $D=\emptyset$, $R=S=(i..j ]\setminus M$,
$\phi=\psi$ and $l=1$,  
we have
\begin{equation}\label{eq:constr:2}
g^{(1)}_{i,j }\bigl((i..j ]{\setminus}M\bigr)\=\prod\nolimits_{t\in(i..j ]\,\setminus\,\psi((i..j ]\setminus M)}(x_i-y_t)\pmod\I,
\end{equation}
where $\I$ is the ideal of $\R$ generated by the polynomials $x_t-y_{\psi(t)}$ with $t\in(i..j ]\setminus M$.
By the condition~\ref{lemma:constr:1:condition:a}, we get
\begin{align*}
\ll x_t-y_{\psi(t)}\rr v=( H_t( H_t-1)-( H_{\psi(t)}+1) H_{\psi(t)})v
\\
=(\lm_t(\lm_t-1)-(\lm_{\psi(t)}+1)\lm_{\psi(t)})v
=(\res_p\lm_t-\res_p(\lm_{\psi(t)}+1)) v=0.
\end{align*}
Thus we have proved $\bigl\llbracket\I\,\bigr\rrbracket \,v=0$.
So, by~(\ref{eq:constr:1}) and~(\ref{eq:constr:2}), we have
\begin{align*}
E^{\delta_i}_i\cdots E^{\delta_{j -1}}_{j -1} S_{i,j }^{\,\epsilon}(M)v
=\cond_{\sum\delta=\epsilon}\ll\prod\nolimits_{t\in(i..j ]\setminus\psi((i..j ]\setminus M)}(x_i-y_t)\rr v
\\
=\cond_{\sum\delta=\epsilon}\prod\nolimits_{t\in(i..j ]\setminus\psi((i..j ]\setminus M)}\bigl(\lm_i(\lm_i-1)-(\lm_t+1)\lm_t\bigr)v
\\
=\cond_{\sum\delta=\epsilon}\prod\nolimits_{t\in(i..j ]\,\setminus\,\psi((i..j ]\setminus M)}\bigl(\res_p\lm_i-\res_p(\lm_t+1)\bigr)v.
\end{align*}

\ref{lemma:constr:1:part:ii}
By Lemmas~\ref{lemma:ops:6} and~\ref{lemma:ops:7}, we get $E_l^\delta S_{i,j }^{\,\epsilon}(M)\,v=0$
if $l+1\in M$.
Let $l+1\notin M$. In the case $l=i$, by Lemma~\ref{lemma:ops:6}\ref{lemma:ops:6:ii}, we get
\begin{align*}
E_l^\delta S_{i,j }^{\,\epsilon}(M)\,v=(-1)^{\1+\delta(\epsilon+\1)} S_{i+1,j }^{\,\epsilon+\delta}(M)\,v
=(-1)^{\1+\delta(\epsilon+\1)}S_{l+1,j }^{\,\epsilon+\delta}(M_{(l+1..j ]})\,v=0.
\end{align*}
In the case $i<l<j -1$, by Lemma~\ref{lemma:ops:7}, we have $E_l^\delta S_{i,j }^{\,\epsilon}(M)v=0$ if $l\notin M$, and
\begin{align*}
E_l^\delta S_{i,j }^{\,\epsilon}(M)\,v
=\sum_{\gamma+\sigma=\epsilon+\delta}(-1)^{\1+(\delta+\gamma)(\epsilon+\1)} S^{\,\gamma}_{i,l}\bigl(M_{(i..l]}\bigr) S^{\,\sigma}_{l+1,j }\bigl(M_{(l+1..j ]}\bigr)v=0
\end{align*}
if $l\in M$.
\end{proof}

\begin{lemma}\label{lemma:constr:2} Let $\lm\in X(n)$, $v^+_\lm\in L(\lm)^\lm$, $1\le i<j\le n$, $\epsilon\in\{\0,\1\}$ and
$(i..j]\supset M\ni j$. 
Let $\psi:[i..j]\setminus M\to[i..j]$ be an injection such that
{\renewcommand{\labelenumi}{{\rm \theenumi}}
\renewcommand{\theenumi}{{\rm(\alph{enumi})}}
\begin{enumerate}
\item\label{lemma:constr:2:condition:a} $\res_p\lm_t=\res_p\bigl(\lm_{\psi(t)}+1\bigr)$ for all $t\in[i..j]\setminus M${\rm;}\\[-10pt]
\item\label{lemma:constr:2:condition:b} $\psi(t)>t$ for all $t\in[i..j]\setminus M$.
\end{enumerate}}
\noindent
Then we have $ S_{i,j}^{\epsilon}(M)\,v^+_\lm=0$.
\end{lemma}
\begin{proof}
By Proposition~\ref{proposition:intro:5}  and weight considerations, it suffices to prove that $E_l^\delta S_{i,j}^{\,\epsilon}(M)v^+_\lm=0$ for all $l\in[i..j-1)$, $\de=\0,\1$,  and
\begin{equation}\label{eq:constr:5}
E^{\delta_i}_i\cdots E^{\delta_{j-1}}_{j-1}\, S_{i,j}^{\,\epsilon}(M)\,v^+_\lm=0
\end{equation}
for all $\delta:[i..j)\to\{\0,\1\}$.
We apply induction on $j-i$, the inductive base $j-i=1$ coming from Lemma~\ref{lemma:constr:1}\ref{lemma:constr:1:part:i}. 
Let $j-i>1$.
By Lemma~\ref{lemma:constr:1}\ref{lemma:constr:1:part:ii}, in proving that $E_l^\delta S_{i,j}^{\,\epsilon}(M)v^+_\lm=0$, we may assume that $l+1\notin M$
and prove under this assumption that $ S^{\,\sigma}_{l+1,j}\bigl(M_{(l+1..j]}\bigr)v^+_\lm=0$ for any $\sigma=\0,\1$. Now,  $l+1\notin M$ implies
$[l{+}1..j]\setminus M_{(l+1..j]}=[l{+}1..j]\setminus M$. So we can consider the restriction
$\psi'=\psi|_{[l{+}1..j]\setminus M_{(l+1..j]}}$, which obviously is an injection of
$[l{+}1..j]\setminus M_{(l+1..j]}$ into $[l{+}1..j]$ and satisfies the following conditions:
{\renewcommand{\labelenumi}{{\rm \theenumi}}
\renewcommand{\theenumi}{{\rm(\alph{enumi}$^\prime$)}}
\begin{enumerate}
\item\label{lemma:constr:2:condition:a'} $\res_p\lm_t=\res_p\bigl(\lm_{\psi'(t)}+1\bigr)$ for any $t\in[l{+}1..j]\setminus M_{(l+1..j]}${\rm;}\\[-10pt]
\item\label{lemma:constr:2:condition:b'} $\psi'(t)>t$ for any $t\in[l{+}1..j]\setminus M_{(l+1..j]}$,
\end{enumerate}}
\noindent
similar to~\ref{lemma:constr:2:condition:a} and~\ref{lemma:constr:2:condition:b}.
So $ S^{\,\sigma}_{l+1,j}\bigl(M_{(l+1..j]}\bigr)\,v^+_\lm=0$ by the inductive hypothesis.

Finally, by Lemma~\ref{lemma:constr:1}\ref{lemma:constr:1:part:i}, we have
\begin{align*}
E^{\delta_i}_i\cdots E^{\delta_{j-1}}_{j-1} S_{i,j}^{\,\epsilon}(M)v^+_\lm
=\cond_{\sum\delta=\epsilon}\prod\nolimits_{t\in(i..j]\,\setminus\,\psi((i..j]\setminus M)}(\res_p\lm_i-\res_p(\lm_t+1))v^+_\lm.
\end{align*}
We have $\psi(i)\in(i..j]\setminus\psi((i..j]{\setminus}M)$, since $\psi$ is an injection.
By condition~\ref{lemma:constr:2:condition:a}, we get $\res_p\lm_i=\res_p(\lm_{\psi(i)}+1)$,
whence~(\ref{eq:constr:5}) follows.
\end{proof}

\begin{theorem}\label{theorem:constr:1}
Let $\lm\in X(n)$ and $1\le i<n$. Set $\beta:=\res_p\lm_i$, and assume that
$\bigl[\prod_{i<k\le n} r_\beta(\lm)_k\bigr]=-^m$ for some $m\ge 0$.
Then there exist nonzero homogeneous $U(n{-}1)$-primitive vectors $v,v'\in L(\lm)$
of weight $\lm-\alpha(i,n)$ such that 
$E_{j,n}^\de v=E_{j,n}^{\de+\1}v'$ for all $1\leq j<n$ and $\de\in\{0,1\}$.
\end{theorem}
\begin{proof} By Lemmas~\ref{lemma:pm:1} and~\ref{lemma:pm:3}, there exists a flow $\Gamma$ on $(i..n]$,
fully coherent with $r_\beta(\lm)|_{(i..n]}$. Let $S$ denote the set of all sources of edges of $\Gamma$. Then for all
$t\in S$ there exists a unique integer $\psi(t)\in(i..n]$ such that $(t,\psi(t))\in\Gamma$. Moreover, $\psi$ is an injection
of $S$ into $(i..n]$, and  
{\renewcommand{\labelenumi}{{\rm \theenumi}}
\renewcommand{\theenumi}{{\rm(\alph{enumi})}}
\begin{enumerate}
\item\label{theorem:constr:1:property:a} $\res_p\lm_t=\res_p\bigl(\lm_{\psi(t)}+1\bigr)=\beta$ for any $t\in S$;\\[-10pt]
\item\label{theorem:constr:1:property:b} $\psi(t)>t$ for any $t\in S$.
\end{enumerate}
}

Set $M:=(i..n]\setminus S$. Note that $n\in M$, as otherwise $n\in S$, and $n<\psi(n)\in(i..n]$ is a contradiction.
Choose any nonzero homogeneous $v^+_\lm\in L(\lm)$ and $\epsilon\in\{\0,\1\}$. Set
$v:=S_{i,n}^{\,\epsilon}(M)\,v^+_\lm$ and $v':=S_{i,n}^{\,\epsilon+\1}(M)\,v^+_\lm$
We claim that $(v,v')$ is the required pair.

Let us prove that $v$ and $v'$ are $U(n-1)$-primitive. We consider only the vector $v$, the argument for $v'$
being similar. It suffices to prove that
$E_l^\delta S_{i,n}^{\,\epsilon}(M)\,v^+_\lm=0$ for $l=i,\ldots,n-2$ and $\delta=\0,\1$.
By Lemma~\ref{lemma:constr:1}\ref{lemma:constr:1:part:ii},
we may assume that $l+1\notin M$ and prove that
$ S^{\,\sigma}_{l+1,j}\bigl(M_{(l+1..j]}\bigr)\,v^+_\lm=0$ for $\sigma=\0,\1$.
But this equality follows from Lemma~\ref{lemma:constr:2} applied to the injection
$\psi|_{[l+1..j]\setminus M_{(l+1..j]}}$. Note that $l+1\notin M$ implies
$[l{+}1..j]\setminus M_{(l{+}1..j]}=[l{+}1..j]\setminus M=[l{+}1..j]\cap S$. Therefore the restriction $\psi|_{[l+1..j]\setminus M_{(l+1..j]}}$ is well defined.

Now by Lemma~\ref{lemma:constr:1}\ref{lemma:constr:1:part:i}, we have (using the $U(n-1)$-primitivity of $v$ and $v'$ already established)
$$
 E_{i,n}^\delta v=
 E^\delta_iE_{i+1}\cdots E_{n-1} S_{i,n}^{\,\epsilon}(M)\,v^+_\lm=\cond_{\delta=\epsilon}\prod\nolimits_{t\in(i..n]\,\setminus\,\psi(S)}\bigl(\beta-\res_p(\lm_t+1)\bigr)v^+_\lm
$$
and similarly
$$
E_{i,n}^{\delta+1}v'=\cond_{\delta=\epsilon}\prod\nolimits_{t\in(i..n]\,\setminus\,\psi(S)}\bigl(\beta-\res_p(\lm_t+1)\bigr)v^+_\lm
$$
for any $\delta\in\{\0,\1\}$.
In particular, $E_{i,n}^\delta v=E_{i,n}^{\delta+1}v'$.

Recall that $\Gamma$ is fully coherent with $r_\beta(\lm)|_{(i..n]}$.
So, by Definition~\ref{definition:pm:2}\ref{definition:pm:2:6} and
Proposition~\ref{proposition:rbeta}\ref{proposition:rbeta:part:2},
every $t\in(i..n]$ with $\res_p(\lm_t+1)=\beta$ is the target of an edge of $\Gamma$,
i.e. it belongs to $\psi(S)$. So
$$
\prod\nolimits_{t\in(i..n]\,\setminus\,\psi(S)}\bigl(\beta-\res_p(\lm_t+1)\bigr)\ne\mathbf 0.
$$
Hence $E_{i,n}^\epsilon v\ne0$ and $E_{i,n}^{\epsilon+\1}v'\ne0$, whence $v\ne0$ and $v'\ne0$.

We now complete the proof of the equality $E_{j,n}^\de v=E_{j,n}^{\de+\1}v'$ for all $1\leq j<n$ and $\de\in\{0,1\}$.
If $j<i$ then $E_{j,n}^\delta v=0=E_{j,n}^{\delta+\1}v'$ by weight consideration.
The case $j=i$ is done above.
Finally, suppose that $j>i$.
Consider the vector $w:=E_{j,n}^\delta v-E_{j,n}^{\delta+\1}v'$.
It has weight $\lambda-\alpha(i,j)$ and is $U(j-1)$-primitive.
Hence it equals zero if and only if $E_{i,j}^\epsilon w=0$
for any $\epsilon\in\{\0,\1\}$. We have
\begin{align*}
E_{i,j}^\epsilon w\!=\!E_{i,j}^\epsilon E_{j,n}^\delta v-E_{i,j}^\epsilon E_{j,n}^{\delta+\1}v'\!
=\![E_{i,j}^\epsilon,E_{j,n}^\delta]v-[E_{i,j}^\epsilon,E_{j,n}^{\delta+\1}]v'\!
=\!E_{i,n}^{\epsilon+\delta}v-E_{i,n}^{\epsilon+\delta+\1}v',
\end{align*}
which is zero by the case $j=i$, which we have already considered.
\end{proof}

\section[Construction: case 2]{Construction: case $\bigl[\prod_{i<k<n} r_\beta(\lm)_k\bigr]=-^m$ and $\lm_i$, $\lm_n$ are not both divisible by $p$}
This construction is similar to the previous one, but uses operators $ S_{i,j}^{\,\epsilon}(M)$, where the signed set $M$ contains exactly one odd element $\bar\jmath$.

\begin{lemma}\label{lemma:constr:5} Let $1\le i<j$ and $v$ be a primitive vector of weight $\lm\in X(j)$ in an arbitrary $U(j)$-supermodule. Pick $\epsilon\in\{\0,\1\}$ and
a signed $(i..j]$-set $M\ni \bar\jmath $ all of whose elements are even except $\bar\jmath$.
\renewcommand{\labelenumi}{{\rm \theenumi}}
\renewcommand{\theenumi}{{\rm(\roman{enumi})}}
\makeatletter
\renewcommand{\p@enumii}{}
\makeatother
\begin{enumerate}
\item\label{lemma:constr:5:part:i}
Suppose that $\psi:(i..j)\setminus M\to(i..j)$ is an injection such that
\renewcommand{\labelenumii}{{\rm \theenumii}}
\renewcommand{\theenumii}{{\rm(\alph{enumii})}}
\begin{enumerate}
\item\label{lemma:constr:5:condition:a} $\res_p\lm_t=\res_p\bigl(\lm_{\psi(t)}+1\bigr)$ for any $t\in(i..j)\setminus M${\rm;}\\[-10pt]
\item\label{lemma:constr:5:condition:b} $\psi(t)>t$ for any $t\in(i..j)\setminus M$.
\end{enumerate}
\noindent
Then for any function $\delta:[i..j)\to\{\0,\1\}$, we have
\begin{align*}
 E^{\delta_i}_i\cdots E^{\delta_{j-1}}_{j-1} S_{i,j}^{\,\epsilon}(M)v=\\
\prod\nolimits_{t\in(i..j)\,\setminus\,\psi((i..j)\setminus M)}\bigl(\res_p\lm_i{-}\res_p(\lm_t{+}1)\bigr)( H_i^{\epsilon+\sum\delta}{-}(-1)^{\epsilon\sum\delta} H_j^{\epsilon+\sum\delta})v.
\end{align*}
\item\label{lemma:constr:5:part:ii} Let
$i\le l<j-1$, $\delta\in\{\0,\1\}$, and  either $l+1\in M$ or
$ S^{\,\sigma}_{l+1,j}(M_{(l+1..j]})v=0$ for $\sigma=\0,\1$.
Then $ E_l^\delta S_{i,j}^{\,\epsilon}(M)v=0$.
\end{enumerate}
\end{lemma}
\begin{proof}\ref{lemma:constr:5:part:i}
By Lemma~\ref{lemma:rcoeff:2}, we have
\begin{equation}\label{eq:constr:5.5}
\begin{array}{l}
 E^{\delta_i}_i\cdots E^{\delta_{j-1}}_{j-1} S_{i,j}^{\,\epsilon}(M)\,v
=P_{i,j}^{\delta,\epsilon}(M)\,v\\
\displaystyle
=\sum_{k\in X(i,j,M)}(-1)^{\cscript_{k>i}+(\1+\epsilon+\sum\delta)\sum\delta|_{[i..k)}}\!\ll g_{i,k,j,j}^{(2)}\bigl((i..j]{\setminus}M\bigr)\rr  H_k^{\epsilon+\sum\delta}v
\end{array}
\end{equation}
We claim that the only non-zero contribution to the sum in the right hand side can come from $k=i$ and $k=j$. Indeed, let $k\in(i..j)\setminus M$ and apply Lemma~\ref{lemma:compol:4}
with $D=\{k\}$, $S=(i..j]\setminus M$, $R=[k..j]{\setminus}M$, $l=l^{(2)}_{i,k,j,j}$ and
$$
\phi(t) =\left\{
{\arraycolsep=0pt
\begin{array}{cl}
\psi(t)&\text{ if }t\in R\setminus\{j\};\\
      j&\text{ if }t=j.
\end{array}}
\right.
$$
Clearly, we have $\phi(t)\ge t+l(t)$ for any $t\in R$, since $\phi(t)=\psi(t)>t$ if $t\in R\setminus\{j\}$
and $l(j)=l^{(2)}_{i,k,j,j}(j)=0$.
The condition $l|_{R\cap D}=0$ in our case takes the form $l(k)=0$, which is satisfied by the definition
of $l^{(2)}_{i,k,q,j}$.
Consider the ideal $\I$ of $\R$ generated by the following polynomials:
\renewcommand{\labelenumi}{{\rm \theenumi}}
\renewcommand{\theenumi}{{\rm(\arabic{enumi})}}
\begin{enumerate}
\item\label{lemma:constr:5:gen:type:1} $x_{\{k\}_{\phi(k)}^i}-y_{\phi(t)}$, where $t\in R$
     such that $k\in\bigl[t{+}l(t)..\phi(t)\bigr)$;
\item\label{lemma:constr:5:gen:type:2} $x_t-y_{\phi(t)}$, where $t\in R\setminus\{j\}$
     such that $k\notin\bigl[t{+}l(t)..\phi(t)\bigr)$.
\end{enumerate}
Note that $k\in R$ and $k\in\bigl[k{+}l(k)..\phi(k)\bigr)$, as $l(k)=0$ and $\phi(k)=\psi(k)>k$. So by Lemma~\ref{lemma:compol:4} we have $g_{i,k,j,j}^{(2)}\bigl((i..j]{\setminus}M\bigr)\in\I$.
It remains to prove that $\bigl\llbracket\I\,\bigr\rrbracket \,v=0$. Indeed, the only generator of
type~\ref{lemma:constr:5:gen:type:1}
corresponds to $t=k$. For this generator:
$$
\bigl\llbracket x_{\{k\}_{\phi(k)}^i}-y_{\phi(k)}\bigr\rrbracket \,v
=\bigl\llbracket x_k-y_{\psi(k)}\bigr\rrbracket \,v
=(\res_p\lm_k-\res_p\bigl(\lm_{\psi(k)}+1\bigr))\,v=0,
$$
by condition~\ref{lemma:constr:5:condition:a}. Similarly, for a generator of
type~\ref{lemma:constr:5:gen:type:2},
we get by condition~\ref{lemma:constr:5:condition:a}:
$$
\llbracket x_t-y_{\phi(t)}\rrbracket \,v=(\res_p\lm_t-\res_p\bigl(\lm_{\psi(t)}+1\bigr))\,v=0.
$$

Now note that $j-1\notin(i..j)\setminus M$, as otherwise by ~\ref{lemma:constr:5:condition:b}, we would have $j-1<\psi(j-1)\in(i..j)$, which is a contradiction. In other words, $j-1\in M$ or $j-1=i$.
So in~(\ref{eq:constr:5.5}), both values $k=i$ and $k=j$ belong to $X(i,j,M)$.
Thus
\begin{equation}\label{eq:constr:12}
\begin{array}{l}
 E^{\delta_i}_i\cdots E^{\delta_{j-1}}_{j-1} S_{i,j}^{\,\epsilon}(M)\,v
 \\
=\Bigl(\ll g_{i,i,j,j}^{(2)}\bigl((i..j]{\setminus}M\bigr)\rr\!  H_i^{\,\epsilon+\sum\delta}\!-\!(-1)^{\epsilon\sum\delta}\ll g_{i,j,j,j}^{(2)}\bigl((i..j]{\setminus}M\bigr)\rr  H_j^{\,\epsilon+\sum\delta}\Bigr)
v.
\end{array}
\end{equation}

To calculate $\ll g_{i,i,j,j}^{(2)}\bigl((i..j]{\setminus}M\bigr)\rr \,v$, we apply Lemma~\ref{lemma:compol:4}
with $D=\{i\}$, $S=R=(i..j]\setminus M$, $l=l^{(2)}_{i,i,j,j}=0$ and
$$
\phi(t) =\left\{
{\arraycolsep=0pt
\begin{array}{cl}
\psi(t)&\text{ if }t\in R\setminus\{j\};\\
      j&\text{ if }t=j.
\end{array}}
\right.
$$
We have $\phi(t)\ge t+l(t)$ for any $t\in R$, since $\phi(t)=\psi(t)\ge t$ and $l=0$.
The condition $l|_{R\cap D}=0$ is now satisfied, since $R\cap D=\emptyset$.
Consider the ideal $\I$ of $\R$ generated by the polynomials
$x_t-y_{\phi(t)}$, where $t\in R\setminus\{j\}=(i..j)\setminus M$.
By Lemma~\ref{lemma:compol:4}, modulo $\I$ we have
\begin{equation}\label{eq:constr:13}
g_{i,i,j,j}^{(2)}\bigl((i..j]{\setminus}M\bigr)\=\prod_{t\in(i..j]\,\setminus\,\phi((i..j]\setminus M)}(x_i-y_t)
=\prod_{t\in(i..j)\,\setminus\,\psi((i..j)\setminus M)}(x_i-y_t).
\end{equation}
The equality $\bigl\llbracket\I\,\bigr\rrbracket \,v=0$ can be proved in the same way as
we proved the similar equality above.
Also, $\bigl\llbracket x_i-y_t \bigr\rrbracket \,v=\Bigl(\res_p\lm_i-\res_p(\lm_t+1)\Bigr)\,v$.
So by~(\ref{eq:constr:13}), we get
\begin{equation}\label{eq:constr:14}
\begin{array}{l}
\displaystyle
\ll g_{i,i,j,j}^{(2)}\bigl((i..j]{\setminus}M\bigr)\rr v
=\ll \prod\nolimits_{t\in(i..j)\,\setminus\,\psi((i..j)\setminus M)}(x_i-y_t)\rr v
\\
\displaystyle=\prod\nolimits_{t\in(i..j)\,\setminus\,\psi((i..j)\setminus M)}\bigl(\res_p\lm_i{-}\res_p(\lm_t+1)\bigr)v.
\end{array}
\end{equation}

To calculate $\ll g_{i,j,j,j}^{(2)}\bigl((i..j]{\setminus}M\bigr)\rr v$, we apply Lemma~\ref{lemma:compol:4} with $D=\{j\}$, $S=R=(i..j]\setminus M$, $l=l^{(2)}_{i,j,j,j}$, and
$$
\phi(t) =\left\{
{\arraycolsep=0pt
\begin{array}{cl}
\psi(t)&\text{ if }t\in R\setminus\{j\};\\
      j&\text{ if }t=j.
\end{array}}
\right.
$$
Clearly, $\phi(t)\ge t+l(t)$ for any $t\in R$, as $\phi(t)=\psi(t)>t$ if $t\in R\setminus\{j\}$ and $l(j)=0$.
By Lemma~\ref{lemma:compol:4}, modulo the ideal $\I$ defined above, we get
$$
g_{i,j,j,j}^{(2)}\bigl((i..j]{\setminus}M\bigr)
\!\=\prod\nolimits_{t\in(i..j)\,\setminus\,\psi((i..j)\setminus M)}(x_i-y_t)\,v
=\prod\nolimits_{t\in(i..j)\,\setminus\,\psi((i..j)\setminus M)}(x_i-y_t).
$$
Hence
\begin{equation}\label{eq:constr:15}
\begin{array}{l}
\displaystyle
\ll g_{i,j,j,j}^{(2)}\bigl((i..j]{\setminus}M\bigr)\rr \,v
=\ll \prod\nolimits_{t\in(i..j)\,\setminus\,\psi((i..j)\setminus M)}(x_i-y_t)\rr v
\\
\displaystyle=\prod\nolimits_{t\in(i..j)\,\setminus\,\psi((i..j)\setminus M)}\bigl(\res_p\lm_i{-}\res_p(\lm_t{+}1)\bigr) v.
\end{array}
\end{equation}
Substituting~(\ref{eq:constr:14}) and~(\ref{eq:constr:15}) to~(\ref{eq:constr:12}), we obtain the required result.

\ref{lemma:constr:5:part:ii} Note that $\bar l,\odd{l+1}\not\in M$.
By Lemmas~\ref{lemma:ops:6} and~\ref{lemma:ops:7}, we have $ E_l^\delta S_{i,j}^{\,\epsilon}(M)\,v=0$ if $l+1\in M$.
Now let $l+1\notin M$. In the case $l=i$, by Lemma~\ref{lemma:ops:6}\ref{lemma:ops:6:ii},
we get
\begin{equation*}\label{eq:constr:16}
 E_l^\delta S_{i,j}^{\,\epsilon}(M)v=(-1)^{\1+\delta\epsilon} S_{i+1,j}^{\,\epsilon+\delta}(M)v=(-1)^{\1+\delta\epsilon} S_{i+1,j}^{\,\epsilon+\delta}(M_{(i+1..j]})v=0.
\end{equation*}
In the case $i<l<j-1$, by Lemma~\ref{lemma:ops:7}, we get $ E_l^\delta S_{i,j}^{\,\epsilon}(M)\,v=0$ if $l\notin M$ and
\begin{align*}
 E_l^\delta S_{i,j}^{\,\epsilon}(M)v
 =\sum_{\gamma+\sigma=\epsilon+\delta}(-1)^{\1+(\delta+\gamma)\epsilon}S^{\,\gamma}_{i,l}\bigl(M_{(i..l]}\bigr) S^{\,\sigma}_{l+1,j}\bigl(M_{(l+1..j]}\bigr)v=0
\end{align*}
if $l\in M$.
\end{proof}

\begin{lemma}\label{lemma:constr:6} Let $\lm\in X(n)$, $v^+_\lm\in L(\lm)^\lm$, $1\le i<j\le n$, $\epsilon\in\{\0,\1\}$, and
$M\ni \bar\jmath$ be a signed $(i..j]$-set all elements of which are even except $\bar\jmath$.
Suppose that $\psi:[i..j)\setminus M\to[i..j)$ is an injection such that
{\renewcommand{\labelenumi}{{\rm \theenumi}}
\renewcommand{\theenumi}{{\rm(\alph{enumi})}}
\begin{enumerate}
\item\label{lemma:constr:6:condition:a} $\res _p\lm_t=\res_p\bigl(\lm_{\psi(t)}+1\bigr)$ for any $t\in[i..j)\setminus M${\rm;}\\[-10pt]
\item\label{lemma:constr:6:condition:b} $\psi(t)>t$ for any $t\in[i..j)\setminus M$.
\end{enumerate}}
\noindent
Then $ S_{i,j}^{\,\epsilon}(M)\,v^+_\lm=0$.
\end{lemma}
\begin{proof}
By Lemma~\ref{lemma:constr:5}\ref{lemma:constr:5:part:i}, for any $\delta:[i..j)\to\{\0,\1\}$, we have
\begin{align*}
E^{\delta_i}_i\cdots E^{\delta_{j-1}}_{j-1}\, S_{i,j}^{\,\epsilon}(M)\,v^+_\lm
\\
=\prod\nolimits_{t\in(i..j)\,\setminus\,\psi((i..j)\setminus M)}\bigl(\res_p\lm_i{-}\res_p(\lm_t{+}1)\bigr)\big( H_i^{\epsilon+\sum\delta}{-}(-1)^{\epsilon\sum\delta} H_j^{\epsilon+\sum\delta}\big)v^+_\lm=0
\end{align*}
because for the factor corresponding to $t=\psi(i)$ is
$(\res_p\lm_i-\res_p(\lm_{\psi(i)}+1))=0$.

By Proposition~\ref{proposition:intro:5} and weight considerations, it now suffices to prove that $E_l^\delta S_{i,j}^{\,\epsilon}(M)\,v^+_\lm=0$ for all $i\le l<j-1$, $\delta\in\{\0,\1\}$. We apply induction on $j-i$. If $j-i=1$, the result follows from the previous paragraph. Let now
$j-i>1$. By Lemma~\ref{lemma:constr:5}\ref{lemma:constr:5:part:ii}, we may assume that $l+1\notin M$
and prove that $ S^{\,\sigma}_{l+1,j}(M_{(l+1..j]})v^+_\lm=0$ for any $\sigma\in\{\0,\1\}$.
As $l+1\notin M$, we have
$[l{+}1..j)\setminus M_{(l+1..j]}=[l{+}1..j)\setminus M$. So we can consider the restriction
$\psi'=\psi|_{[l{+}1..j)\setminus M_{(l+1..j]}}$, which is an injection of
$[l{+}1..j)\setminus M_{(l+1..j]}$ into $[l{+}1..j)$ such that
{\renewcommand{\labelenumi}{{\rm \theenumi}}
\renewcommand{\theenumi}{{\rm(\alph{enumi}$^\prime$)}}
\begin{enumerate}
\item\label{lemma:constr:6:condition:a'} $\res_p\lm_t=\res_p\bigl(\lm_{\psi'(t)}+1\bigr)$ for all  $t\in[l{+}1..j)\setminus M_{(l+1..j]}${\rm;}\\[-10pt]
\item\label{lemma:constr:6:condition:b'} $\psi'(t)>t$ for all $t\in[l{+}1..j)\setminus M_{(l+1..j]}$
\end{enumerate}}
\noindent
Hence $ S^{\,\sigma}_{l+1,j}(M_{(l+1..j]})v^+_\lm=0$ by the inductive hypothesis.
\end{proof}

\begin{theorem}\label{theorem:constr:3}
Let $\lm\in X(n)$, $1\le i<n$, and  $\beta:=\res_p\lm_i$. Suppose that $\bigl[\prod_{i<k<n} r_\beta(\lm)_k\bigr]=-^m$ for some $m\geq 0$.
If $\lm_i$ and $\lm_n$ are not both divisible by $p$, then there exists a nonzero homogeneous $U(n-1)$-primitive
vector $v\in L(\lm)$ of weight $\lm-\alpha(i,n)$.
\end{theorem}
\begin{proof}
By Lemmas~\ref{lemma:pm:1},~\ref{lemma:pm:3}, there exists a flow $\Gamma$ on $(i..n)$
fully coherent with $r_\beta(\lm)|_{(i..n)}$. Let $S$ be the set of all sources of edges of $\Gamma$.
Thus for any $t\in S$, there exists a unique $\psi(t)\in(i..n)$ such that $(t,\psi(t))\in\Gamma$,  and $\psi:S\to (i..n)$ is an injection.
Set $M:=\bigl((i..n)\setminus S\bigr)\cup\{\bar n\}$. As $\lm_i$ and $\lm_n$
are not both divisible by $p$ and $p\ne2$, we have that either $\lm_i-\lm_n\not\=0\pmod p$ or $\lm_i+\lm_n\not\=0\pmod p$.
So there is $\epsilon\in\{\0,\1\}$ such that $\lm_i{-}(-1)^\epsilon\lm_n\not\=0\pmod p$.
Let $v^+_\lm\in L(\lm)^\lm$ be a homogeneous nonzero vector, and set
$
v:= S_{i,n}^{\,\epsilon}(M)\,v^+_\lm.
$

First, we prove that $v\neq 0$. Note $(i..n)\setminus M=S$.
By Proposition~\ref{proposition:rbeta},
{\renewcommand{\labelenumi}{{\rm \theenumi}}
\renewcommand{\theenumi}{{\rm(\alph{enumi})}}
\begin{enumerate}
\item\label{theorem:constr:3:property:a} $\res_p\lm_t=\res_p\bigl(\lm_{\psi(t)}+1\bigr)=\beta$ for all $t\in(i..n)\setminus M$;\\[-10pt]
\item\label{theorem:constr:3:property:b} $\psi(t)>t$ for all $t\in(i..n)\setminus M$;\\[-10pt]
\end{enumerate}
}
\noindent
By Lemma~\ref{lemma:constr:5}\ref{lemma:constr:5:part:i}, for  $\delta:[i..n)\to\{\0,\1\}$ with $\sum\delta=\epsilon$, we have
\begin{align*}
 E^{\delta_i}_i\cdots E^{\delta_{n-1}}_{n-1}\! S_{i,n}^{\epsilon}(M)v^+_\lm
=\prod\nolimits_{t\in(i..n)\setminus\psi(S)}\!(\res_p\lm_i-\res_p(\lm_t+1))( H_i\!-\!(-1)^\epsilon H_n)v^+_\lm\\
=\prod\nolimits_{t\in(i..n)\setminus\psi(S)}\!(\res_p\lm_i-\res_p(\lm_t+1))(\lm_i-(-1)^\epsilon\lm_n)v^+_\lm\neq 0,
\end{align*}
since, by Definition~\ref{definition:pm:2}\ref{definition:pm:2:6} and
Proposition~\ref{proposition:rbeta}\ref{proposition:rbeta:part:2}, any $t\in(i..n)$ such that $\res_p(\lm_t+1)\=\beta$ is a target of an edge of $\Gamma$, that is, belongs
to $\psi(S)$.

Now, we prove that $v$ is a $U(n-1)$-primitive vector.
We need to check that $ E_l^\delta S_{i,n}^{\,\epsilon}(M)\,v^+_\lm=0$ for all $i\le l<n-1$ and all $\de$.
By Lemma~\ref{lemma:constr:5}\ref{lemma:constr:5:part:ii}, we may assume that $l+1\notin M$
and prove that
$ S^{\sigma}_{l+1,n}(M_{(l+1..n]})v^+_\lm=0$ for $\sigma=\0,\1$.
But the last equality follows from Lemma~\ref{lemma:constr:6} applied to the injection
$\psi|_{[l+1..n)\setminus M_{(l+1..n]}}$, which is well defined because $l+1\notin M$ implies
$[l{+}1..n)\setminus M_{(l{+}1..n]}=[l{+}1..n)\cap S$.
\end{proof}

The following result will be used in Theorem~\ref{theorem:main:2}.

\begin{lemma}\label{lemma:constr:7} Let $\lm\in X(n)$, $v^+_\lm\in L(\lm)^\lm$, $1\le i<j\le n$,
$M\ni\bar\jmath$ be a signed $(i..j]$-set all of whose elements are even except $\bar\jmath$, and  $\epsilon\in\{\0,\1\}$.
Suppose that $\lm_i\=\lm_j\=0\pmod p$ and $\psi:(i..j)\setminus M\to(i..j)$ is an injection such that
{\renewcommand{\labelenumi}{{\rm \theenumi}}
\renewcommand{\theenumi}{{\rm(\alph{enumi})}}
\begin{enumerate}
\item\label{lemma:constr:7:condition:a} $\res_p\lm_t=\res_p\bigl(\lm_{\psi(t)}+1\bigr)$ for all $t\in(i..j)\setminus M${\rm;}\\[-10pt]
\item\label{lemma:constr:7:condition:b} $\psi(t)>t$ for all $t\in(i..j)\setminus M$.
\end{enumerate}}
\noindent
Then $ S_{i,j}^{\,\epsilon}(M)\,v^+_\lm=0$.
\end{lemma}
\begin{proof}
The proof that $E^\de_l S_{i,j}^{\,\epsilon}(M)v^+_\lm=0$ for $l\neq j-1$ is the same as in Lemma~\ref{lemma:constr:6}, except that one uses Lemma~\ref{lemma:constr:6} instead of the inductive hypothesis. Now, for any $\delta:[i..j)\to\{\0,\1\}$, by
 Lemma~\ref{lemma:constr:5}\ref{lemma:constr:5:part:i}, we have
\begin{align*}
E^{\delta_i}_i\cdots E^{\delta_{j-1}}_{j-1} S_{i,j}^{\,\epsilon}(M)\,v^+_\lm\\
=\prod\nolimits_{t\in(i..j)\setminus\psi((i..j)\setminus M)}(\res_p\lm_i{-}\res_p(\lm_t{+}1))( H_i^{\epsilon+\sum\delta}{-}(-1)^{\epsilon\sum\delta} H_j^{\epsilon+\sum\delta})v^+_\lm=0
\end{align*}
since $\lm_i\=\lm_j\=0\pmod p$ and $v^+_\lm$ belongs to the $U^0(n)$-supermodule $L(\lm)^\lm$ isomorphic to $\mathfrak u(\lm)$, see Proposition~\ref{proposition:intro:4}.
\end{proof}

\section[Case 3]{Construction: case $\lm_i\=1\pmod p$ and $\bigl[\prod_{i<k\le n}r_\mathbf0(\lm)_k\bigr]=+-^m$}
This construction is a generalization of the previous one: we use operators $ S_{i,j}^{\,\epsilon}(M)$,
where the signed set $M$ contains exactly one odd element $\bar q$. 

\begin{lemma}\label{lemma:constr:3} Let $\lm\in X(n)$, $v^+_\lm\in L(\lm)^\lm$, $1\le i<j\le n$, $\epsilon\in\{\0,\1\}$, and
$M$ be a signed $(i..j]$-set containing either $\bar\jmath$ or $j$, all of whose elements are even except for some $\bar q$.
\renewcommand{\labelenumi}{{\rm \theenumi}}
\renewcommand{\theenumi}{{\rm(\roman{enumi})}}
\makeatletter
\renewcommand{\p@enumii}{}
\makeatother
\begin{enumerate}
\item\label{lemma:constr:3:part:i}
Suppose that $\psi:(i..j]\setminus M\to(i..j]$ is an injection such that
{\renewcommand{\labelenumii}{{\rm \theenumii}}
\renewcommand{\theenumii}{{\rm(\alph{enumii})}}
\begin{enumerate}
\item\label{lemma:constr:3:condition:a} $\res_p\lm_t=\res_p\bigl(\lm_{\psi(t)}+1\bigr)$ for all $t\in(i..j]\setminus M${\rm;}\\[-10pt]
\item\label{lemma:constr:3:condition:b} $\psi(t)\ge t$ for all $t\in(i..j]\setminus M${\rm;}\\[-10pt]
\item\label{lemma:constr:3:condition:c} if $\psi(t)=t$ then $t\le q$.
\end{enumerate}}
\noindent
Then for any $\delta:[i..j)\to\{\0,\1\}$, we have
\begin{align*}
 E^{\delta_i}_i{\cdots} E^{\delta_{j-1}}_{j-1} S_{i,j}^{\epsilon}(M)v^+_\lm=\prod\nolimits_{t\in(i..j]\setminus\psi((i..j]\setminus M)}(\res_p\lm_i-\res_p(\lm_t{+}1)) H_i^{\epsilon+\sum\delta}v^+_\lm.
\end{align*}
\item\label{lemma:constr:3:part:ii} Let $i\le l<j-1$ and $\delta\in\{\0,\1\}$. Suppose that either $l+1\in M$
     or one of the following conditions holds\,{\rm:}
            \begin{itemize}
            \item $l=q-1$ and $\lm_q\=0\pmod p$;
            \item $l\ne q-1$ and $ S^{\,\sigma}_{l+1,j}\bigl(M_{(l+1..j]}\bigr)\,v^+_\lm=0$ for all $\sigma\in\{\0,\1\}$.
            \end{itemize}
            Then $ E_l^\delta S_{i,j}^{\,\epsilon}(M)\,v^+_\lm=0$.
\end{enumerate}
\end{lemma}
\begin{proof}\ref{lemma:constr:3:part:i}
By Lemma~\ref{lemma:rcoeff:2}, we get
\begin{equation}\label{eq:constr:6}
\begin{array}{l}
 E^{\delta_i}_i\cdots E^{\delta_{j-1}}_{j-1} S_{i,j}^{\,\epsilon}(M)\,v^+_\lm
= P_{i,j}^{\delta,\epsilon}(M)\,v^+_\lm\\
\displaystyle
=\sum_{k\in X(i,q,M)}(-1)^{\cscript_{k>i}\cdot\1+\(\1+\epsilon+\sum\delta\)\sum\delta|_{[i..k)}}\!\ll g_{i,k,q,j}^{(2)}\bigl((i..j]{\setminus}M\bigr)\rr  H_k^{\epsilon+\sum\delta}\,v^+_\lm,
\end{array}
\end{equation}
where $X(i,q,M)=\{k\in[i..q]\setminus M\mid k{-}1\in M\cup\{i{-}1,i\}\}$.

We claim that the only non-zero contribution to the sum above can come from the summation index $k=i\in X(i,q,M)$. Indeed, let $k\in(i..q]\setminus M$. If $\psi(k)=k$ then $\res_p\lm_k=\res_p(\lm_k{+}1)$
by condition~\ref{lemma:constr:3:condition:a}. Hence $\lm_k(\lm_k-1)\=(\lm_k+1)\lm_k$ and $\lm_k\=0\pmod p$, since $p\ne2$.
Therefore, $ H_k^{\epsilon+\sum\delta}\,v_\lm^+=0$, since $v_\lm^+$ belongs to the $U^0(n)$-supermodule $L(\lm)^\lm$
isomorphic to $\mathfrak u(\lm)$ (see Proposition~\ref{proposition:intro:4}).
Now suppose that $\psi(k)>k$. To calculate $\ll g_{i,k,q,j}^{(2)}((i..j]{\setminus}M)\rr v^+_\lm$, we apply
Lemma~\ref{lemma:compol:4} for $D=\{k\}$, $S=(i..j]\setminus M$, $R=[k..j]\setminus M$, $\phi=\psi|_R$, and $l=l^{(2)}_{i,k,q,j}$.
Let us check that $\phi(t)\ge t+l(t)$ for any $t\in R$. Indeed, if $t\in R$ and $t>q$, then $\phi(t)=\psi(t)>t$
by conditions~\ref{lemma:constr:3:condition:b} and~\ref{lemma:constr:3:condition:c}.
On the other hand, if $t\in R$ and $t\le q$,
then $k\le t\le q$ and by the definition of $l^{(2)}_{i,k,q,j}$, we get
$l(t)=l^{(2)}_{i,k,q,j}(t)=0$. Hence again $\phi(t)=\psi(t)\ge t=t+l(t)$ by condition~\ref{lemma:constr:3:condition:b}. The condition $l|_{R\cap D}=0$ in our case takes the form $l(k)=0$. This equality obviously holds by the definition of $l^{(2)}_{i,k,q,j}$.
Now, consider the ideal $\I$ of $\R$ generated by the polynomials:
\renewcommand{\labelenumi}{{\rm \theenumi}}
\renewcommand{\theenumi}{{\rm(\arabic{enumi})}}
\begin{enumerate}
\item\label{lemma:constr:3:gen:type:1} $x_{\{k\}_{\phi(t)}^i}\!\!\!-y_{\phi(t)}$, for $t\in R$ such that $k\in\bigl[t{+}l(t)..\phi(t)\bigr)$;
\item\label{lemma:constr:3:gen:type:2} $x_t-y_{\phi(t)}$, for $t\in R\setminus\{j\}$ such that $k\notin\bigl[t{+}l(t)..\phi(t)\bigr)$.
\end{enumerate}
Note that $k\in R$ and $k\in\bigl[k{+}l(k)..\phi(k)\bigr)$, since $l(k)=0$ and $\phi(k)=\psi(k)>k$.
Therefore, by Lemma~\ref{lemma:compol:4}, we get $g_{i,k,q,j}^{(2)}\bigl((i..j]{\setminus}M\bigr)\in\I$.
It remains to prove that $\bigl\llbracket\I\,\bigr\rrbracket \,v^+_\lm=0$. Well, the only generator of type~\ref{lemma:constr:3:gen:type:1}
corresponds to $t=k$, and
$$
\llbracket x_{\{k\}_{\phi(k)}^i}-y_{\phi(k)}\rrbracket v^+_\lm
=\llbracket x_k-y_{\psi(k)}\rrbracket v^+_\lm
=(\res_p\lm_k-\res_p(\lm_{\psi(k)}+1))v^+_\lm=0
$$
by condition~\ref{lemma:constr:3:condition:a}. Similarly, for a generator of
type~\ref{lemma:constr:3:gen:type:2},
we have
$$
\llbracket x_t-y_{\phi(t)}\rrbracket v^+_\lm=(\res_p\lm_t-\res_p(\lm_{\psi(t)}+1))v^+_\lm=0,
$$
which again follows from~\ref{lemma:constr:3:condition:a}.
Thus, we have proved that all summand corresponding to $k\neq i$ in~(\ref{eq:constr:6}) are zero, and so
\begin{equation}\label{eq:constr:7}
 E^{\delta_i}_i\cdots E^{\delta_{j-1}}_{j-1} S_{i,j}^{\,\epsilon}(M)\,v^+_\lm
 =
\ll g_{i,i,q,j}^{(2)}\bigl((i..j]{\setminus}M\bigr)\rr  H_i^{\epsilon+\sum\delta}\\
\,v^+_\lm.
\end{equation}

To calculate $\ll g_{i,i,q,j}^{(2)}((i..j]{\setminus}M)\rr v^+_\lm$,
we apply Lemma~\ref{lemma:compol:4} for $D=\{i\}$, $S=R=(i..j]\setminus M$, $\phi=\psi$ and $l=l^{(2)}_{i,i,q,j}$. As above it is easy to check that
$\phi(t)\ge t+l(t)$ for any $t\in (i..j]\setminus M$. Moreover, the condition $l|_{R\cap D}=0$ holds in this case, since $R\cap D=\emptyset$. Consider the ideal $\I$ of $\R$ generated by the polynomials $x_t-y_{\phi(t)}$ for $t\in R\setminus\{j\}$.
By Lemma~\ref{lemma:compol:4}, we get
\begin{equation}\label{eq:constr:8}
g_{i,i,q,j}^{(2)}\bigl((i..j]{\setminus}M\bigr)\=\prod\nolimits_{t\in(i..j]\setminus\psi((i..j]\setminus M)}(x_i-y_t)\pmod\I.
\end{equation}
Using~\ref{lemma:constr:3:condition:a}, one can easily see that
$\llbracket\I\,\rrbracket v^+_\lm=0$.
Thus by~(\ref{eq:constr:7}) and~(\ref{eq:constr:8}), we get
\begin{align*}
 E^{\delta_i}_i\cdots E^{\delta_{j-1}}_{j-1} S_{i,j}^{\epsilon}(M) v^+_\lm
=\prod\nolimits_{t\in(i..j]\setminus\psi((i..j]\setminus M)}(\res_p\lm_i-\res_p(\lm_t+1)) H_i^{\epsilon+\sum\delta} v^+_\lm.
\end{align*}

\ref{lemma:constr:3:part:ii} In view of Lemmas~\ref{lemma:ops:6} and~\ref{lemma:ops:7},
we
may assume that $l+1\notin M$.

{\it Case~1: $l=q-1$.} Then $\lm_q\=0\pmod p$ and so $ H_q\,v^+_\lm=\bar H_q\,v^+_\lm=0$.
If $l=i$, by Lemma~\ref{lemma:ops:6}\ref{lemma:ops:6:iii}, we have
$$
 E^\delta_l S_{i,j}^{\,\epsilon}(M)\,v^+_\lm=\sum_{\gamma+\sigma=\epsilon+\delta}(-1)^{\1+\delta\epsilon} S^{\,\gamma}_{i+1,j}\(M{\setminus}\left\{\overline{i{+}1}\right\}\) H^{\,\sigma}_q\,v^+_\lm=0.
$$
If $i<l<j-1$, by Lemma~\ref{lemma:ops:7}, we have $ E^\delta_l S^\epsilon_{i,j}(M)\,v^+_\lm=0$ if
$\bar l,l\notin M$ and
\begin{align*}
 E^\delta_l S^{\,\epsilon}_{i,j}(M)v^+_\lm
 =\sum_{\gamma+\sigma+\tau=\epsilon+\delta}(-1)^{\1+(\delta+\gamma)\epsilon}
  S_{i,l}^{\,\gamma}\(M_{(i..l]}\) S^{\,\sigma}_{l+1,j}\(M_{(l+1..j]}\) H_q^\tau\,v^+_\lm=0
\end{align*}
otherwise.

{\it Case~2: $l\ne q-1$.} Then $l+1\ne q$ and so $\odd{l{+}1},l{+}1\notin M$.
If $l=i$, by Lemma~\ref{lemma:ops:6}\ref{lemma:ops:6:ii}, we have
\begin{align*}
 E_l^\delta S_{i,j}^{\,\epsilon}(M)\,v^+_\lm=(-1)^{\1+\delta\epsilon} S_{i+1,j}^{\,\epsilon+\delta}(M)\,v^+_\lm
=(-1)^{\1+\delta\epsilon} S_{i+1,j}^{\,\epsilon+\delta}(M_{(i+1..j]})\,v^+_\lm=0
\end{align*}
by assumption.
If $i<l<j-1$, by Lemma~\ref{lemma:ops:7}, we have $ E^\delta_l S^{\,\epsilon}_{i,j}(M)\,v^+_\lm=0$ if $\bar l,l\notin M$,
and
\begin{align*}
 E^\delta_l S^\epsilon_{i,j}(M)v^+_\lm=\hspace{-3mm}\sum_{\gamma+\sigma=\epsilon+\delta}
 \hspace{-3mm}
 (-1)^{\1+(\delta+\gamma)(\epsilon+\1+\|M_{(l+1..j]}\|)} S^{\gamma}_{i,l}(M_{(i..l]}) S^{\sigma}_{l+1,j}(M_{(l+1..j]})v^+_\lm=0
\end{align*}
otherwise.
\end{proof}

\begin{lemma}\label{lemma:constr:4} Let $\lm\in X(n)$, $v^+_\lm\in L(\lm)^\lm$, $1\le i<j\le n$, $\epsilon\in\{\0,\1\}$
and $M$ be a signed $(i..j]$-set containing either $\bar\jmath$ or $j$, all of whose elements are even, except for some $\odd q$.
Suppose that $\lm_q\=0\pmod p$ and there exists an injection $\psi:[i..j]\setminus M\to[i..j]$
such that
{\renewcommand{\labelenumi}{{\rm \theenumi}}
\renewcommand{\theenumi}{{\rm(\alph{enumi})}}
\begin{enumerate}
\item\label{lemma:constr:4:condition:a} $\res_p\lm_t=\res_p\bigl(\lm_{\psi(t)}+1\bigr)$ for all $t\in[i..j]\setminus M${\rm;}\\[-10pt]
\item\label{lemma:constr:4:condition:b} $\psi(t)\ge t$ for all $t\in[i..j]\setminus M${\rm;}\\[-10pt]
\item\label{lemma:constr:4:condition:c} if $\psi(t)=t$ then $t\le q$.
\end{enumerate}}
\noindent
Then $ S_{i,j}^{\,\epsilon}(M)\,v^+_\lm=0$.
\end{lemma}
\begin{proof}
By Lemma~\ref{lemma:constr:3}\ref{lemma:constr:3:part:i}, for any $\delta:[i..j)\to\{\0,\1\}$,
we have
\begin{align*}
 E^{\delta_i}_i{\cdots} E^{\delta_{j-1}}_{j-1} S_{i,j}^{\,\epsilon}(M)v^+_\lm=\prod\nolimits_{t\in(i..j]\setminus\psi((i..j]\setminus M)}\bigl(\res_p\lm_i-\res_p(\lm_t{+}1)\bigr) H_i^{\,\epsilon+\sum\delta}v^+_\lm.
\end{align*}
We claim that the last expression is zero. Indeed,
if $\psi(i)>i$ then $\psi(i)\in(i..j]\setminus\psi((i..j]{\setminus}M)$, since
$\psi$ is an injection, and by~\ref{lemma:constr:4:condition:a}, we have $\res_p\lm_i=\res_p(\lm_{\psi(i)}+1)$. On the other hand, if $\psi(i)=i$ then by~\ref{lemma:constr:4:condition:a}, we get $\res_p\lm_i\=\res_p(\lm_i+1)$, and
since $p\ne2$, we have $\lm_i\=0\pmod p$, whence $H_i^{\epsilon+\sum\delta}v^+_\lm{=}0$.

It remains to prove that $E_l^\delta S_{i,j}^{\,\epsilon}(M)v^+_\lm=0$ for all $i\le l<j-1$ and all $\de$.
Apply induction on $j-i$, the base case $j-i=1$ coming from the previous paragraph.
Let $j-i>1$.
By Lemma~\ref{lemma:constr:3}\ref{lemma:constr:3:part:ii}, we may assume that $l+1\notin M$ and $l\ne q-1$
and prove that $ S^{\,\sigma}_{l+1,j}\bigl(M_{(l+1..j]}\bigr)\,v^+_\lm=0$
for $\sigma=\0,\1$. 

{\it Case~1: $i\le l<q-1$.} 
Then $\bar q\in M_{(l+1..i]}$, and  $[l{+}1..j]\setminus M_{(l+1..j]}=[l{+}1..j]\setminus M$.
Consider the restriction $\psi'=\psi|_{[l{+}1..j]\setminus M_{(l+1..j]}}$. This restriction is obviously an injection of
$[l{+}1..j]\setminus M_{(l+1..j]}$ into $[l{+}1..j]$ and satisfies the following conditions:
{\renewcommand{\labelenumi}{{\rm \theenumi}}
\renewcommand{\theenumi}{{\rm(\alph{enumi}$^\prime$)}}
\begin{enumerate}
\item\label{lemma:constr:4:condition:a'} $\res_p\lm_t\=\res_p\bigl(\lm_{\psi'(t)}+1\bigr)$ for all $t\in[l{+}1..j]\setminus M${\rm;}\\[-10pt]
\item\label{lemma:constr:4:condition:b'} $\psi'(t)\ge t$ for all $t\in[l{+}1..j]\setminus M${\rm;}\\[-10pt]
\item\label{lemma:constr:4:condition:c'} if $\psi'(t)=t$ then $t\le q$.
\end{enumerate}}
\noindent
So $S^{\,\sigma}_{l+1,j}(M_{(l+1..j]})v^+_\lm=0$ by the inductive hypothesis.

{\it Case~2: $q\le l<j-1$.} We want to apply Lemma~\ref{lemma:constr:2}. Note that $M_{(l+1..j]}$
consists only of even elements. Moreover, $[l{+}1..j]\setminus M_{(l+1..j]}=[l{+}1..j]\setminus M$
and we can consider the restriction
$\psi''=\psi|_{[l{+}1..j]\setminus M_{(l+1..j]}}$. This restriction is obviously an injection of $[l{+}1..j]\setminus M_{(l+1..j]}$ into $[l{+}1..j]$ and satisfies the following conditions:
{\renewcommand{\labelenumi}{{\rm \theenumi}}
\renewcommand{\theenumi}{{\rm(\alph{enumi}$^{\prime\prime}$)}}
\begin{enumerate}
\item\label{lemma:constr:4:condition:a''} $\res_p\lm_t=\res_p\bigl(\lm_{\psi''(t)}+1\bigr)$ for any $t\in[l{+}1..j]\setminus M${\rm;}\\[-10pt]
\item\label{lemma:constr:4:condition:b''} $\psi''(t)>t$ for any $t\in[l{+}1..j]\setminus M$
\end{enumerate}}
\noindent
Note that \ref{lemma:constr:4:condition:b''} follows from~\ref{lemma:constr:4:condition:c}.
Now $ S^{\,\sigma}_{l+1,j}\bigl(M_{(l+1..j]}\bigr)\,v^+_\lm=0$ by Lemma~\ref{lemma:constr:2}.
\end{proof}

\begin{theorem}\label{theorem:constr:2}${}$\!\!
Let $\lm\in X(n)$, $1\le i<n$ and $\lm_i\=1\pmod p$. Suppose that $\bigl[\prod_{i<k\le n}r_{\mathbf 0}(\lm)_k\bigr]=+-^m$. Then there exists a non-zero homogeneous $U(n-1)$-primitive vector $v\in L(\lm)$ of weight $\lm-\alpha(i,n)$.
\end{theorem}
\begin{proof} By Lemma~\ref{lemma:pm:5} there exists a resolution $\Gamma$ of $r_{\mathbf 0}(\lm)|_{(i..n]}$.
Let $S$ be the set of all sources of edges of $\Gamma$.
For any $t\in S$, there exists a unique index $\psi(t)\in(i..n]$ such that $(t,\psi(t))\in\Gamma$.
Moreover, $\psi$ is an injection of $S$ into $(i..n]$.
Let $q$ be the maximal element of $S$ such that $q=\psi(q)$. Such an element exists, because $\Gamma$ is not a flow, see Proposition~\ref{proposition:A}. Pick a non-zero homogeneous vector  $v^+_\lm\in L(\lm)$, $\epsilon\in\{\0,\1\}$ and set
$M:=\bigl((i..n]\setminus S\bigr)\cup\{\bar q\}$ and
$$
v:= S_{i,n}^{\,\epsilon}(M)\,v^+_\lm.
$$

If $q<n$ then $n\notin S$, since otherwise $n<\psi(n)\in(i..n]$
by the choice of $q$. Thus $q<n$ implies $n\in M$. Note also that $q\in S$ and
so $q\notin M$. We have proved that $M$ is a signed $(i..n]$-set containing either $\bar n$ or $n$.
Note also that $(i..n]\setminus M=S$.

We have the following properties of $\psi$ following from the fact that $\Gamma$ is a weak flow coherent with $r_\beta(\lm)|_{(i..n]}$
and the choice of $q$:
{\renewcommand{\labelenumi}{{\rm \theenumi}}
\renewcommand{\theenumi}{{\rm(\alph{enumi})}}
\begin{enumerate}
\item\label{theorem:constr:2:property:a} $\res_p\lm_t=\res_p\bigl(\lm_{\psi(t)}+1\bigr)={\mathbf0}$ for all $t\in(i..n]\setminus M$;\\[-10pt]
\item\label{theorem:constr:2:property:b} $\psi(t)\ge t$ for all $t\in(i..n]\setminus M$;\\[-10pt]
\item\label{theorem:constr:2:property:c} if $\psi(t)=t$ then $t\le q$.
\end{enumerate}
}
\noindent
By Lemma~\ref{lemma:constr:3}\ref{lemma:constr:3:part:i},
for any $\delta:[i..n)\to\{\0,\1\}$ such that $\sum\delta=\epsilon$, we have
$$
 E^{\delta_i}_i{\cdots} E^{\delta_{n-1}}_{n-1} S_{i,n}^{\,\epsilon}(M)\,v^+_\lm
=\prod\nolimits_{t\in(i..n]\,\setminus\,\psi(S)}\bigl(-\res_p(\lm_t{+}1)\bigr)v^+_\lm.
$$

As $\Gamma$ is fully coherent with $r_{\mathbf0}(\lm)|_{(i..n]}$, any $t\in(i..n]$ such that $\res_p(\lm_t+1)\=\mathbf0$ is the target of an edge of $\Gamma$, i.e.
belongs to $\psi(S)$. Therefore,
$$
\prod\nolimits_{t\in(i..n]\,\setminus\,\psi(S)}\bigl(-\res_p(\lm_t{+}1)\bigr)\ne{\mathbf0}.
$$
Hence $ E^{\delta_i}_i\cdots E^{\delta_{n-1}}_{n-1} S_{i,n}^{\,\epsilon}(M)\,v^+_\lm\ne0$ and $v\ne0$.

Let us finally prove that $v$ is a $U(n-1)$-primitive vector. We need to show that
$ E_l^\delta S_{i,n}^{\,\epsilon}(M)\,v^+_\lm=0$ for all $i\le l<n-1$ and all  $\de$.
Note that $\lm_q\=0\pmod p$ by~\ref{theorem:constr:2:property:a} as $\psi(q)=q$.
So, by Lemma~\ref{lemma:constr:3}\ref{lemma:constr:3:part:ii}, we may assume that $l+1\notin M$, $l\ne q-1$,
and prove that
$ S^{\,\sigma}_{l+1,j}\bigl(M_{(l+1..j]}\bigr)\,v^+_\lm=0$ for any $\sigma$.
But the last equality follows from Lemma~\ref{lemma:constr:4} if $l<q-1$ and from Lemma~\ref{lemma:constr:2} if $l>q-1$,
where we consider the injection $\psi|_{[l+1..j]\setminus M_{(l+1..j]}}$, noting that
$[l{+}1..j]\setminus M_{(l{+}1..j]}=[l{+}1..j]\setminus M$, since $l+1\notin M$.
\end{proof}

\section[Extension: case 1]{Extension: case  $\lm_h\=0\pmod p$, $\lm_i\=1\pmod p$, 
$\bigl[\prod_{h<k\le i} r_\mathbf0(\lm)_k\bigr]=-^m$}\label{extending:1}

\begin{lemma}\label{lemma:socle:1.75} Let $\lm\in X(n)$, $1\le h<i<n$, $v$ be a $U(n-1)$-primitive vector of $L(\lm)^{\lm-\alpha(i,n)}$,
$\epsilon\in\{\0,\1\}$, and $M\ni\bar\imath$ be a signed $(h..i]$-set
all of whose elements except $\bar\imath$ are even.
Suppose that $\psi:[h..i)\setminus M\to[h..i)$ is an injection such that
{\renewcommand{\labelenumi}{{\rm \theenumi}}
\renewcommand{\theenumi}{{\rm(\alph{enumi})}}
\begin{enumerate}
\item\label{lemma:socle:1.75:condition:a} $\res_p\lm_t=\res_p\bigl(\lm_{\psi(t)}+1\bigr)$ for any $t\in[h..i)\setminus M${\rm;}\\[-10pt]
\item\label{lemma:socle:1.75:condition:b} $\psi(t)>t$ for any $t\in[h..i)\setminus M$.
\end{enumerate}}
\noindent
Then $ S_{h,i}^{\,\epsilon}(M)\,v=0$.
\end{lemma}
\begin{proof}
Set $\mu:=\lm-\alpha(i,n)$. We first prove that
$
 E_l^\delta S_{h,i}^{\,\epsilon}(M)\,v=0
$
for all $l\neq n-1,i-1$ and $\delta\in\{\0,\1\}$. By weights, we may assume that $h\le l<i-1$.
We apply induction on $i-h$.
By Lemma~\ref{lemma:constr:5}\ref{lemma:constr:5:part:ii}, we may assume that
$l+1\notin M$ and prove that $ S^{\,\sigma}_{l+1,i}(M_{(l+1..i]})v=0$
for any $\sigma$. But this equality follows from the inductive hypothesis if we consider the restriction
$\psi'=\psi|_{[l{+}1..i)\setminus M_{(l+1..i]}}$, which is obviously an injection of
$[l{+}1..i)\setminus M_{(l+1..i]}=[l{+}1..i)\setminus M$ into $[l{+}1..i)$
satisfying the conditions similar to~\ref{lemma:socle:1.75:condition:a} and~\ref{lemma:socle:1.75:condition:b}.

Next, we prove that $ S^{\,\epsilon}_{h,i}(M)v$ is $U(n-1)$-primitive.
By Lemma~\ref{lemma:socle:1.5}, it suffices to show that
$ E_h^{\delta_h}\cdots E_{i-1}^{\delta_{i-1}} S^{\,\epsilon}_{h,i}\bigl(M\bigr)\,v=0$
for all $\delta:[h..i)\to\{\0,\1\}$. By~\ref{lemma:socle:1.75:condition:a} and the fact that
$\psi$ is an injection of $[h..i)\setminus M$ into $[h..i)$, we have:
{\renewcommand{\labelenumi}{{\rm \theenumi}}
\renewcommand{\theenumi}{{\rm(\alph{enumi}$^\prime$)}}
\begin{enumerate}
\item\label{lemma:socle:1.75:condition:a'} $\res_p\mu_t=\res_p\bigl(\mu_{\psi(t)}+1\bigr)$ for all $t\in[h..i)\setminus M$.
\end{enumerate}}
\noindent
Hence by Lemma~\ref{lemma:constr:5}\ref{lemma:constr:5:part:i}, we get
\begin{align*}
 E^{\delta_h}_h\cdots E^{\delta_{i-1}}_{i-1} S_{h,i}^{\,\epsilon}(M)\,v
 \\
=\prod\nolimits_{t\in(h..i)\,\setminus\,\psi((h..i)\setminus M)}\bigl(\res_p\mu_h{-}\res_p(\mu_t{+}1)\bigr)
( H_h^{\epsilon+\sum\delta}{-}(-1)^{\epsilon\sum\delta} H_i^{\epsilon+\sum\delta})v
\\
=\prod\nolimits_{t\in(h..i)\,\setminus\,\psi((h..i)\setminus M)}\bigl(\res_p\lm_h{-}\res_p(\lm_t{+}1)\bigr)
( H_h^{\epsilon+\sum\delta}{-}(-1)^{\epsilon\sum\delta} H_i^{\epsilon+\sum\delta})v=0,
\end{align*}
because $\psi(h)\in(h..i)\setminus\psi((h..i){\setminus}M)$, and
by~\ref{lemma:socle:1.75:condition:a}, we have $\res_p\lm_h=\res_p(\lm_{\psi(h)}+1)$.

Now, it suffice to prove that
$
 E_h^{\delta_h}\cdots E_{n-1}^{\delta_{n-1}} S_{h,i}^{\,\epsilon}(M)\,v=0
$
for all $\delta:[h..n)\to\{\0,\1\}$. By Corollary~\ref{corollary:ops:2}, 
we have
$$
 E_h^{\delta_h}\cdots E_{n-1}^{\delta_{n-1}} S_{h,i}^{\,\epsilon}(M)\,v
=(-1)^{\,\epsilon\sum\delta|_{[i..n)}} E_h^{\delta_h}\cdots E_{i-1}^{\delta_{i-1}} S_{h,i}^{\,\epsilon}(M)\, E_i^{\delta_i}\cdots E_{n-1}^{\delta_{n-1}}v=0
$$
by Lemma~\ref{lemma:constr:6} with $v^+_\lm= E_i^{\delta_i}\cdots E_{n-1}^{\delta_{n-1}}v$.
\end{proof}

\begin{theorem}\label{theorem:socle:0.5}
Let $\lm\in X(n)$, $1\le h<i<n$, $\lm_h\=0\pmod p$, $\lm_i\=1\pmod p$, and $\bigl[\prod_{h<k\le i} r_{\mathbf0}(\lm)_k\bigr]=-^m$.
Then for any nonzero homogeneous $U(n-1)$-primitive
vector $v\in L(\lm)^{\lm-\alpha(i,n)}$ there exists a nonzero homogeneous $U(n-1)$-primitive vector
$w\in U(i)v$ of weight $\lm-\alpha(h,n)$.
\end{theorem}
\begin{proof}
Let $(a,b)\in\Z_{\geq 0}^2$ be a pair with maximal possible sum $a+b$ such that
$\bar{\! H}_h^a\,\bar{\! H}_i^b\,v\ne0$. As $ H_iv= H_hv=0$, we have $a,b<2$. We set $v':=\,\bar{\! H}_h^a\,\bar{\! H}_i^bv$.
Then $v'\in U(i)v$ and $ H_h^{\sigma}v'= H_i^{\sigma}v'=0$ for all $\sigma\in\{\0,\1\}$.

By Lemma~\ref{lemma:pm:3}, there exists a flow $\Gamma$ on $(h..i]$ fully coherent with $r_{\mathbf0}(\lm)|_{(h..i]}$.
Let $S$ be the set of all sources of edges of $\Gamma$. Thus for any $t\in S$ there exists a unique
$\psi(t)\in(h..i]$ with $(t,\psi(t))\in\Gamma$, $\psi$ is an injection of $S$ into $(h..i]$, and
{\renewcommand{\labelenumi}{{\rm \theenumi}}
\renewcommand{\theenumi}{{\rm(\alph{enumi})}}
\begin{enumerate}
\item\label{theorem:socle:0.5:property:a} $\res_p\lm_t=\res_p\bigl(\lm_{\psi(t)}+1\bigr)=\mathbf0$ for all $t\in S$;\\[-10pt]
\item\label{theorem:socle:0.5:property:b} $\psi(t)>t$ for all $t\in S$.
\end{enumerate}
}
Note that $\psi$ is actually an injection of $S$ into $(h..i)$.
Indeed, suppose that $\psi(t)=i$ for some $t\in S$. Then by~\ref{theorem:socle:0.5:property:a},
we get $\mathbf0=\res_p(\lm_i+1)=\mathbf2$, which is a contradiction.

Set $M:=((h..i)\setminus S)\cup\{\bar\imath\}$. Note that $i\notin S$
for otherwise $i<\psi(i)\in(h..i]$, giving a contradiction.
Hence we get $(h..i)\setminus M=S$.
Choose 
$\epsilon\in\{\0,\1\}$ and set
$$
w:= S^{\,\epsilon}_{h,i}(M)\,v'.
$$

To prove that $w$ is a $U(n-1)$-primitive vector, we must show that $ E_l^\delta S^{\,\epsilon}_{h,i}(M)\,v'$\linebreak$=0$
for $l=1,\ldots,n-2$ and $\delta\in\{\0,\1\}$.
As $v'$ is $U(n-1)$-primitive, and by weights,  we may assume that $h\le l<i$.
If $h\le l<i-1$, then by Lemma~\ref{lemma:constr:5}\ref{lemma:constr:5:part:ii},
it suffices to prove that $ S^{\sigma}_{l+1,i}(M_{(l+1..i]})v'=0$ in the case $l+1\notin M$.
But this equality follows from Lemma~\ref{lemma:socle:1.75} if we consider the injection
$\psi|_{[l+1..i)\setminus M_{(l+1..i]}}$ of $[l{+}1..i)\setminus M_{(l+1..i]}=[l{+}1..i)\setminus M$ into $[l{+}1..i)$.

Now, by Lemma~\ref{lemma:socle:1.5}, to prove the $U(n-1)$-primitivity of $ S^{\epsilon}_{h,i}(M)v'$
it suffices to show that $ E_h^{\delta_h}\cdots E_{i-1}^{\delta_{i-1}} S^{\epsilon}_{h,i}(M)v'=0$
for all $\delta:[h..i)\to\{\0,\1\}$.
Applying Lemma~\ref{lemma:constr:5}\ref{lemma:constr:5:part:i} as in the proof of Lemma~\ref{lemma:socle:1.75}, we get
\begin{align*}
 E^{\delta_h}_h\cdots E^{\delta_{i-1}}_{i-1} S_{h,i}^{\,\epsilon}(M)\,v'\\
 =\prod\nolimits_{t\in(h..i)\,\setminus\,\psi((h..i)\setminus M)}\bigl(\res_p\lm_h{-}\res_p(\lm_t{+}1)\bigr)( H_h^{\epsilon+\sum\delta}{-}(-1)^{\epsilon\sum\delta} H_i^{\epsilon+\sum\delta})v'=0
\end{align*}
by the choice of $v'$ in the beginning of this proof.

Finally, we prove that $w\neq 0$. We are going to prove that $ E_h^{\delta_h}\cdots E_{n-1}^{\delta_{n-1}}w\ne0$
for some function $\delta:[h..n)\to\{\0,\1\}$.
Applying Lemma~\ref{lemma:constr:5}\ref{lemma:constr:5:part:i} as in the proof of Lemma~\ref{lemma:socle:1.75}, we get
\begin{align*}
 E_h^{\delta_h}\cdots E_{n-1}^{\delta_{n-1}} S_{h,i}^{\,\epsilon}(M)v'
=(-1)^{\epsilon\sum\delta|_{[i..n)}} E_h^{\delta_h}\cdots E_{i-1}^{\delta_{i-1}} S_{h,i}^{\epsilon}(M) E_i^{\delta_i}\cdots E_{n-1}^{\delta_{n-1}}v'
\\
=(-1)^{\,\epsilon\sum\delta|_{[i..n)}}\prod
_{t\in(h..i)\,\setminus\,\psi((h..i)\setminus M)}(-\res_p(\lm_t{+}1))
\\
\times( H_h^{\,\epsilon+\sum\delta|_{[h..i)}}{-}(-1)^{\epsilon\sum\delta|_{[h..i)}} H_i^{\epsilon+\sum\delta|_{[h..i)}}) E_i^{\delta_i}\cdots E_{n-1}^{\delta_{n-1}}v'.
\end{align*}
Since the flow $\Gamma$ is fully coherent with $r_\beta(\lm)|_{(h..i]}$, we have
$$
\prod\nolimits_{t\in(h..i)\setminus\psi((h..i)\setminus M)}(-\res_p(\lm_t{+}1))\ne\mathbf0
$$
by Proposition~\ref{proposition:rbeta}\ref{proposition:rbeta:part:2}.
Now to complete the proof that $w\neq 0$,  choose $\delta_h,\ldots,\delta_{i-1}$ so that $\sum\delta|_{[h..i)}=\epsilon$, choose $\delta_i,\ldots,\delta_{n-1}$ so that $ E_i^{\delta_i}\cdots E_{n-1}^{\delta_{n-1}}\,v'\ne0$, and recall that $\lm_h\=0\pmod p$ and $\lm_i\=1\pmod p$.
\end{proof}

\section[Extension: Case 2]{Extension: case $\lm_i\not\=0\pmod p$, $\bigl[\prod_{h<k\le i} r_\beta(\lm)_k\bigr]=-^m$, \\
and $\lm_h\=0\pmod p$, $\lm_i\=1\pmod p$ do not both hold}\label{extending:2}

\begin{lemma}\label{lemma:socle:2} Let $\lm\in X(n)$, $1\le h<i<n$, $v$ be a $U(n-1)$-primitive vector in $L(\lm)^{\lm-\alpha(i,n)}$,
$\epsilon\in\{\0,\1\}$, and
$M\subset(h..i]$ be a subset containing $i\in M$.
Suppose that there exists an injection $\psi:[h..i)\setminus M\to[h..i)$ such that
{\renewcommand{\labelenumi}{{\rm \theenumi}}
\renewcommand{\theenumi}{{\rm(\alph{enumi})}}
\begin{enumerate}
\item\label{lemma:socle:2:condition:a} $\res_p\lm_t=\res_p\bigl(\lm_{\psi(t)}+1\bigr)$ for all $t\in[h..i)\setminus M${\rm;}\\[-10pt]
\item\label{lemma:socle:2:condition:b} $\psi(t)>t$ for all $t\in[h..i)\setminus M$.
\end{enumerate}}
\noindent
Then $ S_{h,i}^{\,\epsilon}(M)\,v=0$.
\end{lemma}
\begin{proof}
Set $\mu:=\lm-\alpha(i,n)$.
Note that
{\renewcommand{\labelenumi}{{\rm \theenumi}}
\renewcommand{\theenumi}{{\rm(\alph{enumi}$^\prime$)}}
\begin{enumerate}
\item\label{lemma:socle:2:condition:a'} $\res_p\mu_t=\res_p\bigl(\mu_{\psi(t)}+1\bigr)$ for all $t\in[h..i)\setminus M$.
\end{enumerate}}
\noindent
Hence and by Lemma~\ref{lemma:constr:1}\ref{lemma:constr:1:part:i}, we get for all $\delta_h,\dots, \delta_{i-1}$:
\begin{align*}
 E^{\delta_h}_h\cdots E^{\delta_{i-1}}_{i-1} S_{h,i}^{\epsilon}(M)v
=\cond_{\sum\delta=\epsilon}\prod\nolimits_{t\in(h..i]\,\setminus\,\psi((h..i]\setminus M)}\bigl(\res_p\mu_h-\res_p(\mu_t+1)\bigr)v
\\
=\prod\nolimits_{t\in(h..i)\setminus\psi((h..i)\setminus M)}\!(\res_p\lm_h\!-\res_p(\lm_t+1))(\res_p\lm_h-\res_p(\mu_i+1))v,
\end{align*}
which is zero because $\psi(h)\in(h..i)\setminus\psi((h..i){\setminus}M)$.


Let $\eps,\delta_h,\dots,\delta_{n-1}\in\{\0,\1\}$.
Set $M':=(M\setminus\{i\})\cup\{\bar\imath\}$,
$v':= E_i^{\delta_i}\cdots E_{n-1}^{\delta_{n-1}}v$, and
$v^\sigma:= E_i^{\,\sigma} E_{i+1}^{\delta_{i+1}}\cdots E_{n-1}^{\delta_{n-1}}v$. Note that  $v', v^\sigma\in L(\lm)^\lm$. 
By Corollary~\ref{corollary:ops:2} and Lemma~\ref{lemma:socle:1}, we have:
\begin{align*}
 E_h^{\delta_h}\cdots E_{n-1}^{\delta_{n-1}} S_{h,i}^{\,\epsilon}(M)\,v
=(-1)^{\,\epsilon\sum\delta|_{(i..n)}} E_h^{\delta_h}\cdots E_i^{\delta_i} S_{h,i}^{\,\epsilon}(M)\, E_{i+1}^{\delta_{i+1}}\cdots E_{n-1}^{\delta_{n-1}}\,v
\\
=(-1)^{\,\epsilon\sum\delta|_{[i..n)}} E_h^{\delta_h}\cdots E_{i-1}^{\delta_{i-1}} S_{h,i}^{\,\epsilon}(M)\, E_i^{\delta_i}\cdots E_{n-1}^{\delta_{n-1}}\,v
\\
+(-1)^{\,\epsilon\sum\delta|_{(i..n)}}\sum_{\gamma+\sigma=\epsilon+\delta_i}(-1)^{\1+(\delta_i+\gamma)\epsilon}
E_h^{\delta_h}\cdots E_{i-1}^{\delta_{i-1}} S_{h,i}^{\,\gamma}(M') E_i^{\,\sigma} E_{i+1}^{\delta_{i+1}}\cdots E_{n-1}^{\delta_{n-1}}v
\\
=(-1)^{\epsilon\sum\delta|_{[i..n)}} E_h^{\delta_h}\cdots E_{i-1}^{\delta_{i-1}} S_{h,i}^{\,\epsilon}(M)\,v'
\\
+(-1)^{\,\epsilon\sum\delta|_{(i..n)}}\sum_{\gamma+\sigma=\epsilon+\delta_i}(-1)^{\1+(\delta_i+\gamma)\epsilon}E_h^{\delta_h}\cdots E_{i-1}^{\delta_{i-1}} S_{h,i}^{\,\gamma}(M')\,v^\sigma=0,
\end{align*}
because, by Lemma~\ref{lemma:constr:2}, $ S_{h,i}^{\epsilon}(M)v'\!=\!0$,
and by Lemma~\ref{lemma:constr:6},  $ S_{h,i}^{\gamma}(M')v^\sigma\!=\!0$.

By Proposition~\ref{proposition:intro:5}, Lemma~\ref{lemma:socle:1.5}, and Lemma~\ref{lemma:constr:1}\ref{lemma:constr:1:part:ii}, it now suffices to prove that $
E_l^\delta S_{h,i}^{\,\epsilon}(M)v=0$ for all $h\le l<i-1$ and all $\de$. Apply induction on $i-h$.
By Lemma~\ref{lemma:constr:1}\ref{lemma:constr:1:part:ii}, we may assume that
$l+1\notin M$ and prove that $ S^{\,\sigma}_{l+1,i}\bigl(M_{(l+1..i]}\bigr)v=0$
for any $\sigma$. But this fact follows from the inductive hypothesis by considering the restriction
$\psi'=\psi|_{[l{+}1..i)\setminus M_{(l+1..i]}}$, which is an injection of
$[l{+}1..i)\setminus M_{(l+1..i]}=[l{+}1..i)\setminus M$ into $[l{+}1..i)$, satisfying conditions similar to~\ref{lemma:socle:2:condition:a} and~\ref{lemma:socle:2:condition:b}.
\end{proof}

\begin{theorem}\label{theorem:socle:1}
Let $\lm\in X(n)$, $1\le h<i<n$, $\lm_i\not\=0\pmod p$, and $\lm_h\not\equiv0\pmod p$ or  $\lm_i\not\equiv1\pmod p$. Let $\res_p\lm_h=\res_p\lm_i=:\be$, and $\bigl[\prod_{h<k\le i} r_\beta(\lm)_k\bigr]=-^m$. Then for any nonzero homogeneous $U(n{-}1)$-primitive vector
$v\in L(\lm)^{\lm-\alpha(i,n)}$, there exists a nonzero homogeneous $U(n-1)$-primitive vector $w\in U(i)v$ of weight
$\lm-\alpha(h,n)$.
\end{theorem}
\begin{proof}
By Lemmas~\ref{lemma:pm:1} and~\ref{lemma:pm:3}, there exists a flow $\Gamma$ on $(h..i]$
fully coherent with $r_\beta(\lm)|_{(h..i]}$. Let $S$ be the set of all sources of edges of $\Gamma$.
Thus for any $t\in S$, there exists a unique $\psi(t)\in(h..i]$ such that $(t,\psi(t))\in\Gamma$, and $\psi$ is an injection of $S$ into $(h..i]$.
By Proposition~\ref{proposition:rbeta}, we also have
{\renewcommand{\labelenumi}{{\rm \theenumi}}
\renewcommand{\theenumi}{{\rm(\alph{enumi})}}
\begin{enumerate}
\item\label{theorem:socle:1:property:a} $\res_p\lm_t=\res_p\bigl(\lm_{\psi(t)}+1\bigr)=\beta$ for all $t\in S$;\\[-10pt]
\item\label{theorem:socle:1:property:b} $\psi(t)>t$ for all $t\in S$.
\end{enumerate}}
\noindent
Note further that $\psi(S)\subset (h..i)$, for if $\psi(t)=i$ for some $t\in S$, then by~\ref{theorem:socle:1:property:a}, we have $\res_p(\lm_i+1)=\beta=\res_p\lm_i$, and so $\lm_i\=0\pmod p$, which contradicts our assumptions.

Denote $M:=(h..i]\setminus S$. Note that $i\in M$, for otherwise $i\in S$, and $i<\psi(i)\in(h..i]$ gives a contradiction. Set
$
w:=\,\bar{\! S}_{h,i}(M)v.
$
We first prove that $w$ is $U(n-1)$-primitive. As $v$ is $U(n-1)$-primitive and by weights,
it suffices to prove that $ E_l^\delta\bar{\! S}_{h,i}(M)v=0$ for all $h\le l<i$ and all $\de$.
Assume first that $h\le l<i-1$. By Lemma~\ref{lemma:constr:1}\ref{lemma:constr:1:part:ii},
it suffices to prove that $ S^{\,\sigma}_{l+1,i}\bigl(M_{(l+1..i]}\bigr)\,v=0$ in the case $l+1\notin M$. But this follows from Lemma~\ref{lemma:socle:2} if we consider the injection
$\psi|_{[l+1..i)\setminus M_{(l+1..i]}}$.
Now, by Lemma~\ref{lemma:socle:1.5}, to that $w$ is $U(n-1)$-primitive,
it suffices to prove that $E_h^{\delta_h}\cdots E_{i-1}^{\delta_{i-1}}\,\bar{\! S}_{h,i}\bigl(M\bigr)\,v=0$
for any $\delta:[h..i)\to\{\0,\1\}$.
By Lemma~\ref{lemma:constr:1}\ref{lemma:constr:1:part:i}, we have
\begin{align*}
 E_h^{\delta_h}\cdots E_{i-1}^{\delta_{i-1}}\,\bar{\! S}_{h,i}(M)\,v
 \\
 \cond_{\sum\delta=\1}\prod\nolimits_{t\in(h..i)\setminus\psi((h..i)\setminus M)}
(\res_p\lm_h-\res_p(\lm_t+1))(\res_p\lm_h-\res_p(\mu_i+1))v,
\end{align*}
where $\mu=\lm-\alpha(i,n)$. The last expression is zero because
$\mu_i+1=\lm_i$, and so $\res_p\lm_h-\res_p(\mu_i+1)=\res_p\lm_h-\res_p\lm_i=\mathbf0$.

Finally, we prove that $w\neq 0$. If $w=0$, then $ E_h^{\delta_h}\cdots E_{n-1}^{\delta_{n-1}}w=0$ for all $\delta:[h..n)\to\{\0,\1\}$.
We set
$$
c:=\prod\nolimits_{t\in(h..i)\,\setminus\,\psi(S)}\bigl(\beta-\res_p(\lm_t+1)\bigr).
$$
As the flow $\Gamma$ is fully coherent with $r_\beta(\lm)|_{(h..i]}$, we have $c\neq 0$.
Denoting $M':=(M\setminus\{i\})\cup\{\bar\imath\}$ and applying parts~(i) of Lemmas~\ref{lemma:constr:1} and~\ref{lemma:constr:5},
we get: 
\begin{equation}\label{equation:socle:2}
\begin{array}{l}
0= E_h^{\delta_h}\cdots E_{n-1}^{\delta_{n-1}}w
=(-1)^{\sum\delta|_{[i..n)}} E_h^{\delta_h}\cdots E_{i-1}^{\delta_{i-1}}\,\bar{\! S}_{h,i}(M)\, E_i^{\delta_i}\cdots E_{n-1}^{\delta_{n-1}}\,v
\\[2pt]
\displaystyle
+(-1)^{\sum\delta|_{(i..n)}}\hspace{-5mm}\sum_{\gamma+\sigma=\1+\delta_i}
\hspace{-4mm}
(-1)^{\1+\delta_i+\gamma} E_h^{\delta_h}\cdots E_{i-1}^{\delta_{i-1}}S_{h,i}^{\,\gamma}(M')
 E_i^{\sigma} E_{i+1}^{\delta_{i+1}}\cdots E_{n-1}^{\delta_{n-1}}v
\\[2pt]
=(-1)^{\sum\delta|_{[i..n)}}
\cond_{\sum\delta|_{[h..i)}=\1}\,c\,\bigl(\res_p\lm_h-\res_p(\lm_i+1)\bigr) E_i^{\delta_i}\cdots E_{n-1}^{\delta_{n-1}}v
\\[3pt]
\displaystyle
+(-1)^{\sum\delta|_{[i..n)}}\sum_{\gamma+\sigma=\1+\delta_i}
\hspace{-4.5mm}
(-1)^{\1+\gamma}
c( H_h^{\gamma+\sum\delta|_{[h..i)}}{-}(-1)^{\gamma\sum\delta|_{[h..i)}} H_i^{\gamma+\sum\delta|_{[h..i)}})
\\[4pt]
\times E_i^{\,\sigma}
E_{i+1}^{\delta_{i+1}}\cdots E_{n-1}^{\delta_{n-1}}\,v.
\end{array}
\end{equation}
Choose $\delta_h,\ldots,\delta_{n-1}$ so that
$ E_i^{\,\delta_i+\1} E_{i+1}^{\delta_{i+1}}\cdots E_{n-1}^{\delta_{n-1}}v\ne0$
If $\delta_h+\dots+\delta_{i-1}=\1$ then cancelling out $c\,(-1)^{\sum\delta|_{[i..n)}}$ in~(\ref{equation:socle:2}), we get
\begin{equation}\label{equation:socle:3}
\begin{array}{l}
0=\bigl(\lm_h(\lm_h-1)-(\lm_i+1)\lm_i\bigr)\, E_i^{\delta_i}\cdots E_{n-1}^{\delta_{n-1}}\,v
+(\lm_h+\lm_i)\, E_i^{\delta_i}\cdots E_{n-1}^{\delta_{n-1}}\,v\\
-(\,\bar{\! H}_h-\,\bar{\! H}_i)\, E_i^{\,\delta_i+\1} E_{i+1}^{\delta_{i+1}}\cdots E_{n-1}^{\delta_{n-1}}\,v.
\end{array}
\end{equation}
On the other hand, if $\delta_h+\cdots+\delta_{i-1}=\0$, we get similarly:
\begin{equation}\label{equation:socle:4}
\begin{array}{l}
0=(\,\bar{\! H}_h-\,\bar{\! H}_i)\, E_i^{\delta_i}\cdots E_{n-1}^{\delta_{n-1}}\,v
-(\lm_h-\lm_i)\, E_i^{\,\delta_i+\1} E_{i+1}^{\delta_{i+1}}\cdots E_{n-1}^{\delta_{n-1}}\,v.
\end{array}
\end{equation}
Multiplying~(\ref{equation:socle:3}) by $\,\bar{\! H}_h-\,\bar{\! H}_i$,
using~(\ref{equation:socle:4}) and the equality $(\,\bar{\! H}_h-\,\bar{\! H}_i)^2=H_h+H_i$,
we get
$$
\Bigl[
\bigl(\lm_h(\lm_h-1)-(\lm_i+1)\lm_i+\lm_h+\lm_i\bigr)\,(\lm_h-\lm_i)-(\lm_h+\lm_i)
\Bigr] E_i^{\,\delta_i+\1} E_{i+1}^{\delta_{i+1}}\cdots E_{n-1}^{\delta_{n-1}}\,v=0.
$$
As $ E_i^{\,\delta_i+\1} E_{i+1}^{\delta_{i+1}}\cdots E_{n-1}^{\delta_{n-1}}\,v\ne0$, we have
$$
\bigl(\lm_h(\lm_h-1)-(\lm_i+1)\lm_i+\lm_h+\lm_i\bigr)\,(\lm_h-\lm_i)-(\lm_h+\lm_i)\=0\pmod p.
$$
The equality $\res_p\lm_h=\res_p\lm_i$ yields $\lm_h(\lm_h-1)-\lm_i(\lm_i-1)\=0\pmod p$. This allows us to write
the above formula as
\begin{align*}
0\=&\bigl(\lm_h(\lm_h\!-\!1\!)\!-\!(\lm_i\!+\!1)\lm_i\!-\!\bigl[\lm_h(\lm_h\!-\!1)\!-\!\lm_i(\lm_i\!-\!1)\bigr]\!+\!\lm_h\!+\!\lm_i\bigr)(\lm_h\!-\!\lm_i)\!-\!(\lm_h\!+\!\lm_i)\\
=&(\lm_h-\lm_i)^2\!-\!(\lm_h\!+\!\lm_i)\=\lm_h^2-2\lm_h\lm_i+\lm_i^2-\lm_h-\lm_i-\bigl[\lm_h(\lm_h-1)-\lm_i(\lm_i-1)\bigr]\\
=&-2\lm_i(\lm_h-\lm_i+1)\=0\pmod p.
\end{align*}
Since by assumption $\lm_i\not\=0\pmod p$, we get $\lm_h\=\lm_i-1\pmod p$.
Substituting this value of $\lm_h$ to
$\lm_h(\lm_h-1)-\lm_i(\lm_i-1)\=0\pmod p$, we get
$$
0\=(\lm_i-1)(\lm_i-2)-\lm_i(\lm_i-1)=-2\lm_i+2\pmod p.
$$
Hence $\lm_i\=1\pmod p$ and $\lm_h\=0\pmod p$, which is a contradiction.
\end{proof}

\section[Extension: case 3]{Extension: case  $\lm_h\=1\pmod p$, $\lm_i\=0\pmod p$, and $[\prod_{h<k\le i} r_\mathbf0(\lm)_k]=+-^m$}\label{extending:3}

\begin{lemma}\label{lemma:socle:5}
Let $\lm\in X(n)$, $1\le h<i<n$, $v$ be a $U(n-1)$-primitive vector of $L(\lm)^{\lm-\alpha(i,n)}$, $\epsilon\in\{\0,\1\}$, and $M\subset(h..i]$ be such that $i\in M$.
Suppose that there exists an injection $\psi:[h..i]\setminus M\to[h..i]$ such that
{\renewcommand{\labelenumi}{{\rm \theenumi}}
\renewcommand{\theenumi}{{\rm(\alph{enumi}$_\psi$)}}
\begin{enumerate}
\item\label{lemma:socle:5:psi:condition:a} $\res_p\lm_t=\res_p\bigl(\lm_{\psi(t)}+1\bigr)$ for all $t\in[h..i]\setminus M${\rm;}\\[-10pt]
\item\label{lemma:socle:5:psi:condition:b} $\psi(t)>t$ for all $t\in[h..i]\setminus M$
\end{enumerate}}
\noindent
and an injection $\theta:\bigl([h..i]\setminus M\bigr)\cup\{i\}\to[h..i]$ such that
{\renewcommand{\labelenumi}{{\rm \theenumi}}
\renewcommand{\theenumi}{{\rm(\alph{enumi}$_\theta$)}}
\begin{enumerate}
\item\label{lemma:socle:5:theta:condition:a} $\res_p\lm_t=\res_p\bigl(\lm_{\theta(t)}+1\bigr)$ for all $t\in\bigl([h..i]\setminus M\bigr)\cup\{i\}${\rm;}\\[-10pt]
\item\label{lemma:socle:5:theta:condition:b} $\theta(t)\ge t$ for all $t\in\bigl([h..i]\setminus M\bigr)\cup\{i\}${\rm;}\\[-10pt]
\end{enumerate}}
\noindent
Then $ S_{h,i}^{\,\epsilon}(M)\,v=0$.
\end{lemma}
\begin{proof}
Set $\mu:=\lm-\alpha(i,n)$.
Note the following property of $\psi$:
{\renewcommand{\labelenumi}{{\rm \theenumi}}
\renewcommand{\theenumi}{{\rm(\alph{enumi}$'_\psi$)}}
\begin{enumerate}
\item\label{lemma:socle:5:psi:condition:a'} $\res_p\mu_t=\res_p\bigl(\mu_{\psi(t)}+1\bigr)$ for all $t\in[h..i]\setminus M${\rm;}\\[-10pt]
\end{enumerate}}
\noindent
Indeed, we only need to check that $\res_p\mu_t=\res_p\bigl(\mu_i+1\bigr)$
for $t\in[h..i]\setminus M$ such that $\psi(t)=i$. We have $\theta(i)=i$ by property~\ref{lemma:socle:5:theta:condition:b}
and $t<i$ by property~\ref{lemma:socle:5:psi:condition:b}. Therefore by properties~\ref{lemma:socle:5:psi:condition:a}
and~\ref{lemma:socle:5:theta:condition:a}, we get
\begin{equation}\label{equation:socle:4.5}
\res_p\mu_t\!=\!\res_p\lm_t\!=\!\res_p(\lm_i+1)\!=\!\res_p(\lm_{\theta(i)}+1)\!=\!\res_p\lm_i\!=\!\res_p\bigl(\mu_i+1\bigr).
\end{equation}
Note that in particular $\res_p(\lm_i+1)=\res_p\lm_i$ implies $\lm_i\=0\pmod p$.

Let $\delta_h,\dots,\delta_{i-1}\!\in\!\{\0,\1\}$.
As $v$ has weight $\mu$, we have by~\ref{lemma:socle:5:psi:condition:a'} and Lemma~\ref{lemma:constr:1}\ref{lemma:constr:1:part:i}:
\begin{align*}
 E^{\delta_h}_h\cdots E^{\delta_{i-1}}_{i-1} \!S_{h,i}^{\,\epsilon}(M)v
\!=\!\cond_{\sum\delta=\epsilon}\prod\nolimits_{t\in(h..i]\,\setminus\,\psi((h..i]\setminus M)}\!(\res_p\mu_h-\res_p(\mu_t+1))v\!=\!0,
\end{align*}
since by ~\ref{lemma:socle:5:psi:condition:b}, $\psi(h)\in(h..i]\setminus\psi((h..i]{\setminus}M)$, and by~\ref{lemma:socle:5:psi:condition:a'}, $\res_p\mu_h=\res_p(\mu_{\psi(h)}+1)$.

Let $\delta_h,\dots, \delta_{n-1}\!\in\!\{\0,\1\}$, and set $M'\!:=\!(M\!\setminus\!\{i\})\!\cup\!\{\bar\imath\}$, $v':= E_i^{\delta_i}\cdots E_{n-1}^{\delta_{n-1}}v\in L(\lm)^\lm$, $v^\sigma:= E_i^{\,\sigma} E_{i+1}^{\delta_{i+1}}\cdots E_{n-1}^{\delta_{n-1}}\,v\in L(\lm)^\lm$.
By Lemma~\ref{lemma:socle:1}, we have
\begin{equation}\label{equation:socle:5}
\begin{array}{l}
\displaystyle E_h^{\delta_h}\cdots E_{n-1}^{\delta_{n-1}} S_{h,i}^{\,\epsilon}(M)\,v
=(-1)^{\,\epsilon\sum\delta|_{[i..n)}} E_h^{\delta_h}\cdots E_{i-1}^{\delta_{i-1}} S_{h,i}^{\,\epsilon}(M)\,v'\\
\displaystyle
+(-1)^{\,\epsilon\sum\delta|_{(i..n)}}\sum_{\gamma+\sigma=\epsilon+\delta_i}(-1)^{\1+(\delta_i+\gamma)\epsilon} E_h^{\delta_h}\cdots E_{i-1}^{\delta_{i-1}} S_{h,i}^{\,\gamma}(M')\,v^\sigma.
\end{array}
\end{equation}
Using $\psi$, by Lemma~\ref{lemma:constr:2}, we get $ S_{h,i}^{\,\epsilon}(M)\,v'=0$.
On the other hand, using $\theta$, by Lemma~\ref{lemma:constr:4} with $q=i$, we get $ S_{h,i}^{\,\gamma}(M')\,v^\sigma=0$.

By Proposition~\ref{proposition:intro:5} and Lemma~\ref{lemma:socle:1.5}, it now suffices to prove
$
E_l^\delta S_{h,i}^{\,\epsilon}(M)\,v=0
$
for any $h\le l<i-1$ and $\delta\in\{\0,\1\}$.
We apply induction on $i-h$. By Lemma~\ref{lemma:constr:1}\ref{lemma:constr:1:part:ii} we may assume that
$l+1\notin M$ and prove that $ S^{\,\sigma}_{l+1,i}(M_{(l+1..i]})v=0$
for any $\sigma$. But this equality follows from the inductive hypothesis if we consider the restrictions
$\psi'=\psi|_{[l{+}1..i]\setminus M_{(l+1..i]}}$ and $\theta'=\theta|_{\([l{+}1..i]\setminus M_{(l+1..i]}\)\cup\{i\}}$,
which are injections from $[l{+}1..i]\setminus M_{(l+1..i]}=[l{+}1..i]\setminus M$ and
$\([l{+}1..i]\setminus M_{(l+1..i]}\)\cup\{i\}=\([l{+}1..i]\setminus M\)\cup\{i\}$ to $[l{+}1..i]$
satisfying conditions similar to~\ref{lemma:socle:5:psi:condition:a},
\ref{lemma:socle:5:psi:condition:b}, \ref{lemma:socle:5:theta:condition:a}
and~\ref{lemma:socle:5:theta:condition:b}.
\end{proof}

\begin{theorem}\label{theorem:socle:2}
Let $\lm\in X(n)$ and $1\le h<i<n$,  $\lm_h\=1\pmod p$, $\lm_i\=0\pmod p$,
and $\bigl[\prod_{h<k\le i} r_\mathbf0(\lm)_k\bigr]=+-^m$. Then for any nonzero homogeneous $U(n{-}1)$-primitive
vector $v\in L(\lm)^{\lm-\alpha(i,n)}$, there exists a nonzero homogeneous $U(n-1)$-primitive vector $w\in U(i)\,v$
of weight $\lm-\alpha(h,n)$. 
\end{theorem}
\begin{proof}
We construct a flow $\Gamma$ on $[h..i]$ and a weak flow $\Delta$ on $(h..i]$ as follows.
If $\bigl[\prod_{h<k<i} r_{\mathbf0}(\lm)_k\bigr]=\emptyset$,
then by Lemma~\ref{lemma:pm:3}, there exists a flow $\Gamma_0$ on $(h..i)$ fully coherent with
$r_{\mathbf0}(\lm)|_{(h..i)}$, and we set $\Gamma:=\Gamma_0\cup\{(h,i)\}$ and $\Delta:=\Gamma_0\cup\{(i,i)\}$.
Now let $\bigl[\prod_{h<k<i} r_{\mathbf0}(\lm)_k\bigr]\ne\emptyset$.
We have $\bigl[\prod_{h<k<i} r_{\mathbf0}(\lm)_k\bigr]=+^s-^r$ by Proposition~\ref{proposition:pm:2}.
Then
$$
\textstyle
+-^m=\bigl[\prod\nolimits_{h<k\le i} r_{\mathbf0}(\lm)_k\bigr]
=\bigl[(+^s-^r)+-\bigr]
=\left\{
{\arraycolsep=0pt
\begin{array}{l}
+^{s+1}-\text{ if }r=0;\\[6pt]
+^s-^r\text{ if }r>0.
\end{array}}
\right.
$$
So $r=0$ implies $s=0$. But we cannot have $r=s=0$,
as $\bigl[\prod_{h<k<i} r_{\mathbf0}(\lm)_k\bigr]\ne\emptyset$. So $r>0$ and $s=1$.
Applying Lemma~\ref{lemma:pm:5} to $r_{\mathbf0}(\lm)|_{(h..i)}$, there exist
$h<a_1<\cdots<a_q<i$, with $q>0$, such that
{
\begin{itemize}
\item $r_{\mathbf0}(\lm)_{a_1}=\cdots=r_{\mathbf0}(\lm)_{a_q}=+-${\rm;}\\[-8pt]
\item $\bigl[\prod_{k\in(a_{l-1}..a_l)}r_{\mathbf0}(\lm)_k\bigr]$ contains only $-$'s for all  $l=1,\ldots,q+1$, where $a_0:=h$ and $a_{q+1}:=i${\rm;}\\[-8pt]
\end{itemize}}
\noindent
By Lemma~\ref{lemma:pm:3}, for $l=1,\ldots,q+1$, there exists a flow $\Gamma_l$
fully coherent with $r_{\mathbf0}(\lm)|_{(a_{l-1}..a_l)}$. We set $\Gamma:=\{(h,a_1),(a_1,a_2),\ldots,(a_q{-}1,a_q),(a_q,i)\}\cup\,\bigcup_{\,l=1}^{\,q+1}\Gamma_l$
and $\Delta:=\{(a_1,a_1),\ldots,(a_q,a_q),(i,i)\}\cup\,\bigcup_{\,l=1}^{\,q+1}\Gamma_l$. Then $\Gamma$ is a flow and
$\Delta$ is a weak flow.

We denote by $S$ the set of all sources of edges of $\Gamma$ except $h$.
Then the set of all sources of edges of $\Gamma$ is $S\cup\{h\}$
and the set of all sources of edges of $\Delta$ is $S\cup\{i\}$.
For any $t\in S\cup\{h\}$, there exists a unique  $\psi(t)\in(h..i]$ such that $(t,\psi(t))\in\Gamma$.
For any $t\in S\cup\{i\}$, there exists a unique  $\theta(t)\in(h..i]$ such that $(t,\theta(t))\in\Delta$.
Clearly, $\psi$ and $\theta$ are embeddings of $S\cup\{h\}$ and $S\cup\{i\}$, respectively,  into $(h..i]$.

We note the following properties of $\psi$ and $\theta$, following from the fact that each $\Gamma_l$ is a flow
coherent with $r_{\mathbf0}(\lm)|_{(a_{l-1}..a_l)}$, Definition~\ref{definition:pm:2} and
Proposition~\ref{proposition:rbeta}:
{\renewcommand{\labelenumi}{{\rm \theenumi}}
\renewcommand{\theenumi}{{\rm(\alph{enumi}$_\psi$)}}
\begin{enumerate}
\item\label{theorem:socle:2:psi:property:a} $\res_p\lm_t=\res_p\bigl(\lm_{\psi(t)}+1\bigr)$ for all $t\in S\cup\{h\}$;\\[-10pt]
\item\label{theorem:socle:2:psi:property:b} $\psi(t)>t$ for all $t\in S\cup\{h\}$.
\end{enumerate}
}

{\renewcommand{\labelenumi}{{\rm \theenumi}}
\renewcommand{\theenumi}{{\rm(\alph{enumi}$_\theta$)}}
\begin{enumerate}
\item\label{theorem:socle:2:theta:property:a} $\res_p\lm_t=\res_p\bigl(\lm_{\theta(t)}+1\bigr)$ for all $t\in S\cup\{i\}${\rm;}\\[-10pt]
\item\label{theorem:socle:2:theta:property:b} $\theta(t)\ge t$ for all $t\in S\cup\{i\}${\rm;}\\[-10pt]
\end{enumerate}}

Now set $M:=(h..i]\setminus S$ and, for $\epsilon\in\{\0,\1\}$, define
$
w^\epsilon:= S_{h,i}^{\,\epsilon}(M)\,v.
$
We claim that $w^{\0}$ or $w^{\1}$ is the required $U(n-1)$-primitive vector.

We first prove that $w^\epsilon$ is $U(n-1)$-primitive. By~\ref{theorem:socle:2:psi:property:a} and~(\ref{equation:socle:4.5}), we have
{\renewcommand{\labelenumi}{{\rm \theenumi}}
\renewcommand{\theenumi}{{\rm(\alph{enumi}$'_\psi$)}}
\begin{enumerate}
\item\label{theorem:socle:2:psi:property:a'} $\res_p\mu_t=\res_p\bigl(\mu_{\psi(t)}+1\bigr)$ for any $t\in S\cup\{h\}${\rm;}\\[-10pt]
\end{enumerate}}
\noindent
Note that $\psi(h)\in(h..i]\setminus\psi((h..i]\setminus M)$, as $\psi(h)>h$
and $\psi$ is an injection. So, by Lemma~\ref{lemma:constr:1}\ref{lemma:constr:1:part:i} and~\ref{theorem:socle:2:psi:property:a'}, we get
\begin{align*}
 E^{\delta_h}_h\cdots
 E^{\delta_{i-1}}_{i-1}\! S_{h,i}^{\epsilon}(M)v
\! =\! \cond_{\sum\delta=\epsilon}\prod\nolimits_{t\in(h..i]\,\setminus\,\psi((h..i]\setminus M)}\!(\res_p\mu_i\!-\!\res_p(\mu_t+1))v=0.
\end{align*}
Now by Lemma~\ref{lemma:socle:1.5} and weights, it remains to show that
$ E_l^\delta S^{\,\epsilon}_{h,i}(M)v=0$ for any $h\le l<i-1$ and any $\de$.
By Lemma~\ref{lemma:constr:1}\ref{lemma:constr:1:part:ii},
we may assume that $l+1\notin M$ and prove that $ S^{\,\sigma}_{l+1,i}\bigl(M_{(l+1..i]}\bigr)v=0$
for any $\sigma\in\{\0,\1\}$. But this equality follows from Lemma~\ref{lemma:socle:5} if we consider the restrictions
$\psi'=\psi|_{[l+1..i]\setminus M_{(l+1..i]}}$ and $\theta'=\theta|_{\([l+1..i]\setminus M_{(l+1..i]}\)\cup\{i\}}$.
To check that $\psi'$ and $\theta'$ are well defined, note that $l+1\notin M$, and so
$
[l{+}1..i]\setminus M_{(l+1..i]}=[l{+}1..i]\setminus M\subset(h..i]\setminus M=S.
$

Now suppose that $w^{\0}=w^{\1}=0$. Then $ E_h^{\delta_h}\cdots E_{n-1}^{\delta_{n-1}}w^\epsilon=0$ for all $\delta_h,\dots,\de_{n-1}$, and  the equality~(\ref{equation:socle:5}) holds, where $v'= E_i^{\delta_i}\cdots E_{n-1}^{\delta_{n-1}}v$, $v^\sigma= E_i^{\sigma} E_{i+1}^{\delta_{i+1}}\cdots E_{n-1}^{\delta_{n-1}}v$, $M'\!=\!(M\!\setminus\!\{i\})\!\cup\!\{\bar\imath\}$.
Moreover, note that
$S\cup\{h\}=[h..i]\setminus M$ and~\ref{theorem:socle:2:psi:property:a} and~\ref{theorem:socle:2:psi:property:b} hold, so by Lemma~\ref{lemma:constr:2}, $ S_{h,i}^{\,\epsilon}(M)\,v'=0$.
Thus, by~(\ref{equation:socle:5}) and Lemma~\ref{lemma:constr:3}\ref{lemma:constr:3:part:i}, we have
\begin{align*}
0\!=\! E_h^{\delta_h}\cdots E_{n-1}^{\delta_{n-1}} S_{h,i}^{\epsilon}(M)v
\!=\!(-1)^{\epsilon\sum\delta|_{[i..n)}}
\hspace{-5mm}
\sum_{\gamma+\sigma=\epsilon+\delta_i}
\hspace{-4mm}
(-1)^{\1+(\delta_i+\gamma)\epsilon} E_h^{\delta_h}\cdots E_{i-1}^{\delta_{i-1}} S_{h,i}^{\gamma}(M')v^\sigma
\\
=(-1)^{\epsilon\sum\delta|_{[i..n)}}\!
\prod_{t\in(h..i]\setminus\theta((h..i]\setminus M')}\!(-\res_p(\lm_t{+}1))
\sum_{\gamma+\sigma=\epsilon+\delta_i}
\hspace{-3mm}
(-1)^{\1+\gamma\epsilon}H_h^{\gamma+\sum\delta|_{[h..i)}}v^\sigma.
\end{align*}
\noindent
Note that $(h..i]\setminus M'=S\cup\{i\}$.
Since every $\Gamma_l$ is fully coherent with $r_{\mathbf0}(\lm)|_{(a_{l-1}..a_l)}$, we see  that
$\Delta$ is fully coherent with $r_{\mathbf0}(\lm)|_{(h..i]}$.
So, by
Proposition~\ref{proposition:rbeta}\ref{proposition:rbeta:part:2},
we get
$
\prod_{t\in(h..i]\,\setminus\,\theta((h..i]\setminus M')}\res_p(\lm_t{+}1)\ne\mathbf0.
$
Hence
$$
\sum_{\gamma+\sigma=\epsilon+\delta_i}(-1)^{\gamma\epsilon} H_h^{\,\gamma+\sum\delta|_{[h..i)}}v^\sigma=0.
$$
Choosing $\delta_h,\ldots,\delta_{i-1}$ so that $\delta_h+\dots+\delta_{i-1}=\epsilon$, we get
$
(-1)^\epsilon v^{\delta_i}+\bar{\! H}_hv^{\delta_i+\1}=0
$
for any $\epsilon\in\{\0,\1\}$. As $p>2$, have $v^{\delta_i}= E_i^{\delta_i}\cdots E_{n-1}^{\delta_{n-1}}\,v=0$.
As $\delta_i,\ldots,\delta_{n-1}$ are arbitrary and
$v$ is $U(n-1)$-primitive, we get $v=0$, which is a contradiction.
\end{proof}

\section{Normal indices and primitive vectors}\label{normal_indices}

\begin{theorem}\label{theorem:constr:4} Let $\lm\in X(n)$ and $i$ be a $\lm$-normal index.
Then in $L(\lm)$, there exists a nonzero $U(n-1)$-primitive homogeneous vector $v$ of weight
$\lm-\alpha(i,n)$.
\end{theorem}
\begin{proof} Set $\beta:=\res_p\lm_i$. As $i$ is $\la$-normal, $\bigl[\prod_{i\le k<n}r_\beta(\lambda)_k\bigr]$ contains $-_i$.
Assume first that $\beta\ne\mathbf 0$. Then $\bigl[\prod_{i<k<n}r_\beta(\lm)_k\bigr]$ cannot contain any $+$'s.

{\it Case~1:} $\bigl[\prod_{i<k<n}r_\beta(\lm)_k\bigr]=-^r$, with $r>0$.
In  this case, $\bigl[\prod_{i<k\le n} r_\beta(\lm)_k\bigr]$ does not contain $+$ for any value of $\lm_n$.
The required result follows now from Theorem~\ref{theorem:constr:1}.

{\it Case~2: $\bigl[\prod_{i<k<n}r_\beta(\lm)_k\bigr]=\emptyset$.}
As $\lm_i\not\equiv {\mathbf 0}\pmod{p}$,
the required result follows from Theorem~\ref{theorem:constr:3}.

Now let $\beta=\mathbf0$. Then one of the following two conditions must holds:
\begin{itemize}
\item $\bigl[\prod_{i<k<n}r_\mathbf0(\lambda)_k\bigr]=-^r$ for some even non-negative $r$ (cf. Corollary~\ref{corollary:pm:2});\\[-6pt]
\item $r_\mathbf0(\lm)_i=--$ and $\bigl[\prod_{i<k<n}r_\mathbf0(\lambda)_k\bigr]$ contains exactly one $+$.
\end{itemize}

{\it Case~1: $\bigl[\prod_{i<k<n}r_\mathbf0(\lambda)_k\bigr]=-^r$, with $r\ge2$.} In this case,
$\Bigl[\prod_{i<k\le n}r_\mathbf0(\lambda)_k\Bigr]$ does not contain the sign $+$ for any  $\lm_n$, and we can apply Theorem~\ref{theorem:constr:1}.

{\it Case~2: $\bigl[\prod_{i<k<n}r_\mathbf 0(\lm)_k\bigr]=\emptyset$}. In this case,
$\lm_i$ and $\lm_n$ are not both divisible by $p$ by the definition of a normal index, and we can apply Theorem~\ref{theorem:constr:3}.

{\it Case~3:} $\lm_i\=1\pmod p$ (equivalently $r_\mathbf0(\lm)_i=--$), $\bigl[\prod_{i<k<n}r_\mathbf0(\lm)_k\bigr]$ is of the form $+-^m$, and $\lm_n\not\=-1\pmod p$.
Then $m>0$ by Corollary~\ref{corollary:pm:2}, and $r_\mathbf0(\lm)_n\ne++$. So $\bigl[\prod_{i<k\le n}r_\mathbf0(\lm)_k\bigr]$
contains exactly one $+$. The required result now follows from Theorem~\ref{theorem:constr:2}

{\it Case~4:} $\lm_i\=1\pmod p$, $\bigl[\prod_{i<k<n}r_\mathbf0(\lm)_k\bigr]$ is of the form $+-^m$, and $\lm_n\=-1\pmod p$.
By Lemma~\ref{lemma:pm:5}, there exists a section of $r_\mathbf0(\lm)|_{(i..n)}$, see Definition~\ref{definition:pm:3}.
Let $a$ be its maximal element. We have the following properties:
\begin{itemize}
\item $\lm_a\=0\pmod p$;\\[-8pt]
\item $\bigl[\prod_{a<k<n}r_\mathbf0(\lm)_k\bigr]$ does not contain any $+$'s;\\[-6pt]
\item $\bigl[\prod_{i<k\le a}r_\mathbf0(\lm)_k\bigr]=+-$.
\end{itemize}
Applying cases~1 and~2, which we have already considered, we obtain that there exists a nonzero $U(n-1)$-primitive
homogeneous vector $w\in L(\lm)$ of weight $\lm-\alpha(a,n)$.
Next, we apply Theorem~\ref{theorem:socle:2} to obtain a nonzero homogeneous $U(n-1)$-primitive  vector
$v\in U(a)\,w\subset L(\lm)$ of weight $\lm-\alpha(i,n)$.
\end{proof}

\begin{theorem}\label{theorem:constr:5} Let $\lm\in X(n)$, $h$ be a $\lm$-normal index, $h<i<n$, and $\res_p\lm_h=\res_p\lm_i$.
Then for any nonzero homogeneous $U(n-1)$-primitive vector $v\in L(\lm)^{\lm-\alpha(i,n)}$,
there exists a nonzero homogeneous $U(n-1)$-primitive vector $w\in U(i)v$ of weight $\lm-\alpha(h,n)$.
\end{theorem}
\begin{proof} Set $\beta:=\res_p\lm_i$.
As $h$ is $\la$-normal, $\bigl[\prod_{h\le k<n}r_\beta(\lambda)_k\bigr]$ contains $-_h$.
Assume first that $\beta\ne\mathbf 0$.
Then $\bigl[\prod_{h<k\le i}r_\beta(\lm)_k\bigr]$ does not contain any $+$'s.
Now the required result follows from Theorem~\ref{theorem:socle:1}.

Now let $\beta=\mathbf0$. Note that $\lm_h\=0$ or $1\pmod p$, and $\lm_i\=0$ or $1\pmod p$.
Moreover, one of the following conditions holds:
\begin{itemize}
\item $\bigl[\prod_{h<k\le i}r_\mathbf0(\lambda)_k\bigr]$ does not contain any $+$'s;\\[-6pt]
\item $r_\mathbf0(\lm)_h=--$ and $\bigl[\prod_{h<k\le i}r_\mathbf0(\lambda)_k\bigr]$ contains exactly one $+$.
\end{itemize}

{\it Case~1:} $\bigl[\prod_{h<k\le i} r_\mathbf0(\lm)_k\bigr]=-^m$, $\lm_h\=1\pmod p$ and $\lm_i\=1\pmod p$.
In this case the result follows from Theorem~\ref{theorem:socle:1}.

{\it Case~2:} $\bigl[\prod_{h<k\le i} r_{\mathbf0}(\lm)_k\bigr]=-^m$, $\lm_h\=0\pmod p$ and $\lm_i\=1\pmod p$.
In this case the result follows from Theorem~\ref{theorem:socle:0.5}.

{\it Case~3:} $\bigl[\prod_{h<k\le i} r_\mathbf0(\lm)_k\bigr]=-^m$ and $\lm_i\=0\pmod p$.
If $\bigl[\prod_{h<k<i} r_\mathbf0(\lm)_k\bigr]=+^s-^r$, then
$$
-^m=\bigl[\prod_{h<k\le i} r_\mathbf0(\lm)_k\bigr]=\bigl[\bigl[\prod_{h<k<i} r_\mathbf0(\lm)_k\bigr]+-\bigr]=[(+^s-^r)+-]
=\left\{
{\arraycolsep=0pt
\begin{array}{ll}
+^{s+1}-&\text{ if }r=0;\\[6pt]
+^s-^r&\text{ if }r>0.
\end{array}}
\right.
$$
Hence $s=0$ and $r=m>0$. Considering the map $r_\mathbf0(\lm)|_{(h..i)}$,
By Lemma~\ref{lemma:pm:3.5}, we get that there exists an index $a\in(h..i)$ such that
\begin{itemize}
\item $r_\mathbf0(\lm)_a=--$ (that is, $\lm_a\=1\pmod p$);
\item $\bigl[\prod_{a<k<i} r_\mathbf0(\lm)_k\bigr]$ equals either $\emptyset$ or $+-$;
\item $\bigl[\prod_{h<k\le a} r_\mathbf0(\lm)_k\bigr]=-^m$.
\end{itemize}
Regardless of which possibility in the second line holds, we have $\bigl[\prod_{a<k\le i} r_\mathbf0(\lm)_k\bigr]=+-$.
By Theorem~\ref{theorem:socle:2}, there exists a nonzero homogeneous $U(n-1)$-primitive vector
$u\in U(i)\,v$ of weight $\lm-\alpha(a,n)$.
To finish the proof in this case, it remains to apply cases~1 and~2, where we consider $a$ instead of $i$.

{\it Case~4:} $\bigl[\prod_{h<k\le i} r_\mathbf0(\lm)_k\bigr]=+-^m$ and $\lm_i\=1\pmod p$.
In this case $\lm_h\=1\pmod p$.
By Lemma~\ref{lemma:pm:5}, there exists a section of $r_\mathbf0(\lm)|_{(h..i]}$, see Definition~\ref{definition:pm:3}. Let $a$ be its maximal element.
We have the following properties:
\begin{itemize}
\item $\lm_a\=0\pmod p$;\\[-6pt]
\item $\bigl[\prod_{a<k\le i}r_\mathbf0(\lm)_k\bigr]=-^{m-1}$;\\[-6pt]
\item $\bigl[\prod_{h<k\le a}r_\mathbf0(\lm)_k\bigr]=+-$.
\end{itemize}
By Theorem~\ref{theorem:socle:0.5}, there is a nonzero homogeneous $U(n-1)$-primitive vector
$u\in U(i)\,v$ of weight $\lm-\alpha(a,n)$.
Now, by Theorem~\ref{theorem:socle:2}, there is a nonzero homogeneous $U(n-1)$-primitive vector
$w\in U(a)\,u\subset U(a)\,U(i)\,v=U(i)\,v$ of weight $\lm-\alpha(h,n)$.

{\it Case~5: $\bigl[\prod_{h<k\le i} r_\mathbf0(\lm)_k\bigr]=+-^m$ and $\lm_i\=0\pmod p$.}
In this case $\lm_h\=1\pmod p$. The result now follows from Theorem~\ref{theorem:socle:2}.
\end{proof}

\chapter{Main results on $U(n)$}\label{main_results}

\section{Criterion for existence of nonzero $U(n-1)$-primitive vectors}

\begin{proposition}\label{proposition:main:1}
Let $\lm,\mu=(\mu_1,\dots,\mu_n)\in X(n)$ and $\mu\leq \la$. Let $\bar\mu:=(\mu_1,\dots,\mu_{n-1})\in X(n-1)$.
A vector $v\in L(\la)$ is a $\mu$-weight vector of $L(\la)$ if and only if it is a $\bar\mu$-weight vector of the restriction $L(\la)_{U(n-1)}$.
\end{proposition}
\begin{proof}
This follows from the fact that $\nu_1+\dots+\nu_n=\la_1+\dots+\la_n$ for any weight $\nu$ of $L(\la)$, and so any weight of $L(\la)$ can be recovered from its first $n-1$ components.
\end{proof}

Recall
functions $\cf^{\,\delta}$ from Section~\ref{some_coefficients}.
We will abbreviate $\cf=\cf^\0$ and $\odd\cf=\cf^\1$.

\begin{theorem}\label{theorem:main:1}
Let $\lm\in X(n)$ and $1\le i<n$. Suppose that there exist $F,F'\in U^{\leq 0}(n)^{-\alpha(i,n)}$ such that $FL(\lm)^\lm=F'L(\lm)^\lm=0$ and
$$
\cf_{i,n}(F)\in\mathbb F^\times,\quad \overline\cf_{i,n}(F)=0,\quad
\cf_{i,n}(F')=0,\quad\overline\cf_{i,n}(F')\in\mathbb F^\times.
$$
Then there is no nonzero $U(n{-}1)$-primitive vector of weight $\lm-\alpha(i,n)$ in $L(\lm)$.
\end{theorem}
\begin{proof}
Suppose on the contrary that such a vector exists.
Then the vector space $W$ of all $U(n{-}1)$-primitive vectors in $L(\lm)^{\lm-\alpha(i,n)}$ is nontrivial.
In fact, $W$ is a $U^0(n-1)$-subsupermodule of $L(\lm)$. Let $W_0$ be an irreducible
$U^0(n{-}1)$-subsupermodule of $W$. By Proposition~\ref{proposition:intro:4},
we have $W_0\cong{\mathfrak u}(\mu)$, where $\mu:=(\lm_1,\ldots,\lm_{i-1},$ $\lm_i{-}1,\lm_{i+1},\ldots,\lm_{n-1})$. By the universality of Verma modules,
there is a non-zero $U(n-1)$-homomorphism from the $U(n-1)$-Verma module $M(\mu)$ to the restriction $L(\la)_{U(n-1)}$. 
So, using (\ref{EVermaDual}), we have
\begin{align*}
0\ne\Hom_{U(n-1)}\bigl(M(\mu),L(\lm)_{U(n-1)}\bigr)
&\cong\Hom_{U(n-1)}\bigl(L(\lm)_{U(n-1)},M(\mu)^\tau\bigr)
\\
&\cong\Hom_{U^{\leq 0}(n-1)}\bigl(L(\lm)_{U^{\leq 0}(n-1)},{\mathfrak u}(\mu)\bigr).
\end{align*}
So there exists a non-zero $U^{\leq 0}(n-1)$-homomorphism $\phi:L(\lm)\to{\mathfrak u}(\mu)$.
We may assume that $\phi$ homogeneous.
Noting that $\mu=(\lm-\alpha(i,n))|_{U^0(n-1)}$, Proposition~\ref{proposition:main:1} implies that $\phi$ is zero on each $U^0(n)$-weight space of $L(\lm)$ except $L(\lm)^{\lm-\alpha(i,n)}$.
The space $L(\lm)^{\lm-\alpha(i,n)}$ is spanned by vectors of the form
\begin{equation}\label{lemma:main:1}
 F_{i,i_1}^{\,\epsilon_0} F_{i_1,i_2}^{\,\epsilon_1}\cdots F_{i_{N-1},n}^{\,\epsilon_{N-1}}w,
\end{equation}
where $i<i_1<\cdots<i_{N-1}<n$, $\epsilon_0,\ldots,\epsilon_{N-1}\in\{\0,\1\}$ and $w\in L(\lm)^\lm$.
If $N>1$ then $i_1<n$ and we get
$$
\phi\bigl( F_{i,i_1}^{\,\epsilon_0} F_{i_1,i_2}^{\,\epsilon_1}\cdots F_{i_{N-1},n}^{\,\epsilon_{N-1}}w\bigr)=(-1)^{\|\phi\|\epsilon_0}
 F_{i,i_1}^{\,\epsilon_0}\,\phi\bigl( F_{i_1,i_2}^{\,\epsilon_1}\cdots F_{i_{N-1},n}^{\,\epsilon_{N-1}}w\bigr)=0,
$$
since the vector $ F_{i_1,i_2}^{\,\epsilon_1}\cdots F_{i_{N-1},n}^{\,\epsilon_{N-1}}w$
has weight $\lm-\alpha(i_1,n)\neq \lm-\alpha(i,n)$.
Hence
$$
0=\phi(Fw)=\phi\bigl( F_{i,n}\cf_{i,n}(F)\,w+\bar{\! F}_{i,n}\overline\cf_{i,n}(F)\,w\bigr)=\cf_{i,n}(F)\,\phi( F_{i,n}w).
$$
Since $\cf_{i,n}(F)\ne0$, we get $\phi( F_{i,n}w)=0$. Similarly from $0=\phi(F'w)$ and
$\overline\cf_{i,n}(F')\ne0$, we get $\phi(\bar{F}_{i,n}w)=0$.
Hence $\phi$ take all vectors~(\ref{lemma:main:1}) to zero. Thus we have proved that $\phi=0$, which is a contradiction.
\end{proof}


Let $\la\in X(n)$. We introduce a linear map $\ev_\lm$ on $U^{\leq 0}(n)^{-\alpha(i,j)}$ as follows. For
$$
H=\tbinom{ H_1}{a_1}\cdots\tbinom{ H_n}{a_n}\,\,\bar{H}^{b_1}_1\cdots\,\bar{H}^{b_n}_n
$$
with $a_1,\ldots,a_n\ge0$ and $b_1,\ldots,b_n\in\{0,1\}$, we set
$$
\ev_\lm(H):=\tbinom{\lm_1}{a_1}\cdots\tbinom{\lm_n}{a_n}\,\,\bar{H}^{b_1}_1\cdots\,\bar{H}^{b_n}_n,
$$
and extend by linearity to the whole $U^0(n)$.
Now we set
$$
\ev_\lm\Bigl( F^{\,\epsilon_0}_{i,a_1} F^{\,\epsilon_1}_{a_1,a_2}\cdots F^{\,\epsilon_m}_{a_m,j}\,H_{i,\,a_1,\ldots,a_m,j}^{\,\epsilon_0,\ldots,\epsilon_m}\Bigr)
:= F^{\,\epsilon_0}_{i,a_1} F^{\,\epsilon_1}_{a_1,a_2}\cdots F^{\,\epsilon_m}_{a_m,j}\,\ev_\lm\bigl(H_{i,\,a_1,\ldots,a_m,j}^{\,\epsilon_0,\ldots,\epsilon_m}\bigr)
$$
and extend $\ev_\lm$ linearly to the whole $U^{\geq 0}(n)$.
By definitions, we have:

\begin{proposition}\label{proposition:constr:2}
Let $F\in U^{\leq 0}(n)^{-\alpha(i,j)}$ and $\lm\in X(n)$.
Then
$$\cf_{i,j}^\delta(\ev_\lm(F))=\ev_\lm(\cf_{i,j}^\delta(F)).$$
Moreover, $Fv=\ev_\lm(F)v$
for any vector $v$ of weight $\lm$ belonging to some $U(n)$-supermodule.
\end{proposition}

We now prove one of our main results.

\begin{theorem}\label{theorem:main:2} Let $\lm\in X(n)$ and $1\le i<n$.
There exists a nonzero $U(n-1)$-primitive vector of weight $\lm-\alpha(i,n)$ in $L(\lm)$ if and only if
$i$ is a $\lm$-normal index.
\end{theorem}
\begin{proof} In view of Theorem~\ref{theorem:constr:4}, it suffices to prove that if $i$ is not $\lm$-normal, then there is no nonzero
$U(n-1)$-primitive vector of weight $\lm-\alpha(i,n)$ in $L(\lm)$.

So let $i$ be not $\lm$-normal. We set $\beta:=\res_p\lm_i$.
By Definition~\ref{definition:constr:4}, one of the following conditions holds:
{
\renewcommand{\labelenumi}{{\rm \theenumi}}
\renewcommand{\theenumi}{{\rm(\alph{enumi})}}
\begin{enumerate}
\itemsep=4pt
\item\label{theorem:main:2:case:a} $\beta\ne\mathbf0$ and $\bigl[\prod_{i<k<n}r_\beta(\lambda)_k\bigr]$ contains at least one sign $+$;
\item\label{theorem:main:2:case:b} $\beta=\mathbf0$ and $\bigl[\prod_{i<k<n}r_\beta(\lambda)_k\bigr]$ contains at least two signs $+$;
\item\label{theorem:main:2:case:c} $r_\beta(\lambda)_i=+-$ and $\bigl[\prod_{i<k<n}r_\beta(\lambda)_k\bigr]$ contain exactly one sign $+$.
\item\label{theorem:main:2:case:d} $\bigl[\prod_{i<k<n}r_\beta(\lm)_k\bigr]=\emptyset$ and $\lm_i\=\lm_n\=0\pmod p$.
\end{enumerate}}

First, we consider the cases~\ref{theorem:main:2:case:a} and~\ref{theorem:main:2:case:b}.
By Lemmas~\ref{lemma:pm:2} and~\ref{lemma:pm:6}, there exist
$j\in\{i+1,\ldots,n-1\}$ and a flow $\Gamma$ on $(i..j]$ coherent but not fully with $r_\beta(\lm)|_{(i..j]}$,
and not having buds on $(i..j]$.
Let $S$ be the set of all sources of edges of $\Gamma$ and set $M:=\bigl((i..j]\setminus S\bigr)\cup\{\odd{j+1}\}$.
We define the injection $\psi:\{i\}\cup S\to[i..j{+}1)$ as follows.
If $s\in S$ then we set $\psi(s)$ equal to the unique index such that $(s,\psi(s))\in\Gamma$.
As $\Gamma$ is not fully coherent with $r_\beta(\lm)|_{(i..j]}$, there exists $e\in(i..j]$ such that
$r_\beta(\lm)_e$ contains $+$ and no edge of $\Gamma$ ends at $e$. We set $\psi(i):=e$.
Now we can apply Lemma~\ref{lemma:constr:6} to conclude that for any $\epsilon\in\{\0,\1\}$ and $v^+_\lm\in L(\lm)^\lm$, we have
\begin{equation}\label{equation:main:1}
 S^{\,\epsilon}_{i,j+1}(M)\,v^+_\lm=0
\end{equation}

Now, we set
$$
F:=\left\{
\begin{array}{ll}
\ev_\lm( F_{j+1,n} S_{i,j+1}(M))&\text{ if }j+1<n;\\
\ev_\lm( S_{i,j+1}(M))&\text{ if }j+1=n,
\end{array}
\right.
$$
$$
F':=\left\{
\begin{array}{ll}
\ev_\lm( F_{j+1,n}\,\bar{\! S}_{i,j+1}(M))&\text{ if }j+1<n;\\
\ev_\lm(\bar{\! S}_{i,j+1}(M))&\text{ if }j+1=n.
\end{array}
\right.
$$
\noindent
By~(\ref{equation:main:1}) and Proposition~\ref{proposition:constr:2}, we get $FL(\lm)^\lm=F'L(\lm)^\lm=0$.
If $j+1<n$ then by Proposition~\ref{proposition:constr:2} and Lemmas~\ref{lemma:main:2'} and~\ref{lemma:rev:1}, we get
\begin{align*}
&\cf_{i,n}^{\delta}\bigl(\ev_\lm( F_{j+1,n} S^{\,\epsilon}_{i,j+1}(M))\bigr)
\\
=&\ev_\lm\bigl(\cf_{i,n}^{\,\delta}( F_{j+1,n} S^{\,\epsilon}_{i,j+1}(M))\bigr)
=\ev_\lm\bigl(\cf_{i,j+1}^{\,\delta}( S^{\,\epsilon}_{i,j+1}(M))\bigr)
\\
=&\cond_{\epsilon=\delta}\prod\nolimits_{t\in M\setminus\{\odd{j+1}\}}\ev_\lm( C(i,t))
=\cond_{\epsilon=\delta}\prod\nolimits_{t\in(i..j]\setminus S}(\beta-\res_p\lm_t).
\end{align*}
If $j+1=n$ then by Proposition~\ref{proposition:constr:2} and Lemma~\ref{lemma:rev:1}, we still get
\begin{align*}
&\cf_{i,n}^{\,\delta}\bigl(\ev_\lm( S^{\,\epsilon}_{i,j+1}(M))\bigr)
=\ev_\lm\bigl(\cf_{i,j+1}^{\,\delta}( S^{\,\epsilon}_{i,j+1}(M))\bigr)
\\
=&\cond_{\epsilon=\delta}\prod\nolimits_{t\in M\setminus\{\odd{j+1}\}}\ev_\lm( C(i,t))
=\cond_{\epsilon=\delta}\prod\nolimits_{t\in(i..j]\setminus S}(\beta-\res_p\lm_t).
\end{align*}
Thus we have obtained the following formulas:
$$
\begin{array}{ll}
\displaystyle
\cf_{i,n}(F)=\prod\nolimits_{t\in(i..j]\setminus S}(\beta-\res_p\lm_t),&\overline\cf_{i,n}(F)=0,
\\
\cf_{i,n}(F')=0,&\displaystyle\overline\cf_{i,n}(F')=\prod\nolimits_{t\in(i..j]\setminus S}(\beta-\res_p\lm_t).
\end{array}
$$
We claim that $\prod\nolimits_{t\in(i..j]\setminus S}(\beta-\res_p\lm_t)\ne\mathbf0$. Otherwise we would have
$\res_p\lm_t=\beta$ for some $t\in(i..j]\setminus S$.
By Proposition~\ref{proposition:rbeta}\ref{proposition:rbeta:part:1},
we obtain that $r_\beta(\lambda)_t$ contains $-$
ant therefore $t$ is a bud of $\Gamma$ on $(i..j]$ with respect to $r_\beta(\lambda)|_{(i..j]}$, see Definition~\ref{definition:pm:2}\ref{definition:pm:2:7},\ref{definition:pm:2:8}.
This is a contradiction. By Theorem~\ref{theorem:main:1}, we now conclude that there are no nonzero
$U(n-1)$-primitive vectors of weight $\lm-\alpha(i,n)$ in $L(\lm)$.

Now we consider the remaining cases~\ref{theorem:main:2:case:c} and~\ref{theorem:main:2:case:d}.
In the case \ref{theorem:main:2:case:c}, we set $j$ equal to the minimal element of some section of $r_\mathbf0(\lm)|_{(i..n)}$, see Definition~\ref{definition:pm:3}, and in the case  \ref{theorem:main:2:case:d} we set $j:=n$.
In both cases, we have
\begin{itemize}
\item $\lm_i\=\lm_j\=0\pmod p$;
\item $\bigl[\prod_{i<k<j}r_\mathbf0(\lm)_k\bigr]=\emptyset$.
\end{itemize}
By Lemma~\ref{lemma:pm:3}, there exists a flow $\Gamma$ on $(i..j)$ fully coherent with $r_\mathbf0(\lm)|_{(i..j)}$ and having no buds on $(i..j)$. Let $S$ be the set of all sources of edges of $\Gamma$ and define $M:=\bigl((i..j)\setminus S\bigr)\cup\{\bar\jmath\}$.
For any $s\in S$, let $\psi(s)$ be the unique number such that $(s,\psi(s))\in\Gamma$.
Clearly, $\psi$ is an injection of $S=(i..j)\setminus M$ into $(i..j)$.
Now we can apply Lemma~\ref{lemma:constr:7} to conclude that for any $\epsilon\in\{\0,\1\}$ and $v^+_\lm\in L(\lm)^\lm$, we have
\begin{equation}\label{equation:main:2}
 S^{\,\epsilon}_{i,j}(M)\,v^+_\lm=0
\end{equation}

Now, we set
$$
{
\arraycolsep=0pt
F:=\left\{
\begin{array}{ll}
\ev_\lm( F_{j,n} S_{i,j}(M))&\text{ if }j<n;\\[6pt]
\ev_\lm( S_{i,j}(M))&\text{ if }j=n,
\end{array}
\right.}
\quad
{
\arraycolsep=0pt
F':=\left\{
\begin{array}{ll}
\ev_\lm( F_{j,n}\,\bar{\! S}_{i,j}(M))&\text{ if }j<n;\\[6pt]
\ev_\lm(\bar{\! S}_{i,j}(M))&\text{ if }j=n.
\end{array}
\right.}
$$
\noindent
By~(\ref{equation:main:2}) and Proposition~\ref{proposition:constr:2}, we have $FL(\lm)^\lm=F'L(\lm)^\lm=0$.
If $j<n$ then by Proposition~\ref{proposition:constr:2} and Lemmas~\ref{lemma:main:2'} and~\ref{lemma:rev:1}, we get
\begin{align*}
&\cf_{i,n}^{\,\delta}\bigl(\ev_\lm( F_{j,n} S^{\,\epsilon}_{i,j}(M))\bigr)
=\ev_\lm\bigl(\cf_{i,n}^{\,\delta}( F_{j,n} S^{\,\epsilon}_{i,j}(M))\bigr)
=\ev_\lm\bigl(\cf_{i,j}^{\,\delta}( S^{\,\epsilon}_{i,j}(M))\bigr)\\
=&\cond_{\epsilon=\delta}\prod\nolimits_{t\in M\setminus\{\bar\jmath\}}\ev_\lm( C(i,t))
=\cond_{\epsilon=\delta}\prod\nolimits_{t\in(i..j)\setminus S}(\beta-\res_p\lm_t).
\end{align*}
If $j=n$ then by Proposition~\ref{proposition:constr:2} and Lemma~\ref{lemma:rev:1}, we still get
\begin{align*}
&\cf_{i,n}^{\,\delta}\bigl(\ev_\lm( S^{\,\epsilon}_{i,j}(M))\bigr)
=\ev_\lm\bigl(\cf_{i,j}^{\,\delta}( S^{\,\epsilon}_{i,j}(M))\bigr)\\
=&\cond_{\epsilon=\delta}\prod\nolimits_{t\in M\setminus\{\bar\jmath\}}\ev_\lm( C(i,t))
=\cond_{\epsilon=\delta}\prod\nolimits_{t\in(i..j)\setminus S}(\beta-\res_p\lm_t).
\end{align*}

Thus we have obtained the following formulas:
$$
\begin{array}{ll}
\displaystyle
\cf_{i,n}(F)=\prod\nolimits_{t\in(i..j)\setminus S}(\beta-\res_p\lm_t),&\overline\cf_{i,n}(F)=0,
\\
\cf_{i,n}(F')=0,&\displaystyle\overline\cf_{i,n}(F')=\prod\nolimits_{t\in(i..j)\setminus S}(\beta-\res_p\lm_t).
\end{array}
$$
Now, we have $\prod\nolimits_{t\in(i..j)\setminus S}(\beta-\res_p\lm_t)\ne\mathbf0$, since the flow $\Gamma$
has no buds on $(i..j)$.
By Theorem~\ref{theorem:main:1}, we conclude that there are no
 nonzero $U(n-1)$-primitive vector of weight $\lm-\alpha(i,n)$  in $L(\lm)$.
\end{proof}

\section{The socle of the first level}\label{The_socle_of_the_first_level}
Before proving the main result, we need to establish the following simple fact.

\begin{proposition}\label{proposition:main:2}
Let $\lm\in X(n)$ and $v$ be a nonzero $U(n-1)$-primitive vector in $L(\la)^\tau$.
Then $\tau=\lm-\alpha(i_1,n)-\cdots-\alpha(i_m,n)$ for some indices  $i_1,\ldots,i_m\in\{1,\ldots,n-1\}$.
\end{proposition}
\begin{proof} By Proposition~\ref{proposition:intro:1}, the operators of the form
$$
\prod_{1\le k<n} E_{k,n}^{(a_{k,n})}
\cdot\prod_{1\le k<n}\,\bar{\! E}_{k,n}^{\,b_{k,n}}
\cdot\prod_{1\le k<l<n} E_{k,l}^{(a_{k,l})}
\cdot\prod_{1\le k<l<n}\,\bar{\! E}_{k,l}^{\,b_{k,l}},
$$
with $a_{k,l}\in\Z_{\geq 0}$ and $b_{k,l}\in\{0,1\}$, form a basis of $U^+(n)$
(the order of factors in each of these four products is arbitrary but fixed). So one of these operators must move $v$ to a nonzero vector of $L(\lm)^\lm$. As $v$ is $U(n-1)$-primitive, the last two products must be empty.
Thus this operator has weight
$\sum_{1\le k<n}a_{k,n}\alpha(k,n)+\sum_{1\le k<n}b_{k,n}\alpha(k,n)$ and the vector $v$ has weight
$\lm-\sum_{1\le k<n}a_{k,n}\alpha(k,n)-\sum_{1\le k<n}b_{k,n}\alpha(k,n)$.
\end{proof}

Now we prove our second main result:

\begin{theorem}\label{TSocleQ}
Let $\lm\in X(n)$,  $1\leq i<n$, and
$$
\mu:=(\lm_1,\ldots,\lm_{i-1},\lm_i-1,\lm_{i+1},\ldots,\lm_{n-1}).
$$
There exists a $U(n-1)$-subsupermodule of $L(\lm)$ isomorphic to $L(\mu)$ if and only if $i$ is a $\lm$-good index.
\end{theorem}
\begin{proof}
Suppose first that $i$ is not a $\lm$-good index. If $i$ is not $\lm$-normal, then by Theorem~\ref{theorem:main:2},
there is no nonzero $U(n-1)$-primitive vector in $L(\lm)^{\lm-\alpha(i,n)}$.
Then by Proposition~\ref{proposition:main:1}, there is no nonzero $U(n-1)$-primitive vector of weight
$\mu$ with respect to $U^0(n-1)$ in $L(\lm)$. So $L(\mu)$ is not a submodule of the restriction $L(\la)_{U(n-1)}$.
So we may assume that $i$ is $\lm$-normal, but there exists another $\lm$-normal index $h<i$ such that
$\res_p\lm_h=\res_p\lm_i$, see Definition~\ref{definition:main:5}. Assume that $L(\mu)\subseteq L(\la)_{U(n-1)}$. Pick a  nonzero vector $v\in L(\mu)^\mu$. By Proposition~\ref{proposition:main:1},
we get $v\in L(\la)^{\lm-\alpha(i,n)}$. By Theorem~\ref{theorem:constr:5}, there exists a nonzero
$U(n-1)$-primitive vector $w\in U(i)v\subset L(\mu)$ of weight $\lm-\alpha(h,n)$. This contradicts the irreducibility of $L(\mu)$.

Conversely, let $i$ be a $\lm$-good node. Then by Theorem~\ref{theorem:constr:4}, there exists
a nonzero $U(n-1)$-primitive homogeneous vector $v\in L(\lm)^{\lm-\alpha(i,n)}$.
Let $W$ be an irreducible subsupermodule of the $U^0(n-1)$ supermodule $U^0(n-1)v$ generated by $v$.
By Proposition~\ref{proposition:intro:4}, we have $W\cong{\mathfrak u}(\mu)$.
Consider the $U(n-1)$-subsupermodule $M$ of $L(\lm)$ generated by $W$.
It suffices to prove that $M$ is an irreducible $U(n-1)$-supermodule.

If not, then $M$ contains a proper nonzero subsupermodule $M'$. Note that $W\cap M'=0$.
Choose any nonzero weight vector $v'$ in $M'$
of maximal possible weight.
Then $v'$ has weight $\mu-\beta$, where $\beta$ is a sum of positive roots of the form $\alpha(k,l)$ with $1\leq k<l<n$. The maximality of the weight of $v'$
implies that $v'$ is $U(n-1)$-primitive.
Since $\mu=\bigl(\lm-\alpha(i,n)\bigr)|_{[1..n)}$, we obtain $\mu-\beta=\bigl(\lm-\alpha(i,n)-\bigr)|_{[1..n)}-\beta$.
By Proposition~\ref{proposition:main:1}, we get $v'\in L(\lm)^{\lm-\alpha(i,n)-\beta}$. So, by Proposition~\ref{proposition:main:2}, we must have $\beta=\al(h,i)$ and
$v'\in L(\lm)^{\lm-\alpha(h,n)}$ for some $h<i$.
By Theorem~\ref{theorem:main:2}, $h$ is a $\lm$-normal index.
Moreover, by the universal property of Verma modules,
$M$ is a quotient of $M(\mu)$, and so $L(\mu-\al(h,i))$ is a composition factor of $M(\mu)$. By the Linkage  Principle of \cite[Theorem 8.10]{Kleshchev_Brundan_Modular_Representations_of_the_supergroup_Q(n)_I}, it is easy to see that $\res_p\lm_h=\res_p\lm_i$. This contradicts the fact that $i$ is a $\lm$-good index, and so the proof is complete.
\end{proof}

Let $\la\in X(n)$ and consider the irreducible supermodule $L(\la)$. Recall the simple roots $\al_1,\dots,\al_{n-1}$. Define the {\em $j$th level} of $L(\la)$ to be
$$
L(\la)_j:=\bigoplus_{\mu\in X(n)}L(\la)^\mu,
$$
where the sum is over all weights $\mu$ of the form $\mu=\la-j\al_{n-1}-\sum_{i=1}^{n-2}m_i\al_i$. Note that $L(\la)_j$ is invariant with resect to the action of $U(n-1)\subset U(n)$. So we have
$$
L(\la)_{U(n-1)}=\bigoplus_{j\geq 0} L(\la)_j.
$$

We now restate the main results on $U(n)$ obtained above as the results on the first level $L(\la)_1$.

\begin{theorem}\label{TFirstLevelMain}
Let $\la\in X(n)$, and $\mu\in X(n-1)$. Then
\begin{enumerate}
\item[{\rm (i)}] $\Hom_{U(n-1)}(M(\mu),L(\la)_1)\neq 0$ if and only if $\mu=(\la_1,\dots,\la_{n-1})-\eps_i$ for some $\la$-normal index $i$.
\item[{\rm (ii)}] $\Hom_{U(n-1)}(L(\mu),L(\la)_1)\neq 0$ if and only if $\mu=(\la_1,\dots,\la_{n-1})-\eps_i$ for some $\la$-good index $i$.
\item[{\rm (iii)}] Assume in addition that $\la\in X^+_p(n)$ and $\mu\in X^+_p(n-1)$. Then we have $\Hom_{U(n-1)}(V(\mu),L(\la)_1)\neq 0$ if and only if $\mu=(\la_1,\dots,\la_{n-1})-\eps_i$ for some $\la$-normal index $i$.
\end{enumerate}
\end{theorem}
\begin{proof}
Note, using Proposition~\ref{proposition:main:2}, that $\Hom_{U(n-1)}(L(\mu),L(\la)_1)\neq 0$ implies that $\Hom_{U(n-1)}(M(\mu),L(\la)_1)\neq 0$, which in turn  implies that $\mu$ is of the form $(\la_1,\dots,\la_{n-1})-\eps_i$. Now the result follows from Theorems~\ref{theorem:main:2} and \ref{TSocleQ} and the universality of Verma modules $M(\mu)$ and Weyl modules $V(\mu)$.
\end{proof}

\section{Complement pairs}\label{Complement_pairs}
Let $M$ be a $U(n)$-supermodule. We call a pair of vectors $(v,v')$ of $M$ a {\it complement pair} if
\begin{equation}\label{equation:tensor:1}
E^\epsilon_{i,n}v=E_{i,n}^{\epsilon+\1}v'
\end{equation}
for all $i=1,\ldots,n-1$ and $\epsilon\in\{\0,\1\}$.
The set of all complement pairs of $M$ is denoted by $\cp(M)$. Clearly, $\cp(M)$ is an $\mathbb F$-linear space
under componentwise sum and multiplication by scalars.

The space $\cp(M)$ has the following $\Z_2$-grading:
$$
\begin{array}{l}
\cp(M)_\0=\bigl\{(v,v')\in\cp(M)\suchthat v\in M_\0\ \text{and}\ v'\in M_\1\bigl\},\\[3pt]
\cp(M)_\1=\bigl\{(v,v')\in\cp(M)\suchthat v\in M_\1\ \text{and}\ v'\in M_\0\bigl\}.
\end{array}
$$
To prove that $\cp(M)=\cp(M)_\0\oplus\cp(M)_\1$, let $(v,v')\in \cp(M)$ and  decompose  $v=v_\0+v_\1$ and $v'=v'_\0+v'_\1$, where $v_\0,v'_\0\in M_\0$ and $v_\1,v'_\1\in M_\1$. Then $(v,v')=(v_\0,v'_\1)+(v_\1,v'_\0)$ and~(\ref{equation:tensor:1}) gives
$
E^\epsilon_{i,n}v_\0+E^\epsilon_{i,n}v_\1=E^{\epsilon+\1}_{i,n}v'_\0+E^{\epsilon+\1}_{i,n}v'_\1,
$
whence
$
E^\epsilon_{i,n}v_\0=E^{\epsilon+\1}_{i,n}v'_\1,\quad
E^\epsilon_{i,n}v_\1=E^{\epsilon+\1}_{i,n}v'_\0,
$
i.e. $(v_\0,v'_\1)\in\cp(M)_\0$ and $(v_\1,v'_\0)\in\cp(M)_\1$.

Consider the new action of $U(n)$ on $M$ given by $x\bullet m:=(-1)^{\|x\|}xm$ for any homogeneous $x\in U(n)$.
Under this action $M$ is still a $U(n)$-supermodule.
When dealing with this action, it is convenient to consider the opposite grading on $M$.
This new supermodule module is then denoted by $\Pi M$, see~\cite[Section 12.1]{Kbook}.

We denote by $U^0_{ev}(n)$ the subalgebra of $U(n)$ generated by all~$\tbinom{H_i}m$.

\begin{lemma}\label{lemma:tensor:1}
The space
$\cp(M)$ is a $U(n-1)U^0_{ev}(n)$-module under the action $x(v,v')=(xv,x\bullet v')$.
\end{lemma}
\begin{proof} It suffices to prove that $(xv,x\bullet v')\in \cp(M)$ for $x$ being a generator of $U(n-1)U^0_{ev}(n)$
and $(v,v')\in\cp(M)$. We fix some $i\in\{1,\ldots,n-1\}$ and $\epsilon\in\{\0,\1\}$ and
check that $E^\epsilon_{i,n}xv=E_{i,n}^{\epsilon+\1}(x\bullet v')$.

{\em Case~1:} $x$ supercommutes with $E_{i,n}$ and $\bar E_{i,n}$.
Multiplying both sides of~(\ref{equation:tensor:1}) by $x$, we get
$$
(-1)^{\epsilon\|x\|}E^\epsilon_{i,n}xv=(-1)^{(\epsilon+\1)\|x\|}E^{\epsilon+\1}_{i,n}xv',
$$
whence
$$
E^\epsilon_{i,n}xv=(-1)^{\|x\|}E^{\epsilon+\1}_{i,n}xv'=E^{\epsilon+\1}_{i,n}(x\bullet v').
$$

{\em Case~2:} $x=\tbinom{H_i}m$. Apply induction on $m$, the case $m=0$ being obvious.
Multiplying both sides of~(\ref{equation:tensor:1}) by $x$, we get
$$
E^\epsilon_{i,n}\tbinom{H_i+1}mv=E_{i,n}^{\epsilon+\1}\tbinom{H_i+1}mv'.
$$
Hence
$$
E^\epsilon_{i,n}\tbinom{H_i}mv+E^\epsilon_{i,n}\tbinom{H_i}{m-1}v=E_{i,n}^{\epsilon+\1}\tbinom{H_i}mv'+E_{i,n}^{\epsilon+\1}\tbinom{H_i}{m-1}v',
$$
and we are done by the inductive hypothesis.

{\em Case~3:} $x=\tbinom{H_n}m$. Apply induction on $m$, the case $m=0$ being obvious.
Multiplying both sides of~(\ref{equation:tensor:1}) by $x$, we get
$$
E^\epsilon_{i,n}\tbinom{H_n-1}mv=E_{i,n}^{\epsilon+\1}\tbinom{H_n-1}mv'.
$$
Hence
$$
E^\epsilon_{i,n}\tbinom{H_n}mv-E^\epsilon_{i,n}\tbinom{H_n-1}{m-1}v=E_{i,n}^{\epsilon+\1}\tbinom{H_n}mv'-E_{i,n}^{\epsilon+\1}\tbinom{H_n-1}{m-1}v'
$$
and we are done by the inductive hypothesis since 
$\tbinom{H_n-1}{m-1}=\sum_{k+l=m-1}\tbinom{H_n}k\tbinom{-1\vphantom{H_n}}l$.

{\em Case~4:} $x=\bar H_i$. Multiplying both sides of~(\ref{equation:tensor:1}) by $x$, we get
$$
(-1)^\epsilon E^\epsilon_{i,n}\bar H_iv+E^{\epsilon+\1}_{i,n}v=(-1)^{\epsilon+\1}E_{i,n}^{\epsilon+\1}\bar H_iv'+E^\epsilon_{i,n}v',
$$
and the result follows.

{\em Case~5:} $x=E_{l,i}^{(m)}$. Apply induction on $m$, the case $m=0$ being obvious.
Multiplying both sides of~(\ref{equation:tensor:1}) by $x$, we get
$$
E^\epsilon_{i,n}E_{l,i}^{(m)}v+E^\epsilon_{l,n}E_{l,i}^{(m-1)}v=E_{i,n}^{\epsilon+\1}E_{l,i}^{(m)}v'+E_{l,n}^{\epsilon+\1}E_{l,i}^{(m-1)}v'.
$$
Applying the inductive hypothesis, we obtain the required formula.

{\em Case~6:} $x=\bar E_{l,i}$. Multiplying both sides of~(\ref{equation:tensor:1}) by $x$, we get
$$
(-1)^\epsilon E^\epsilon_{i,n}\bar E_{l,i}v+E^{\epsilon+\1}_{l,n}v=(-1)^{\epsilon+\1}E_{i,n}^{\epsilon+\1}\bar E_{l,i}v'+E_{l,n}^\epsilon v'
$$
or
$$
E^\epsilon_{i,n}\bar E_{l,i}v=-E_{i,n}^{\epsilon+\1}\bar E_{l,i}v'=E_{i,n}^{\epsilon+\1}(\bar E_{l,i}\bullet v').
$$

{\em Case~7:} $x=F_{i,l}^{(m)}$. Apply induction on $m$, the case $m=0$ being obvious.
Multiplying both sides of~(\ref{equation:tensor:1}) by $x$, we get
$$
E^\epsilon_{i,n}F_{i,l}^{(m)}v+E^\epsilon_{l,n}F_{i,l}^{(m-1)}v=E_{i,n}^{\epsilon+\1}F_{i,l}^{(m)}v'+E^{\epsilon+\1}_{l,n}F_{i,l}^{(m-1)}v'.
$$
Applying the inductive hypothesis, we obtain the required formula.

{\em Case~8:} $x=\bar F_{i,l}$. Multiplying both sides of~(\ref{equation:tensor:1}) by $x$, we get
$$
(-1)^\epsilon E^\epsilon_{i,n}\bar F_{i,l}v+E^{\epsilon+\1}_{l,n}v=(-1)^{\epsilon+\1}E_{i,n}^{\epsilon+\1}\bar F_{i,l}v'+E_{l,n}^\epsilon v',
$$
or
$$
E^\epsilon_{i,n}\bar F_{i,l}v=-E_{i,n}^{\epsilon+\1}\bar F_{i,l}v'=E_{i,n}^{\epsilon+\1}(\bar F_{i,l}\bullet v'),
$$
as required.
\end{proof}

\begin{remark}\label{remark:tensor:0} {\rm The action of $\bar H_n$ on $\cp(M)$ is not well-defined.
However, if $M$ has a weight space decomposition, then the action of $U^0_{ev}(n)$ yields the weight space decomposition $\cp(M)=\bigoplus_{\mu\in X(n)}\cp(M)^\mu$ such that
a pair $(v,v')\in\cp(M)$ has weight $\mu\in X(n)$ if and only if $v$ and $v'$ both have weight $\mu$.
In fact, the $\Z_2$-grading of $\cp(M)$ and the action of $U(n-1)U^0_{ev}(n)$ on $\cp(M)$ make $\cp(M)$
into a $U(n-1)U^0_{ev}(n)$-subsupermudule of $M\oplus\Pi M$.}
\end{remark}

Denote by $V$ the {\em natural}\, $U(n)$-supermodule. By definition, $V_\0$ has basis $v_1,\ldots,v_n$ and $V_\1$
has basis $\bar v_1,\ldots,\bar v_n$ so that, setting $v_i^\0:=v_i$ and $v_i^\1:=\bar v_i$, we have
$$
X^\delta_{i,j}v^\epsilon_k=\cond_{j=k}v_i^{\epsilon+\delta},
$$
and all elements $X^{(m)}_{i,j}$ and $\tbinom{H_i}m$ with $m>0$ act on $V$ as zero.

Let us consider in more detail the dual natural module $V^*$. It has the grading $V^*=V^*_\0\oplus V^*_\1$, where
$V^*_\0=\{f\in V^*\suchthat f(V_\1)=0\}$ and $V^*_\1=\{f\in V^*\suchthat f(V_\0)=0\}$. The action of $U(n)$ on $V^*$
is then given by
the formula $(xf)(v)=(-1)^{\|x\|\|f\|}f(\eta(x)v)$ for homogeneous $f$ and $x$. Let us write down this action
explicitly. Let $f_i^\epsilon:V\to\mathbb F$ be the linear map given by $f^\epsilon_i(v^\delta_j)=\cond_{i=j\and \epsilon=\delta}1_\mathbb F$.
We have
\begin{align*}
\bigl(X^\delta_{i,j}f_l^\epsilon\bigr)(v^\tau_k)=(-1)^{\epsilon\delta}f_l^\epsilon\bigl(\eta(X^\delta_{i,j})v^\tau_k\bigr)
=(-1)^{\epsilon\delta+\1}f_l^\epsilon\bigl(X^\delta_{i,j}v^\tau_k\bigr)\\
=(-1)^{\epsilon\delta+\1}\cond_{j=k}f_l^\epsilon\bigl(v^{\tau+\delta}_i\bigr)
=(-1)^{\epsilon\delta+\1}\cond_{j=k}\cond_{i=l\and\epsilon=\tau+\delta}1_\mathbb F\\
=(-1)^{\epsilon\delta+\1}\cond_{i=l}\cond_{j=k\and\epsilon+\delta=\tau}=(-1)^{\epsilon\delta+\1}\cond_{i=l}f^{\epsilon+\delta}_j(v^\tau_k).
\end{align*}
Hence we get
$$
X^\delta_{i,j}f_l^\epsilon=\cond_{i=l}(-1)^{\1+\epsilon\delta}f^{\epsilon+\delta}_j.
$$
Clearly all elements $X^{(m)}_{i,j}$ and $\tbinom{H_i}m$ with $m>0$ act on $V^*$ as zero.
As usual, we set $f_i:=f_i^\0$ and $\bar f_i:=f_i^\1$.

We introduce the  linear map $e:\cp(M)\to M\otimes V^*$ by
\begin{equation}\label{equation:tensor:2}
e(v,v'):=v\otimes f_n+\sum_{h=1}^{n-1}E_{h,n}v\otimes f_h-(-1)^{\|v'\|}v'\otimes \bar f_n-(-1)^{\|v'\|}\sum_{h=1}^{n-1}E_{h,n}v'\otimes \bar f_h
\end{equation}
for any homogeneous pair $(v,v')\in \cp(M)$, which extends by linearity. Note that this formula works even
for $v'=0$, in which case $\|v'\|$ is not defined.

For any $U(n)$-module $W$, we define the {\em $Y$-invariants} of $W$ to be
$$
W^Y:=\bigl\{w\in W\suchthat E_{i,n}^{(m)}w=\bar E_{i,n} w=0\text{ for all $1\leq i<n$ and }m>0\bigr\}.
$$

The following lemma is an analogue of \cite[Lemma 5.3]{BKtr}.

\begin{lemma}\label{lemma:tensor:2}
For any $U(n)$-supermodule $M$ with weight space decomposition, the map $e:\cp(M)\to M\otimes V^*$ is an even injective homomorphism of $U(n-1)$-supermodules whose image
contains $(M\otimes V^*)^Y$.
\end{lemma}
\begin{proof}
The injectivity is straightforward from the linear independence of $f_1,\ldots,f_n,\bar f_1,\ldots,\bar f_n$.

To show that $e$ is a homomorphism of $U(n-1)$-modules, it suffices to prove that
$
xe(v,v')=e(x v,x\bullet v'),
$
for any generator $x$ of $U(n-1)$ and homogeneous pair $(v,v')$.
Note that the last condition implies $\|v\|+\|v'\|=\1$ if $v\ne0$ and $v'\ne0$.

{\em Case 1:} $x=\tbinom{H_i}m$ with $i<n$ and $m>0$. We may assume that $v$ and $v'$ are weight vectors. Then
the result follows from the fact that $f^\delta_h$ has weight $-\epsilon_h$ and $E_{h,n}$
has weight $\alpha(h,n)=\epsilon_h-\epsilon_n$.

{\em Case 2:} $x=\bar H_i$ with $i<n$.
Multiplying~(\ref{equation:tensor:2}) by $x$, we get
\begin{align*}
\bar H_ie(v,v')=\bar H_iv\otimes f_n+\sum_{h=1}^{n-1}\bar H_iE_{h,n}v\otimes f_h-(-1)^{\|v\|}E_{i,n}v\otimes\bar f_i
\\
-(-1)^{\|v'\|}\bar H_iv'\otimes \bar f_n-(-1)^{\|v'\|}\sum_{h=1}^{n-1}\bar H_iE_{h,n}v'\otimes \bar f_h- E_{i,n}v'\otimes f_i\\
=\bar H_iv\otimes f_n+\sum_{h=1}^{n-1}E_{h,n}\bar H_iv\otimes f_h+\bar E_{i,n}v\otimes f_i-(-1)^{\|v\|}E_{i,n}v\otimes\bar f_i
\\
-(-1)^{\|v'\|}\bar H_iv'\otimes \bar f_n\!-\!(-1)^{\|v'\|}\sum_{h=1}^{n-1}E_{h,n}\bar H_iv'\otimes \bar f_h
\!-\!(-1)^{\|v'\|}\bar E_{i,n}v'\otimes \bar f_i
\!-\!E_{i,n}v'\otimes f_i
\\
=e(\bar H_iv,\bar H_i\bullet v')+\Bigl[\bar E_{i,n}v\otimes f_i-E_{i,n}v'\otimes f_i\Bigr]
\\
-\Bigl[(-1)^{\|v\|}E_{i,n}v\otimes\bar f_i+(-1)^{\|v'\|}\bar E_{i,n}v'\otimes \bar f_i\Bigr].
\end{align*}
The sums in the square brackets equal zero by~(\ref{equation:tensor:1}) and $\|v\|+\|v'\|=\1$ (if $v=0$ or $v'=0$, then both summands in the second pair of square brackets equal zero).

{\em Case 3:} $x=E^{(m)}_{i,j}$ with $m>0$, $1\leq i<j<n$. Multiplying~(\ref{equation:tensor:2}) by $x$, we get
\begin{align*}
E^{(m)}_{i,j}e(v,v')=E^{(m)}_{i,j}v\otimes f_n+\sum_{h=1}^{n-1}E^{(m)}_{i,j}E_{h,n}v\otimes f_h-E^{(m-1)}_{i,j}E_{i,n}v\otimes f_j
\end{align*}
\begin{align*}
-(-1)^{\|v'\|}E^{(m)}_{i,j}v'\otimes \bar f_n-(-1)^{\|v'\|}\sum_{h=1}^{n-1}E^{(m)}_{i,j}E_{h,n}v'\otimes \bar f_h
\\
+(-1)^{\|v'\|}E^{(m-1)}_{i,j}E_{i,n}v'\otimes\bar f_j\\
=E^{(m)}_{i,j}v\otimes f_n+\sum_{h=1}^{n-1}E_{h,n}E^{(m)}_{i,j}v\otimes f_h-(-1)^{\|v'\|}E^{(m)}_{i,j}v'\otimes \bar f_n\\
-(-1)^{\|v'\|}\sum_{h=1}^{n-1}E_{h,n}E^{(m)}_{i,j}v'\otimes \bar f_h
=e(E^{(m)}_{i,j}v,E^{(m)}_{i,j}\bullet v').
\end{align*}

{\em Case 4:} $x=\bar E_{i,j}$ with $1\leq i<j<n$. Multiplying~(\ref{equation:tensor:2}) by $x$, we get
\begin{align*}
\bar E_{i,j}e(v,v')=
\bar E_{i,j}v\otimes f_n+\sum_{h=1}^{n-1}\bar E_{i,j}E_{h,n}v\otimes f_h-(-1)^{\|v\|}E_{i,n}v\otimes\bar f_j
\\
-(-1)^{\|v'\|}\bar E_{i,j}v'\otimes \bar f_n-(-1)^{\|v'\|}\sum_{h=1}^{n-1}\bar E_{i,j}E_{h,n}v'\otimes \bar f_h
- E_{i,n}v'\otimes f_j
\\
=\bar E_{i,j}v\otimes f_n+\sum_{h=1}^{n-1}E_{h,n}\bar E_{i,j}v\otimes f_h+\bar E_{i,n}v\otimes f_j-(-1)^{\|v\|}E_{i,n}v\otimes\bar f_j
\\
-(-1)^{\|v'\|}\bar E_{i,j}v'\otimes \bar f_n-(-1)^{\|v'\|}\sum_{h=1}^{n-1}E_{h,n}\bar E_{i,j}v'\otimes \bar f_h
-(-1)^{\|v'\|}\bar E_{i,n}v'\otimes \bar f_j
\\
- E_{i,n}v'\otimes f_j
\,=\,e(\bar E_{i,j}v,\bar E_{i,j}\bullet v')+[\bar E_{i,n}v\otimes f_j
- E_{i,n}v'\otimes f_j]
\\-[(-1)^{\|v\|}E_{i,n}v\otimes\bar f_j+(-1)^{\|v'\|}\bar E_{i,n}v'\otimes \bar f_j]
\end{align*}
The sums in the square brackets equal zero by~(\ref{equation:tensor:1}) and $\|v\|+\|v'\|=\1$.

{\em Case 5:} $x=F^{(m)}_{i,j}$ with $m>0$, $1\leq i<j<n$. Multiplying~(\ref{equation:tensor:2}) by $x$, we get
\begin{align*}
F^{(m)}_{i,j}e(v,v')=F^{(m)}_{i,j}v\otimes f_n+\sum_{h=1}^{n-1}F^{(m)}_{i,j}E_{h,n}v\otimes f_h-F^{(m-1)}_{i,j}E_{j,n}v\otimes f_i\\
-(-1)^{\|v'\|}F^{(m)}_{i,j}v'\otimes \bar f_n-(-1)^{\|v'\|}\sum_{h=1}^{n-1}F^{(m)}_{i,j}E_{h,n}v'\otimes \bar f_h\\
+(-1)^{\|v'\|}F^{(m-1)}_{i,j}E_{j,n}v'\otimes \bar f_i,
\end{align*}
which is easily checked to equal $e(F^{(m)}_{i,j}v,F^{(m)}_{i,j}\bullet v')$.

{\em Case 6:} $x=\bar F_{i,j}$ with $1\leq i<j<n$. Multiplying~(\ref{equation:tensor:2}) by $x$, we get
\begin{align*}
\bar F_{i,j}e(v,v')=\bar F_{i,j}v\otimes f_n+\sum_{h=1}^{n-1}\bar F_{i,j}E_{h,n}v\otimes f_h-(-1)^{\|v\|}E_{j,n}v\otimes\bar f_i
\\
-(-1)^{\|v'\|}\bar F_{i,j}v'\otimes \bar f_n-(-1)^{\|v'\|}\sum_{h=1}^{n-1}\bar F_{i,j}E_{h,n}v'\otimes \bar f_h- E_{j,n}v'\otimes f_i\\
=\bar F_{i,j}v\otimes f_n+\sum_{h=1}^{n-1}E_{h,n}\bar F_{i,j}v\otimes f_h+\bar E_{j,n}v\otimes f_i-(-1)^{\|v\|}E_{j,n}v\otimes\bar f_i
\end{align*}
\begin{align*}
-(-1)^{\|v'\|}\bar F_{i,j}v'\otimes \bar f_n-(-1)^{\|v'\|}\sum_{h=1}^{n-1}E_{h,n}\bar F_{i,j}v'\otimes \bar f_h
\\
-(-1)^{\|v'\|}\bar E_{j,n}v'\otimes\bar f_i- E_{j,n}v'\otimes f_i
=e(\bar F_{i,j}v,\bar F_{i,j}\bullet v')+\Bigl[\bar E_{j,n}v\otimes f_i-E_{j,n}v'\otimes f_i\Bigr]
\\
-\Bigl[(-1)^{\|v\|}E_{j,n}v\otimes\bar f_i+(-1)^{\|v'\|}\bar E_{j,n}v'\otimes\bar f_i\Bigr].
\end{align*}
The sums in the square brackets equal zero by~(\ref{equation:tensor:1}) and $\|v\|+\|v'\|=\1$.

Finally, we prove that every element $w\in(M\otimes V^*)^Y$ is in the image of~$e$.
We may assume that $w$ is homogeneous, say, of parity $\epsilon$.
We can write
$$w=\sum_{h=1}^nv_h\otimes f_h+\sum_{h=1}^nv'_h\otimes\bar f_h,
$$
where $v_h\in M_\epsilon$ and $v'_h\in M_{\epsilon+\1}$. For any $i=1,\ldots,n-1$, we have
$$
0=E_{i,n}w=\sum_{h=1}^nE_{i,n}v_h\otimes f_h-v_i\otimes f_n+\sum_{h=1}^nE_{i,n}v'_h\otimes\bar f_h-v'_i\otimes \bar f_n.
$$
Hence $v_i=E_{i,n}v_n$ and $v'_i=E_{i,n}v'_n$. On the other hand,
$$
0=\bar E_{i,n}w=\sum_{h=1}^n\bar E_{i,n}v_h\otimes f_h-(-1)^{\epsilon}v_i\otimes\bar f_n+\sum_{h=1}^n\bar E_{i,n}v'_h\otimes\bar f_h+(-1)^{\epsilon+\1}v'_i\otimes f_n.
$$
Hence $v_i=(-1)^{\epsilon}\bar E_{i,n}v'_n$ and $v'_i=(-1)^\epsilon\bar E_{i,n}v_n$. Combining these formulas with our previous formulas
for $v_i$ and $v'_i$, we get
$$
\left\{
\arraycolsep=2pt
\begin{array}{l}
E_{i,n}v_n=\bar E_{i,n}(-1)^{\epsilon} v'_n;\\[3pt]
\bar E_{i,n}v_n=E_{i,n}(-1)^\epsilon v'_n.
\end{array}
\right.
$$
These formulas show that $(v_n,(-1)^{\epsilon}v'_n)\in\cp(M)_\epsilon$. Finally, we now have
\begin{align*}
e(v_n,(-1)^{\epsilon}v'_n)=v_n\otimes f_n+\sum_{h=1}^{n-1}E_{h,n}v_n\otimes f_h
\\
-(-1)^{\|v'\|+\epsilon}v'_n\otimes \bar f_n-(-1)^{\|v'\|+\epsilon}\sum_{h=1}^{n-1}E_{h,n}v'_n\otimes \bar f_h
\\
=v_n\otimes f_n+\sum_{h=1}^{n-1}E_{h,n}v_n\otimes f_h+v'_n\otimes \bar f_n+\sum_{h=1}^{n-1}E_{h,n}v'_n\otimes \bar f_h=w,
\end{align*}
as desired.
\end{proof}

\begin{lemma} \label{LNew}
Let $\la\in X(n)$.
Every $U(n)$-primitive vector of
$L(\lm)\otimes V^*$ has weight of the form $\lm-\epsilon_j$ for some $1\leq j\leq n$.
\end{lemma}
\begin{proof}
Apply Corollary~\ref{CVermaTensHom} and the universality of Verma modules.
\end{proof}

Define the {\em $j$th level} of $\cp(L(\lm))$ to be
$$
\cp(L(\lm))_j:=\bigoplus_{\mu\in X(n)}\cp(L(\lm))^\mu,
$$
where the sum is over all weights $\mu$ of the form $\mu=\la-j\al_{n-1}-\sum_{i=1}^{n-2}m_i\al_i$
(cf. the definition of the $j$th level of $L(\lm)$ in Section~\ref{The_socle_of_the_first_level}).
The following lemma is an analogue of \cite[Proposition 5.4]{BKtr}.

\begin{lemma}\label{lemma:tensor:3}
For any $\lm\in X(n)$, the restriction of $e$ to $\cp(L(\lm))_0\oplus \cp(L(\lm))_1$ gives an even
isomorphism of $U(n-1)$-supermodules
$$
e':\cp(L(\lm))_0\oplus \cp(L(\lm))_1\to(L(\la)\otimes V^*)^Y,
$$
which takes vectors of weight $\mu\in X(n)$ to vectors of weight $\mu-\epsilon_n$. In particular, $e'$ establishes an isomorphism between the space of $U(n-1)$-primitive vectors of $\cp(L(\lm))_0\oplus \cp(L(\lm))_1$ and the space of $U(n)$-primitive vectors in $L(\la)\otimes V^*$.
\end{lemma}
\begin{proof} We first check that the image of $e'$ lies in $(L(\la)\otimes V^*)^Y$. Take some homogeneous weight pair $(v,v')\in\cp(L(\lm))_0\oplus \cp(L(\lm))_1$.
As usual, $\|v\|+\|v'\|=\1$ if $v\ne0$ and $v'\ne0$.
By the definition of the level of $L(\lm)$ in Section~\ref{The_socle_of_the_first_level} and Remark~\ref{remark:tensor:0}, we have $v,v'\in L(\lm)_0\oplus L(\lm)_1$.

Take some $E_{i,n}^{(m)}$ with $i<n$ and $m>1$.
We have $E_{i,n}^{(k)}f_n=E_{i,n}^{(k)}\bar f_n=0$ and $E_{i,n}^{(k)}E_{h,n}v=E_{i,n}^{(k)}E_{h,n}v'=0$ for any $k>0$.
This formulas and~(\ref{equation:tensor:2}) imply
\begin{multline*}
E_{i,n}^{(m)}e(v,v')=E_{i,n}^{(m)}v\otimes f_n+\sum_{h=1}^{n-1}E_{h,n}v\otimes E_{i,n}^{(m)}f_h\\
-(-1)^{\|v'\|}E_{i,n}^{(m)}v'\otimes \bar f_n-(-1)^{\|v'\|}\sum_{h=1}^{n-1}E_{h,n}v'\otimes E_{i,n}^{(m)}\bar f_h=0
\end{multline*}
in view of $m>1$.

Now consider $E_{i,n}^\epsilon$. Multiplying (\ref{equation:tensor:2}) by $E_{i,n}^\epsilon$, we get
\begin{align*}
E_{i,n}^\epsilon e(v,v')=E_{i,n}^\epsilon v\otimes f_n+(-1)^{\epsilon\|v\|}\sum_{h=1}^{n-1}E_{h,n}v\otimes E_{i,n}^\epsilon f_h
\\
-(-1)^{\|v'\|}E_{i,n}^\epsilon v'\otimes \bar f_n-(-1)^{\|v'\|+\epsilon\|v'\|}\sum_{h=1}^{n-1}E_{h,n}v'\otimes E_{i,n}^\epsilon\bar f_h
\\
=E_{i,n}^\epsilon v\otimes f_n-(-1)^{\epsilon\|v\|}E_{i,n}v\otimes f^\epsilon_n
\\
-(-1)^{\|v'\|}E_{i,n}^\epsilon v'\otimes \bar f_n+(-1)^{\|v'\|+\epsilon\|v'\|+\epsilon}E_{i,n}v'\otimes f^{\epsilon+\1}_n.
\end{align*}
In the case $\epsilon=\0$, the right hand side is trivially zero, so we consider the case $\epsilon=\1$.
We have
\begin{multline*}
E_{i,n}^\epsilon e(v,v')=\Bigl[\bar E_{i,n} v\otimes f_n- E_{i,n}v'\otimes f_n\Bigr]\\
-\Bigl[(-1)^{\|v\|}E_{i,n}v\otimes \bar f_n+(-1)^{\|v'\|}\bar E_{i,n} v'\otimes \bar f_n\Bigr],
\end{multline*}
which equals zero by~(\ref{equation:tensor:1}) and~$\|v\|+\|v'\|=\1$.

It remains now to prove the surjectivity of $e'$. By Lemma~\ref{lemma:tensor:2}, any vector $w\in (M\otimes V^*)^Y$ is of the form $e(v,v')$ for $(v,v')\in\cp(L(\la))$. We may assume that $w$ is a homogeneous weight vector of weight $\nu$.
We shall prove by induction on $\nu$ that $(v,v')$ lies in $\cp(L(\lm))_0\oplus \cp(L(\lm))_1$.

If $(v,v')$ is $U(n-1)$-primitive, then $e(v,v')$ is a $U(n)$-primitive
vector in $L(\la)\otimes V^*$, since we have already shown in Lemma~\ref{lemma:tensor:2} that it is $Y$-invariant. By Lemma~\ref{LNew}, every $U(n)$-primitive vector of
$L(\lm)\otimes V^*$ has weight of the form $\lm-\epsilon_j$ for some $1\leq j\leq n$. So $\nu=\la-\eps_j$. But then the weight of $(v,v')$ is $\la-\eps_j+\eps_n$, so $(v,v')$ lies in $\cp(L(\lm))_0\oplus \cp(L(\lm))_1$, as required.
Otherwise, if $(v,v')\in\cp(L(\la))$ is not $U(n-1)$-primitive, then we can find $1 < i < j < n$ and $k > 0$
such that $E_{i,j}^{(k)}(v,v')\neq 0$.
By Lemma~\ref{lemma:tensor:2}, we have $e(E_{i,j}^{(k)}(v,v'))=E_{i,j}^{(k)}w\in(M\otimes V^*)^Y$.
Hence by induction $E_{i,j}^{(k)}(v,v')$ lies in $\cp(L(\lm))_0\oplus \cp(L(\lm))_1$,
so $(v,v')$ does too by weights.
\end{proof}

\section{Primitive vectors in $L(\lm)\otimes V^*$}\label{Primitive_vectors_in_L(lm)otimesV*}
In view of Lemma~\ref{lemma:tensor:3}, $U(n)$-primitive vectors in $L(\lm)\otimes V^*$
are closely connected with $U(n-1)$-primitive
pairs in $\cp(L(\lm))_1$.
First, we prove the following simple result.

\begin{lemma}\label{lemma:tensor:4}
Let $\mu\in X(n{+}1)$ and $w\in L(\mu)^{\mu}$ be a nonzero homogeneous vector such that
$U^0(n)w$ is irreducible as a $U^0(n)$-supermodule.
Then $U(n)w\cong L(\mu|_{[1..n)})$ as $U(n)$-supermodules.
\end{lemma}
\begin{proof}
We just need to prove that $U(n)w$ is irreducible. Suppose this is not true.
By the universality of Verma modules (Lemma~\ref{Vermamod}),
$U(n)w$ is a quotient of $M(\mu|_{[1..n)})$. Hence $U(n)w$ contains a nonzero $U(n)$-primitive
vector $v$ of $U(n)$-weight $\nu<\mu|_{[1..n)}$.
Clearly, $v$ has $U(n{+}1)$-weight $(\nu,\mu_{n+1})<\mu$ and is $U(n{+}1)$-primitive.
This contradicts the irreducibility of $L(\mu)$ by Proposition~\ref{proposition:intro:5}.
\end{proof}

To deal with $U(n-1)$-primitive pairs in $\cp(L(\lm))_1$, we will consider
a weight $\tilde\lm=(\lm_1,\ldots,\lm_n,d)\in X(n+1)$, where $d\=0\pmod p$.
Note that if $p>0$, then $d$ can be chosen so that $d\le\lm_n$,
in which case $\tilde\lm\in X_p^+(n+1)$ as long as $\la\in X^+_p(n)$.
On the other hand, if $p=0$ then $d=0$, and so it could happen that $\tilde\la\not \in X^+_0(n+1)$ even if
$\la\in X^+_0(n)$.

\begin{lemma}\label{lemma:tensor:5}
Let $\lm\in X(n)$ and $i$ be a tensor $\lm$-normal index.
Then $L(\lm)\otimes V^*$ contains a nonzero primitive homogeneous vector of weight
$\lm-\epsilon_i$.
\end{lemma}
\begin{proof}
If $i=n$ and $v\in L(\la)^\la$, then $v\otimes f_n$ is primitive of weight $\lm-\epsilon_n$.
Now let $i<n$. By Lemma~\ref{lemma:tensor:3}, it suffices to construct a nonzero $U(n-1)$-primitive
homogeneous pair in $\cp(L(\lm))$ of weight $\lm-\alpha(i,n)$. Let $\beta:=\res_p\la_i$.
If $\bigl[\prod_{i<k\le n}r_\beta(\lambda)_k\bigr]=\emptyset$ the required complement
pair comes from Theorem~\ref{theorem:constr:1}. So we may assume that
 $\bigl[\prod_{i<k\le n}r_\beta(\lambda)_k\bigr]\ne\emptyset$.

Note that
$
\textstyle
\bigl[\prod_{i\le k<n+1}r_\beta(\tilde\lambda)_k\bigr]=\bigl[\prod_{i\le k\le n}r_\beta(\lambda)_k\bigr]
$
contains $-_i$, and
$$
\textstyle
\bigl[\prod_{i<k<n+1}r_\beta(\tilde\lambda)_k\bigr]=\bigl[\prod_{i<k\le n}r_\beta(\lambda)_k\bigr]\ne\emptyset.
$$
Hence by Definition~\ref{definition:constr:4}, we have that $i$ is a
$\tilde\lm$-normal index.
By Theorem~\ref{theorem:constr:4}, there is a nonzero $U(n)$-primitive homogeneous vector $u$
of weight~$\tilde\lm-\alpha(i,n+1)$ in $L(\tilde\lm)$. Set $v:=\bar E_nu$ and $v':=E_nu$. Because of its weight,
the vector $u$ is not $U(n+1)$-primitive, whence $v\ne0$ or $v'\ne0$.
If $1\leq i<n$ and $\epsilon\in\{\0,\1\}$ then $E_{i,n}^\epsilon v=[E_{i,n}^\epsilon,\bar E_n]u=E_{i,n+1}^{\epsilon+\1}u$
and $E_{i,n}^{\epsilon+\1} v'=[E_{i,n}^{\epsilon+\1},E_n]u=E_{i,n+1}^{\epsilon+\1}u$. Hence $E_{i,n}^\epsilon v=E_{i,n}^{\epsilon+\1} v'$,
which means that $(v,v')$ is a complement pair for $L(\tilde\lm)$.

We claim that $v,v'\in U(n)w$
for any nonzero homogeneous vector $w\in L(\tilde\lm)^{\tilde\lm}$. Indeed
by the irreducibility of $L(\tilde\lm)$ and the triangular decomposition of $U(n+1)$,
we have $u=Fw$ for some $F\in U^{\leq0}(n+1)^{-\alpha(i,n+1)}$. Hence $E_nF\=K$ and $\bar E_nF\=K'\pmod{I^+_n}$
for some $K,K'\in U^{\leq0}(n+1)^{-\alpha(i,n)}=U^-(n)^{-\alpha(i,n)}U^0(n+1)$. Hence $v=\bar E_nFw=K'w$
and $v'=E_nFw=Kw$.
However $\tbinom{H_{n+1}}m$ and $\bar H_{n+1}$ act as zero on
$w\in L(\tilde\lm)^{\tilde\lm}$ by~Proposition~\ref{proposition:intro:4}(ii).
So we can rewrite $v=L'w$ and $v=Lw$ for some $L,L'\in U^-(n)^{-\alpha(i,n)}U^0(n)=U^{\leq0}(n)^{-\alpha(i,n)}$.

Since $\tilde\lm_{n+1}\=0$, we get that $U^0(n)w=U^0(n{+}1)w\cong\mathfrak u(\tilde\lm)$
is irreducible as a $U^0(n)$-supermodule. Hence by Lemma~\ref{lemma:tensor:4}, we get $U(n)w\cong L(\lm)$.
Thus we have actually proved that $(v,v')\in\cp(L(\lm))_1$. Finally, for all $1\leq j<n-1$ we have
$$
E_j^\delta e(v,v')=e(E^\delta_jv,E^\delta_j\bullet v')=e(E^\delta_jv,(-1)^\delta E^\delta_j v')=e(0,0)=0.
$$
Therefore $(v,v')$ is $U(n-1)$-primitive.
\end{proof}

Now we are going to prove a theorem similar to Theorem~\ref{theorem:main:1}

\begin{theorem}\label{theorem:tensor:1}
Let $\lm\in X(n)$ and $1\le i<n$. If there exists $F\in U^{\leq 0}(n)^{-\alpha(i,n)}$
such that $FL(\lm)^\lm=0$, $\cf_{i,n}(F)$ is odd, $\overline\cf_{i,n}(F)$ is even and
$$
\ev_\lm\Bigl(\cf_{i,n}(F)^2-\overline\cf_{i,n}(F)^2\Bigr)\in\mathbb F^\times,
$$
then there is no nonzero $U(n{-}1)$-primitive pair in $\cp(L(\lm))$ of weight $\lm-\alpha(i,n)$.
\end{theorem}
\begin{proof} Suppose on the contrary that such a pair exists.
Then the vector space $W$ of all $U(n{-}1)$-primitive pairs in $\cp(L(\lm))^{\lm-\alpha(i,n)}$ is nonzero.
In fact, $W$ is a $U^0(n-1)$-subsupermodule of $\cp(L(\lm))$.
Let $W_0$ be an irreducible
$U^0(n{-}1)$-subsupermodule of $W$. By Proposition~\ref{proposition:intro:4},
we have $W_0\cong{\mathfrak u}(\mu)$,
where $\mu:=(\lm_1,\ldots,\lm_{i-1},$ $\lm_i{-}1,\lm_{i+1},\ldots,\lm_{n-1})$.
By 
the universality of Verma modules, there is a non-zero homogeneous $U(n{-}1)$-homomorphism $\theta:M(\mu)\to\cp(L(\la))$.

Consider the projections $\pr:\cp(L(\lm))\to L(\lm)$ and $\pr':\cp(L(\lm))\to\Pi L(\lm)$ to the first and
the second components, respectively. These maps are degree $\0$ homomorphisms of  $U(n-1)$-supermodules. We set $\psi:=\pr\circ\,\theta$
and $\psi':=\pr'\circ\,\theta$. Thus $\psi:M(\mu)\to L(\lm)$ and $\psi':M(\mu)\to \Pi L(\lm)$ are homogeneous
$U(n-1)$-homomorphisms of the same parity such that
\begin{equation}\label{equation:tensor:2.5}
E^\epsilon_{i,n}\psi(v)=E^{\epsilon+\1}_{i,n}\psi'(v)\qquad\qquad(\epsilon\in\{\0,\1\},\ v\in M(\mu)).
\end{equation}

Let $T(\mu)=(M(\mu)^\mu)^*$. So $T(\mu)$ is the superspace of all linear functions $f:M(\mu)^\mu\to\mathbb F$ with natural grading
$T(\mu)_\0=\{f\in T(\mu)\suchthat f(M(\mu)^\mu_\1)=0\}$, $T(\mu)_\1=\{f\in T(\mu)\suchthat f(M(\mu)^\mu_\0)=0\}$.
We make $T(\mu)$ into a $U^{\leq 0}(n-1)$-supermodule by the following rules:
any element $x\in U^{\leq 0}(n-1)$ of strictly negative weight acts as zero on $T(\mu)$; $(xf)(v)=f(\tau_{n-1}(x)v)$ for any $x\in U^0(n-1)$
and $v\in M(\mu)^\mu$.

Consider the map $\rho:M(\mu)^{\tau_{n-1}}\to T(\mu)$ which sends
a linear map $f:M(\mu)\to\mathbb F$ to its restriction $f|_{M(\mu)^\mu}$. We claim that $\rho$ is an even homomorphism
of $U^{\leq0}(n-1)$-supermodules. Indeed, take an arbitrary $v\in M(\mu)^\mu$.
If $x\in U^{\leq 0}(n-1)$ has weight $<0$, then $\tau_{n-1}(x)$ has weight $>0$,
so
$
\rho(xf)(v)=(xf)(v)=f(\tau_{n-1}(x)v)=f(0)=0.
$
On the other hand, $x\rho(f)=0$ by definition. Now let $x\in U^0(n-1)$. Then we have
$
\rho(xf)(v)=(xf)(v)=f(\tau_{n-1}(x)v).
$
On the other hand, by the definition of the action of $U^{\leq0}(n-1)$ on $T(\mu)$,
we have $(x\rho(f))(v)=\rho(f)(\tau_{n-1}(x)v)=f(\tau_{n-1}(x)v)$, since $\tau_{n-1}(x)v\in M(\mu)^\mu$.

Now, we set $\phi:=\rho\circ\psi^{\tau_{n-1}}$ and $\phi':=\rho\circ(\psi')^{\tau_{n-1}}$.
Then $\phi$ is a $U^{\leq0}(n-1)$-homomorphism from $L(\lm)^{\tau_n}$ to $T(\mu)$
and $\phi'$ is a $U^{\leq0}(n-1)$-homomorphism from $\bigl(\Pi L(\lm)\bigr)^{\tau_n}$ to $T(\mu)$.
Note that $\phi$ and $\phi'$ are homogeneous of the same parity equal to that of $\theta$.

We claim that $\phi\ne0$ or $\phi'\ne0$. Indeed, we have $\psi\ne0$ or $\psi'\ne0$ since $\theta\ne0$. Suppose for definiteness that $\psi\ne0$, the argument for $\psi'$ being similar.
Then $\psi^{\tau_{n-1}}\ne0$.
The image of this homogeneous homomorphism of $U(n-1)$-supermodules is a $U(n-1)$-subsupermodule of $M(\mu)^{\tau_{n-1}}$.
But $M(\mu)^{\tau_{n-1}}$ has a unique minimal subsupermodule,
which is generated by $(M(\mu)^{\tau_{n-1}})^\mu$. Hence $(M(\mu)^{\tau_{n-1}})^\mu\subseteq \operatorname{im}\psi^{\tau_{n-1}}$. Take a nonzero $f\in (M(\mu)^{\tau_{n-1}})^\mu$.
We can write $f=\psi^{\tau_{n-1}}(g)$ for some $g\in L(\lm)^{\tau_n}$.
Since $f(M(\mu)^\mu)\ne0$, we have $\phi(g)=\rho(f)\ne0$.

Since $L(\lm)^{\tau_n}\cong L(\lm)$ and $(\Pi L(\lm))^{\tau_n}\cong\Pi L(\lm)$,
we get
$$
F(L(\lm)^{\tau_n})^\lm=F((\Pi L(\lm))^{\tau_n})^\lm=0.
$$
For the last equality, note that $FL(\lm)^\lm=0$ implies $F\bullet L(\lm)^\lm=0$ by considering homogeneous components.
Take an arbitrary linear function $f:L(\lm)\to\mathbb F$
such that $f(L(\lm)^\gamma)=0$ for any $\gamma<\lm$.
We can consider $f$ as an element of both $L(\lm)^{\tau_n}$ and $(\Pi L(\lm))^{\tau_n}$.
Note that $f$ has weight $\lm$ in either of these supermodules.
Arguing as in Theorem~\ref{theorem:main:1}, we get
\begin{equation}\label{equation:tensor:3}
0=\phi(Ff)=\phi(F_{i,n}af+\bar F_{i,n}bf),
\end{equation}
\begin{equation}\label{equation:tensor:4}
0=\phi'(Ff)=\phi'(F_{i,n}af+\bar F_{i,n}bf),
\end{equation}
where $a:=\cf_{i,n}(F)$ and $b:=\overline\cf_{i,n}(F)$.
Recall that $a$ is odd and $b$ is even. This and $a,b\in U^0(n)$ imply that $a$ and $b$ commute.

Now let $v\in M(\mu)^\mu$. By~(\ref{equation:tensor:3}) and~(\ref{equation:tensor:2.5}), we get
\begin{align*}
0=\phi (F_{i,n}af+\bar F_{i,n}bf)(v)=\psi^{\tau_{n-1}}(F_{i,n}af+\bar F_{i,n}bf)(v)
\\
=(F_{i,n}af+\bar F_{i,n}bf)(\psi(v))=f\bigl(\tau_n(F_{i,n}a)\psi(v)+\tau_n(\bar F_{i,n}b)\psi(v)\bigr)
\\
=f\bigl(\tau_n(a)E_{i,n}\psi(v)+\tau_n(b)\bar E_{i,n}\psi(v)\bigl)
=f\bigl(\tau_n(a)\bar E_{i,n}\psi'(v)+\tau_n(b)E_{i,n}\psi'(v)\bigl)
\\
=f\bigl((\tau_n(a)\bar E_{i,n})\bullet\psi'(v)+(\tau_n(b)E_{i,n})\bullet\psi'(v)\bigl)
\\
=f\bigl(\tau_n(\bar F_{i,n}a)\bullet\psi'(v)+\tau_n(F_{i,n}b)\bullet\psi'(v)\bigl)
=(\bar F_{i,n}af+ F_{i,n}bf)(\psi'(v))
\\
=(\psi')^{\tau_{n-1}}(\bar F_{i,n}af+ F_{i,n}bf)(v)=\phi'(\bar F_{i,n}af+ F_{i,n}bf)(v),
\end{align*}
where we have used the assumptions $\|a\|=\1$ and $\|b\|=\0$. Therefore we have
\begin{equation}\label{equation:tensor:5}
\phi'\bigl( F_{i,n}bf+\bar F_{i,n}af\bigr)=0.
\end{equation}
Similarly, by~(\ref{equation:tensor:4}) and~(\ref{equation:tensor:2.5}), we get
\begin{align*}
0=\phi'(F_{i,n}af+\bar F_{i,n}bf)(v)=(\psi')^{\tau_{n-1}}(F_{i,n}af+\bar F_{i,n}bf)(v)
\\
=(F_{i,n}af+\bar F_{i,n}bf)(\psi'(v))=f\bigl(\tau_n(F_{i,n}a)\bullet\psi'(v)+\tau_n(\bar F_{i,n}b)\bullet\psi'(v)\bigr)
\\
=\!f\bigl((\tau_n(a)E_{i,n})\bullet\psi'(v)+(\tau_n(b)\bar E_{i,n})\bullet\psi'(v)\bigr)
\\
=f\bigl(-\tau_n(a)E_{i,n}\psi'(v)-\tau_n(b)\bar E_{i,n}\psi'(v)\bigr)
\\
=f\bigl(-\tau_n(a)\bar E_{i,n}\psi(v)-\tau_n(b)E_{i,n}\psi(v)\bigr)
\\
=f\bigl(-\tau_n(\bar F_{i,n}a)\psi(v)-\tau_n(F_{i,n}b)\psi(v)\bigr)=(-\bar F_{i,n}af-F_{i,n}bf)(\psi(v))
\\
=\psi^{\tau_{n-1}}(-\bar F_{i,n}af-F_{i,n}bf)(v)=\phi(-\bar F_{i,n}af-F_{i,n}bf)(v).
\end{align*}
Hence, we have
\begin{equation}\label{equation:tensor:6}
\phi(F_{i,n}bf+ \bar F_{i,n}af)=0.
\end{equation}
We apply the substitution $f\mapsto af$ in~(\ref{equation:tensor:3}) and
the substitution $f\mapsto bf$ in~(\ref{equation:tensor:6}), subtract the latter
from the former, note that $a$ and $b$ commute, and use the assumption $\ev_\lm(a^2- b^2)\in\mathbb F^\times$ to get
$$
0=
\phi(F_{i,n}(a^2- b^2)f)=\ev_\lm(a^2- b^2)\phi(F_{i,n}f).
$$
Hence $\phi(F_{i,n}f)=0$. Now apply the substitution $f\mapsto bf$ in~(\ref{equation:tensor:3}) and
the substitution $f\mapsto af$ in~(\ref{equation:tensor:6}), subtract the latter from the former to get
$$
0=
\phi(\bar F_{i,n}b^2f-\bar F_{i,n}a^2f)
=-\ev_\lm(a^2- b^2)\phi(\bar F_{i,n}f).
$$
Hence $\phi(\bar F_{i,n}f)=0$ and, arguing as in Theorem~\ref{theorem:main:1}, we prove that $\phi=0$.
The equality $\phi'=0$ is proved similarly, using~(\ref{equation:tensor:4}) and~(\ref{equation:tensor:5}).
We get a contradiction.
\end{proof}

\begin{theorem}\label{theorem:tensor:2} Let $\lm,\mu\in X(n)$.
Then $L(\lm)\otimes V^*$ contains a nonzero primitive vector of weight $\mu$ if and only if $\mu =\lm-\epsilon_i$
for some tensor $\lm$-normal index $i$.
\end{theorem}
\begin{proof}
Suppose that $i\in\{1,\dots,n\}$ is not tensor $\lm$-normal. By Lemmas~\ref{LNew} and~\ref{lemma:tensor:5}, it suffices prove that there is no primitive vector of weight $\lm-\epsilon_i$ in $L(\lm)\otimes V^*$. Assume for a contradiction that such vector exists.
We have $i<n$, as $n$ is always tensor $\lm$-normal.

By Lemma~\ref{lemma:tensor:3}, there exists a nonzero $U(n-1)$-primitive
pair of weight $\lm-\alpha(i,n)$ in $\cp(L(\lm))_0\oplus\cp(L(\lm))_1$.
In particular, there exists a nonzero $U(n-1)$-primitive vector of weight $\lm-\alpha(i,n)$ in $L(\lm)$.
So by Theorem~\ref{theorem:main:2}, the index $i$ is
$\lm$-normal. Set $\beta:=\res_p\lm_i$. We need to establish the following

\vspace{1 mm}
{\em Claim.} There exists
a flow $\Gamma$ on $[i..n]$ coherent with $r_\beta(\lm)|_{[i..n]}$ having no buds on $[i..n]$
and one of the following happens:
{
\renewcommand{\labelenumi}{{\rm \theenumi}}
\renewcommand{\theenumi}{{\rm(\alph{enumi})}}
\begin{enumerate}
\item\label{theorem:tensor:1:case:a} $\beta\ne\mathbf 0$, $\res_p(\lm_n+1)=\beta$;
\item\label{theorem:tensor:1:case:b} $\beta=\mathbf 0$,   $\lm_n\=-1\pmod p$.
\end{enumerate}}

\vspace{1mm}
To prove the claim, we first suppose that $\beta\ne\mathbf 0$. Then $\bigl[\prod_{i<k<n}r_\beta(\lm)_k\bigl]=-^s$,
as $i$ is
$\lm$-normal. If $s>0$ then $\bigl[\prod_{i<k\le n}r_\beta(\lm)_k\bigl]$
does not contain $+$ regardless of the value of $\lm_n$.
In this case, $i$ is tensor $\lm$-normal contrary to the assumption.
Hence $s=0$. If $r_\beta(\lm)_n$ is distinct from $+$, then
$\bigl[\prod_{i<k\le n}r_\beta(\lm)_k\bigl]$ does not contain $+$ and we have a contradiction again.
Hence $r_\beta(\lm)_n=+$, whence $\res_p(\lm_n+1)=\beta$
by Proposition~\ref{proposition:rbeta}\ref{proposition:rbeta:part:2}.
Since
$$
\textstyle
\bigl[\prod_{i\le k\le n}r_\beta(\lm)_k\bigl]=\bigl[-\bigl[\prod_{i<k<n}r_\beta(\lm)_k\bigl]+\bigl]=[-+]=\emptyset,
$$
Lemma~\ref{lemma:pm:1} implies that there exists
a flow $\Gamma$ on $[i..n]$ (fully) coherent with $r_\beta(\lm)|_{[i..n]}$ and having no buds on $[i..n]$.

Now suppose that $\beta=\mathbf 0$. There are two cases: $r_\mathbf0(\lm)_i=--$ and $r_\mathbf0(\lm)_i=+-$. Suppose first that $r_\mathbf0(\lm)_i=--$. Then we have
$\bigl[\prod_{i<k<n}r_\mathbf 0(\lm)_k\bigl]=+^r-^s$ with $r\le1$, as $i$ is
$\lm$-normal.
If $s\ge2$ then $\bigl[\prod_{i<k\le n}r_\mathbf0(\lm)_k\bigl]$ contains at most one $+$
regardless of $\lm_n$. In this case, $i$ is tensor $\lm$-normal contrary to the assumption.
Hence $s\le1$. Corollary~\ref{corollary:pm:2} shows that only the following two cases are possible: $r=s=0$ and $r=s=1$.
In the former case, the reduction $\bigl[\prod_{i<k\le n}r_\mathbf0(\lm)_k\bigl]=r_\mathbf 0(\lm)_n$ must be $++$,
since $i$ is not tensor $\lm$-normal. Hence $\lm_n\=-1\pmod p$.
In the former case, the reduction
$$\textstyle\bigl[\prod_{i<k\le n}r_\mathbf0(\lm)_k\bigl]=[+-r_\mathbf 0(\lm)_n]$$
must contain at least two signs $+$, which is possible only if $r_\mathbf 0(\lm)_n=++$. Hence again $\lm_n\=-1\pmod p$.
In both cases, we have $\bigl[\prod_{i\le k\le n}r_\mathbf0(\lm)_k\bigl]=\emptyset$.
Now, by Lemma~\ref{lemma:pm:3}, there exists a flow $\Gamma$ on $[i..n]$ (fully) coherent with $r_\mathbf0(\lm)|_{[i..n]}$ and
having no buds on $[i..n]$.

If $r_\mathbf 0(\lm)_i=+-$ then $\bigl[\prod_{i<k<n}r_\mathbf0(\lm)_k\bigl]=-^s$,
as $i$ is
$\lm$-normal. If $s>0$ then also $s\ge2$ by Corollary~\ref{corollary:pm:2}. In that case,
$\bigl[\prod_{i<k\le n}r_\mathbf0(\lm)_k\bigl]$ does not contain $+$ regardless of $\lm_n$.
Hence $i$ is tensor $\lm$-normal contrary to the assumption. So
$\bigl[\prod_{i<k<n}r_\mathbf0(\lm)_k\bigl]=\emptyset$. Definition~\ref{definition:constr:4} implies that
$r_\mathbf 0(\lm)_n\ne+-$.
The cases $r_\mathbf 0(\lm)_n=\emptyset$ and $r_\mathbf 0(\lm)_n=--$ are also impossible,
since then $\bigl[\prod_{i<k\le n}r_\mathbf0(\lm)_k\bigl]$ does not contain the sign $+$,
which makes $i$ a tensor $\lm$-normal index. Thus we conclude that $r_\mathbf 0(\lm)_n=++$, i.e. $\lm_n\=-1\pmod p$.
By Lemma~\ref{lemma:pm:3}, there exists a flow $\Gamma'$ on $(i..n)$ (fully) coherent with $r_\mathbf0(\lm)|_{(i..n)}$
having no buds on $(i..n)$. It suffices to set $\Gamma:=\Gamma'\cup\{(i,n)\}$, to complete the proof of the claim.

\vspace{1mm}
Let $S$ be the set of all sources of edges of $\Gamma$ and set $M:=[i..n]\setminus S$. Note that $i\in S$,
as $r_\beta(\lm)_i$ contains $-$ and $\Gamma$ has no buds. On the other hand, $n\notin S$ as $\Gamma$ is a flow.
Thus we see that $M\subset(i..n]$ and $n\in M$. 
We define the injection $\psi:S\to[i..n]$ as follows.
If $s\in S$ then we set $\psi(s)$ equal to the unique index such that $(s,\psi(s))\in\Gamma$.
Now we apply Lemma~\ref{lemma:constr:2} to conclude that for any $v^+_\lm\in L(\lm)^\lm$,
we have $\bar S_{i,n}(M)\,v^+_\lm=0$.

By Lemma~\ref{lemma:rev:1.5}, we have
$$
\begin{array}{l}
\cf_{i,n}(\bar S_{i,n}(M))=-\(\bar H_i-\bar H_n\)\prod\nolimits_{t\in M\setminus\{n\}}C(i,t),
\\
\overline\cf_{i,n}(\bar S_{i,n}(M))=\(H_i-H_n\)\prod\nolimits_{t\in M\setminus\{n\}}C(i,t).
\end{array}
$$
We have
$$
\ev_\lm\left(\prod\nolimits_{t\in M\setminus\{n\}}C(i,t)\right)=\prod\nolimits_{t\in [i..n)\setminus S}(\beta-\res_p\lm_t).
$$
The right-hand side is a nonzero element of $\mathbb F$, since $\Gamma$ has no buds on $[i..n)$, see Proposition~\ref{proposition:rbeta}\ref{proposition:rbeta:part:1}. We denote this non-zero element by $c$.
Then we have
\begin{align*}
\ev_\lm\left(\cf_{i,n}(\bar S_{i,n}(M))^2-\overline\cf_{i,n}(\bar S_{i,n}(M))^2\right)
=c^2\ev_\lm\Bigl(\(\bar H_i-\bar H_n\)^2-\(H_i-H_n\)^2\Bigr)\\
= c^2(\lm_i+\lm_n-\lm_i^2+2\lm_i\lm_n-\lm_n^2)\cdot1_\mathbb F
\\
= c^2(-\res_p\lm_i-\res_p(\lm_n+1)+2\lm_n+2\lm_i\lm_n)\cdot1_\mathbb F
\\
= c^2(-2\beta+2\lm_n+2\lm_i\lm_n)\cdot1_\mathbb F
=-2 c^2(\beta-\lm_n(\lm_i+1))\cdot1_\mathbb F.
\end{align*}
Suppose temporarily that this coefficient equals zero. Then $\beta=\lm_n(\lm_i+1)\cdot1_\mathbb F$.
Since we also have $\beta=\lm_n(\lm_n+1)\cdot1_\mathbb F$, we obtain $\lm_n\lm_i=\lm_n^2\pmod p$.
Recall that in both cases~\ref{theorem:tensor:1:case:a} and~\ref{theorem:tensor:1:case:b},
we have $\lm_n\not\=0\pmod p$. Thus we have proved $\lm_i\=\lm_n\pmod p$. Now we have simultaneously
$\res_p\lm_n=\beta$ and $\res_p(\lm_n+1)=\beta$. As $p\ne2$, we get $\lm_n\=0\pmod p$, a contradiction.

Now we can apply Theorem~\ref{theorem:tensor:1} to conclude that there is no nonzero $U(n-1)$-primitive
pair in $\cp(L(\lm))$ of weight $\lm-\alpha(i,n)$, which is a contradiction.
\end{proof}

\section{Primitive vectors in $L(\la)\otimes V$}
In order to translate from $L(\la)\otimes V^*$ to $L(\la)\otimes V$, we recall the automorphism $\si$ of $U(n)$ from (\ref{ESigma}).

\begin{corollary} \label{CConormPrim}
Let $\la,\mu\in X(n)$. Then the $U(n)$-supermodule $L(\lm)\otimes V$ contains a nonzero primitive vector of weight $\mu$ if and only if $\mu =\lm+\epsilon_j$
for some tensor $\lm$-conormal index $j$.
\end{corollary}
\begin{proof}
By the universality of Verma modules, $L(\lm)\otimes V$ contains a nonzero primitive vector of weight $\mu$ if and only if
$$\Hom_{U(n)}(M(\mu),L(\la)\otimes V)\neq 0.$$
Twisting with the automorphism $\si$ and applying Lemma~\ref{LSigmaAppl}, this is equivalent~to
$$\Hom_{U(n)}(M(-w_0\mu),L(-w_0\la)\otimes V^*)\neq 0.$$
Again by the universality of Verma modules, the last statement is in turn equivalent to the existence of a non-zero $U(n)$-primitive vector of weight $-w_0\mu$ in $L(-w_0\la)\otimes V^*$. By Theorem~\ref{theorem:tensor:2}, this is equivalent to $-w_0\mu=-w_0\la-\eps_i$ for some tensor $-w_0\la$-normal index $i$. The last equality is equivalent to $\mu=\la+\eps_{w_0i}$. Now  apply Corollary~\ref{corollary:conormal}
\end{proof}

\begin{remark} 
{\rm
We point out that if we stay in the category of integrable finite dimensional $U(n)$-modules then it is possible to use usual dualities and avoid twisting with $\si$. Indeed, let $\la,\mu\in X^+_p(n)$. It suffices to note that
\begin{align*}
\Hom_{U(n)}(V(\mu),L(\la)\otimes V)&\cong \Hom_{U(n)}((L(\la)\otimes V))^\tau,V(\mu)^\tau)
\\
&\cong \Hom_{U(n)}(L(\la)\otimes V,H^0(\mu))
\\
&\cong \Hom_{U(n)}(L(\la),H^0(\mu)\otimes V^*)
\\
&\cong \Hom_{U(n)}((H^0(\mu)\otimes V^*)^*,L(\la)^*)
\\
&\cong \Hom_{U(n)}(H^0(\mu)^*\otimes V,L(-w_0\la))
\\
&\cong \Hom_{U(n)}(V(-w_0\mu)\otimes V,L(-w_0\la))
\\
&\cong \Hom_{U(n)}(V(-w_0\mu),L(-w_0\la)\otimes V^*).
\end{align*}
Similar arguments do not work for infinite dimensional modules, as for example $L(\la)^*$ is not even a highest weight module in general.
}
\end{remark}

\begin{theorem} 
Let $\la,\mu\in X(n)$. Then:
\begin{enumerate}
\item[{\rm (i)}] $\Hom_{U(n)}(L(\mu),L(\la)\otimes V^*)\neq 0$ if and only if $\mu=\la-\eps_i$ for some $\la$-good index $i$.
\item[{\rm (ii)}] $\Hom_{U(n)}(L(\mu),L(\la)\otimes V)\neq 0$ if and only if $\mu=\la+\eps_i$ for some $\la$-cogood index $i$.
\end{enumerate}
\end{theorem}
\begin{proof}
Note using contravariant duals that
\begin{align*}
\Hom_{U(n)}(L(\mu),L(\la)\otimes V)&\cong \Hom_{U(n)}(L(\mu)\otimes V^*,L(\la))
\\
&\cong \Hom_{U(n)}(L(\la),L(\mu)\otimes V^*).
\end{align*}
So, taking into account Corollary~\ref{CGoodCogood}, it suffices to prove (i).

Assume that $\Hom_{U(n)}(L(\mu),L(\la)\otimes V^*)\neq 0$. Since $L(\mu)$ is a quotient of $M(\mu)$, we also have
$$\Hom_{U(n)}(M(\mu),L(\la)\otimes V^*)\neq 0.$$
Hence $\mu=\la-\eps_i$ for some index $i$ which is tensor $\la$-normal by Theorem~\ref{theorem:tensor:2}. Moreover, since $L(\la)$ is a submodule of $M(\la)^\tau$, we also have
$$\Hom_{U(n)}(L(\mu),M(\la)^\tau\otimes V^*)\neq 0.$$
But
\begin{align*}
\Hom_{U(n)}(L(\mu),M(\la)^\tau\otimes V^*)&\cong \Hom_{U(n)}((M(\la)^\tau\otimes V^*)^\tau, L(\mu)^\tau)
\\
&\cong \Hom_{U(n)}(M(\la)\otimes V^*, L(\mu))
\\
&\cong \Hom_{U(n)}(M(\la), L(\mu)\otimes V).
\end{align*}
Hence by Corollary~\ref{CConormPrim}, the index $i$ is conormal for $\mu=\la-\eps_i$. We have proved that $\mu=\la-\eps_i$ for a tensor $\la$-good index $i$.

Conversely, suppose that $i$ is tensor normal for $\la$ and tensor conormal for $\la-\eps_i$. Since $i$ is tensor normal, Theorem~\ref{theorem:tensor:2} yields a non-zero homomorphism
$$
\theta\in\Hom(M(\la-\eps_i),L(\la)\otimes V^*).
$$
We claim that $\theta$ factors through the quotient $L(\la-\eps_i)$ of $M(\la-\eps_i)$. Indeed, otherwise, there is a composition factor $L(\nu)$ of $M(\la-\eps_i)$ with $\nu\neq \la-\eps_i$ such that
$$\Hom(L(\nu),L(\la)\otimes V^*)\neq 0.$$
By the forward direction in (i) which we have already proved, it follows that $\nu=\la-\eps_k$ for a tensor $\la$-good index $k$. Since $\nu\leq \la-\eps_i$ in the dominance order, it follows that $k<i$. Moreover, $L(\nu)=L(\la-\eps_i-(\eps_k-\eps_i))$ can only be a composition factor of $M(\la-\eps_i)$ if $\res_p\la_i=\res_p\la_k$, see the Linkage Principle of \cite[Section~8]{Kleshchev_Brundan_Modular_Representations_of_the_supergroup_Q(n)_I}. We have now got a contradiction with Lemma~\ref{LGoodTopNormal}.
\end{proof}

\chapter[Projective representations of symmetric groups]{Main results on projective representations of symmetric groups}\label{main_resultsSn}

\section{Representations of Sergeev supealgebras}\label{SMainYTN}
Here we freely use notions defined in the Introduction and Chapter~\ref{ChPrel} above. In particular, $\Se_n$ is the Sergeev superalgebra and $G=Q(n)$ is the algebraic supergroup of type $Q(n)$. We now review in more detail the theory developed in \cite{Brundan_Kleshchev_Projective_representations,Kleshchev_Brundan_Modular_Representations_of_the_supergroup_Q(n)_I}, which will allow us to apply the results on $Q(n)$ obtained above to get new results about $\Se_n$ and eventually projective representations of symmetric and alternating groups.

The category ${\mathcal Pol}(n)$ of finite dimensional polynomial supermodules over the supergroup $G$ is defined in \cite[Section~10]{Brundan_Kleshchev_Projective_representations}. This category splits as
$${\mathcal Pol}(n)=\bigoplus_{d\geq 0}{\mathcal Pol}(n,d),$$
where ${\mathcal Pol}(n,d)\subset {\mathcal Pol}(n)$ is the full subcategory of polynomial supermodules of degree $d$. The category ${\mathcal Pol}(n,d)$ is equivalent to the category $\smod{S(n,d)}$ of finite dimensional supermodules over the Schur superalgebra $S(n,d)$  introduced in \cite{Brundan_Kleshchev_Projective_representations} (denoted $Q(n,d)$ there).

Let
$$\Lambda^+_p(n)=\{\la=(\la_1,\dots,\la_n)\in X^+_p(n)\mid \la_1,\dots,\la_n\geq 0\},$$ and
$$
\Lambda^+_p(n,d)=\{\la=(\la_1,\dots,\la_n)\in \Lambda^+_p(n)\mid \la_1+\dots+\la_n=d\}.
$$
If $\la\in\Lambda^+_p(n)$, then $L(\la),V(\la)\in {\mathcal Pol}(n,d)$, and
$$
\{L(\la)\mid \la\in \Lambda^+_p(n,d)\}
$$
is a complete and irredundant set of irreducible supermodules in ${\mathcal Pol}(n,d)$ (equivalently, irreducible $S(n,d)$-supermodules) up to isomorphism.

Recall the $j$th level $L(\la)_j$ introduced in the previous section.

\begin{lemma} \label{LLevPolDeg}
Let $n\geq 2$, $\la\in \Lambda_p^+(n,d)$ for $d\leq n$, and $j\in\Z_{\geq 0}$. Then, as a $Q(n-1)$-module, $L(\la)_j\in {\mathcal Pol}(n-1,d-j)$.
\end{lemma}
\begin{proof}
Since $L(\la)$ is a polynomial $Q(n)$-module, we have that each $L(\la)_j$ is a polynomial $Q(n-1)$-module. Moreover, it follows from the assumption $d\leq n$ that if
$$\mu=(\mu_1,\dots,\mu_n)=\la-j\al_{n-1}-\sum_{i=1}^{n-2}m_i\al_i$$
for some $m_i$'s, then $\mu_n=j$, and hence $\mu_1+\dots+\mu_{n-1}=d-j$. So $L(\la)_j$ must be polynomial of degree $d-j$.
\end{proof}

Let $d\leq n$. Then the Schur superalgebra $S(n,d)$ has an idempotent $e_{n,d}$ (denoted $\xi_\omega$ in \cite{Brundan_Kleshchev_Projective_representations}) with the property $e_{n,d}S(n,d)e_{n,d}\cong \Se_d$. In fact, the isomorphism can be made explicit, see \cite[Theorem 6.2(ii)]{Brundan_Kleshchev_Projective_representations}. This allows us to define the ``Schur functor''
$$
{\mathcal F}_{n,d}: \smod{S(n,d)}\to \smod{\Se_d},\ V\mapsto e_{n,d}V.
$$
Moreover, consider the special weight
$$\omega_{n,d}:=\eps_1+\dots+\eps_d\in X(n).$$
Then, in fact, $e_{n,d}V$ is the $\omega_{n,d}$-weight space $V_{\omega_{n,d}}$ on which $\Se_d$ acts via the isomorphism of \cite[Theorem 6.2(ii)]{Brundan_Kleshchev_Projective_representations}.
The special case $d=n$
will be especially important for us. We use the notation $e_n:=e_{n,n}, {\mathcal F}_n:={\mathcal F}_{n,n}, \omega_n:=\omega_{n,n}$.

Let $\la\in\Lambda^+_p(n,n)$. By Lemma~\ref{LLevPolDeg}, the first level $L(\la)_1$ is in the category $\smod{S(n-1,n-1)}$, so we can apply the Schur functor ${\mathcal F}_{n-1}$ to it. The following proposition shows that the resulting $\Se_{n-1}$-supermodule is simply the restriction of ${\mathcal F}_n(L(\la))$ from $\Se_n$ to $\Se_{n-1}$. This is the main reason why the structure of the first level $L(\la)_1$ as a $U(n-1)$-supermodule is so important for us.

\begin{proposition} \label{PSchFunFirstLevel}
Let $n\geq 2$, $\la\in\Lambda^+_p(n,n)$. Then
$${\mathcal F}_{n-1}(L(\la)_1)={\mathcal F}_n(L(\la))_{\Se_{n-1}}.$$
\end{proposition}
\begin{proof}
Note that, as vector superspaces,
$${\mathcal F}_{n-1}(L(\la)_1)=(L(\la)_1)_{\omega_{n-1}}=L(\la)_{\omega_n}={\mathcal F}_n(L(\la))$$
and use the explicit formulas for the actions of $\Se_{n-1}$ and $\Se_n$ coming from the isomorphism of  \cite[Theorem 6.2(ii)]{Brundan_Kleshchev_Projective_representations}.
\end{proof}

While Proposition~\ref{PSchFunFirstLevel} connects the restriction $\operatorname{res}^{\Se_n}_{\Se_{n-1}}$ to the first level via Schur functors, the following result establishes a similar connection between the induction $\operatorname{ind}^{\Se_n}_{\Se_{n-1}}$ to tensoring with the natural module $V$, cf. \cite[Theorem 4.13]{BKLR}.

\begin{proposition} \label{PIndTens} 
Let $M\in {\mathcal Pol}(n,n-1)$. Then
$$
{\mathcal F}_n(M\otimes V)\cong \operatorname{ind}^{\Se_n}_{\Se_{n-1}}\big({\mathcal F}_{n,n-1}(M)\big).
$$
\end{proposition}
\begin{proof}
We have
\begin{equation}\label{EDecompTensOmega}
{\mathcal F}_n(M\otimes V)=(M\otimes V)_{\omega_n}=\bigoplus_{i=1}^n(M_{\omega_n-\eps_i}\otimes v_i)\,\oplus\, \bigoplus_{i=1}^n(M_{\omega_n-\eps_i}\otimes \bar v_i),
\end{equation}
where $\{v_1,\dots,v_n,\bar v_1,\dots,\bar v_n\}$ is the natural basis of $V$. So

Note that $\omega_n-\eps_n=\omega_{n,n-1}$ and so $M_{\omega_n-\eps_n}\otimes v_n$ is the submodule of the restriction $\operatorname{res}^{\Se_n}_{\Se_{n-1}}{\mathcal F}_n(M\otimes V)$ isomorphic to ${\mathcal F}_{n,n-1}(M)$. Now the Frobenius reciprocity allows us to extend this embedding of ${\mathcal F}_{n,n-1}(M)$ into ${\mathcal F}_n(M\otimes V)$  to a homomorphism $\phi:\operatorname{ind}^{\Se_n}_{\Se_{n-1}}\big({\mathcal F}_{n,n-1}(M)\big)\to {\mathcal F}_n(M\otimes V)$.

Moreover, the weights $\omega_n-\eps_i$ and $\omega_n-\eps_n$ are conjugate under the action of the symmetric group (the Weyl group of the even part $Q(n)_{ev}$), and so $\dim M_{\omega_n-\eps_i}=\dim M_{\omega_n-\eps_n}$. Hence, by (\ref{EDecompTensOmega}), we have
$$
\dim {\mathcal F}_n(M\otimes V)\!=2n\dim M_{\omega_n-\eps_n}\!=2n\dim {\mathcal F}_{n,n-1}(M)\!=\dim \operatorname{ind}^{\Se_n}_{\Se_{n-1}}\!\big({\mathcal F}_{n,n-1}(M)\big).
$$
So it suffices to prove that $\phi$ is surjective. Let $\eps\in\{\0,\1\}$, $(i,n)\in S_n$ be the transposition of $i$ and $n$, and $c_n$ be the $n$th generator of the Clifford algebra $\mathcal C_n$ as in the introduction. Then the explicit identification of $\Se_n$ with $e_nS(n,n)e_n$ obtained in \cite[Theorem 6.2(ii)]{Brundan_Kleshchev_Projective_representations} implies that
$$(i,n)c_n^\eps\big(M_{\omega_n-\eps_n}\otimes v_n\big)=M_{\omega_n-\eps_i}\otimes v_i^\eps.$$
This yields the surjectivity of $\phi$.
\end{proof}

Now, denote
$$
D(\la):={\mathcal F}_n(L(\la)),\quad S(\la):={\mathcal F}_n(V(\la))\qquad\qquad(\la\in\Lambda^+_p(n,n)).
$$
The set of weights $\Lambda^+_p(n,n)$ can and {\em will} be identified with the set ${\mathcal P}_p(n)$ of $p$-strict partitions of $d$. A $p$-strict weight $\la\in \Lambda^+_p(n,n)$ is {\em restricted} if
 $\la$ is a restricted $p$-strict partition, i.e. $\la\in  \mathcal{RP}_p(n)$. An irreducible $S(n,n)$-supermodule $L(\la)$ is {\em restricted} if $\la$ is restricted. Finally, an $S(n,n)$-supermodule is {\em restricted} if all its composition factors are restricted.

\begin{theorem} 
\label{TSchFApplication}
We have:
\begin{enumerate}
\item[{\rm (i)}] $\{\la\in {\mathcal P}_p(n)\mid D(\la)\neq 0\}={\mathcal{RP}}_p(n)$;
\item[{\rm (ii)}] $\{D(\la)\mid \la\in {\mathcal{RP}}_p(n)\}$
is a complete and irredundant set of irreducible $\Se_n$-supermodules up to isomorphism;
\item[{\rm (iii)}] if $\la\in {\mathcal{RP}}_p(n)$ then $D(\la)$ is the simple head of the Specht module $S(\la)$;
\item[{\rm (iv)}] if 
$V,W\in \smod{S(n,n)}$ are supermodules such that $V$ has restricted head and $W$ has restricted socle, then $$\Hom_{S(n,n)}(V,W)\cong\Hom_{\Se_n}(\mathcal{F}_n(V),\mathcal{F}_n(W)).$$
\end{enumerate}
\end{theorem}
\begin{proof}
(i) is \cite[Theorem 9.5]{Brundan_Kleshchev_Projective_representations}, and  (ii) is \cite[Theorem 10.2]{Brundan_Kleshchev_Projective_representations}. On the other hand,  (iv) is a standard property of Schur functors, see for example \cite[Lemma 2.17(ii)]{BKLR}. Finally, (iii) follows from (iv) and the fact that $L(\la)$ is the simple head of $V(\la)$.
\end{proof}

\begin{proposition}\label{PRestrSocle}
If $\la\in {\mathcal{RP}}_p(n),\mu\in {\mathcal{RP}}_p(n-1)$, then the socles of the $U(n-1)$-module $L(\la)_1$ and the $U(n)$-module $L(\mu)\otimes V$ are restricted.
\end{proposition}
\begin{proof}
Let $\nu\in{\mathcal{P}}_p(n-1)$ and $\kappa\in {\mathcal{P}}_p(n)$ be such that
$$\Hom_{U(n-1)}(L(\nu),L(\la)_1)\neq0\quad \text{and}\quad
\Hom_{U(n)}(L(\kappa),L(\mu)\otimes V)\neq0.$$
We need to prove that $\nu\in{\mathcal{RP}}_p(n-1)$ and $\kappa\in {\mathcal{RP}}_p(n)$. By Theorem~\ref{TFirstLevelMain}, $\nu$ is obtained from $\la$ by removing a good node, as defined in the Introduction. Now it is an easy combinatorial check to see that $\nu\in{\mathcal{RP}}_p(n-1)$. This also follows from \cite[Theorem 22.1.2]{Kbook}. Similarly we check that $\kappa$, which is obtained from $\mu$ by adding a tensor good node is restricted.
\end{proof}

\begin{remark}
{\rm
There is a general conceptual argument which shows that if $\la$ is restricted, then the socle of {\em any} level $L(\la)_j$ is restricted. For $GL(n-1)\subset GL(n)$ and some other natural embeddings this was first proved in \cite[Theorem B]{KlResI}. The proof given in \cite{KlResI} goes through for $Q(n-1)\subset Q(n)$ using theory developed in \cite[Section 9]{Kleshchev_Brundan_Modular_Representations_of_the_supergroup_Q(n)_I}. We will not need this here. Similarly, one can see that the socle of $L(\mu)\otimes V$ is restricted because it is a submodule of $V^{\otimes n}$, which has a restricted socle in view of  \cite[Theorem 6.2(i)]{Brundan_Kleshchev_Projective_representations}.
}
\end{remark}

\begin{corollary} \label{CTransl}
Let $\la\in {\mathcal{RP}}_p(n)$ and $\mu\in  {\mathcal{RP}}_p(n-1)$. Then
\begin{enumerate}
\item[{\rm (i)}] $\Hom_{\Se_{n-1}}(D(\mu),D(\la)_{\Se_{n-1}})\cong \Hom_{U(n-1)}(L(\mu),L(\la)_1)$.
\item[{\rm (ii)}] $\Hom_{\Se_{n-1}}(S(\mu),D(\la)_{\Se_{n-1}})\cong \Hom_{U(n-1)}(V(\mu),L(\la)_1)$.
\item[{\rm (iii)}] $\Hom_{\Se_{n}}(D(\la),\operatorname{ind}_{\Se_{n-1}}^{\Se_n}D(\mu))\cong \Hom_{U(n)}(L(\la),L(\mu)\otimes V)$.
\item[{\rm (iv)}] $\Hom_{\Se_{n}}(S(\la),\operatorname{ind}_{\Se_{n-1}}^{\Se_n}D(\mu))\cong \Hom_{U(n)}(V(\la),L(\mu)\otimes V)$.
\end{enumerate}
\end{corollary}
\begin{proof}
Apply Theorem~\ref{TSchFApplication} and Propositions~\ref{PRestrSocle},~\ref{PSchFunFirstLevel},~\ref{PIndTens}.
\end{proof}

We now claim that the definitions of good, cogood, normal, and conormal nodes of  $p$-strict partitions given in the Introduction match those of good, cogood, normal, and conormal indices used in the main body of this paper.

To be more precise, recall first that the set of all possible contents is $I:=\{0,1,\dots,\ell\}\subset \Z$, where $\ell=(p-1)/2$. As the content of the node $A=(r,s)$ depends only on the column $s$, we can speak of the content of $s\in\Z_{>0}$, denoted $\Res_p s$, so that $\Res_p(r,s)=\Res_p s$. Recall also that for any $s\in\Z$ we have by definition that  $\res_p s=s(s-1)\pmod{p}$.
It is pointed out in \cite[(8.8)]{Kleshchev_Brundan_Modular_Representations_of_the_supergroup_Q(n)_I} that for any $s\in \Z$ there exists a unique $i\in I$ such that
$$s(s-1)\equiv i^2+i\pmod{p},$$
i.e. for any $r,s\in\Z$ we have that $\res_pr=\res_ps$ if and only if $\Res_pr=\Res_ps$. For a residue $\beta$ of the form $\beta=s(s-1)\pmod{p}$ for $s\in \Z$, we denote by $i(\beta)$ the unique $i\in I$ with $\beta\equiv i^2+i\pmod{p}$. Thus
$$\beta\equiv i(\beta)^2+i(\beta)\pmod{p}.$$ Conversely, for $i\in I$ we denote by $\be(i)$ the residue
$$\be(i):=i^2+i\pmod{p}.$$

Now, let $\la\in {\mathcal{RP}}_p(n)$ be a $p$-strict partition of $n$ considered also as a weight in $X^+_p(n)$. The case $n=1$ is boring, so let us assume that $n>1$. Then we always have $\la_n=0$. Moreover, comparing the definitions, we see that $A$ is a $\beta$-removable node for $\la$ in the sense of Section~\ref{SNormGood} if and only if either $A$ is an $i(\beta)$-removable node for $\la$ in the sense of the Introduction or $A=(n,0)$. Similarly, $A$ is a $\beta$-addable node for $\la$ in the sense of Section~\ref{SNormGood} if and only if $A$ is an $i(\beta)$-addable node for $\la$ in the sense of the Introduction.
So, the reduced $i$-signature of $\la$ in the sense of the Introduction is obtained from the the reduced $\beta(i)$-signature $\si$ of $\la$ in the sense of Section~\ref{SNormGood} by removing one $-_n$ from the end of $\si$.
It follows from Lemma~\ref{LNormRemovAdd} and Corollary~\ref{corollary:main:2} that an $i$-removable node $A=(r,s)$ is $i$-normal for $\la$ in the sense of the Introduction if and only if $\res_ps=\beta(i)$ and $r$ is a $\la$-normal index in the sense of Section~\ref{SNormGood}. The similar statement holds for good nodes. Also, an $i$-addable node $A=(r,s)$ is $i$-normal (resp. $i$-cogood) for $\la$ in the sense of the Introduction if and only if $\res_ps=\beta(i)$ and $r$ is a tensor $\la$-conormal (resp. $\la$-cogood) index in the sense of Section~\ref{SSTensConCog}.

Now, Corollary~\ref{CTransl} and Theorem~\ref{TFirstLevelMain} imply Theorem A from the Introduction.

To deduce Theorem $B$, recall the Jucys-Murphy elements
$$
L_1,\dots,L_n\in \Se_n
$$
from \cite[(13.22)]{Kbook}. The eigenvalues of the elements $L_k^2$ on finite dimensional $\Se_n$-supermodules are all of the form $\beta(i)=i^2+i$ (considered as an element of $\mathbb{F}$) for $i\in I$. It is known that the Jucys-Murphy elements commute and so we can decompose an arbitrary finite dimensional $\Se_n$-supermodule $V$ into the corresponding simultaneous generalized eigenspaces:
$$
V=\bigoplus_{\bi\in I^n}V_\bi,
$$
where for $\bi=(i_1,\dots,i_n)\in I^n$, we define:
$$
V_{\bi}=\{v\in V\mid (L_k^2-\beta(i_k))^N=0\ \text{for $N\gg0$},\ k=1,\dots,n\}.
$$

Denote by $\Gamma_n$ the set of all tuples $\ga=(\ga_i)_{i\in I}$ with $\ga_i\in\Z_{\geq 0}$ and $\sum_{i\in I}\ga_i=n$.
For any $\ga=(\ga_i)_{i\in I}\in\Gamma_n$, denote by $I^\ga$ the subset of $I^n$, which consists of all $\bi=(i_1,\dots,i_n)\in I^n$ such that for each $i\in I$ there are exactly $\ga_i$ entries among $i_1,\dots,i_n$ which are equal to $i$. Note that
$$
I^n=\bigsqcup_{\ga\in \Gamma_n}I^\ga
$$
is just the decomposition of $I^n$ into the orbits of the natural action of $S_n$ on $I^n$ by place permutations.

Now, given a finite dimensional $\Se_n$-supermodule $V$ and $\ga\in\Gamma_n$, let
$$
V[\ga]:=\bigoplus_{\bi\in I^\ga}V_\bi.
$$
 It is known that the symmetric polynomials in $L_1^2,\dots,L_n^2$ are central in $\Se_n$, see \cite[Remark 15.4.7, Theorem 14.3.1]{Kbook}. It follows that $V[\ga]$ is a subsupermodule of $V$, and we have the decomposition of $\Se_n$-supermodules
 $$
 V=\bigoplus_{\ga\in\Gamma_n}V[\ga].
 $$
It is known that the (non-zero) $V[\ga]$ are precisely the superblock components of $V$, but we will not need this fact.

Given a partition $\la\in {\mathcal{P}}_p(n)$, define $\ga(\la)$ to be the tuple $(\ga_i)_{i\in I}\in\Gamma_n$, where $\ga_i$ is the number of $i$-nodes of $\la$, cf. \cite[(22.1)]{Kbook}.
The following result follows immediately from the definition of the supermodules $G(\la)$:

\begin{lemma} \label{LGLaBlock}
Let $\la\in {\mathcal{RP}}_p(n)$ and $\ga=\ga(\la)$. Then $G(\la)[\gamma]=G(\la)$.
\end{lemma}

With the goal of identifying $D(\la)$ and $G(\la)$ in mind, we want to prove a similar result for $D(\la)$. For this we need the following:

\begin{lemma} \label{LCharZero}
Theorem $B$ holds in charecteristic zero.
\end{lemma}
\begin{proof}
We note that both approaches to representation theory of Sergeev superalgebras $\Se_n$, described in the Introduction, work the case where $p=\operatorname{char}\mathbb{F}=0$. In this case we interpret $I$ as $\{0,1,2,\dots\}$, and $ {\mathcal{RP}}_0(n)= {\mathcal{P}}_0(n)$ as strict partitions of $n$, i.e. partitions with distinct parts. The Schur functor approach in this case has been developed originally by Sergeev \cite{Sergeev}. It leads to the $\Se_n$-supermodules $D(\la)$ parametrized by $\la\in  {\mathcal{P}}_0(n)$. On the other hand, \cite[Theorem 7.2]{Nazarov} (see also \cite[Section 2.6]{SergeevSpecht}) describes the spectrum of the squares of Jucys-Murphy elements on $D(\la)$. This description implies that $D(\la)[\gamma(\la)]=D(\la)$, and so $D(\la)=G(\la)$, since in the case $p=0$ all columns have distinct contents, and therefore $\ga(\la)$ determines $\la$ uniquely.
\end{proof}

\begin{lemma} \label{LDBlock}
If $\la\in {\mathcal{RP}}_p(n)$ and $\ga=\ga(\la)$, then $D(\la)[\ga]=D(\la)$.
\end{lemma}
\begin{proof}
In this proof it will be important to distinguish between characteristic $0$ and characteristic $p$, so we will use the corresponding indices in our notation, for example $I_0$ vs. $I_p$, $\Gamma^p_n$, vs. $\Gamma^0_n$, $\gamma_0(\la)$ vs. $\ga_p(\la)$, etc.

For $\mu\in {\mathcal{P}}_0(n)$, denote by $D_0(\mu)$ the irreducible module corresponding to $\mu$ over a field of characteristic zero, and denote by $\bar D_0(\mu)$ its reduction modulo $p$, see \cite[p.65]{Brundan_Kleshchev_Projective_representations}. As $D_0(\mu)=G_0(\mu)$ by Lemma~\ref{LCharZero}, it follows that $D_0(\mu)[\ga_0(\mu)]=D_0(\mu)$, where $\ga_0(\mu)\in\Gamma_n^0$.
Let $\bar\ga_0(\mu)$ denote the element of $\Gamma^p_n$ obtained from $\gamma_0(\mu)$ as follows. If $\ga_0(\mu)=(\ga_i')_{i\in I_0}$, then $\bar\ga_0(\mu)=(\ga_i)_{i\in I_p}$, where
$$\ga_i:=\sum_{\{j\in I_0\mid i^2+i\equiv j^2+j\pmod p\}}\ga_j' \qquad(i\in I_p).$$

When we reduce a representation modulo $p$, the eigenvalues of the squares of Jucys-Murphy elements will get reduced modulo $p$, and so $\bar D_0(\mu)=\bar D_0(\mu)[\bar\ga_0(\mu)]$. Hence $D[\bar\ga_0(\mu)]=D$ for any composition factor $D$ of $\bar D_0(\mu)$.
Now, $D(\la)$ must appear as a composition factor of some reduction $\bar D(\mu)$, in which case, by \cite[Thorem 10.8]{Brundan_Kleshchev_Projective_representations} and \cite[Theorems  4.3, 6.3]{BrundanQII}, we have $\ga_p(\la)=\bar\ga_0(\mu)$. Therefore $D(\la)[\ga_p(\la)]=D(\la)$, as required.
\end{proof}

\begin{corollary} \label{CContD}
Let $\la,\mu\in {\mathcal{RP}}_p(n)$. Then $D(\la)\cong G(\mu)$ implies $\ga(\la)=\ga(\mu)$.
\end{corollary}
\begin{proof}
Combine Lemmas ~\ref{LGLaBlock} and \ref{LDBlock}.
\end{proof}

Now, we complete the proof of Theorem B. We apply induction on $n$. The base of induction is clear, as $\Se_1$ has only one irreducible supermodule up to isomorphism. For the inductive step, assume that $n>1$ and Theorem B is true for $n-1$. Let $\la\in {\mathcal{RP}}_p(n)$, and $\mu$ be obtained from $\la$ be removing an $i$-good node of $\la$. By the inductive assumption we have $D(\mu)=G(\mu)$. We need to prove that $G(\la)=D(\la)$. Let $G(\la)=D(\nu)$ for some  $\nu\in {\mathcal{RP}}_p(n)$. Using the notation of \cite[(17.6) and Theorem 22.2.2]{Kbook}, we see that $G(\mu)$ is in the socle of $\operatorname{res}_iG(\la)$.
So $D(\mu)$ is in the socle of $\res_iD(\nu)$, which by Theorem A implies that $\mu$ is obtained from $\nu$ by removing a good node $A$. By Corolary~\ref{CContD}, $\ga(\nu)=\ga(\la)$, and so $A$ must have content $i$. Thus $\mu$ is obtained from $\la$ by removing an $i$-good node, and $\mu$ is also obtained from $\nu$ by removing an $i$-good node. This implies that $\la=\nu$, either by an easy combinatorial exercise or by \cite[Corollary 17.2.3]{Kbook}. The proof of Theorem B is complete.

\section{Projective representations of symmetric groups}\label{SFinal}

Finally, we use the categorical superequivalence from \cite[Section 13.2]{Kbook} to translate from $\Se_n$ to $\T_n$. Recall from the introduction that $\Se_n\cong\T_n\otimes \mathcal{C}_n$. The Clifford superalgebra $\mathcal{C}_n$ is simple as a superalgebra, so it has only one irreducible supermodule denoted by $U_n$, see \cite[Example 12.2.14]{Kbook}. The supermodule $U_n$ is of type $\Mtype$ if $n$ is even and of type $\Qtype$ if $n$ is odd.

We have functors
$$
?\boxtimes U_n=\mathfrak{F}_n:\smod{\T_n}\to\smod{\Se_n}$$
and  
$$
\Hom_{\mathcal{C}_n}(U_n,?)=\mathfrak{G}_n:\smod{\Se_n}\to\smod{\T_n}.
$$
Main properties of these functors are described in \cite[Proposition 13.2.2]{Kbook}. In particular, $\mathfrak{F}_n$ and $\mathfrak{G}_n$ are exact, left and right adjoint to each other, and behave nicely with respect to restriction and induction.

If $n$ is even, then $\mathfrak{F}_n$ and $\mathfrak{G_n}$ are quasi-inverse equivalences of categories which induce a type-preserving bijection between the set of isomorphism classes of $\T_n$-supermodules and the set of isomorphism classes of $\Se_n$-supermodules. In this case we define
$$
S^\la:=\mathfrak{G}_n(S(\la)),\quad D^\mu:=\mathfrak{G}_n(D(\mu))\qquad(\la\in\mathcal{P}_p(n),\ \mu\in\mathcal{RP}_p(n)).
$$

If $n$ is odd, we have in the notation of \cite{Kbook} that $\mathfrak{G}_n\circ\mathfrak{F}_n\simeq \operatorname{Id}\oplus\operatorname{\Pi}$ and $\mathfrak{F}_n\circ\mathfrak{G}_n\simeq \operatorname{Id}\oplus\operatorname{\Pi}$, where $\Pi$ is the parity change functor. In this case, if $D(\la)$ is of type $\Mtype$, i.e. $h_{p'}(\la)$ is even, then $D^\la:=\mathfrak{G}_n(D(\la))$ is an irreducible $\T_n$-supermodule of type $\Qtype$. If $D(\la)$ is of type $\Qtype$, i.e. $h_{p'}(\la)$ is odd, then $\mathfrak{G}_n(D(\la))\simeq D^\la\oplus \Pi D^\la$ for a unique irreducible $\T_n$-supermodule $D^\la$ of type~$\Mtype$.


\begin{theorem} 
{\rm \cite[Theorem 22.3.1]{Kbook}} 
For any $n$ we have that
$$\{D^\la\mid \la\in \mathcal{RP}_p(n)\}$$ is a complete and irredundant set of irreducible $\T_n$-supermodules up to isomorphism. Moreover, $D^\la$ is of type $\Mtype$ if and only if $n-h_{p'}(\la)$ is even.
\end{theorem}

Let again $n$ be odd. If $\la\in\mathcal{P}_p(n)$ and $h_{p'}(\la)$ is even, we define $S^\la:=\mathfrak{G}_n(S(\la))$. However, if $h_{p'}(\la)$ is odd, we have to be more careful.

Let $V$ be a finite dimensional supermodule over a superalgebra $A$. A map $J:V\to V$  is called a {\em $\Qtype$-map} if $J$ is an odd $A$-endomorphism of $V$ such that $J^2=\id_V$. The Schur's Lemma for superalgebras (see e.g. \cite[Lemma 12.2.3]{Kbook}) states that for a finite dimensional irreducible $A$-supermodule $V$, its endomorphism algebra $\operatorname{End}(V)$ is spanned by $\id_V$ if $V$ is of type $\Mtype$, and it is spanned by $\id_V$ and a $\Qtype$-map $J$ on $V$ if $V$ is of type $\Qtype$ (to get to this statement of Schur's lemma, one needs to replace $J$ from \cite[Lemma 12.2.3]{Kbook} by $\sqrt{-1}J$). We are interested in $\Qtype$-maps on some supermodules which are not necessarily irreducible.

\begin{lemma} \label{LSJ}
Let $n\in\Z_{>0}$.
\begin{enumerate}
\item[{\rm (i)}] If $\la\in X_p^+(n)$ and $h_{p'}(\la)$ is odd then the $U(n)$-supermodule $V(\la)$ has a $\Qtype$-map.
\item[{\rm (ii)}] If $\la\in \mathcal{P}_p(n)$ and $h_{p'}(\la)$ is odd then the $\Se_n$-supermodule $S(\la)$ has a $\Qtype$-map.
\end{enumerate}
\end{lemma}
\begin{proof}
(i) Under our assumptions, $\mathfrak{u}(\la)$ is an irreducible $H$-supermodule of type $\Qtype$, see Proposition~\ref{PUMu}(v), so it has a $\Qtype$-map by Schur's Lemma. This map induces a $\Qtype$-map on the induced module $H^0(\la)$, and now (i) follows by passing to duals.

(ii) Let $J$ be a $\Qtype$-map on $V(\la)$, which exists by (i). By definition $S(\la)=\mathcal F(V(\la))=V(\la)_{\omega_n}$, and $\mathcal{F}(J)$ is just the restriction of $J$ to the weight space $V(\la)_{\omega_n}$, so it is also a $\Qtype$-map.
\end{proof}

\begin{lemma} \label{LNasty}
Let $n$ be odd, and $M\in\smod{\Se_n}$ possess a $\Qtype$-map  $J_M$. Then there exists a unique up to isomorphism $N\in\smod{\T_n}$ such that $\mathfrak{G}_n(M)\simeq N\oplus \Pi N$.
\end{lemma}
\begin{proof}
Note that $U_n$ is an irreducible type $\Qtype$ module over $\mathcal{C}_n$, so it possesses a $\Qtype$-map  $J_U$ by Schur's Lemma. Let $\delta_M:M\to M$ be a linear map which maps an arbitrary homogeneous $m\in M$ to $(-1)^{\| m\|}m$.

Observe that the linear operator
$$\Hom_{\mathcal{C}_n}(U_n,M)\to \Hom_{\mathcal{C}_n}(U_n,M),\ f\mapsto \de_MJ_MfJ_U$$
squares to $-\id$. So we can decompose $\mathfrak{G}_n(M)=\Hom_{\mathcal{C}_n}(U_n,M)$ as $N_+\oplus N_-$, where
$$
N_\pm:=\{f\in \Hom_{\mathcal{C}_n}(U_n,M)\mid f=\pm\sqrt{-1} \delta_M J _MfJ_U\}.
$$
One easily checks that this is decomposition respects the structure of $\mathfrak{G}_n(M)$ as a $\T_n$-supermodule. It remains to observe that $N_-\simeq\Pi N_+$, the isomorphism given by $f\mapsto J_M f$. The uniqueness of $N$ (up to a not necessarily even isomorphism) follows from the Krull-Schmidt Theorem.
\end{proof}

It now follows from Lemmas~\ref{LSJ} and~\ref{LNasty}, that in the case where $n$ is odd and $\la\in\mathcal{P}_p(n)$ has odd $p'$-height $h_{p'}(\la)$, we have $\mathfrak{G}_n(S(\la))\cong S^\la\oplus\Pi S^\la$ for some $\T_n$-module $S^\la$ defined uniquely up to isomorphism. We have now defined the modules $S^\la$ for $\la\in\mathcal{P}_p(n)$ in all cases. We refer to these modules as {\em Specht modules} for $\T_n$.

\begin{proposition}
Let $\la\in\mathcal{RP}_p(n)$. Then the Specht module $S^\la$ has simple head $D^\la$.
\end{proposition}
\begin{proof}
An easy computation involving the properties of the functors $\mathfrak{G}_n$ and $\mathfrak{F}_n$ established in \cite[Proposition 13.2.2]{Kbook} shows that $\dim\Hom_{\T_n}(S^\la,D^\mu)=\dim\Hom_{\T_n}(D^\la,D^\mu)$ for any $\mu$, using the similar fact for the supermodules $S(\la)$ and $D(\mu)$ over $\Se_n$.
\end{proof}

Another application of \cite[Proposition 13.2.2]{Kbook} now yields our main result for $\T_n$-supermodules:

\begin{theorem} 
Let $\la\in\mathcal{RP}_p(n)$ and $\mu\in \mathcal{RP}_p(n-1)$. Then
\begin{enumerate}
\item[{\rm (i)}] $\mu$ is obtained from $\la$ by removing a good node if and only if
$$
\Hom_{\T_{n-1}}(D^\mu,\operatorname{res}^{\T_n}_{\T_{n-1}}D^\la)\neq 0.
$$
\item[{\rm (ii)}] $\mu$ is obtained from $\la$ by removing a normal node if and only if
$$
\Hom_{\T_{n-1}}(S^\mu,\operatorname{res}^{\T_n}_{\T_{n-1}}D^\la)\neq 0.
$$
In particular, if $\mu$ is obtained from $\la$ by removing a normal node then $D^\mu$ is a composition factor of the restriction $\operatorname{res}^{\T_n}_{\T_{n-1}}D^\la$.
\item[{\rm (iii)}] $\la$ is obtained from $\mu$ by adding a cogood node if and only if
$$
\Hom_{\T_{n}}(D^\la,\operatorname{ind}^{\T_n}_{\T_{n-1}}D^\mu)\neq 0.
$$
\item[{\rm (iv)}] $\la$ is obtained from $\mu$ by adding a conormal node if and only if
$$
\Hom_{\T_{n}}(S^\la,\operatorname{ind}^{\T_n}_{\T_{n-1}}D^\mu)\neq 0.
$$
In particular, if $\la$ is obtained from $\mu$ by adding a conormal node then $D^\la$ is a composition factor of the rinduction $\operatorname{ind}^{\T_n}_{\T_{n-1}}D^\mu$.
\end{enumerate}
\end{theorem}

\appendix

\backmatter

\begin{thebibliography}{ABC}

\bibitem
{ACat}
S. Ariki, On the decomposition numbers of the Hecke algebra of $G(m,1,n)$, {\em J. Math. Kyoto Univ.} {\bf 36} (1996), 789--808.

\bibitem
{ABr}
S. Ariki, Proof of the modular branching rule for cyclotomic Hecke algebras, {\em J. Algebra} {\bf 306} (2006), 290--300.

\bibitem
{ArS}
H. Arisha and M. Schaps,
Maximal strings in the crystal graph of spin representations of the symmetric and
alternating groups, {\em Comm. Algebra} {\bf 37} (2009), 3779--3795.

\bibitem
{BarK}
A. Baranov and A. Kleshchev, Maximal ideals in modular group algebras of the finitary symmetric and alternating groups, {\em Trans. Amer. Math. Soc.} {\bf 351} (1999),  595--617.

\bibitem
{BKZ}
A. Baranov, A. Kleshchev and A.E. Zalesskii, Asymptotic results on modular representations of symmetric groups and almost simple modular group algebras, {\em J. Algebra} {\bf 219} (1999), 506--530.


\bibitem
{BrBr}
J. Brundan, Modular branching rules and the Mullineux map for Hecke algebras of type $A$,  {\em Proc. London Math. Soc. (3)} {\bf 77} (1998), 551--581.

\bibitem
{Brundan_operators}
J. Brundan, Lowering operators for {\rm GL}(n) and quantum {\rm GL}(n), {\em
  Proc. Simposia in Pure Math.}, {\bf 63} (1998), 95--114.

\bibitem
{BrundanQII}
J. Brundan,  Modular representations of the supergroup $Q(n)$. II, {\em Pacific J. Math.} {\bf 224} (2006), 65--90.

\bibitem
{BKLR}
J. Brundan and A. Kleshchev,   Modular Littlewood-Richardson coefficients, {\em Math. Z.} {\bf 232} (1999), 287--320.

\bibitem
{BKtr}
J. Brundan and A. Kleshchev,   On translation functors for general linear and symmetric groups, {\em Proc. London Math. Soc. (3)} {\bf 80} (2000), 75--106.

\bibitem
{BKIrrRes}
J. Brundan and A. Kleshchev, Representations of the symmetric group which are irreducible over subgroups, {\em J. Reine Angew. Math.} {\bf 530} (2001), 145--190.

\bibitem
{Brundan_Kleshchev_Hecke-Clifford_superalgebras}
J. Brundan and A. Kleshchev, Hecke-Clifford superalgebras, crystals of type
  $A^{(2)}_{2\ell}$ and modular branching rules for ${\widehat S}_n$, {\em
  Represent. Theory}, {\bf 5} (2001), 317--403.

\bibitem
{Brundan_Kleshchev_Projective_representations}
J. Brundan and A. Kleshchev, Projective representations of symmetric via Sergeev
  duality, {\em Math. Z.}, {\bf 239} (2002), 27--68.

\bibitem
{BKDurham}
J. Brundan and A. Kleshchev, Representation theory of symmetric groups and their double covers, pp. 31--53 in {\em Groups, Combinatorics \& Geometry (Durham, 2001)}, World Sci. Publ., River Edge, NJ, 2003.

\bibitem
{BKCartan}
J. Brundan, A. Kleshchev, Cartan determinants and Shapovalov forms, {\em Math. Ann.} {\bf 324} (2002), 431--449.

\bibitem
{Kleshchev_Brundan_Modular_Representations_of_the_supergroup_Q(n)_I}
J. Brundan and A. Kleshchev, Modular representations of the supergroup $Q(n)$, I,
  {\em J. of Algebra}, {\bf 260} (2003), 64--98.


\bibitem
{BKReg}
J. Brundan and A. Kleshchev,  James' regularization theorem for double covers of symmetric groups, {\em J. Algebra} {\bf 306} (2006), 128--137.

\bibitem
{BKllt}
J. Brundan and A. Kleshchev, Graded decomposition numbers for cyclotomic Hecke algebras, {\em Adv. Math.} {\bf 222} (2009), 1883--1942.

\bibitem
{Carter}
R. W. Carter, Raising and lowering operators for $sl_n$, with applications to orthogonal bases of $sl_n$-modules, {\em Proceedings of the Arcata conference on representations of finite groups}, vol. 47, 1987, pp. 351Р366.

\bibitem
{Carter_Lusztig}
R.W. Carter and G.~Lusztig, On the modular representations of the general
  linear and symmetric groups, {\em Math. Z.}, {\bf 136} (1974), 193--242.

\bibitem
{Gr}
I. Grojnowski, Affine $\mathfrak{sl}_p$ controls the representation theory of the symmetric group and related Hecke algebras, {\tt arXiv:math.RT/9907129}.


\bibitem
{Jantzen_Representations_of_Algebraic_Groups}
Jantzen J.C., Representations of algebraic groups, Second Edition (Mathematical
  Surveys and Monographs 107), {\em Amer. Math. Soc.}, 2003.


\bibitem
{Kang}
S.-J. Kang, Crystal bases for quantum affine algebras and combinatorics of Young walls,  {\em Proc. London Math. Soc. (3)} {\bf 86} (2003), 29--69.

 \bibitem
{KlResI}
 A. Kleshchev, On restrictions of irreducible modular
representations of semisimple algebraic groups and symmetric groups to some natural subgroups, I, {\em Proc. London Math. Soc. (3)} {\bf 69} (1994), 515--540.

  \bibitem
{KBrI}
A. Kleshchev, Branching rules for modular representations of symmetric groups. I,
{\em J. Algebra} {\bf 178} (1995), 493--511.

  \bibitem
{KBrII}
A. Kleshchev, Branching rules for modular representations of symmetric
  groups. {II}, {\em J. reine angew. Math.} {\bf 459} (1995), 163--212.

  \bibitem
{KBrIII}
A. Kleshchev, Branching rules for modular representations of symmetric groups. III. Some corollaries and a problem of Mullineux,
{\em J. London Math. Soc.} (2) {\bf 54} (1996), 25--38.

\bibitem
{KDec}
A. Kleshchev, On decomposition numbers and branching coefficients for symmetric and special linear groups, {\em Proc. London Math. Soc. (3)} {\bf 75} (1997), 497--558.

\bibitem
{KBrIV}
A. Kleshchev, Branching rules for modular representations of symmetric groups. IV, {\em J. Algebra} {\bf 201} (1998), 547--572.

\bibitem
{Kbook}
A. Kleshchev, {\em
Linear and Projective Representations of Symmetric Groups}, Cambridge University Press, Cambridge, 2005.

\bibitem
{KT}
A. Kleshchev and P.-H. Tiep, On restrictions of modular spin representations of symmetric and alternating groups, {\em Trans. Amer. Math. Soc.} {\bf 356} (2004),  1971--1999.

\bibitem
{Kujawa}
J. Kujawa, Crystal structures arising from representations of ${\rm GL}(m\vert n)$, {\em Represent. Theory} {\bf 10} (2006), 49--85.


\bibitem
{LT}
B. Leclerc and J.-Y. Thibon, $q$-Deformed Fock spaces and modular
representations of spin symmetric groups, {\em J. Phys. A} {\bf 30} (1997), 6163--6176.

\bibitem
{Nazarov}
M. Nazarov, YoungХs symmetrizers for projective representations of the symmetric group,
{\em Adv. Math.} {\bf 127} (1997), 190--257.

\bibitem
{Ph}
A.M. Phillips, Restricting modular spin representations of symmetric and alternating groups to Young-type subgroups, {\em Proc. London Math. Soc. (3)} {\bf 89} (2004), 623--654.

\bibitem
{RT}
A. Ram and P. Tingley, Universal Verma modules and the Misra-Miwa Fock space, {\tt arXiv:1002.0558}.

\bibitem
{Sergeev_The_center_of_enveloping_algebra_for_Lie_superalgebra_Q}
A.N. Sergeev, The center of enveloping algebra for Lie superalgebra
  $Q(n,{\mathbb C})$, {\em Lett. Math. Phys.}, {\bf 7} (1983), 177--179.

\bibitem
{Sergeev}
A.N. Sergeev,  Tensor algebra of the identity representation as a module over the Lie superalgebras ${\rm Gl}(n,\,m)$ and $Q(n)$, {\em (Russian) Mat. Sb. (N.S.)} {\bf 123}(165) (1984), 422--430.

\bibitem
{SergeevSpecht}
A.N. Sergeev, The Howe duality and the projective representations of symmetric groups. {\em Represent. Theory} {\bf 3} (1999), 416--434 (electronic).

\bibitem
{ShchItLow}
V. Shchigolev, Iterating lowering operators, {\em J. Pure Appl. Algebra} {\bf 206} (2006),  111--122.

\bibitem
{ShchGenLow}
V. Shchigolev, Generalization of modular lowering operators for ${\rm GL}_n$, {\em Comm. Algebra} {\bf 36} (2008), 1250--1288.

\bibitem
{Shchigolev_Rectangular_low_level_case}
V. Shchigolev, Rectangular low level case of modular branching problem for ${\rm GL}_n(K)$,
{\em J. Algebra}, {\bf 321} (2009), no. 1, 28--85.

\bibitem
{ShchWeyl}
V. Shchigolev,  Weyl submodules in restrictions of simple modules, {\em J. Algebra} {\bf 321}  (2009), 1453--1462.


\end{thebibliography}

\bibliographystyle{amsalpha}

\printindex

\end{document}
